\documentclass[11pt, leqno]{article}
\usepackage{amssymb, amsthm}
\usepackage{graphicx}
\begin{document}
\newtheorem{tm}{Theorem}
\newtheorem{la}{Lemma}
\newtheorem{cy}{Corollary}
\newtheorem{pn}{Proposition}
\newcommand{\bx}{{\bf x}}
\newcommand{\by}{{\bf y}}
\newcommand{\bw}{{\bf w}}
\newcommand{\ekp}{\enskip}
\newcommand{\wtd}{\widetilde}
\newcommand{\noi}{\noindent}
\thispagestyle{empty}
\pagestyle{myheadings}
\markright{RANDOMIZED ORTHOGONAL ARRAYS}

\begin{center}
{\large {\bf A MULTIVARIATE CENTRAL LIMIT THEOREM FOR}} \\
{\large {\bf RANDOMIZED
ORTHOGONAL ARRAY SAMPLING}} \\
{\large {\bf DESIGNS IN COMPUTER EXPERIMENTS}}\footnote{
{\em AMS} 2000 {\em subject classifications.} Primary 62E20;
secondary 60F05, 65C05.
} \\
\vspace{0.3cm}
{\sc By Wei-Liem Loh} \\
\vspace{0.3cm}
{\em National University of Singapore} \\
\end{center}
\begin{quote}
Let $f:[0,1)^d \rightarrow {\mathbb R}$ be an integrable function. An objective of many computer 
experiments is to estimate $\int_{[0,1)^d} f(x) dx$ by evaluating $f$ at a finite number of points 
in $[0,1)^d$. There is a design issue in the choice of these points and a popular choice is via the use
of randomized orthogonal arrays.
This article proves a multivariate central limit theorem for a class of randomized orthogonal array sampling designs
[Owen (1992a)] as well as for a class of OA-based Latin hypercubes [Tang (1993)].
\end{quote}

\section{Introduction}

Let $X$ be a random vector uniformly distributed on the $d$-dimensional unit hypercube
$[0,1)^d$ and $f$ be an integrable function from $[0,1)^d$ to ${\mathbb R}$. An objective of
many computer experiments [see, for example, McKay, Conover and Beckman (1979), Stein (1987),
Sacks, Welch, Mitchell and Wynn (1989) and Santner, Williams and Notz (2003)] is to estimate
\begin{equation}
\mu = E f(X) = \int_{[0,1)^d} f(x) dx,
\label{eq:1.1}
\end{equation}
using a finite number of function evaluations. It is well known 
that as the dimension $d$
increases, Monte Carlo methods and (deterministic) equidistribution methods become competitive
and ultimately dominant. Indeed Davis and Rabinowitz (1984), Chapter 5.10, consider
$d>15$ to be a high enough dimensionality that sampling or equidistribution methods are 
indicated. 

For definiteness, let $n, d, q$ and $t$ be positive integers such that $t\leq d$. An orthogonal array
of strength $t$ is a matrix of $n$ rows and $d$ columns with elements taken from the set of symbols
$\{0,1,\ldots, q-1\}$ such that in any $n\times t$ submatrix, each of the $q^t$ possible rows occurs
the same number of times. The class of all such arrays is denoted by OA$(n, d, q, t)$. Comprehensive accounts
of orthogonal arrays can be found in the books by Raghavarao (1971) and Hedayat, Sloane and Stufken (1999).

Owen (1992a), (1994) and Tang (1993) independently proposed the use of randomized orthogonal arrays in 
computer experiment sampling designs. The main attraction of these designs is that they, in contrast to simple random sampling,
stratify on all $t$-variate margins simultaneously.
A class of randomized orthogonal array sampling designs proposed by Owen (1992a) is as follows.
Let 
\begin{enumerate}
\item[(a)] $A\in {\rm OA}(q^t, d, q, t)$ where $a_{i,j}$ denotes the $(i,j)$th element of $A$,
\item[(b)] $\pi_1,\ldots, \pi_d$ be random permutations of $\{0,\ldots, q-1\}$, each
uniformly distributed on all the $q!$ possible permutations,
\item[(c)] $\{ U_{i,j}: i=1,\ldots, q^t, j =1,\ldots, d \}$, be $[0, 1)$ uniform
random variables,
\item[(d)] and all the $U_{i,j}$'s and $\pi_k$'s are independent.
\end{enumerate}
We randomize the symbols of $A$ by applying the permutation $\pi_j$ to the $j$th column of $A$, $j=1,\ldots, d$.
This gives us another orthogonal array $A^*$ such that its $(i,j)$th element
 satisfies $a^*_{i,j} = \pi_j( a_{i,j})$. 
An orthogonal array based sample of size $q^t$ (taken from $[0,1)^d$) is defined to be
$\{ X_1,\ldots, X_{q^t}\}$ where for $i=1,\ldots, q^t$,
$X_i=(X_{i,1},\ldots, X_{i,d})'$,
\begin{equation}
X_{i,j} = \frac{a^*_{i,j} + U_{i,j} }{q}, \hspace{0.5cm}\forall j=1,\ldots, d.
\label{eq:1.4}
\end{equation}

For $t \geq 2$, Tang (1993) observed that the above sampling designs may not stratify well on $s$-variate margins if $s<t$.
He suggested modified designs that 
stratify on $t$-variate margins as well as $1$-variate margins simultaneously. 
He called these designs OA-based Latin hypercubes.
Finally, Owen (1997a), (1997b), in a series of articles, proposed the use of scrambled nets.
Given $t \in {\mathbb Z}^+$, the  scrambled nets
stratify on $s$-variate margins whenever $t/s$ is a positive integer.

A class of OA-based Latin hypercubes 
can be constructed as follows.
Let $A\in {\rm OA}(q^t, d, q,$ $t)$. 
As before, we randomize its symbols to obtain the orthogonal array $A^*$.
Then for each column of $A^*$, we replace the $q^{t-1}$ positions with entry $k$ by a random permutation
(with each such permutation having an equal probability of being chosen)
of $\{k q^{t-1}, kq^{t-1} +1, \ldots, (k+1) q^{t-1} -1\}$,  for all $k= 0,\ldots, q-1$.
After the replacement is done for all $d$ columns of $A^*$, the newly obtained matrix, say $A^{**}$, 
satisfies $A^{**} \in {\rm OA}(q^t, d, q^t, 1)$.

One version of OA-based Latin hypercubes that was considered by Owen (1997a), page 1906, is of the form
$\{ Y_1,\ldots, Y_{q^t}\}$ where for $i=1,\ldots, q^t$,
$Y_i = (Y_{i,1},\ldots, Y_{i,d})'$,
\begin{equation}
Y_{i,j} = \frac{ a^{**}_{i,j} + U_{i,j} }{ q^t}, \hspace{0.5cm}\forall j=1,\ldots, d,
\label{eq:1.2}
\end{equation}
$\{U_{i,j}: i=1,\ldots, q^t, j=1,\ldots, d\}$ are $U [0,1)$ random variables independent of one another and all other permutations,
and $a^{**}_{i,j}$ denotes the $(i,j)$th element of $A^{**}$.
The class of OA-based Latin hypercubes proposed by Tang (1993) requires one more level of randomization where the columns of  $A^{**}$ are randomized. We denote the
resulting matrix by $A^{***}$.
Tang's OA-based Latin hypercubes can be expressed as
$\{ Y_1^*,\ldots, Y_{q^t}^*\}$ where for $i=1,\ldots, q^t$,
$Y_i^* = (Y_{i,1}^*,\ldots, Y_{i,d}^*)'$,
\begin{equation}
Y_{i,j}^* = \frac{ a^{***}_{i,j} + U_{i,j} }{ q^t}, \hspace{0.5cm}\forall j=1,\ldots, d,
\label{eq:1.21}
\end{equation}
$\{U_{i,j}: i=1,\ldots, q^t, j=1,\ldots, d\}$ are, as before, $U [0,1)$ random variables independent of one another and all other permutations,
and $a^{***}_{i,j}$ denotes the $(i,j)$th element of $A^{***}$.
We note that $\{ Y_1,\ldots, Y_{q^t} \}$ and $\{ Y_1^*,\ldots, Y_{q^t}^* \}$
are Latin hypercube samples [see, for example, McKay, Conover and Beckman (1979) and Owen (1992b)].

The estimators for $\mu$ in (\ref{eq:1.1}) that we are concerned with are
\begin{equation}
\mbox{$\hat{\mu}_{oas} = q^{-t} \sum_{i=1}^{q^t} f( X_i)$,
$\hat{\mu}_{oal} = q^{-t} \sum_{i=1}^{q^t} f( Y_i)$,
and
$\hat{\mu}_{oal}^* = q^{-t} \sum_{i=1}^{q^t} f( Y_i^*)$,}
\label{eq:1.3}
\end{equation}
where the $X_i$'s, $Y_i$'s and $Y_i^*$'s are as in (\ref{eq:1.4}), (\ref{eq:1.2}) and (\ref{eq:1.21}) respectively. It is easily seen that
$\hat{\mu}_{oas}$, $\hat{\mu}_{oal}$ and $\hat{\mu}_{oal}^*$ are all unbiased estimators for
$\mu$. For simplicity, we write $\sigma_{oas}^2 = {\rm Var} (\hat{\mu}_{oas})$, $\sigma_{oal}^2 = {\rm Var}(\hat{\mu}_{oal})$ and $\sigma_{oal}^{*2} = {\rm Var} (\hat{\mu}_{oal}^*)$.

In this article, we shall assume that $t=2$. This significantly simplifies the notation as well
as the theoretical
arguments that follow. Also as Owen (1992a) and Tang (1993) noted, orthogonal arrays of strength $t=2$ lead to the most economical  
sample size $q^2$. This is important in practice especially when $q$ is large.
The following theorem is due to Owen (1992a) and Tang (1993).

\begin{tm} \label{tm:3.1}
Let $d\geq 3$, $f$ be a bounded continuous function on $[0,1)^d$ and $\hat{\mu}_{oas}$, $\hat{\mu}_{oal}^*$ be as in (\ref{eq:1.3}) with $A\in$ OA$(q^2, d, q, 2)$. Then
as $q\rightarrow \infty$, we have
\begin{displaymath}
q^2 \sigma^2_{oas} =  \int_{[0,1)^d} f_{rem}^2 (x) dx + o(1)
\hspace{0.5cm}\mbox{and} \hspace{0.5cm}
q^2 \sigma^{*2}_{oal} =  \int_{[0,1)^d} f_{rem}^2 (x) dx + o(1),
\end{displaymath}
where for all $x = (x_1,\ldots, x_d)' \in [0, 1)^d$, $1\leq j\leq d$, $1\leq k<l\leq d$,
\begin{eqnarray}
f_j (x_j) &=& \int_{[0,1)^{d-1}} [ f(x) - \mu ] \prod_{1\leq k\leq d: k\neq j} dx_k,
\nonumber \\
f_{k,l} (x_k, x_l) &=& \int_{[0,1)^{d-2}} [ f(x) - \mu - f_k(x_k) - f_l(x_l) ]
\prod_{1\leq j\leq d: j\neq k, l} dx_j,
\nonumber \\
f_{rem} (x) &=&
f(x) - \mu -\sum_{j=1}^d f_j (x_j) -\sum_{1\leq k< l\leq d} f_{k,l} (x_k, x_l).
\label{eq:2.39}
\end{eqnarray}
\end{tm}
Theorem \ref{tm:3.1} implies that (i) the asymptotic variances of $\hat{\mu}_{oas}$ and $\hat{\mu}_{oal}^*$ are always less than or equal to the asymptotic variance of an analogous estimator based
on a simple random sample of the same size, (ii) they are dramatically smaller 
if the integrand $f$ can be approximated by a sum of bivariate functions, and (iii) $\sigma_{oas}^2 \sim \sigma^{*2}_{oal}$ if $\int_{[0,1)^d} f_{rem}^2 (x) dx >0$.
Tang (1993), page 1395, further showed that $\sigma_{oal}^{*2} \leq \sigma_{oas}^2$ if $f$ is additive.

The aim of this article is to study the asymptotic distributions of $\hat{\mu}_{oas}$, $\hat{\mu}_{oal}$ and $\hat{\mu}_{oal}^*$.
For instance, such a result will be useful in the construction of confidence intervals
for $\mu$.

{\sc Definition.} A function $f: [0,1)^d\rightarrow {\mathbb R}$ is smooth with a Lipschitz continuous mixed partial of order $d$ if there exist 
finite constants $B\geq 0$ and $\beta \in (0,1]$ such that
\begin{displaymath}
\sup_{j_1,\ldots,j_d\in \{0,\ldots, d\} : j_1+\cdots + j_d = d} | \frac{\partial^d }{\partial x_1^{j_1} \ldots \partial x_d^{j_d}} f(x) - \frac{\partial^d}{ \partial x_1^{j_1}\ldots \partial x_d^{j_d}} f(y) |
\leq B \|x-y\|^\beta,
\end{displaymath}
$\forall x,y\in [0,1)^d$
where $\|.\|$ is the usual Euclidean norm.
We shall now state the main result of this article, the proof of which is deferred to the Appendix.
\begin{tm} \label{tm:5.2}
Suppose $d\geq 3$ and $f:[0,1)^d\rightarrow {\mathbb R}$ is smooth with a Lipschitz continuous mixed partial of
order $d$ such that $\int_{{\mathbb R}^d} f_{rem}^2 (x) dx > 0$. Define $W_{oas} = (\hat{\mu}_{oas}-\mu)/\sigma_{oas}$, $W_{oal}= (\hat{\mu}_{oal}-\mu)/\sigma_{oal}$
and
$W^*_{oal} = (\hat{\mu}_{oal}^* -\mu)/\sigma_{oal}^*$  with $A\in$ OA$(q^2, d, q, 2)$.
Then $W_{oas}$, $W_{oal}$ and $W_{oal}^*$ each converges in law to the standard normal distribution as $q\rightarrow \infty$.
\end{tm}

The remainder of this article proceeds as follows.
In Section 2 we shall first establish base $q$ expansions for $\{X_1,\ldots, X_{q^2}\}$ and
$\{Y_1,\ldots, Y_{q^2}\}$. The main point here is that the difference between these two base $q$ expansions is of order $O(1/q)$.
Following Owen (1997a),  a $d$-dimensional base $q$ Haar multiresolution analysis is applied to $f$ and 
 an ANOVA decomposition of $f$ is obtained.
This ANOVA decomposition facilitates much of the theoretical analysis that ensue.

In Section 3, a proxy statistic $W$ for $W_{oas}$ and $W_{oal}$ is introduced.
Proposition \ref{pn:4.1} shows that to prove the asymptotic normality of $W_{oas}$ and $W_{oal}$ as $q\rightarrow \infty$,
it suffices to prove that $W$ is asymptotically normal.
Stein (1972) proposed a powerful and general method for obtaining a bound for the error in the normal approximation to the distribution 
of a sum of dependent random variables. Since then, Stein's method has found considerable applications in combinatorics, probability and statistics [see Stein (1986)].
We shall use the multivariate normal version of Stein's method as given in G\"{o}tze (1991) and Bolthausen and G\"{o}tze (1993). In particular,
Theorem \ref{tm:5.1} establishes a multivariate central limit theorem for the ``components'' of $W$ under the conditions of Theorem \ref{tm:5.2}. 
This result is needed in the proof of the latter theorem.
Finally, the Appendix contains the proof of Theorem \ref{tm:5.2} as well as some more technical results that used in this article.

We would like to add that Loh (1996) has established the asymptotic normality of $\hat{\mu}_{oas}$ when $d =3$ and $t=2$ under moment conditions on $f$. 
However the approach in Loh (1996), which uses directly the univariate version of Stein's method, does
not seem to be extendable to $d\geq 4$. For example, the inequality (11)
in Loh (1996) is valid for $d=3$ but not for $d\geq 4$. 

We conclude the Introduction with a note on notation. In this article,
the indicator function is denoted by ${\cal I}\{.\}$ and
if $x$ is a vector, then $x'$ is its transpose.
$\|.\|$ denotes the Euclidean norm in ${\mathbb R}^p$ where $p$ is either $d-2$ or $d$ (depending on the context).

\section{ANOVA decomposition}

We shall first establish base $q$ expansions for randomized orthogonal array samples
as well as for OA-based Latin hypercubes.
Let $A\in {\rm OA}(q^2, d, q, 2)$, $a_{i,j}$ be the $(i,j)$th element of $A$ and
\begin{equation}
\{\pi_j, \pi_{j;b}, \pi_{i,j,k}: i=1,\ldots, q^2, j=1,\ldots, d, b=0,\ldots, q-1, k=2,3,\ldots \}
\label{eq:2.21}
\end{equation}
be a set of mutually independent random permutations of $\{0,1,\ldots, q-1\}$, where
each of these permutations is uniformly distributed over its $q!$ possible values.
We observe that the randomized orthogonal array sample $X_1,\ldots, X_{q^2}$ in (\ref{eq:1.4}) can be expressed as
$X_i = (X_{i,1}, \ldots, X_{i,d})'$ where
\begin{equation}
X_{i,j} = \sum_{k=1}^\infty x_{i,j,k} q^{-k},\hspace{0.5cm} \forall
i=1,\ldots, q^2, j=1,\ldots d,
\label{eq:2.3}
\end{equation}
$x_{i,j,1} = \pi_j (a_{i,j} )$ and 
$x_{i,j,k} = \pi_{i,j,k} (0)$ for all $k \geq 2$.
Let $A^{**} \in {\rm OA}(q^2, d, q^2, 1)$ be as in Section 1 with $t=2$.
Since $0\leq a^{**}_{i,j}/q^2 <1$, we observe that
$a^{**}_{i,j}/q^2 = \sum_{k=1}^\infty b_{i,j,k} q^{-k}$
for suitable integers $0\leq b_{i,j,1}, b_{i,j,2} \leq q-1$ and $b_{i,j,k}=0$ for all $k\geq 3$.
Owen (1997a), page 1907, observed that
an OA-based Latin hypercube defined as in (\ref{eq:1.2}) has the form  $\{ Y_i= (Y_{i,1},\ldots, Y_{i,d})': i=1,\ldots, q^2\}$,
where
\begin{equation}
Y_{i,j} = \sum_{k=1}^\infty y_{i,j,k} q^{-k}
\label{eq:2.2}
\end{equation}
and for $1\leq i\leq q^2$, $1\leq j\leq d$,
$y_{i,j,1} = \pi_j (a_{i,j} )$,
$y_{i,j,2} = \pi_{j; a_{i,j} } (b_{i,j,2})$,
$y_{i,j,k} = \pi_{i,j,k} (0)$ for all $k \geq 3$.
We observe from (\ref{eq:2.3}) and (\ref{eq:2.2}) that $\sup_{1\leq i\leq q^2, 1\leq j\leq d} |X_{i,j} - Y_{i,j}| \leq (q-1)/q^2$.

Let $f:[0,1)^d \rightarrow \mathbb R$ be a square integrable function.
Inspired by Owen (1997a), 
we apply a $d$-dimensional base $q$ Haar multiresolution analysis
to $f$. More precisely, 
for any integer $k\geq 0$, let ${\cal Y}_k$ denote the linear span of the functions
$\{\psi_{k,t,c}: t= 0, 1,\ldots,$ and $c=0,\ldots, q -1\}$ where
\begin{displaymath}
\psi_{k,t,c}(x) 
= q^{(k+1)/2} {\cal I}\{ x \in [\frac{ qt+c}{q^{k+1}}, \frac{ qt + c+1}{q^{k+1}}) \}
- q^{(k-1)/2} {\cal I}\{ x \in [\frac{ t}{q^k}, \frac{ t+1}{q^k} )\},
\end{displaymath}
$\forall x\in [0,1)$.
We observe that
the functions in ${\cal Y}_k$ are constant on
$[t q^{-k-1}, (t+1) q^{-k-1})$ and integrate to zero over 
$[t q^{-k}, (t+1) q^{-k})$.
Next let ${\cal U}_0$ denote the space of functions that are constant on $[0,1)$
and
\begin{displaymath}
{\cal U}_k = \{g + g_0+\cdots + g_{k-1}: g\in {\cal U}_0, g_j\in {\cal Y}_j, j=0,\ldots,k-1\},
\hspace{0.5cm} \forall k= 1, 2,\ldots.
\end{displaymath}
Then it is well known that
$\bigcup_{k=0}^\infty {\cal U}_k$ is dense in $L^2 ([0,1))$ and
$\bigcap_{k=0}^\infty {\cal U}_k = {\cal U}_0$.
We further observe from Owen (1997a), page 1897,
that a typical basis function for $L^2([0,1)^d)$ is of the form
$\prod_{r=1}^l \psi_{k_{j_r}, t_{j_r},c_{j_r}} (x_{j_r})$
for all $(x_1,\ldots,x_d)'\in [0,1)^d$,
where $1\leq j_1<\cdots<j_l \leq d$, 
and $k_{j_r} \geq 0$, $0\leq t_{j_r} \leq q^{k_{j_r}}-1$, $0\leq c_{j_r} \leq q-1$ whenever $1\leq r\leq l$.
Here by convention, an empty product (that is $l=0$) is taken to be $1$.
Hence for each $f \in L^2([0,1)^d)$, it follows from (6.6) of Owen (1997a), page 1898, that
\begin{eqnarray}
f(x) 
&=& \mu + \sum_{l=1}^d \sum_{1\leq j_1<\cdots <j_l\leq d} 
( \sum_{k_{j_1}= 0}^\infty
\sum_{t_{j_1}=0}^{q^{k_{j_1}}-1} 
\sum_{c_{j_1}=0}^{q-1} )
\cdots
( \sum_{k_{j_l}= 0}^\infty
\sum_{t_{j_l}=0}^{q^{k_{j_l}}-1} 
\sum_{c_{j_l}=0}^{q-1} )
\nonumber \\
&&\hspace{0.5cm}\times
\langle f,
\prod_{r=1}^l \psi_{k_{j_r}, t_{j_r},c_{j_r}} \rangle
\prod_{r=1}^l \psi_{k_{j_r}, t_{j_r},c_{j_r}} (x_{j_r}),
\hspace{0.5cm}\mbox{a.e.\ $x\in [0,1)^d$},
\label{eq:2.1}
\end{eqnarray}
where $\mu$ is as in (\ref{eq:1.1}) and
\begin{equation}
\langle f,
\prod_{r=1}^l \psi_{k_{j_r}, t_{j_r},c_{j_r}} \rangle
=
\int_{[0,1)^d} f(x)
[\prod_{r=1}^l \psi_{k_{j_r}, t_{j_r},c_{j_r}} (x_{j_r})] dx.
\label{eq:2.4}
\end{equation}
 Without loss of generality, we can assume that equality in (\ref{eq:2.1}) holds for all $x\in [0,1)^d$ since
changing the value of $f$ on a set of Lebesgue measure zero will not alter the value of $\mu$.
For simplicity let
\begin{displaymath}
\{U[\tilde{c}_{j,1},\ldots, \tilde{c}_{j,u_j}: 1\leq j\leq d]:
0\leq \tilde{c}_{j,1}, \ldots, \tilde{c}_{j,u_j} \leq q-1, u_j \geq 0, 1\leq j\leq d\}
\end{displaymath}
be a set of mutually independent random vectors where each
$U[\tilde{c}_{j,1},\ldots, \tilde{c}_{j,u_j}: 1\leq j\leq d]$
has the uniform distribution
on the $d$-dimensional interval
$\prod_{j=1}^d [\sum_{k=1}^{u_j} \tilde{c}_{j,k} q^{-k}, 
q^{-u_j} + \sum_{k=1}^{u_j} \tilde{c}_{j,k} q^{-k})$.
Here
$\sum_{k=1}^{u_j} \tilde{c}_{j,k} q^{-k} = 0$ 
if $u_j=0$.
Furthermore we assume that the above $U$'s are independent of $\pi$'s [defined as in (\ref{eq:2.21})]. 
For nonnegative integers $u_1^*, u_1, \ldots, u_d^*, u_d$,
we write 
\begin{itemize}
\item[(i)] $(u_1^*,\ldots, u_d^*)\preceq (u_1,\ldots, u_d)$ if and
only if $u_j^*\leq u_j$ for all $j=1,\ldots,d$,
\item[(ii)]
$(u_1^*,\ldots, u_d^*)\prec (u_1,\ldots, u_d)$ if and
only if $u_j^*\leq u_j$ for all $j=1,\ldots,d$ with at least one strict inequality.
\end{itemize}
The following construction establishes an ANOVA decomposition of $E f\circ U = E f(U)$
where $E$ denotes expectation. 
For integers $u_j \geq 0$, $0\leq \tilde{c}_{j,k} \leq q-1$,
$1\leq j\leq d, k\geq 1$, define recursively 
\begin{eqnarray}
&& \nu_{u_1,\ldots,u_d} [\tilde{c}_{j,1},\ldots,\tilde{c}_{j,u_j}: 1\leq j\leq d] 
\nonumber \\
&=& 
E f\circ U[\tilde{c}_{j,1},\ldots,\tilde{c}_{j,u_j}: 1\leq j\leq d]
\nonumber \\
&&\hspace{0.5cm} 
- \sum_{u_1^*,\ldots, u_d^*: (0,\ldots,0)\preceq (u_1^*,\ldots, u_d^*)\prec (u_1,\ldots, u_d)} 
\nu_{u^*_1,\ldots, u_d^*}
[\tilde{c}_{j,1},\ldots,\tilde{c}_{j,u_j^*}: 1\leq j\leq d],
\label{eq:3.761}
\end{eqnarray}
and hence
\begin{eqnarray*}
&& \lim_{u_1,\ldots, u_d\rightarrow\infty}
E f\circ U[\tilde{c}_{j,1},\ldots, \tilde{c}_{j,u_j}:1\leq j\leq d]
\nonumber \\
&=& \sum_{u_1,\ldots, u_d\geq 0 } 
\nu_{u_1,\ldots, u_d}[\tilde{c}_{j,1},\ldots,\tilde{c}_{j,u_j} : 1\leq j\leq d].
\end{eqnarray*}
Writing $|u| = \sum_{j=1}^d {\cal I} \{u_j\geq 1\}, u=(u_1,\ldots, u_d)',$ such that $1\leq j_1<\cdots <j_{|u|} \leq d$ and
$u_j \geq 1$ if and only if $j\in \{j_1,\ldots,j_{|u|} \}$,
it follows from (\ref{eq:2.1}) that $\nu_{u_1,\ldots,u_d}$ can be written down explicitly as
$\nu_{0,\ldots, 0} [\hspace{0.3cm} ] 
= \mu$ if $|u|=0$
and
\begin{eqnarray}
&& \nu_{u_1,\ldots, u_d} [\tilde{c}_{j,1},\ldots,\tilde{c}_{j,u_j}: 1\leq j\leq d] 
\nonumber \\
&=& ( \sum_{t_{j_1}=0}^{q^{u_{j_1}-1}-1} \sum_{c_{j_1}=0}^{q-1} )
\cdots ( \sum_{t_{j_{|u|} }=0}^{q^{u_{j_{|u|} }-1}-1} \sum_{c_{j_{|u|} }=0}^{q-1} )
\langle f,  \prod_{l=1}^{|u|}  \psi_{u_{j_l} -1,t_{j_l},c_{j_l}} \rangle  
\nonumber \\
&&\times
E\{ \prod_{l=1}^{|u|}  \psi_{u_{j_l}-1,t_{j_l},c_{j_l}} \circ 
U_{j_l} [\tilde{c}_{j,1},\ldots, \tilde{c}_{j,u_j}
:1\leq j\leq d]\}
\nonumber \\
&=& ( \sum_{t_{j_1}=0}^{q^{u_{j_1}-1}-1} \sum_{c_{j_1}=0}^{q-1} )
\cdots ( \sum_{t_{j_{|u|} }=0}^{q^{u_{j_{|u|} }-1}-1} \sum_{c_{j_{|u|} }=0}^{q-1} )
\langle f,  \prod_{l=1}^{|u|}  \psi_{u_{j_l} -1,t_{j_l},c_{j_l}} \rangle  
\nonumber \\
&&\times
\prod_{l=1}^{|u|}  \psi_{u_{j_l}-1,t_{j_l},c_{j_l}} 
( \sum_{k=1}^{u_{j_l} } \tilde{c}_{j_l,k} q^{-k} ),
\label{eq:3.85}
\end{eqnarray}
if $|u| \geq 1$. Here $U_{j_l}$ denotes the $j_l$th co-ordinate of $U$ and the last equality uses the fact that
$\psi_{u_{j_l}-1,t_{j_l},c_{j_l}}$ is constant on 
$[t q^{-u_{j_l}}, (t+1) q^{-u_{j_l}})$ for an arbitrary but fixed integer $t$.
An important consequence of the ANOVA decomposition (that will be applied repeatedly in the sequel) is
if $u_k \geq 1$ for some $1\leq k\leq d$, then
\begin{equation}
\sum_{\tilde{c}_{k, u_k}=0}^{q-1}
\nu_{u_1,\ldots,u_d} [\tilde{c}_{j,1},\ldots,\tilde{c}_{j,u_j}:1\leq j\leq d] 
= 0.
\label{eq:3.76}
\end{equation}
Writing $W_{oal} = (\hat{\mu}_{oal} -\mu)/\sigma_{oal}$, we observe from (\ref{eq:1.3}), (\ref{eq:2.2}),
(\ref{eq:3.761}), (\ref{eq:3.85})  and (\ref{eq:3.76}) that
\begin{eqnarray}
W_{oal}  
&=& \frac{1}{q^2 \sigma_{oal}} \sum_{i=1}^{q^2} [ f(Y_i ) - \mu]
\nonumber \\
&=&
\frac{1}{q^2 \sigma_{oal} } \sum_{i=1}^{q^2} \sum_{u_1,\ldots, u_d: (0,\ldots, 0)\prec (u_1,\ldots, u_d)}
\nu_{u_1,\ldots, u_d} 
[ \pi_j (a_{i,j}), \pi_{j;a_{i,j}} (b_{i,j,2}), 
\nonumber \\
&&\hspace{0.5cm}
\pi_{i,j,3} (0),\ldots, 
\pi_{i,j,u_j} (0): 1\leq j\leq d ]
\nonumber \\
&=&
\frac{1}{q^2 \sigma_{oal}} \sum_{i=1}^{q^2} \sum_{u_1,\ldots, u_d\geq 0: u_1+\cdots+ u_d\geq 3}
\nu_{u_1,\ldots, u_d} 
[ \pi_j (a_{i,j}), \pi_{j;a_{i,j}} (b_{i,j,2}), 
\nonumber \\
&&\hspace{0.5cm}
\pi_{i,j,3} (0),\ldots, 
\pi_{i,j,u_j} (0): 1\leq j\leq d ].
\label{eq:4.5}
\end{eqnarray}
Writing
$W_{oas} = (\hat{\mu}_{oas} -\mu)/\sigma_{oas}$
and in a similar manner to (\ref{eq:4.5}), we have
\begin{eqnarray}
W_{oas} &=&
\frac{1}{q^2 \sigma_{oas} } \sum_{i=1}^{q^2 } 
\sum_{u_1,\ldots, u_d \geq 0: u_1+\cdots+ u_d\geq 3}
\nu_{u_1,\ldots, u_d} 
[ \pi_j (a_{i,j}), 
\nonumber \\
&&\hspace{0.5cm}
\pi_{i,j,2} (0),\ldots, 
\pi_{i,j,u_j} (0): 1\leq j\leq d ]
\nonumber \\
&& + \frac{1}{q^2 \sigma_{oas} } \sum_{i=1}^{q^2} \sum_{1\leq k\leq d: u_k=2, u_l=0 \forall l\neq k} \nu_{u_1,\ldots, u_d}
[\pi_j (a_{i,j}), 
\nonumber \\
&& \hspace{0.5cm}
\pi_{i,j,2} (0),\dots, \pi_{i,j,u_j} (0): 1\leq j\leq d].
\label{eq:4.6}
\end{eqnarray}

For brevity of notation,  we write in the sequel
\begin{eqnarray}
&& \nu_{u_1,\ldots, u_d} 
[ \pi_j (a_{i,j}), \pi_{j;a_{i,j}} (b_{i,j,2}), 
\pi_{i,j,3} (0),\ldots, 
\pi_{i,j,u_j} (0): 1\leq j\leq d ]
\nonumber \\
&=& 
\nu^*_{u_{j_1},\ldots, u_{j_{|u|}} } 
[ \pi_{j_1} (a_{i,j_1}), \pi_{j_1;a_{i,j_1}} (b_{i,j_1,2}), 
\pi_{i,j_1,3} (0),\ldots, 
\pi_{i,j_1,u_{j_1}} (0); \ldots;
\nonumber \\
&&\hspace{0.5cm}
\pi_{j_{|u|}} (a_{i,j_{|u|}}), \pi_{j_{|u|};a_{i,j_{|u|}}} (b_{i,j_{|u|},2}), 
\pi_{i,j_{|u|},3} (0),\ldots, 
\pi_{i,j_{|u|},u_{j_{|u|}} } (0) ],
\label{eq:3.36}
\end{eqnarray}
$\nu^* [.] = \nu^*_{u_{j_1},\ldots, u_{j_{|u|}}} [.]$ if $u_{j_1} =\cdots = u_{j_{|u|}} =1$,
and
\begin{equation}
\sigma^2 =
E \Big\{ \{\frac{1}{q^2 } \sum_{i=1}^{q^2 } 
\sum_{0\leq u_1,\ldots, u_d \leq 1: |u| \geq 3}
\nu^*
[ \pi_{j_1} (a_{i,j_1}); \ldots; \pi_{j_{|u|}}(a_{i,j_{|u|}})] \}^2 \Big\},
\label{eq:3.23}
\end{equation}
where $1\leq j_1 <\cdots < j_{|u|}\leq d$ are exactly those coordinates of $u=(u_1,\ldots, u_d)'$ in which $u_j\geq 1$
and $|u|$ denotes the cardinality of that set.
We end this section with the following proposition.

\begin{pn} \label{pn:3.1}
Let $f:[0,1)^d \rightarrow {\mathbb R}$ be smooth with a Lipschitz continuous mixed partial of order $d$.
Then
\begin{eqnarray*}
\sigma^2_{oal} &=&
q^{-2}
\sum_{0\leq u_1,\ldots, u_d \leq 1: |u|\geq 3}
E \{ \nu^*
[ \pi_{j_1} (a_{1,j_1});\ldots;
\pi_{j_{|u|}} (a_{1, j_{|u|}})]^2 \} + O(q^{-3}),
\nonumber \\
\sigma^2_{oas} &=& \mbox{$\sigma^2_{oal} + O(q^{-3})$ and $\sigma^2 = \sigma^2_{oal} + O(q^{-3})$
as $q\rightarrow \infty$.}
\end{eqnarray*}
\end{pn}
{\sc Proof.}
Since $E(W_{oal}^2)=1$, we observe from (\ref{eq:4.5}) that
\begin{eqnarray}
\sigma^2_{oal} &=&
\frac{1}{q^4 } \sum_{i_1=1}^{q^2 } \sum_{i_2=1}^{q^2}
\sum_{u_1,\ldots, u_d \geq 0: u_1+\cdots+ u_d\geq 3}
E \Big\{ \nu_{u_1,\ldots, u_d} 
[ \pi_j (a_{i_1,j}), \pi_{j;a_{i_1,j}} (b_{i_1,j,2}), 
\nonumber \\
&&\hspace{0.5cm}
\pi_{i_1,j,3} (0),\ldots, 
\pi_{i_1,j,u_j} (0): 1\leq j\leq d ]
\nonumber \\
&& \times
\nu_{u_1,\ldots, u_d} 
[ \pi_j (a_{i_2,j}), \pi_{j;a_{i_2,j}} (b_{i_2,j,2}), 
\pi_{i_2,j,3} (0),\ldots, 
\pi_{i_2,j,u_j} (0): 1\leq j\leq d ] \Big\}
\nonumber \\
&=&
\frac{1}{q^4 } \sum_{i=1}^{q^2 }
\sum_{u_1,\ldots, u_d \geq 0: u_1+\cdots+ u_d\geq 3}
E \Big\{ \nu_{u_1,\ldots, u_d} 
[ \pi_j (a_{i,j}), \pi_{j;a_{i,j}} (b_{i,j,2}), 
\nonumber \\
&&\hspace{0.5cm}
\pi_{i,j,3} (0),\ldots, 
\pi_{i,j,u_j} (0): 1\leq j\leq d ]^2 \Big\}
\nonumber \\
&&
+ \frac{1}{q^4 } \sum_{i_1=1}^{q^2 } \sum_{i_2\neq i_1}
\sum_{u_1,\ldots, u_d \geq 0: u_1+\cdots+ u_d\geq 3}
E \Big\{ \nu_{u_1,\ldots, u_d} 
[ \pi_j (a_{i_1,j}), \pi_{j;a_{i_1,j}} (b_{i_1,j,2}), 
\nonumber \\
&&\hspace{0.5cm}
\pi_{i_1,j,3} (0),\ldots, 
\pi_{i_1,j,u_j} (0): 1\leq j\leq d ]
\nonumber \\
&&\times
\nu_{u_1,\ldots, u_d} 
[ \pi_j (a_{i_2,j}), \pi_{j;a_{i_2,j}} (b_{i_2,j,2}), 
\pi_{i_2,j,3} (0),\ldots, 
\pi_{i_2,j,u_j} (0): 1\leq j\leq d ] \Big\}.
\label{eq:3.37}
\end{eqnarray}

Since $A\in$ OA$(q^2, d, q, 2)$, we have
\begin{eqnarray}
&&
\frac{1}{q^4 } \sum_{i_1=1}^{q^2 } \sum_{i_2\neq i_1}
\sum_{u_1,\ldots, u_d \geq 0: u_1+\cdots+ u_d\geq 3}
E \Big\{ \nu_{u_1,\ldots, u_d} 
[ \pi_j (a_{i_1,j}), \pi_{j;a_{i_1,j}} (b_{i_1,j,2}), 
\nonumber \\
&&\hspace{0.5cm}
\pi_{i_1,j,3} (0),\ldots, 
\pi_{i_1,j,u_j} (0): 1\leq j\leq d ]
\nonumber \\
&& \times
\nu_{u_1,\ldots, u_d} 
[ \pi_j (a_{i_2,j}), \pi_{j;a_{i_2,j}} (b_{i_2,j,2}), 
\pi_{i_2,j,3} (0),\ldots, 
\pi_{i_2,j,u_j} (0): 1\leq j\leq d ] \Big\}
\nonumber \\
&=&
\frac{1}{q^4 } \sum_{i_1=1}^{q^2 }
\sum_{0\leq u_1,\ldots, u_d \leq 2: u_1+\cdots+ u_d\geq 3}
\sum_{i_2\neq i_1}
E \Big\{ \nu_{u_1,\ldots, u_d} 
[ \pi_j (a_{i_1,j}), \pi_{j;a_{i_1,j}} (b_{i_1,j,2}), 
\nonumber \\
&&\hspace{0.5cm}
\pi_{i_1,j,3} (0),\ldots, 
\pi_{i_1,j,u_j} (0): 1\leq j\leq d ]
\nonumber \\
&& \times
\nu_{u_1,\ldots, u_d} 
[ \pi_j (a_{i_2,j}), \pi_{j;a_{i_2,j}} (b_{i_2,j,2}), 
\pi_{i_2,j,3} (0),\ldots, 
\pi_{i_2,j,u_j} (0): 1\leq j\leq d ] \Big\}
\nonumber \\
&=&
\frac{1}{q^4 } \sum_{i_1=1}^{q^2 }
\sum_{0\leq u_1,\cdots, u_d \leq 2: u_1+\cdots+ u_d\geq 3}
\sum_{i_2\neq i_1} \sum_{l=1}^{|u|}
{\cal I} \{ a_{i_1,j_l} = a_{i_2,j_l}, u_{j_l}=1 \}
\nonumber \\
&&\times
E \Big\{ \nu_{u_1,\ldots, u_d} 
[ \pi_j (a_{i_1,j}), \pi_{j;a_{i_1,j}} (b_{i_1,j,2}), 
\pi_{i_1,j,3} (0),\ldots, 
\pi_{i_1,j,u_j} (0): 1\leq j\leq d ]
\nonumber \\
&& \hspace{0.5cm}\times
\nu_{u_1,\ldots, u_d} 
[ \pi_j (a_{i_2,j}), \pi_{j;a_{i_2,j}} (b_{i_2,j,2}), 
\pi_{i_2,j,3} (0),\ldots, 
\pi_{i_2,j,u_j} (0): 1\leq j\leq d ] \Big\}
\nonumber \\
&&
+ \frac{1}{q^4 } \sum_{i_1=1}^{q^2 }
\sum_{0\leq u_1,\ldots, u_d \leq 2: u_1+\cdots+ u_d\geq 3}
\sum_{i_2\neq i_1} \sum_{l=1}^{|u|}
{\cal I} \{ a_{i_1,j_l} = a_{i_2,j_l}, u_{j_l}=2 \}
\nonumber \\
&&\times
E \Big\{ \nu_{u_1,\ldots, u_d} 
[ \pi_j (a_{i_1,j}), \pi_{j;a_{i_1,j}} (b_{i_1,j,2}), 
\pi_{i_1,j,3} (0),\ldots, 
\pi_{i_1,j,u_j} (0): 1\leq j\leq d ]
\nonumber \\
&& \hspace{0.5cm}\times
\nu_{u_1,\ldots, u_d} 
[ \pi_j (a_{i_2,j}), \pi_{j;a_{i_2,j}} (b_{i_2,j,2}), 
\pi_{i_2,j,3} (0),\ldots, 
\pi_{i_2,j,u_j} (0): 1\leq j\leq d ] \Big\}
\nonumber \\
&&
+ \frac{1}{q^4 } \sum_{i_1=1}^{q^2 }
\sum_{0\leq u_1,\ldots, u_d \leq 2: u_1+\cdots+ u_d\geq 3}
\sum_{i_2\neq i_1} 
{\cal I} \{ a_{i_1,j_l} \neq a_{i_2,j_l}, \forall l=1,\ldots, |u| \}
\nonumber \\
&&\times
E \Big\{ \nu_{u_1,\ldots, u_d} 
[ \pi_j (a_{i_1,j}), \pi_{j;a_{i_1,j}} (b_{i_1,j,2}), 
\pi_{i_1,j,3} (0),\ldots, 
\pi_{i_1,j,u_j} (0): 1\leq j\leq d ]
\nonumber \\
&& \hspace{0.5cm} \times
\nu_{u_1,\ldots, u_d} 
[ \pi_j (a_{i_2,j}), \pi_{j;a_{i_2,j}} (b_{i_2,j,2}), 
\pi_{i_2,j,3} (0),\ldots, 
\pi_{i_2,j,u_j} (0): 1\leq j\leq d ] \Big\}.
\label{eq:3.38}
\end{eqnarray}

We further note that
\begin{displaymath}
{\cal I} \{ a_{i_1,j_l} = a_{i_2,j_l} \} =
{\cal I} \{ \mbox{$a_{i_1,j_l} = a_{i_2,j_l}$ and
$a_{i_1,j_m} \neq a_{i_2,j_m}$ $\forall m\neq l$} \}.
\end{displaymath}
Consequently,
\begin{eqnarray*}
&& \frac{1}{q^4 } \sum_{i_1=1}^{q^2 }
\sum_{0\leq u_1,\ldots, u_d \leq 2: u_1+\cdots+ u_d\geq 3}
\sum_{i_2\neq i_1} \sum_{l=1}^{|u|}
{\cal I} \{ a_{i_1,j_l} = a_{i_2,j_l}, u_{j_l}=1 \}
\nonumber \\
&&\times
E \Big\{ \nu_{u_1,\ldots, u_d} 
[ \pi_j (a_{i_1,j}), \pi_{j;a_{i_1,j}} (b_{i_1,j,2}), 
\pi_{i_1,j,3} (0),\ldots, 
\pi_{i_1,j,u_j} (0): 1\leq j\leq d ]
\nonumber \\
&& \hspace{0.5cm}\times
\nu_{u_1,\ldots, u_d} 
[ \pi_j (a_{i_2,j}), \pi_{j;a_{i_2,j}} (b_{i_2,j,2}), 
\pi_{i_2,j,3} (0),\ldots, 
\pi_{i_2,j,u_j} (0): 1\leq j\leq d ] \Big\}
\nonumber \\
&=&
\frac{1}{q^4 } \sum_{i_1=1}^{q^2 }
\sum_{0\leq u_1,\ldots, u_d \leq 2: u_1+\cdots+ u_d\geq 3}
\sum_{i_2\neq i_1} 
\nonumber \\
&&\hspace{0.5cm}\times
\sum_{l=1}^{|u|}
{\cal I} \{ \mbox{$a_{i_1,j_l} = a_{i_2,j_l}$ and $u_{j_1}=\cdots = u_{j_{|u|}} =1$} \}
\nonumber \\
&&\hspace{0.5cm}\times
E \{ \nu^*
[ \pi_{j_1} (a_{i_1,j_1});\ldots; \pi_{j_{|u|}} (a_{i_1, j_{|u|}}) ] 
\nu^*
[ \pi_{j_1} (a_{i_2,j_1});\ldots; \pi_{j_{|u|}} (a_{i_2, j_{|u|}}) ] 
\}
\nonumber \\
&=&
\frac{1}{q^4 } \sum_{i_1=1}^{q^2 }
\sum_{0\leq u_1,\ldots, u_d \leq 1: u_1+\cdots+ u_d\geq 3}
\sum_{i_2\neq i_1} 
\sum_{l=1}^{|u|}
{\cal I} \{ a_{i_1,j_l} = a_{i_2,j_l} \}
\nonumber \\
&&\hspace{0.5cm}\times
E \Big\{ \nu^*
[ \pi_{j_1} (a_{i_1,j_1});\ldots; \pi_{j_{|u|}} (a_{i_1, j_{|u|}}) ] 
\nonumber \\
&& \hspace{1cm}\times
(\frac{1}{q-1})^{|u|-1}
\Big[ \prod_{1\leq m\leq |u|: m\neq l}
\sum_{0\leq \tilde{c}_{j_m} \leq q-1: \tilde{c}_{j_m} \neq \pi_{j_m} (a_{i_1, j_m}) }
\Big]
\nonumber \\
&&\hspace{1cm}\times
\nu^*
[ \tilde{c}_{j_1};\ldots; \tilde{c}_{j_{l-1}}; \pi_{j_l} (a_{i_1, j_l}); \tilde{c}_{j_{l+1}}; \ldots; \tilde{c}_{j_{|u|}} ] 
\Big\}
\nonumber \\
&=&
\frac{1}{q^4 } \sum_{i_1=1}^{q^2 }
\sum_{0\leq u_1,\ldots, u_d \leq 1: u_1+\cdots+ u_d\geq 3}
\frac{ (-1)^{|u|-1} }{ (q-1)^{|u|-1} }
\nonumber \\
&&\hspace{0.5cm}\times
\sum_{i_2\neq i_1} \sum_{l=1}^{|u|}
{\cal I} \{ a_{i_1,j_l} = a_{i_2,j_l} \}
E \{ \nu^*
[ \pi_{j_1} (a_{i_1,j_1});\ldots; \pi_{j_{|u|}} (a_{i_1, j_{|u|}}) ]^2 \} 
\nonumber \\
&=&
\sum_{0\leq u_1,\ldots, u_d \leq 1: |u|\geq 3}
\frac{ (-1)^{|u|-1} |u| }{q^2 (q-1)^{|u|-2} } 
E \{ \nu^* 
[ \pi_{j_1} (a_{1,j_1});\ldots; \pi_{j_{|u|}} (a_{1, j_{|u|}}) ]^2 \}, 
\end{eqnarray*}

\begin{eqnarray*}
&&
\frac{1}{q^4 } \sum_{i_1=1}^{q^2 }
\sum_{0\leq u_1,\ldots, u_d \leq 2: u_1+\cdots+ u_d\geq 3}
\sum_{i_2\neq i_1} \sum_{l=1}^{|u|}
{\cal I} \{ a_{i_1,j_l} = a_{i_2,j_l}, u_{j_l}=2 \}
\nonumber \\
&&\times
E \{ \nu_{u_1,\ldots, u_d} 
[ \pi_j (a_{i_1,j}), \pi_{j;a_{i_1,j}} (b_{i_1,j,2}), 
\pi_{i_1,j,3} (0),\ldots, 
\pi_{i_1,j,u_j} (0): 1\leq j\leq d ]
\nonumber \\
&& \hspace{0.5cm}\times
\nu_{u_1,\ldots, u_d} 
[ \pi_j (a_{i_2,j}), \pi_{j;a_{i_2,j}} (b_{i_2,j,2}), 
\pi_{i_2,j,3} (0),\ldots, 
\pi_{i_2,j,u_j} (0): 1\leq j\leq d ] \}
\nonumber \\
&=&
\frac{1}{q^4 } \sum_{i_1=1}^{q^2 }
\sum_{0\leq u_1,\ldots, u_d \leq 2: u_1+\cdots+ u_d\geq 3}
\sum_{i_2\neq i_1} 
\nonumber \\
&&\hspace{0.5cm}\times
\sum_{l=1}^{|u|}
{\cal I} \{ \mbox{$a_{i_1,j_l} = a_{i_2,j_l}, u_{j_l}=2$ and $u_{j_k}=1, \forall k\neq l$} \}
\nonumber \\
&&\hspace{0.5cm}\times
E \Big\{ \nu_{u_{j_1},\ldots, u_{j_{|u|} }}^* 
[ \pi_{j_1} (a_{i_1,j_1}); \ldots; \pi_{j_{l-1}} (a_{i_1,j_{l-1}} );
\nonumber \\
&&\hspace{1cm}
\pi_{j_l} (a_{i_1,j_l}), \pi_{j_l; a_{i_1, j_l} } (b_{i_1, j_l, 2});
\pi_{j_{l+1} } (a_{i_1, j_{l+1}});\ldots; \pi_{j_{|u|}} (a_{i_1, j_{|u|}}) ]
\nonumber \\
&& \hspace{0.5cm} \times
\nu_{u_{j_1},\cdots, u_{j_{|u|} }}^* 
[ \pi_{j_1} (a_{i_2,j_1}); \ldots; \pi_{j_{l-1} } (a_{i_2,j_{l-1}} );
\nonumber \\
&&\hspace{1cm}
\pi_{j_l} (a_{i_2,j_l}), \pi_{j_l; a_{i_2, j_l} } (b_{i_2, j_l, 2});
\pi_{j_{l+1} } (a_{i_2, j_{l+1}});\ldots; \pi_{j_{|u|}} (a_{i_2, j_{|u|}}) ]
\Big\}
\nonumber \\
&=&
\frac{1}{q^4 } \sum_{i_1=1}^{q^2 }
\sum_{0\leq u_1,\ldots, u_d \leq 2: u_1+\cdots+ u_d\geq 3}
\sum_{i_2\neq i_1} 
\nonumber \\
&&\hspace{0.5cm}\times
\sum_{l=1}^{|u|}
{\cal I} \{ \mbox{$a_{i_1,j_l} = a_{i_2,j_l}, u_{j_l}=2$ and $u_{j_k}=1, \forall k\neq l$} \}
\nonumber \\
&&\hspace{0.5cm}\times
E\Big\{ \nu_{u_{j_1},\ldots, u_{j_{|u|} }}^* 
[ \pi_{j_1} (a_{i_1,j_1}); \ldots; \pi_{j_{l-1} } (a_{i_1,j_{l-1}} );
\nonumber \\
&&\hspace{1cm}
\pi_{j_l} (a_{i_1,j_l}), \pi_{j_l; a_{i_1, j_l } } (b_{i_1, j_l, 2});
\pi_{j_{l+1} } (a_{i_1, j_{l+1}});\ldots; \pi_{j_{|u|}} (a_{i_1, j_{|u|}}) ]
\nonumber \\
&& \hspace{0.5cm} \times
(\frac{1}{q-1})^{|u|}
\sum_{0\leq \tilde{c}_{j_l} \leq q-1:\tilde{c}_{j_l} \neq \pi_{j_l; a_{i_1, j_l}} (b_{i_1, j_l, 2} ) }
\nonumber \\
&&\hspace{0.5cm}\times
\Big[ \prod_{1\leq m\leq |u|: m\neq l}
\sum_{0\leq \tilde{c}_{j_m} \leq q-1:\tilde{c}_{j_m} \neq \pi_{j_m} (a_{i_1, j_m} ) }
\Big]
\nonumber \\
&&\hspace{0.5cm}\times
\nu_{u_{j_1},\ldots, u_{j_{|u|} }}^* 
[ \tilde{c}_{j_1}; \ldots; \tilde{c}_{j_{l-1}};
\pi_{j_l} (a_{i_1,j_l}), \tilde{c}_{j_l};
\tilde{c}_{j_{l+1}};\ldots; \tilde{c}_{j_{|u|}} ]
\Big\}
\nonumber \\
&=&
\frac{1}{q^4 } \sum_{i_1=1}^{q^2 }
\sum_{0\leq u_1,\ldots, u_d \leq 2: u_1+\cdots+ u_d\geq 3}
\frac{ (-1)^{|u|} }{ (q-1)^{|u|} }
\sum_{i_2\neq i_1} 
\nonumber \\
&& \hspace{0.5cm}\times
\sum_{l=1}^{|u|}
{\cal I} \{ \mbox{$a_{i_1,j_l} = a_{i_2,j_l}, u_{j_l}=2$ and $u_{j_k}=1, \forall k\neq l$} \}
\nonumber \\
&&\hspace{0.5cm}\times
E\Big\{ \nu_{u_{j_1},\ldots, u_{j_{|u|} }}^* 
[ \pi_{j_1} (a_{i_1,j_1}); \ldots; \pi_{j_{l-1}} (a_{i_1,j_{l-1}} );
\nonumber \\
&&\hspace{1cm}
\pi_{j_l} (a_{i_1,j_l}), \pi_{j_l; a_{i_1, j_l} } (b_{i_1, j_l, 2});
\pi_{j_{l+1} } (a_{i_1, j_{l+1}});\ldots; \pi_{j_{|u|}} (a_{i_1, j_{|u|}}) ]^2 \Big\}
\nonumber \\
&=&
\frac{1}{q^2 }
\sum_{l=1}^d \sum_{u_1,\ldots, u_d: u_l=2, 0\leq u_k\leq 1 \forall k\neq l, |u|+1 \geq 3}
\frac{ (-1)^{|u|} }{ (q-1)^{|u|-1}}
\nonumber \\
&&\hspace{0.5cm}\times
E\Big\{ \nu_{u_{j_1},\ldots, u_{j_{|u|} }}^* 
[ \pi_{j_1} (a_{1,j_1}); \ldots; \pi_{j_{l-1}} (a_{1,j_{l-1}} );
\nonumber \\
&&\hspace{1cm}
\pi_{j_l} (a_{1,j_l}), \pi_{j_l; a_{1, j_l } } (b_{1, j_l, 2});
\pi_{j_{l+1}} (a_{1, j_{l+1}});\ldots; \pi_{j_{|u|}} (a_{1, j_{|u|}}) ]^2 \Big\},
\end{eqnarray*}
and
\begin{eqnarray}
&& \frac{1}{q^4 } \sum_{i_1=1}^{q^2 }
\sum_{0\leq u_1,\ldots, u_d \leq 2: u_1+\cdots+ u_d\geq 3}
\sum_{i_2\neq i_1} 
{\cal I} \{ a_{i_1,j_l} \neq a_{i_2,j_l}, \forall l=1,\ldots, |u| \}
\nonumber \\
&&\times
E \Big\{ \nu_{u_1,\ldots, u_d} 
[ \pi_j (a_{i_1,j}), \pi_{j;a_{i_1,j}} (b_{i_1,j,2}), 
\pi_{i_1,j,3} (0),\ldots, 
\pi_{i_1,j,u_j} (0): 1\leq j\leq d ]
\nonumber \\
&& \hspace{0.5cm}\times
\nu_{u_1,\ldots, u_d} 
[ \pi_j (a_{i_2,j}), \pi_{j;a_{i_2,j}} (b_{i_2,j,2}), 
\pi_{i_2,j,3} (0),\ldots, 
\pi_{i_2,j,u_j} (0): 1\leq j\leq d ] \Big\}
\nonumber \\
&=&
\frac{1}{q^4 } \sum_{i_1=1}^{q^2 }
\sum_{0\leq u_1,\ldots, u_d \leq 2: u_1+\cdots+ u_d\geq 3}
\sum_{i_2\neq i_1} 
{\cal I} \{ a_{i_1,j_l} \neq a_{i_2,j_l}, u_{j_l}=1, \forall l=1,\ldots, |u| \}
\nonumber \\
&&\times
E \{ \nu^* 
[ \pi_{j_1} (a_{i_1,j_1}); \ldots; \pi_{j_{|u|}} (a_{i_1, j_{|u|}}) ]
\nu^* 
[ \pi_{j_1} (a_{i_2,j_1}); \ldots; \pi_{j_{|u|}} (a_{i_2, j_{|u|}}) ]
\}
\nonumber \\
&=&
\frac{1}{q^4 } \sum_{i_1=1}^{q^2 }
\sum_{0\leq u_1,\ldots, u_d \leq 1: |u| \geq 3}
\sum_{i_2\neq i_1} 
{\cal I} \{ a_{i_1,j_l} \neq a_{i_2,j_l}, \forall l=1,\ldots, |u| \}
\nonumber \\
&&\times
E \Big\{ \nu^* 
[ \pi_{j_1} (a_{i_1,j_1}); \ldots; \pi_{j_{|u|}} (a_{i_1, j_{|u|}}) ]
\nonumber \\
&& \hspace{0.5cm} \times
(\frac{1}{q-1})^{|u|}
\Big[ \prod_{1\leq m\leq |u|}
\sum_{0\leq \tilde{c}_{j_m} \leq q-1:\tilde{c}_{j_m} \neq \pi_{j_m} (a_{i_1, j_m} ) }
\Big]
\nu^* 
[ \tilde{c}_{j_1}; \ldots;
\tilde{c}_{j_{|u|}} ]
\Big\}
\nonumber \\
&=&
\frac{1}{q^4 } \sum_{i_1=1}^{q^2 }
\sum_{0\leq u_1,\ldots, u_d \leq 1: |u| \geq 3}
\frac{ (-1)^{|u|} }{(q-1)^{|u|} }
\sum_{i_2\neq i_1} 
{\cal I} \{ a_{i_1,j_l} \neq a_{i_2,j_l}, \forall l=1,\ldots, |u| \}
\nonumber \\
&&\hspace{0.5cm}\times
E \{ \nu^* 
[ \pi_{j_1} (a_{i_1,j_1}); \ldots; \pi_{j_{|u|}} (a_{i_1, j_{|u|}}) ]^2 \}
\nonumber \\
&=&
\frac{1}{q^2 }
\sum_{0\leq u_1,\ldots, u_d \leq 1: |u| \geq 3}
\frac{(-1)^{|u|} [ q^2-1 - |u| (q-1)] }{ (q-1)^{|u|} }
E \{ \nu^* 
[ \pi_{j_1} (a_{1,j_1}); \ldots; \pi_{j_{|u|}} (a_{1, j_{|u|}}) ]^2 \}
\nonumber \\
&=&
\sum_{0\leq u_1,\ldots, u_d \leq 1: |u| \geq 3}
\frac{ (-1)^{|u|}  ( q+1 - |u|) }{ q^2 (q-1)^{|u|-1}}
E \{ \nu^* 
[ \pi_{j_1} (a_{1,j_1}); \ldots; \pi_{j_{|u|}} (a_{1, j_{|u|}}) ]^2 \}.
\label{eq:3.39}
\end{eqnarray}

We conclude from (\ref{eq:3.37}), (\ref{eq:3.38}) and (\ref{eq:3.39}) that
\begin{eqnarray}
\sigma^2_{oal} &=&
\frac{1}{q^2 }
\sum_{u_1,\ldots, u_d \geq 0: u_1+\cdots+ u_d\geq 3}
E \Big\{ \nu_{u_1,\ldots, u_d} 
[ \pi_j (a_{1,j}), \pi_{j;a_{1,j}} (b_{1,j,2}), 
\nonumber \\
&&\hspace{0.5cm}
\pi_{1,j,3} (0),\ldots, 
\pi_{1,j,u_j} (0): 1\leq j\leq d ]^2 \Big\}
\nonumber \\
&&
+ \sum_{0\leq u_1,\ldots, u_d \leq 1: |u|\geq 3}
\frac{ (-1)^{|u|-1} |u| }{q^2 (q-1)^{|u|-2} } 
E \{ \nu^* 
[ \pi_{j_1} (a_{1,j_1});\ldots, \pi_{j_{|u|}} (a_{1, j_{|u|}}) ]^2 \}
\nonumber \\
&&
+ \sum_{l=1}^d \sum_{u_1,\ldots, u_d: u_l=2, 0\leq u_k\leq 1 \forall k\neq l, |u| \geq 2}
\frac{ (-1)^{|u|} }{ q^2 (q-1)^{|u|-1}}
\nonumber \\
&&\hspace{0.5cm}\times
E\Big\{ \nu_{u_{j_1},\ldots, u_{j_{|u|} }}^* 
[ \pi_{j_1} (a_{1,j_1}); \ldots; \pi_{j_{l-1}} (a_{1,j_{l-1}} );
\nonumber \\
&&\hspace{1cm}
\pi_{j_l} (a_{1,j_l}), \pi_{j_l; a_{1, j_l} } (b_{1, j_l, 2});
\pi_{j_{l+1}} (a_{1, j_{l+1}});\ldots; \pi_{j_{|u|}} (a_{1, j_{|u|}}) ]^2 \Big\}
\nonumber \\
&&
+ \sum_{0\leq u_1,\ldots, u_d \leq 1: |u| \geq 3}
\frac{ (-1)^{|u|} (q+1 -|u|) }{ q^2 (q-1)^{|u|-1} }
E \{ \nu^* 
[ \pi_{j_1} (a_{1,j_1}); \ldots; \pi_{j_{|u|}} (a_{1, j_{|u|}}) ]^2 \}
\nonumber \\
&=&
\frac{1}{q^2 }
\sum_{0\leq u_1,\ldots, u_d \leq 1: |u|\geq 3}
E \{ \nu^* 
[ \pi_{j_1} (a_{1,j_1});\ldots;
\pi_{j_{|u|}} (a_{1, j_{|u|}})]^2 \}
\nonumber \\
&&
+ \frac{1}{q^2 }
\sum_{u_1,\ldots, u_d \geq 0: u_1+\cdots+ u_d\geq 3\vee (|u|+1)}
E \Big\{ \nu_{u_1,\ldots, u_d} 
[ \pi_j (a_{1,j}), \pi_{j;a_{1,j}} (b_{1,j,2}), 
\nonumber \\
&&\hspace{0.5cm}
\pi_{1,j,3} (0),\ldots, 
\pi_{1,j,u_j} (0): 1\leq j\leq d ]^2 \Big\}
\nonumber \\
&&
+ \sum_{0\leq u_1,\ldots, u_d \leq 1: |u|\geq 3}
\frac{ (-1)^{|u|-1} |u| }{q^2 (q-1)^{|u|-2} } 
E \{ \nu^* 
[ \pi_{j_1} (a_{1,j_1});\ldots, \pi_{j_{|u|}} (a_{1, j_{|u|}}) ]^2 \}
\nonumber \\
&&
+ \sum_{l=1}^d \sum_{u_1,\ldots, u_d: u_l=2, 0\leq u_k\leq 1 \forall k\neq l, |u| \geq 2}
\frac{ (-1)^{|u|} }{ q^2 (q-1)^{|u|-1}}
\nonumber \\
&&\hspace{0.5cm}\times
E\Big\{ \nu_{u_{j_1},\ldots, u_{j_{|u|} }}^* 
[ \pi_{j_1} (a_{1,j_1}); \ldots; \pi_{j_{l-1}} (a_{1,j_{l-1}} );
\nonumber \\
&&\hspace{1cm}
\pi_{j_l} (a_{1,j_l}), \pi_{j_l; a_{1, j_l} } (b_{1, j_l, 2});
\pi_{j_{l+1}} (a_{1, j_{l+1}});\ldots; \pi_{j_{|u|}} (a_{1, j_{|u|}}) ]^2 \Big\}
\nonumber \\
&&
+ \sum_{0\leq u_1,\ldots, u_d \leq 1: |u| \geq 3}
\frac{ (-1)^{|u|} (q+1 -|u|) }{ q^2 (q-1)^{|u|-1} }
E \{ \nu^* 
[ \pi_{j_1} (a_{1,j_1}); \ldots; \pi_{j_{|u|}} (a_{1, j_{|u|}}) ]^2 \}.
\label{eq:3.40}
\end{eqnarray}
Thus it follows from (\ref{eq:3.40}) and Lemma \ref{la:a.5} (see Appendix) that
\begin{eqnarray*}
\sigma^2_{oal} &=&
\frac{1}{q^2 }
\sum_{0\leq u_1,\ldots, u_d \leq 1: |u|\geq 3}
E \{ \nu^* 
[ \pi_{j_1} (a_{1,j_1});\ldots;
\pi_{j_{|u|}} (a_{1, j_{|u|}})]^2 \} + O(\frac{1}{q^3})
\nonumber \\
&=&
\frac{1}{q^2 }
\sum_{0\leq u_1,\ldots, u_d \leq 1: |u| \geq 3} \sum_{c_{j_1}=0}^{q-1} 
\cdots \sum_{c_{j_{|u|} }=0}^{q-1} 
\langle f,  \prod_{l=1}^{|u|}  \psi_{0, 0,c_{j_l}} \rangle^2
+ O(\frac{1}{q^3})
\nonumber \\
&=&
O(\frac{1}{q^2}) 
\sum_{0\leq u_1,\ldots, u_d \leq 1: |u|\geq 3}
q^{3 |u| - 3 \sum_{l=1}^{|u|} u_{j_l}}
\nonumber \\
&=&
O(\frac{1}{q^2}), 
\end{eqnarray*}
as $q\rightarrow \infty$.
Next we observe from (\ref{eq:4.6}) that
\begin{eqnarray}
\sigma^2_{oas} &=& E \Big\{ \{\frac{1}{q^2 } \sum_{i=1}^{q^2 } 
\sum_{u_1,\ldots, u_d \geq 0: u_1+\cdots+ u_d\geq 3}
\nu_{u_1,\ldots, u_d} 
[ \pi_j (a_{i,j}), 
\pi_{i,j,2} (0),\ldots, 
\pi_{i,j,u_j} (0): 
\nonumber \\
&&\hspace{0.5cm}
1\leq j\leq d ]
+ \frac{1}{q^2 } \sum_{i=1}^{q^2 } 
\sum_{k=1}^d
\nu_2^* 
[ \pi_k (a_{i,k}), 
\pi_{i,k,2} (0)] \}^2 \Big\} 
\nonumber \\
&=&
\frac{1}{q^4 } \sum_{i_1=1}^{q^2 } \sum_{i_2=1}^{q^2}
\sum_{u_1,\ldots, u_d \geq 0: u_1+\cdots+ u_d\geq 3}
E\Big\{ \nu_{u_1,\ldots, u_d} 
[ \pi_j (a_{i_1,j}), 
\pi_{i_1,j,2} (0),\ldots, 
\pi_{i_1,j,u_j} (0): 
\nonumber \\
&&\hspace{1cm}
1\leq j\leq d ]
\nu_{u_1,\ldots, u_d} 
[ \pi_j (a_{i_2,j}), 
\pi_{i_2,j,2} (0),\ldots, 
\pi_{i_2,j,u_j} (0): 1\leq j\leq d ] \Big\}
\nonumber \\
&&
+ \frac{1}{q^4 } \sum_{i_1=1}^{q^2 } \sum_{i_2=1}^{q^2}
\sum_{k=1}^d
E \Big\{ \nu_2^*
[ \pi_k (a_{i_1,k}), 
\pi_{i_1,k,2} (0)]
\nu_2^* 
[ \pi_k (a_{i_2,k}), 
\pi_{i_2,k,2} (0) ]
\Big\}
\nonumber \\
&=&
\frac{1}{q^4 } \sum_{i_1=1}^{q^2 }
\sum_{u_1,\ldots, u_d \geq 0: u_1+\cdots+ u_d\geq 3}
E\Big\{ \nu_{u_1,\ldots, u_d} 
[ \pi_j (a_{i_1,j}), 
\pi_{i_1,j,2} (0),\ldots, 
\pi_{i_1,j,u_j} (0): 
\nonumber \\
&&\hspace{1cm}
1\leq j\leq d ]^2 \Big\}
\nonumber \\
&&
+ \frac{1}{q^4 } \sum_{i_1=1}^{q^2 }
\sum_{k=1}^d
E \{ \nu_2^* 
[ \pi_k (a_{i_1,k}), 
\pi_{i_1,k,2} (0)]^2 \}
\nonumber \\
&& + \frac{1}{q^4 } \sum_{i_1=1}^{q^2 } \sum_{i_2\neq i_1}
\sum_{0\leq u_1,\ldots, u_d \leq 1: u_1+\cdots+ u_d\geq 3}
E\Big\{ \nu^* 
[ \pi_{j_1} (a_{i_1,j_1}); \ldots;  \pi_{j_{|u|}} (a_{i_1,j_{|u|}}) ]
\nonumber \\
&&\hspace{0.5cm}\times
\nu^* [ \pi_{j_1} (a_{i_2,j_1}); \ldots;  \pi_{j_{|u|}} (a_{i_2,j_{|u|}}) ]
\Big\}
\nonumber \\
&&
+ \frac{1}{q^4 } \sum_{i_1=1}^{q^2 } \sum_{i_2\neq i_1}
\sum_{k=1}^d
E \{ \nu_2^* 
[ \pi_k (a_{i_1,k}), 
\pi_{i_1,k,2} (0) ]
\nu_2^* [ \pi_k (a_{i_2,k}), 
\pi_{i_2,k,2} (0)] \}.
\label{eq:3.401}
\end{eqnarray}
From Lemma \ref{la:a.5}, we obtain
\begin{eqnarray}
\frac{1}{q^4 } \sum_{i_1=1}^{q^2 }
\sum_{k=1}^d
E \{ \nu_2^* 
[ \pi_k (a_{i_1,k}), 
\pi_{i_1,k,2} (0)]^2 \} &=& O(\frac{1}{q^4}),
\label{eq:3.402} \\
\frac{1}{q^4 } \sum_{i_1=1}^{q^2 } \sum_{i_2\neq i_1}
\sum_{k=1}^d
E \{ \nu_2^* [ \pi_k (a_{i_1,k}), 
\pi_{i_1,k,2} (0) ]
\nu_2^* [ \pi_k (a_{i_2,k}), 
\pi_{i_2,k,2} (0)] \} &=& 0,
\nonumber 
\end{eqnarray}
and
\begin{eqnarray*}
&& \frac{1}{q^4 } \sum_{i_1=1}^{q^2 } \sum_{i_2\neq i_1}
\sum_{0\leq u_1,\ldots, u_d \leq 1: u_1+\cdots+ u_d\geq 3}
E\Big\{ \nu^* 
[ \pi_{j_1} (a_{i_1,j_1}); \ldots;  \pi_{j_{|u|}} (a_{i_1,j_{|u|}}) ]
\nonumber \\
&&\hspace{0.5cm}\times
\nu^* [ \pi_{j_1} (a_{i_2,j_1}); \ldots;  \pi_{j_{|u|}} (a_{i_2,j_{|u|}}) ]
\Big\} 
\nonumber \\
&=&
\frac{1}{q^4 } \sum_{i_1=1}^{q^2 } \sum_{i_2\neq i_1}
\sum_{0\leq u_1,\ldots, u_d \leq 1: u_1+\cdots+ u_d\geq 3} \sum_{l=1}^{|u|} {\cal I}\{ a_{i_1,j_l} = a_{i_2, j_l} \}
\nonumber \\
&&\hspace{0.5cm}\times
E \{ \nu^* 
[ \pi_{j_1} (a_{i_1,j_1}); \ldots;  \pi_{j_{|u|}} (a_{i_1,j_{|u|}}) ]
\nu^* [ \pi_{j_1} (a_{i_2,j_1}); \ldots;  \pi_{j_{|u|}} (a_{i_2,j_{|u|}}) ] \} 
\nonumber \\
&& + \frac{1}{q^4 } \sum_{i_1=1}^{q^2 } \sum_{i_2\neq i_1}
\sum_{0\leq u_1,\ldots, u_d \leq 1: u_1+\cdots+ u_d\geq 3}  {\cal I}\{ a_{i_1,j_l} \neq a_{i_2, j_l}, l=1,\ldots, |u| \}
\nonumber \\
&&\hspace{0.5cm}\times
E \{ \nu^* 
[ \pi_{j_1} (a_{i_1,j_1}); \ldots;  \pi_{j_{|u|}} (a_{i_1,j_{|u|}}) ]
\nu^* [ \pi_{j_1} (a_{i_2,j_1}); \ldots;  \pi_{j_{|u|}} (a_{i_2,j_{|u|}}) ] \} 
\nonumber \\
&=& \frac{1}{q^4 } \sum_{i_1=1}^{q^2 } \sum_{i_2\neq i_1}
\sum_{0\leq u_1,\ldots, u_d \leq 1: u_1+\cdots+ u_d\geq 3} \sum_{l=1}^{|u|} {\cal I}\{ a_{i_1,j_l} = a_{i_2, j_l} \}
\nonumber \\
&&\hspace{0.5cm}\times
(-\frac{1}{q-1})^{|u|-1}
E \{ \nu^* 
[ \pi_{j_1} (a_{i_1,j_1}); \ldots;  \pi_{j_{|u|}} (a_{i_1,j_{|u|}}) ]^2 \}
\nonumber \\
&& + \frac{1}{q^4 } \sum_{i_1=1}^{q^2 } \sum_{i_2\neq i_1}
\sum_{0\leq u_1,\ldots, u_d \leq 1: u_1+\cdots+ u_d\geq 3}  {\cal I}\{ a_{i_1,j_l} \neq a_{i_2, j_l}, l=1,\ldots, |u| \}
\nonumber \\
&&\hspace{0.5cm}\times
(-\frac{1}{q-1})^{|u|}
E \{ \nu^* 
[ \pi_{j_1} (a_{i_1,j_1}); \ldots;  \pi_{j_{|u|}} (a_{i_1,j_{|u|}}) ]^2 \}
\nonumber \\
&=& \frac{1}{q^2 }
\sum_{0\leq u_1,\ldots, u_d \leq 1: |u|\geq 3} |u| (q-1)
(-\frac{1}{q-1})^{|u|-1}
E \{ \nu^* 
[ \pi_{j_1} (a_{1,j_1}); \ldots;  \pi_{j_{|u|}} (a_{1,j_{|u|}}) ]^2 \}
\nonumber \\
&& + \sum_{0\leq u_1,\ldots, u_d \leq 1: |u|\geq 3} \frac{ (-1)^{|u|} [ q^2 -1 - |u| (q-1) ] }{
 q^2 ( q-1)^{|u|} }
E \{ \nu^* 
[ \pi_{j_1} (a_{1,j_1}); \ldots;  \pi_{j_{|u|}} (a_{1,j_{|u|}}) ]^2 \}
\nonumber \\
&=& O(\frac{1}{q^3}),
\end{eqnarray*}
as $q\rightarrow \infty$.
Hence we conclude from (\ref{eq:3.40}) and (\ref{eq:3.401}) that
\begin{eqnarray*}
\sigma^2_{oas} &=&
\frac{1}{q^2 }
\sum_{u_1,\ldots, u_d \geq 0: u_1+\cdots+ u_d\geq 3}
E\Big\{ \nu_{u_1,\ldots, u_d} 
[ \pi_j (a_{1,j}), 
\pi_{1,j,2} (0),\ldots, 
\pi_{1,j,u_j} (0): 
\nonumber \\
&&\hspace{1cm}
1\leq j\leq d ]^2 \Big\} + O(\frac{1}{q^3})
\nonumber \\
&=& \sigma_{oal}^2 + O(\frac{1}{q^3}),
\end{eqnarray*}
as $q\rightarrow\infty$.
The remaining case, namely $\sigma^2$, can be shown in a similar 
(though simpler) manner.
This proves Proposition \ref{pn:3.1}. \hfill $\Box$

\section{A multivariate central limit theorem}

First we define the proxy statistic
\begin{eqnarray}
W &=&
\frac{1}{q^2 \sigma } \sum_{i=1}^{q^2 } 
\sum_{0\leq u_1,\ldots, u_d \leq 1: |u| \geq 3}
\nu^*
[ \pi_{j_1} (a_{i,j_1}); \ldots; \pi_{j_{|u|}}(a_{i,j_{|u|}})],
\label{eq:4.55}
\end{eqnarray}
where
$\sigma^2$ is as in (\ref{eq:3.23}). 
Clearly we have $E( W^2) =1$.

\begin{pn} \label{pn:4.1}
Let $f:[0,1)^d \rightarrow {\mathbb R}$ be smooth with a Lipschitz continuous mixed partial of order $d$.
Suppose further that $W_{oal}, W_{oas}$ and  $W$ are as defined by (\ref{eq:4.5}), (\ref{eq:4.6}) and (\ref{eq:4.55})
respectively with $A\in$ OA$(q^2, d, q, 2)$.
Then
$q ( \sigma_{oal} W_{oal}- \sigma W) \rightarrow 0$
and
$q ( \sigma_{oas} W_{oas}- \sigma W) \rightarrow 0$
in probability as $q\rightarrow \infty$.
\end{pn}
{\sc Proof.}
We observe that
\begin{displaymath}
q ( \sigma_{oal} W_{oal} - \sigma W ) = \sum_{r=1}^d \sum_{1\leq k_1<\cdots < k_r\leq d} \Delta_{k_1,\ldots, k_r},
\end{displaymath}
where
\begin{eqnarray*}
\Delta_{k_1,\ldots, k_r} &=& \frac{1}{q } \sum_{i=1}^{q^2} \sum_{u_1,\ldots, u_d \geq 0:
u_k \geq 2 \Leftrightarrow k\in \{k_1,\ldots, k_r \}, u_1+\cdots+ u_d\geq 3 }
\nu_{u_1,\ldots, u_d} [ \pi_j (a_{i,j}), 
\nonumber \\
&&\hspace{0.5cm} 
\pi_{j; a_{i,j}} (b_{i,j,2}), 
\pi_{i,j,3} (0), \ldots, \pi_{i,j,u_j}(0):
1\leq j\leq d].
\end{eqnarray*}
Here $\Leftrightarrow$ denotes if and only if and that given $u_1,\ldots, u_d\geq 0$, we write 
$u_k \geq 1 \Leftrightarrow k\in \{j_1,\ldots, j_{|u|}\}$ where $|u| = \sum_{i=1}^d {\cal I}\{u_i \geq 1\}$.
Now for $r=1,\ldots, d$, we have
\begin{eqnarray}
E(\Delta^2_{1,\ldots, r}) 
&=&\frac{1}{q^2}
\sum_{i_1 =1}^{q^2} \sum_{i_2=1}^{q^2}
\sum_{u_1,\ldots, u_d \geq 0: u_k \geq 2 \Leftrightarrow k\in \{1,\ldots, r\},
u_1+\cdots + u_d \geq 3 }
E\Big\{ \nu_{u_1,\ldots, u_d} [\pi_j(a_{i_1,j}), 
\nonumber \\
&& \hspace{0.5cm}
\pi_{j;a_{i_1,j}} (b_{i_1,j,2}), \pi_{i_1,j,3} (0),\ldots,
\pi_{i_1,j,u_j} (0): 1\leq j\leq d]
\nonumber \\
&&\times
\nu_{u_1,\ldots, u_d} [\pi_j(a_{i_2,j}), 
\pi_{j;a_{i_2,j}} (b_{i_2,j,2}), 
\pi_{i_2,j,3} (0),\ldots,
\pi_{i_2,j,u_j} (0): 1\leq j\leq d]
\Big\}
\nonumber \\
&=&\frac{1}{q^2}
\sum_{i =1}^{q^2}
\sum_{u_1,\ldots, u_d \geq 0: u_k \geq 2 \Leftrightarrow k\in \{1,\ldots, r\},
u_1+\cdots + u_d\geq 3 }
E\Big\{ \nu_{u_1,\ldots, u_d} [\pi_j(a_{i,j}), 
\nonumber \\
&& \hspace{0.5cm}
\pi_{j;a_{i,j}} (b_{i,j,2}), \pi_{i,j,3} (0),\ldots,
\pi_{i,j,u_j} (0): 1\leq j\leq d]^2 \Big\}
\nonumber \\
&&
+ \frac{1}{q^2}
\sum_{i_1 =1}^{q^2} \sum_{i_2\neq i_1}
\sum_{u_1,\ldots, u_d \geq 0: u_k \geq 2 \Leftrightarrow k\in \{1,\ldots, r\},
u_1+\cdots + u_d \geq 3 }
E\Big\{ \nu_{u_1,\ldots, u_d} [\pi_j(a_{i_1,j}), 
\nonumber \\
&& \hspace{0.5cm}
\pi_{j;a_{i_1,j}} (b_{i_1,j,2}), \pi_{i_1,j,3} (0),\ldots,
\pi_{i_1,j,u_j} (0): 1\leq j\leq d]
\nonumber \\
&&\times
\nu_{u_1,\ldots, u_d} [\pi_j(a_{i_2,j}), 
\pi_{j;a_{i_2,j}} (b_{i_2,j,2}), 
\pi_{i_2,j,3} (0),\ldots,
\pi_{i_2,j,u_j} (0): 1\leq j\leq d]
\Big\}.
\label{eq:3.67}
\end{eqnarray}

Using the fact that $A\in$ OA$(q^2, d, q, 2)$, we further observe that
\begin{eqnarray}
&& \frac{1}{q^2}
\sum_{i_1 =1}^{q^2} \sum_{i_2\neq i_1}
\sum_{u_1,\ldots, u_d \geq 0: u_k \geq 2 \Leftrightarrow k\in \{1,\ldots, r\},
u_1+\cdots + u_d \geq 3 }
E\Big\{ \nu_{u_1,\ldots, u_d} [\pi_j(a_{i_1,j}), 
\nonumber \\
&& \hspace{0.5cm}
\pi_{j;a_{i_1,j}} (b_{i_1,j,2}), \pi_{i_1,j,3} (0),\ldots,
\pi_{i_1,j,u_j} (0): 1\leq j\leq d]
\nonumber \\
&&\times
\nu_{u_1,\ldots, u_d} [\pi_j(a_{i_2,j}), 
\pi_{j;a_{i_2,j}} (b_{i_2,j,2}), 
\pi_{i_2,j,3} (0),\ldots,
\pi_{i_2,j,u_j} (0): 1\leq j\leq d]
\Big\}
\nonumber \\
&=&
\frac{1}{q^2}
\sum_{i_1 =1}^{q^2}
\sum_{0\leq u_1,\ldots, u_d \leq 2: u_k = 2 \Leftrightarrow k\in \{1,\ldots, r\},
u_1+\cdots + u_d \geq 3 }
\sum_{i_2\neq i_1} E\Big\{ \nu_{u_1,\ldots, u_d} [\pi_j(a_{i_1,j}), 
\nonumber \\
&& \hspace{0.5cm}
\pi_{j;a_{i_1,j}} (b_{i_1,j,2}), \pi_{i_1,j,3} (0),\ldots,
\pi_{i_1,j,u_j} (0): 1\leq j\leq d]
\nonumber \\
&&\times
\nu_{u_1,\ldots, u_d} [\pi_j(a_{i_2,j}), 
\pi_{j;a_{i_2,j}} (b_{i_2,j,2}), 
\pi_{i_2,j,3} (0),\ldots,
\pi_{i_2,j,u_j} (0): 1\leq j\leq d]
\Big\}
\nonumber \\
&=&
\frac{ {\cal I}\{ r=1\} }{q^2}
\sum_{i_1 =1}^{q^2}
\sum_{0\leq u_1,\ldots, u_d \leq 2: u_k = 2 \Leftrightarrow k= 1,
|u| + 1 \geq 3 }
\sum_{i_2\neq i_1} 
{\cal I}\{ a_{i_1,1} = a_{i_2, 1}\}
\nonumber \\
&&\times
E\Big\{ \nu_{u_1,\ldots, u_d} [\pi_j(a_{i_1,j}), 
\pi_{j;a_{i_1,j}} (b_{i_1,j,2}), \pi_{i_1,j,3} (0),\ldots,
\pi_{i_1,j,u_j} (0): 1\leq j\leq d]
\nonumber \\
&&\times
\nu_{u_1,\ldots, u_d} [\pi_j(a_{i_2,j}), 
\pi_{j;a_{i_2,j}} (b_{i_2,j,2}), 
\pi_{i_2,j,3} (0),\ldots,
\pi_{i_2,j,u_j} (0): 1\leq j\leq d]
\Big\}
\nonumber \\
&=&
\frac{ {\cal I}\{ r=1\} }{q^2}
\sum_{i_1 =1}^{q^2}
\sum_{0\leq u_1,\ldots, u_d \leq 2: u_k = 2 \Leftrightarrow k= 1, |u| \geq 2 }
\sum_{i_2\neq i_1} 
{\cal I}\{ a_{i_1,1} = a_{i_2, 1}\}
\nonumber \\
&&\times
E\Big\{ \nu^*_{u_1, u_{j_2},\ldots, u_{j_{|u|}} } [
\pi_{1} (a_{i_1,1}), 
\pi_{1; 
a_{i_1,1} } ( b_{i_1, 1, 2} );  
\pi_{j_2} ( a_{i_1, j_2});\ldots;
\pi_{j_{|u|}} ( a_{i_1, j_{|u|}}) ]
\nonumber \\
&&\times
\nu^*_{u_1, u_{j_2},\ldots, u_{j_{|u|}} } [
\pi_{1} (a_{i_2,1}), 
\pi_{1; 
a_{i_2,1} } ( b_{i_2, 1, 2} );  
\pi_{j_2} ( a_{i_2, j_2});\ldots;
\pi_{j_{|u|}} ( a_{i_2, j_{|u|}}) ] \Big\}
\nonumber \\
&=&
\frac{ {\cal I}\{ r=1\} }{q^2}
\sum_{i_1 =1}^{q^2}
\sum_{0\leq u_1,\ldots, u_d \leq 2: u_k = 2 \Leftrightarrow k= 1, |u| \geq 2 }
\sum_{i_2\neq i_1} 
{\cal I}\{ a_{i_1,1} = a_{i_2, 1}\}
\nonumber \\
&&\times
E\Big\{ \nu^*_{u_1, u_{j_2},\ldots, u_{j_{|u|}} } [
\pi_{1} (a_{i_1,1}), 
\pi_{1; 
a_{i_1,1} } ( b_{i_1, 1, 2} );  
\pi_{j_2} ( a_{i_1, j_2});\ldots;
\pi_{j_{|u|}} ( a_{i_1, j_{|u|}}) ]
\nonumber \\
&&\times
(\frac{1}{q-1})^{|u|} \sum_{0\leq \tilde{c}_1 \leq q-1: \tilde{c}_1 \neq 
\pi_{1; 
a_{i_1,1} } ( b_{i_1, 1, 2} ) } 
\Big[ \prod_{2\leq m\leq |u|}
\sum_{0\leq \tilde{c}_{j_m} \leq q-1: \tilde{c}_{j_m} \neq \pi_{j_m} (a_{i_1, j_m}) }
\Big]
\nonumber \\
&&\times
\nu^*_{u_1, u_{j_2},\ldots, u_{j_{|u|}} } [
\pi_{1} (a_{i_1,1}), 
\tilde{c}_1; 
\tilde{c}_{j_2};\ldots;
\tilde{c}_{j_{|u|}} ]
\Big\}
\nonumber \\
&=&
\frac{ {\cal I}\{ r=1\} }{q^2}
\sum_{i_1 =1}^{q^2}
\sum_{0\leq u_1,\ldots, u_d \leq 2: u_k = 2 \Leftrightarrow k= 1, |u|\geq 2 }
\frac{(-1)^{|u|} }{ (q-1)^{|u|} }
\sum_{i_2\neq i_1} 
{\cal I}\{ a_{i_1,1} = a_{i_2, 1}\}
\nonumber \\
&&\times
E \{ \nu^*_{u_1,u_{j_2}, \ldots, u_{j_{|u|}} } [
\pi_{1} (a_{i_1,1}), 
\pi_{1; 
a_{i_1,1} } ( b_{i_1, 1, 2} );  
\pi_{j_2} ( a_{i_1, j_2});\ldots;
\pi_{j_{|u|}} ( a_{i_1, j_{|u|}}) ]^2 \}
\nonumber \\
&=& {\cal I}\{ r=1\}
\sum_{0\leq u_1,\ldots, u_d \leq 2: u_k = 2 \Leftrightarrow k= 1, |u| \geq 2 }
\frac{ (-1)^{|u|} }{(q-1)^{|u|-1} } 
\nonumber \\
&&\times
E \{ \nu^*_{u_1, u_{j_2},\ldots, u_{j_{|u|}} } [
\pi_{1} (a_{1,1}), 
\pi_{1; 
a_{1,1} } ( b_{1, 1, 2} );  
\pi_{j_2} ( a_{1, j_2});\ldots;
\pi_{j_{|u|}} ( a_{1, j_{|u|}}) ]^2 \}.
\label{eq:3.68}
\end{eqnarray} 
It follows from (\ref{eq:3.67}) and (\ref{eq:3.68}) that
\begin{eqnarray}
&& q^2 E [ ( \sigma_{oal} W_{oal} - \sigma W)^2]  
\nonumber \\
&=& E [ ( \sum_{r=1}^d \sum_{1\leq j_{k_1}<\cdots < j_{k_r}\leq d} \Delta_{j_{k_1},\ldots, j_{k_r}} )^2 ]
\nonumber \\
&=& \sum_{r=1}^d \sum_{1\leq j_{k_1}<\cdots < j_{k_r}\leq d} E( \Delta_{j_{k_1},\ldots, j_{k_r}}^2 )
\nonumber \\
&=& \sum_{r=1}^d \sum_{1\leq j_{k_1} <\cdots < j_{k_r} \leq d} 
\sum_{u_1,\ldots, u_d \geq 0: u_k \geq 2 \Leftrightarrow k\in \{j_{k_1},\ldots, j_{k_r}\},
u_1+\cdots + u_d\geq 3 }
\nonumber \\
&& \hspace{0.5cm}\times
E \{ \nu_{u_1,\ldots, u_d} [\pi_j(a_{1,j}), 
\pi_{j;a_{1,j}} (b_{1,j,2}), 
\pi_{1,j,3} (0),\ldots,
\pi_{1,j,u_j} (0): 1\leq j\leq d]^2 \}
\nonumber \\
&&
+ \sum_{r=1}^d
\sum_{0\leq u_1,\ldots, u_d \leq 2: u_k = 2 \Leftrightarrow k= r, |u| \geq 2 }
\frac{ (-1)^{|u|} }{(q-1)^{|u|-1} } 
\nonumber \\
&&\hspace{0.5cm} \times
E \{ \nu_{u_1,\ldots, u_d} [\pi_j(a_{1,j}), 
\pi_{j;a_{1,j}} (b_{1,j,2}), 
\pi_{1,j,3} (0),\ldots,
\pi_{1,j,u_j} (0): 1\leq j\leq d]^2 \}.
\label{eq:3.69}
\end{eqnarray}
Finally from
(\ref{eq:3.69}) and Lemma \ref{la:a.5} (see Appendix), we have
\begin{eqnarray*}
&& q^2 E [ ( \sigma_{oal} W_{oal} - \sigma W)^2]  
\nonumber \\
&=&
\sum_{r=1}^d  \sum_{1\leq j_{k_1}< \cdots < j_{k_r} \leq d}
\sum_{u_1,\ldots, u_d \geq 0: u_k \geq 2 \Leftrightarrow k\in \{j_{k_1},\ldots, j_{k_r} \},
u_1+\cdots + u_d\geq 3 }
\sum_{t_{j_1}=0}^{q^{u_{j_1}-1}-1} \sum_{c_{j_1}=0}^{q-1} 
\cdots
\nonumber \\
&&\hspace{0.5cm}
\sum_{t_{j_{|u|} }=0}^{q^{u_{j_{|u|} }-1}-1} 
\sum_{c_{j_{|u|} }=0}^{q-1} 
\langle f,  \prod_{l=1}^{|u|}  \psi_{u_{j_l} -1,t_{j_l},c_{j_l}} \rangle^2
\sum_{\alpha=0}^{|u|} {|u| \choose \alpha} \frac{1}{q^\alpha } (1 - \frac{1}{q})^{|u|-\alpha}
\nonumber \\
&&
+ \sum_{r=1}^d \sum_{0\leq u_1,\ldots, u_d \leq 2: u_k = 2 \Leftrightarrow k= r, |u| \geq 2 }
\frac{ (-1)^{|u|} }{(q-1)^{|u|-1} } 
\sum_{t_{j_1}=0}^{q^{u_{j_1}-1}-1} \sum_{c_{j_1}=0}^{q-1} 
\cdots
\nonumber \\
&&\hspace{0.5cm}
\sum_{t_{j_{|u|} }=0}^{q^{u_{j_{|u|} }-1}-1} 
\sum_{c_{j_{|u|} }=0}^{q-1} 
\langle f,  \prod_{l=1}^{|u|}  \psi_{u_{j_l} -1,t_{j_l},c_{j_l}} \rangle^2
\sum_{\alpha=0}^{|u|} {|u| \choose \alpha} \frac{1}{q^\alpha } (1 - \frac{1}{q})^{|u|-\alpha}
\nonumber \\
&=&
\sum_{r=1}^d  {d\choose r}
\sum_{u_1,\ldots, u_d \geq 0: u_k \geq 2 \Leftrightarrow k\in \{1,\ldots, r\},
u_1+\cdots + u_d\geq 3 }
O( q^{2|u| - 2(u_1+\cdots+ u_d)} )
\nonumber \\
&&
+ \sum_{0\leq u_1,\ldots, u_d \leq 2: u_k = 2 \Leftrightarrow k= 1, |u| \geq 2 }
\frac{ (-1)^{|u|} d}{(q-1)^{|u|-1} } 
O( q^{2|u| - 2(u_1+\cdots+ u_d)} )
\nonumber \\
&=& O(q^{-2}),
\end{eqnarray*}
as $q\rightarrow \infty$.
Using Chebyshev's inequality, we conclude that
$q ( \sigma_{oal} W_{oal} - \sigma W) \rightarrow 0$ in probability as $q\rightarrow \infty$.
Next we observe from (\ref{eq:3.402}) that
\begin{displaymath}
E \Big\{ \{ \frac{1}{q } \sum_{i=1}^{q^2 } 
\sum_{k=1}^d
\nu_2^* [ \pi_k (a_{i,k}), 
\pi_{i,k,2} (0) ] \}^2 \Big\} = O(\frac{1}{q^2}),
\end{displaymath}
as $q\rightarrow\infty$. Again by Chebyshev's inequality, we have $q^{-1} \sum_{i=1}^{q^2 } 
\sum_{k=1}^d
\nu_2^* [ \pi_k (a_{i,k}), 
\pi_{i,k,2} (0) ] \rightarrow 0$
in probability as $q\rightarrow \infty$. Thus we conclude from (\ref{eq:4.6}) that to prove that
$q ( \sigma_{oas} W_{oas} - \sigma W) \rightarrow 0$ in probability as $q\rightarrow \infty$, it suffices to show
that
\begin{eqnarray*}
&& \frac{1}{q } \sum_{i=1}^{q^2 } 
\sum_{u_1,\ldots, u_d \geq 0: u_1+\cdots+ u_d\geq 3}
\nu_{u_1,\ldots, u_d} 
[ \pi_j (a_{i,j}), 
\pi_{i,j,2} (0),\ldots, 
\pi_{i,j,u_j} (0): 1\leq j\leq d ] - q \sigma W
\nonumber \\
&\rightarrow & 0,
\end{eqnarray*}
in probability as $q\rightarrow \infty$. The proof of the latter statement is similar to the
 proof that $q(\sigma_{oal} W_{oal} - \sigma W) \rightarrow 0$ and hence will be omitted. 
This proves Proposition \ref{pn:4.1}. \hfill $\Box$ 

{\sc Remark.}
We observe from Theorem \ref{tm:3.1} and Proposition \ref{pn:3.1} that 
for $\int_{[0,1)^d} f^2_{rem} (x) dx >0$, we have
\begin{displaymath}
\frac{\sigma^2_{oal}}{\sigma^2} = 1 + O(\frac{1}{q}),\hspace{0.5cm}
\frac{\sigma^2_{oas}}{\sigma^2} = 1 + O(\frac{1}{q}),\hspace{0.5cm} \mbox{as $q\rightarrow \infty$.}
\end{displaymath}
Thus we conclude from Proposition \ref{pn:4.1} that to show 
$W_{oal}$ and $W_{oas}$ both tend in law  to the standard (univariate) normal distribution
as $q\rightarrow \infty$, it suffices to
show that the proxy statistic $W$ tends in law to that distribution.

Suppose $d\geq 3$. For $\ell =1,\ldots, d-2$, we define
\begin{eqnarray}
\sigma_\ell^2 &=& E \Big\{ \{ \frac{1}{q^2} \sum_{i=1}^{q^2} \sum_{0\leq u_1,\ldots, u_d \leq 1: |u|=\ell +2} 
\nu^* [\pi_{j_1} (a_{i,j_1});\ldots; \pi_{j_{|u|}} (a_{i, j_{|u|}})] \}^2 \Big\},
\nonumber \\
V_\ell &=& \frac{1}{ q^2 \sigma_\ell } \sum_{i=1}^{q^2} \sum_{0\leq u_1,\ldots, u_d \leq 1: |u|=\ell +2} 
\nu^* [\pi_{j_1} (a_{i,j_1});\ldots; \pi_{j_{|u|}} (a_{i, j_{|u|}})],
\nonumber \\
&& \mbox{and $V = ( V_1,\ldots, V_{d-2})'$ where $|u| = \sum_{i=1}^d {\cal I}\{ u_i\geq 1\}$.}
\label{eq:5.30}
\end{eqnarray}
We shall prove that the random vector $V$ converges weakly to the standard $(d-2)$-variate
normal distribution $\Phi_{d-2}$ as $q$ tends to infinity.
To do so, we shall use the multivariate normal version of Stein's method  [see Stein (1972), (1986)] as given
in G\"{o}tze (1991) and Bolthausen and G\"{o}tze (1993).

Let ${\cal A}$ be a class of measurable functions from ${\mathbb R}^{d-2} \rightarrow {\mathbb R}$ such that
$\sup_{v\in {\mathbb R}^{d-2}} |g(v)| \leq 1$ for all $g\in {\cal A}$.
For $g\in {\cal A}$ and $\delta>0$, define
\begin{eqnarray*}
g^+_\delta (v) &=& \sup\{ g(v+ y): \|y\| \leq \delta \}, \hspace{0.5cm} \forall v\in {\mathbb R}^{d-2},
\nonumber \\
g^-_\delta (v) &=& \inf\{ g(v+y): \|y\| \leq \delta \}, \hspace{0.5cm} \forall v \in {\mathbb R}^{d-2},
\nonumber \\
\omega (g, \delta) &=& \int_{{\mathbb R}^{d-2}} [ g^+_\delta (y) - g^-_\delta (y) ] \Phi_{d-2} (dy).
\end{eqnarray*}
We further assume that ${\cal A}$ is closed under supremum and affine transformations, that is,
$g\in {\cal A}$ implies that $g^+_\delta \in {\cal A}$, $g^-_\delta \in {\cal A}$ and $g\circ T \in {\cal A}$ 
whenever $T: {\mathbb R}^{d-2} \rightarrow {\mathbb R}^{d-2}$ is affine. Finally we assume that there exists a constant $\Delta \geq 2 \sqrt{d-2}$
such that 
\begin{equation}
\sup\{ \omega (g,\delta): g \in {\cal A} \} \leq \Delta \delta, \hspace{0.5cm} \forall \delta>0.
\label{eq:5.5}
\end{equation}
We observe from Bolthausen and G\"{o}tze (1993) that 
${\cal A}$ can be taken to be the class of all indicator functions of measurable
convex sets in ${\mathbb R}^{d-2}$.
For $h\in {\cal A}$ and $0\leq t\leq 1$, define
\begin{eqnarray}
\chi_t (v |h) &=& \int_{{\mathbb R}^{d-2}} [ h(y) - h( t^{1/2} y + (1-t)^{1/2} v )] \Phi_{d-2} (dy),
\label{eq:5.1} \\
\psi_t (v) &=& \frac{1}{2} \int_t^1 \chi_s (v|h) \frac{ ds}{1-s}, \hspace{0.5cm}\forall v\in {\mathbb R}^{d-2}.
\nonumber 
\end{eqnarray}
Then $-\chi_0 (v|h) = h(v) - \Phi_{d-2} (h)$ where  $\Phi_{d-2} (h) = Eh(Z)$ and $Z$ is a random vector having distribution $\Phi_{d-2}$.
The following two lemmas are due to G\"{o}tze (1991).
Since the proofs are only briefly sketched in G\"{o}tze (1991),
detailed proofs of Lemmas \ref{la:5.1} and
\ref{la:5.2} are given below.

\begin{la} \label{la:5.1}
For $0< t< 1$ and  $v= (v_1,\ldots, v_{d-2})'\in {\mathbb R}^{d-2}$, we have
\begin{equation}
\sum_{i=1}^{d-2} \frac{ \partial^2}{\partial v_i^2} \psi_t (v) -\sum_{i=1}^{d-2} v_i \frac{\partial}{\partial v_i} \psi_t (v)
= -\chi_t (v|h).
\label{eq:5.4}
\end{equation}
There exists a  constant $c$ (depending only on $d$) such that 
\begin{displaymath}
\sup_{1\leq i,j\leq d-2} \sup_{v\in {\mathbb R}^{d-2}} | \frac{\partial^2}{\partial v_i \partial v_j} \psi_t (v) | \leq \|h\|_\infty \log(1/t),
\end{displaymath}
where $\|h\|_\infty = \sup_{v\in {\mathbb R}^{d-2}} |h(v)|$ and
\begin{eqnarray*}
&& 
\sup_{1\leq i,j,k\leq d-2} | \int_{{\mathbb R}^{d-2}} [ \frac{ \partial^3}{ \partial v_i\partial v_j \partial v_k} \psi_t (v) ] Q(dv)|
\nonumber \\
&\leq & \frac{c}{t^{1/2} } \sup\{ | \int_{{\mathbb R}^{d-2}} h( s v + y) Q(dv) |: 0\leq s \leq 1, y\in {\mathbb R}^{d-2} \}
\end{eqnarray*}
for all finite signed measures $Q$ on ${\mathbb R}^{d-2}$.
\end{la}
{\sc Proof.}  We first observe from (\ref{eq:5.1}) that
\begin{displaymath}
\chi_t (v|h) = \int_{{\mathbb R}^{d-2}} h(y) \Phi_{d-2}(dy) -
\int_{{\mathbb R}^{d-2}} \frac{ h(y)}{ (2 \pi t)^{(d-2)/2}} e^{-\sum_{l=1}^{d-2} (y_l - (1 -t)^{1/2} v_l )^2/(2t) } dy.
\end{displaymath}
Now we consider the partial differential equation
\begin{equation}
2 (1-t) \frac{\partial}{\partial t} u_t (v) = \sum_{i=1}^{d-2} \frac{\partial^2}{\partial v_i^2} u_t(v)
-\sum_{i=1}^{d-2} v_i \frac{\partial}{\partial v_i} u_t (v),\hspace{0.5cm}\forall t>0.
\label{eq:5.2}
\end{equation}
We observe that a trivial  solution to (\ref{eq:5.2}) is $u_t(v) = \int_{{\mathbb R}^{d-2}} h(y) \Phi_{d-2} (dy)$.
Writing
\begin{displaymath}
\tilde{u}_{y,t} (v) = \frac{ 1}{(2 \pi t)^{(d-2)/2}} e^{-\sum_{l=1}^{d-2} (y_l - (1-t)^{1/2} v_l )^2/(2t) },
\end{displaymath}
we have
\begin{eqnarray*}
\sum_{i=1}^{d-2} v_i \frac{\partial}{\partial v_i} \tilde{u}_{y,t} (v)
& = &
\frac{ (1-t)^{1/2} }{ t (2 \pi t)^{(d-2)/2}}  
 e^{-\sum_{l=1}^{d-2} (y_l - (1-t)^{1/2} v_l )^2/(2t) }
\sum_{i=1}^{d-2}
v_i (y_i - (1-t)^{1/2} v_i ),
\nonumber \\
\sum_{i=1}^{d-2} \frac{\partial^2}{\partial v_i^2} \tilde{u}_{y,t} (v)
&=& \frac{1-t}{t^2  (2 \pi t)^{(d-2)/2}} 
e^{-\sum_{l=1}^{d-2} (y_l - (1-t)^{1/2} v_l )^2/(2t) }
\sum_{i=1}^{d-2}
[ (y_i - (1-t)^{1/2} v_i)^2 - t]
\nonumber \\
&=& \frac{1-t}{t^2  (2 \pi t)^{(d-2)/2}} 
e^{-\sum_{l=1}^{d-2} (y_l - (1-t)^{1/2} v_l )^2/(2t) }
\sum_{i=1}^{d-2}
(y_i - (1-t)^{1/2} v_i)^2 
\nonumber \\
&&
- \frac{(d-2) (1-t) }{t  (2 \pi t)^{(d-2)/2}} 
e^{-\sum_{l=1}^{d-2} (y_l - (1-t)^{1/2} v_l )^2/(2t) },
\end{eqnarray*}
and
\begin{eqnarray*} 
2(1-t) \frac{\partial}{\partial t} \tilde{u}_{y,t} (v) 
&=& - \frac{ (d-2 ) (1-t)}{ t (2 \pi t)^{(d-2)/2} }
e^{-\sum_{l=1}^{d-2} (y_l - (1-t)^{1/2} v_l )^2/(2t) }
\nonumber \\
&&
- \frac{ 1-t}{  (2 \pi t)^{(d-2)/2} }
e^{-\sum_{l=1}^{d-2} (y_l - (1-t)^{1/2} v_l )^2/(2t) }
\nonumber \\
&&\hspace{0.5cm}\times
\sum_{i=1}^{d-2} [
- \frac{ (y_i - (1-t)^{1/2} v_i )^2 }{ t^2}
+ \frac{ v_i (y_i - (1- t)^{1/2} v_i) }{ t (1- t)^{1/2} } ].
\end{eqnarray*}
Thus $\tilde{u}_{y,t} (v)$ is a solution to (\ref{eq:5.2}) too for all $y\in {\mathbb R}^{d-2}$.
Using the dominated convergence theorem, we have
\begin{eqnarray}
2 (1-t) \frac{\partial}{\partial t} \chi_t (v| h) 
&=& 
-2 (1-t) \frac{\partial}{\partial t} \int_{{\mathbb R}^{d-2}} h(y) \tilde{u}_{y,t} (v) dy
\nonumber \\
&=&
-2 (1-t) \int_{{\mathbb R}^{d-2}} h(y) \frac{\partial}{\partial t} \tilde{u}_{y,t} (v) dy
\nonumber \\
&=& \sum_{i=1}^{d-2} \int_{{\mathbb R}^{d-2}} h(y) v_i \frac{\partial}{\partial v_i} \tilde{u}_{y,t} (v) dy
- \sum_{i=1}^{d-2} \int_{{\mathbb R}^{d-2}} h(y) \frac{\partial^2 }{\partial v_i^2} \tilde{u}_{y,t} (v) dy
\nonumber \\
&=& \sum_{i=1}^{d-2} v_i \frac{\partial}{\partial v_i}
\int_{{\mathbb R}^{d-2}} h(y) \tilde{u}_{y,t} (v) dy
- \sum_{i=1}^{d-2} \frac{\partial^2 }{\partial v_i^2} \int_{{\mathbb R}^{d-2}} h(y) \tilde{u}_{y,t} (v) dy
\nonumber \\
&=& 
\sum_{i=1}^{d-2} \frac{\partial^2 }{\partial v_i^2} \chi_t( v|h)
- \sum_{i=1}^{d-2} v_i \frac{\partial}{\partial v_i}
\chi_t (v|h).
\label{eq:4.68}
\end{eqnarray}
We further observe that
\begin{eqnarray*}
\frac{\partial}{\partial v_i} \chi_t (v| h)
&=&
-\frac{\partial}{\partial v_i} \int_{{\mathbb R}^{d-2}}  \frac{ h(y)}{ (2 \pi t)^{(d-2)/2}} 
e^{-\sum_{l=1}^{d-2} (y_l -(1-t)^{1/2} v_l)^2/(2t) }
dy
\nonumber \\
&=&
- \frac{ (1-t)^{1/2} }{t (2 \pi t)^{(d-2)/2}}
\int_{{\mathbb R}^{d-2}}  h(y) (y_i - (1-t)^{1/2} v_i) e^{-\sum_{l=1}^{d-2} (y_l -(1-t)^{1/2} v_l)^2/(2t) } dy,
\nonumber \\
\frac{\partial^2 }{\partial v_i^2 } \chi_t (v|h) 
&=& - \frac{ (1-t)^{1/2} }{t (2 \pi t)^{(d-2)/2}} \frac{\partial}{\partial v_i} 
\int_{{\mathbb R}^{d-2}}  h(y) (y_i - (1-t)^{1/2} v_i) e^{-\sum_{l=1}^{d-2} (y_l -(1-t)^{1/2} v_l)^2/(2t) } dy
\nonumber \\
&=& \frac{ 1-t }{t (2 \pi t)^{(d-2)/2}}
\int_{{\mathbb R}^{d-2}}  h(y) e^{-\sum_{l=1}^{d-2} (y_l -(1-t)^{1/2} v_l)^2/(2t) } dy
\nonumber \\
&& 
- \frac{ 1-t }{t^2 (2 \pi t)^{(d-2)/2}}
\int_{{\mathbb R}^{d-2}}  h(y) (y_i - (1-t)^{1/2} v_i)^2 e^{-\sum_{l=1}^{d-2} (y_l -(1-t)^{1/2} v_l)^2/(2t) } dy,
\nonumber \\
\frac{\partial^3 }{\partial v_i^3 } \chi_t (v|h) 
&=& \frac{3 (1-t)^{3/2} }{t^2 (2 \pi t)^{(d-2)/2}}
\int_{{\mathbb R}^{d-2}}  h(y) (y_i - (1-t)^{1/2} v_i) e^{-\sum_{l=1}^{d-2} (y_l -(1-t)^{1/2} v_l)^2/(2t) } dy
\nonumber \\
&&
- \frac{ (1-t)^{3/2} }{t^3 (2 \pi t)^{(d-2)/2}}
\int_{{\mathbb R}^{d-2}}  h(y) (y_i - (1-t)^{1/2} v_i)^3 e^{-\sum_{l=1}^{d-2} (y_l -(1-t)^{1/2} v_l)^2/(2t) } dy.
\end{eqnarray*}
For $1\leq i\neq j\leq d-2$, we have
\begin{eqnarray*}
\frac{\partial^3 }{\partial v_i^2 \partial v_j } \chi_t (v|h) 
&=& \frac{ (1-t)^{3/2} }{t^2 (2 \pi t)^{(d-2)/2}}
\int_{{\mathbb R}^{d-2}}  h(y) (y_j - (1-t)^{1/2} v_j) e^{-\sum_{l=1}^{d-2} (y_l -(1-t)^{1/2} v_l)^2/(2t) } dy
\nonumber \\
&&
- \frac{ (1-t)^{3/2} }{t^3 (2 \pi t)^{(d-2)/2}}
\int_{{\mathbb R}^{d-2}}  h(y) (y_i - (1-t)^{1/2} v_i)^2
(y_j - (1-t)^{1/2} v_j)  
\nonumber \\
&&\hspace{0.5cm}\times
e^{-\sum_{l=1}^{d-2} (y_l -(1-t)^{1/2} v_l)^2/(2t) } dy,
\end{eqnarray*}
and for $i,j,k$ all distinct, we have
\begin{eqnarray*}
\frac{\partial^3 }{\partial v_i \partial v_j \partial v_k} \chi_t (v|h) 
&=&
- \frac{ (1-t)^{3/2} }{t^3 (2 \pi t)^{(d-2)/2}}
\int_{{\mathbb R}^{d-2}}  h(y) (y_i - (1-t)^{1/2} v_i)
(y_j - (1-t)^{1/2} v_j)  
\nonumber \\
&&\hspace{0.5cm}\times
(y_k - (1-t)^{1/2} v_k)  
e^{-\sum_{l=1}^{d-2} (y_l -(1-t)^{1/2} v_l)^2/(2t) } dy.
\end{eqnarray*}
Consequently,
\begin{eqnarray}
| \frac{\partial}{\partial v_i} \chi_t (v|h) |
&\leq &
 \frac{ (1-t)^{1/2} \| h\|_\infty }{t (2 \pi t)^{(d-2)/2}}
\int_{{\mathbb R}^{d-2}}  |y_i | e^{-\sum_{l=1}^{d-2} y_l^2/(2t) } dy
\hspace{0.1cm} = \hspace{0.1cm}
 \frac{ (1-t)^{1/2} \| h\|_\infty }{t^{1/2} },
\nonumber \\
| \frac{\partial^2}{\partial v_i^2 } \chi_t (v|h) |
&\leq &
 \frac{2 (1-t) \| h\|_\infty }{t }.
\label{eq:5.3}
\end{eqnarray}
Since $\chi_1 (v|h) = 0$, it follows from (\ref{eq:4.68}) that
\begin{eqnarray*}
-\chi_t (v|h) &=&
\chi_1 (v|h) - \chi_t (v|h) 
\nonumber \\
&=& \int_t^1 [ \frac{\partial}{\partial s} \chi_s (v|h) ] ds
\nonumber \\
&=& \frac{1}{2} \int_t^1 \sum_{i=1}^{d-2} \frac{\partial^2}{\partial v_i^2} \chi_s (v|h) \frac{ds}{1-s}
-\frac{1}{2} \int_t^1 \sum_{i=1}^{d-2} v_i \frac{\partial}{\partial v_i} \chi_s (v|h) \frac{ds}{1-s}
\nonumber \\
&=& \sum_{i=1}^{d-2} \frac{\partial^2}{\partial v_i^2} \psi_t (v)
- \sum_{i=1}^{d-2} v_i \frac{\partial}{\partial v_i} \psi_t (v), \hspace{0.5cm}\forall 0<t<1.
\end{eqnarray*}
The interchange of integration and partial differentiation is justified via (\ref{eq:5.3}) and the dominated convergence
theorem.
This proves (\ref{eq:5.4}).
Now we observe that for $v\in {\mathbb R}^{d-2}$,
\begin{eqnarray*}
\frac{\partial^2}{\partial v_i^2} \psi_t (v) &=& \frac{1}{2} \int_t^1 [ \frac{\partial^2}{\partial v_i^2} \chi_s (v| h) ] \frac{ds}{1-s}
\nonumber \\
&=& \frac{1}{2} \int_t^1 \frac{ds}{ s (2 \pi s)^{(d-2)/2} } \int_{{\mathbb R}^{d-2}}
h(y) e^{-\sum_{l=1}^{d-2} [y_l - (1-s)^{1/2} v_l]^2/(2s)}
\nonumber \\
&&\hspace{0.5cm}\times 
\{ 1 -  \frac{[y_i - (1-s)^{1/2} v_i ]^2 }{ s} \} dy,
\\
\frac{\partial^2}{\partial v_i\partial v_j} \psi_t (v) &=& \frac{1}{2} \int_t^1 \frac{ds}{s(2\pi s)^{(d-2)/2}} \int_{{\mathbb R}^{d-2}} h(y)
e^{-\sum_{l=1}^{d-2} [y_l -(1-s)^{1/2} v_l]^2/(2s)}
\\
&&\hspace{0.5cm}\times \frac{[y_i- (1-s)^{1/2} v_i] [y_j-(1-s)^{1/2} v_j]}{s} dy, \hspace{0.5cm}\mbox{if $1\leq i\neq j\leq d-2$},
\end{eqnarray*}
and hence
\begin{displaymath}
\sup_{1\leq i,j\leq d-2} \sup_{v\in {\mathbb R}^{d-2}} | \frac{\partial^2}{\partial v_i \partial v_j} \psi_t (v) |
\leq \|h\|_\infty \log(1/t).
\end{displaymath}

Next we observe that
\begin{eqnarray*}
\int_{{\mathbb R}^{d-2}} [ \frac{\partial^3}{\partial v_i^3 } \psi_t (v) ]Q(dv)
&=& \frac{1}{2}
\int_{{\mathbb R}^{d-2}} Q(dv) \frac{\partial^3}{\partial v_i^3} 
\int_t^1 \chi_s (v|h) \frac{ds}{1-s} 
\nonumber \\
&=&
\frac{1}{2} \int_{{\mathbb R}^{d-2}} Q(dv)  
\int_t^1
\frac{ (1-s)^{1/2} ds }{s^2 (2 \pi s)^{(d-2)/2}} 
\nonumber \\
&& \times
\Big[ 3 
\int_{{\mathbb R}^{d-2}}  h(y) (y_i - (1-s)^{1/2} v_i) e^{-\| y -(1-s)^{1/2} v \|^2/(2s) } dy
\nonumber \\
&& \hspace{0.5cm}
- \frac{ 1 }{s }
\int_{{\mathbb R}^{d-2}}  h(y) (y_i - (1-s)^{1/2} v_i)^3 e^{-\| y -(1-s)^{1/2} v \|^2/(2s) } dy
\Big]
\nonumber \\
&=&
\frac{1}{2}
\int_t^1
\frac{ (1-s)^{1/2} ds }{s^2 (2 \pi s)^{(d-2)/2}} 
\int_{{\mathbb R}^{d-2}} \Big[ \int_{{\mathbb R}^{d-2}} h( y + (1-s)^{1/2} v) Q(dv) \Big]
\nonumber \\
&& \times \Big[ 3 y_i e^{-\| y \|^2/(2s) } - \frac{1}{s} y_i^3 e^{-\| y \|^2/(2s)} \Big] dy,
\end{eqnarray*}
where $\|.\|$ is the usual Euclidean norm in ${\mathbb R}^{d-2}$ and
\begin{eqnarray*}
| \int_{{\mathbb R}^{d-2}} [ \frac{\partial^3}{\partial v_i^3 } \psi_t (v) ]Q(dv) |
&\leq &
\frac{1}{2} \sup\{ | \int_{{\mathbb R}^{d-2}} h( y + s v) Q(dv) | : 0\leq s\leq 1, y\in {\mathbb R}^{d-2} \}
\nonumber \\
&&\times
\int_t^1
\frac{ (1-s)^{1/2} ds }{s^{3/2} (2 \pi )^{(d-2)/2}} 
\int_{{\mathbb R}^{d-2}}
( 3 |y_i|  + |y_i^3| ) e^{-\| y \|^2/2} dy
\nonumber \\
&\leq & \frac{c}{t^{1/2}}
\sup\{ | \int_{{\mathbb R}^{d-2}} h( y + s v) Q(dv) | : s\in [0.1], y\in {\mathbb R}^{d-2} \}.
\end{eqnarray*}
In a similar way, we have for $i,j, k$ all distinct,
\begin{eqnarray*}
\int_{{\mathbb R}^{d-2}} [ \frac{\partial^3}{\partial v_i^2 \partial v_j } \psi_t (v) ]Q(dv)
&=& \frac{1}{2} \int_{{\mathbb R}^{d-2}} Q(dv) \int_t^1 \frac{(1-s)^{1/2} ds}{ s^2 (2 \pi s)^{(d-2)/2}}
\nonumber \\
&&\times 
\int_{{\mathbb R}^{d-2}}  h(y + (1-s)^{1/2} v) (y_j - \frac{ y_i^2 y_j}{s})  e^{-\| y \|^2/(2s) } dy,
\nonumber \\
\int_{{\mathbb R}^{d-2}} [ \frac{\partial^3}{\partial v_i \partial v_j \partial v_k} \psi_t (v) ]Q(dv)
&=& \frac{1}{2} \int_{{\mathbb R}^{d-2}} Q(dv) \int_t^1 \frac{(1-s)^{1/2} ds}{ s^3 (2 \pi s)^{(d-2)/2}}
\nonumber \\
&&\times 
\int_{{\mathbb R}^{d-2}}  h(y + (1-s)^{1/2} v) y_i y_j y_k  e^{-\| y \|^2/(2s) } dy,
\end{eqnarray*}
and
\begin{eqnarray*}
| \int_{{\mathbb R}^{d-2}} [ \frac{\partial^3}{\partial v_i^2 \partial v_j } \psi_t (v) ]Q(dv) |
&\leq &
\frac{1}{2} \sup\{ | \int_{{\mathbb R}^{d-2}} h( y + s v) Q(dv) | : s\in [0,1], y\in {\mathbb R}^{d-2} \}
\nonumber \\
&&\times
\int_t^1
\frac{ (1-s)^{1/2} ds }{s^{3/2} (2 \pi )^{(d-2)/2}} 
\int_{{\mathbb R}^{d-2}}
(  |y_j| + |y_i^2 y_j| ) e^{-\| y \|^2/2} dy
\nonumber \\
&\leq & \frac{c}{t^{1/2}}
\sup\{ | \int_{{\mathbb R}^{d-2}} h( y + s v) Q(dv) | : s\in [0,1], y\in {\mathbb R}^{d-2} \},
\nonumber \\
| \int_{{\mathbb R}^{d-2}} [ \frac{\partial^3}{\partial v_i \partial v_j \partial v_k} \psi_t (v) ]Q(dv) |
&\leq & \frac{c}{t^{1/2}}
\sup\{ | \int_{{\mathbb R}^{d-2}} h( y + s v) Q(dv) | : s\in [0,1], y\in {\mathbb R}^{d-2} \}.
\end{eqnarray*}
This proves Lemma \ref{la:5.1}. \hfill $\Box$

\begin{la} \label{la:5.2}
Suppose that (\ref{eq:5.5}) holds. Let $0 < \varepsilon < 1/2$ and $Q$ be a probability distribution on ${\mathbb R}^{d-2}$. Then
there exists a constant $c_d$ depending only on $d$ such that
\begin{displaymath}
\sup_{g\in {\cal A}} | \int_{{\mathbb R}^{d-2}} g(v) [ Q(dv) - \Phi_{d-2} (dv) ] |
\leq  c_d [
\sup_{h\in {\cal A}} | \int_{{\mathbb R}^{d-2}} \chi_{\varepsilon^2} (v|h) Q(dv) |
+ \Delta \varepsilon].
\end{displaymath}
\end{la}
{\sc Proof.}
Let $m_d \in {\mathbb Z}^+$ and $\alpha >1/2$ be constants satisfying
\begin{displaymath}
\Phi_{d-2} (\{ \| x\| \leq \sqrt{3 m_d/4 } : x\in {\mathbb R}^{d-2} \} ) = \alpha.
\end{displaymath}
We further write $b= (1-\varepsilon^2)^{1/2}$, $\delta = \varepsilon/b$
and $\Phi_{d-2, \varepsilon}(A) = \Phi_{d-2} (A/\varepsilon), Q_\varepsilon (A) = Q(A/\varepsilon)$
for all Borel sets $A \subset {\mathbb R}^{d-2}$.
For $g\in {\cal A}$, we define
\begin{eqnarray*}
\gamma^* (g;\varepsilon) &=& \sup \Big\{  \max\{ \int_{{\mathbb R}^{d-2}} g^+_\varepsilon (v+y) (Q- \Phi_{d-2})* 
\Phi_{d-2, \varepsilon/\sqrt{m_d (1-\varepsilon^2)} } (dv),
\nonumber \\
&&\hspace{0.5cm}
-\int_{{\mathbb R}^{d-2}} g^-_\varepsilon (v+y) (Q- \Phi_{d-2})* 
\Phi_{d-2, \varepsilon/\sqrt{m_d (1- \varepsilon^2)} } (dv) \} : y \in {\mathbb R}^{d-2} \Big\},
\nonumber \\
\tau^* (g; \varepsilon) &=& \sup \Big\{
\max\{ \int_{{\mathbb R}^{d-2}} [ g^+_\varepsilon (v+y) - g(v + y) ] \Phi_{d-2}(dv),
\nonumber \\
&&\hspace{0.5cm}
- \int_{{\mathbb R}^{d-2}} [ g^-_\varepsilon (v+y) - g(v + y) ] \Phi_{d-2}(dv) \} : y\in {\mathbb R}^{d-2} \Big\}.
\end{eqnarray*}
Since ${\cal A}$ is closed under affine transformations, it follows from (\ref{eq:5.5}) that
$\tau^*(g; \varepsilon) \leq  \Delta \varepsilon$.
We further observe that
\begin{eqnarray*}
&&
\sup_{h \in {\cal A}}
| \int_{{\mathbb R}^{d-2}} h (v) (Q- \Phi_{d-2})* \Phi_{d-2, \varepsilon/\sqrt{1-\varepsilon^2}} (dv) |
\nonumber \\
&= &
\sup_{ h \in {\cal A}}
| \int_{{\mathbb R}^{d-2}} \int_{{\mathbb R}^{d-2}} h( v + \frac{\varepsilon y}{\sqrt{1- \varepsilon^2}} ) Q(dv) \Phi_{d-2} (dy)
\nonumber \\
&&\hspace{0.5cm}
- \int_{{\mathbb R}^{d-2}} \int_{{\mathbb R}^{d-2}}  h( v + \frac{\varepsilon y}{\sqrt{1-\varepsilon^2}} )  \Phi_{d-2} (dv) \Phi_{d-2} (dy) |
\nonumber \\
&=&
\sup_{h \in {\cal A}}
| \int_{{\mathbb R}^{d-2}} \int_{{\mathbb R}^{d-2}} h( (1-\varepsilon^2)^{1/2} v + \varepsilon y ) Q(dv) \Phi_{d-2} (dy)
\nonumber \\
&&\hspace{0.5cm}
- \int_{{\mathbb R}^{d-2}} \int_{{\mathbb R}^{d-2}}
h( (1-\varepsilon^2)^{1/2} v + \varepsilon y )  \Phi_{d-2} (dv) \Phi_{d-2} (dy) |
\nonumber \\
&= &
\sup_{h \in {\cal A}}
| \int_{{\mathbb R}^{d-2}} \int_{{\mathbb R}^{d-2}} [ h( (1-\varepsilon^2)^{1/2} v + \varepsilon y ) - h(y) ] Q(dv) \Phi_{d-2} (dy) |
\nonumber \\
&=& \sup_{h\in {\cal A}} | \int_{{\mathbb R}^{d-2}} \chi_{\varepsilon^2} (v|h) Q(dv) |.
\end{eqnarray*}
The last equality uses (\ref{eq:5.1}). 
Finally using Lemma 11.4 of Bhattacharya and Rao (1986), page 95, we obtain
\begin{eqnarray*}
&& \sup_{g \in {\cal A}} | \int_{{\mathbb R}^{d-2}} g(v) ( Q - \Phi_{d-2} ) (dv) |
\nonumber \\
&\leq &
\sup_{g\in {\cal A}}
\frac{1}{2\alpha-1} [ \gamma^* (g; \varepsilon) + \alpha \tau^* (g; 2 \varepsilon) + (1-\alpha) \tau^* (g; \varepsilon)
]
\nonumber \\
&\leq & 
\frac{1}{2\alpha-1} \sup_{h\in {\cal A}}
| \int_{{\mathbb R}^{d-2}} h (v) (Q- \Phi_{d-2})* \Phi_{d-2, \varepsilon/\sqrt{m_d (1-\varepsilon^2)} } (dv) |
+ \frac{ (1+ \alpha) \Delta \varepsilon }{ 2\alpha -1}
\nonumber \\
&\leq & 
\frac{1}{ (2\alpha-1)^2} \sup_{ h\in {\cal A}}
| \int_{{\mathbb R}^{d-2}} h (v) (Q- \Phi_{d-2})* \Phi_{d-2, \varepsilon\sqrt{2}/\sqrt{m_d (1- \varepsilon^2)} }  
(dv) | 
\nonumber \\
&& + (1+ \alpha) \Delta \varepsilon [ \frac{1}{2\alpha-1} + \frac{1}{(2\alpha -1)^2} ] 
\nonumber \\
&\leq & \frac{1}{ (2\alpha-1)^{m_d} } \sup_{ h\in {\cal A}}
| \int_{{\mathbb R}^{d-2}} h (v) (Q- \Phi_{d-2})* \Phi_{d-2, \varepsilon/\sqrt{1- \varepsilon^2} }  
(dv) |
+ \sum_{i=1}^{m_d} \frac{(1+\alpha) \Delta \varepsilon }{ (2\alpha-1)^i }
\nonumber \\
&=&
\frac{1}{ (2\alpha-1)^{m_d} }
\sup_{h\in {\cal A}} | \int_{{\mathbb R}^{d-2}} \chi_{\varepsilon^2} (v|h) Q(dv) |
+ \sum_{i=1}^{m_d} \frac{(1+\alpha) \Delta \varepsilon }{ (2\alpha-1)^i }.
\end{eqnarray*}
This proves Lemma \ref{la:5.2}.\hfill $\Box$

The theorem below is the main result of this section and is needed in the proof of Theorem \ref{tm:5.2}.
\begin{tm} \label{tm:5.1}
Suppose $d\geq 3$. Let $f:[0,1)^d\rightarrow {\mathbb R}$ be smooth with a Lipschitz continuous mixed partial of
order $d$ such that $\int_{{\mathbb R}^d} f_{rem}^2 (x) dx > 0$ and the $(d-2)$-variate random vector $V$ be as in (\ref{eq:5.30}).
Then $V$ converges to $\Phi_{d-2}$ in distribution as $q\rightarrow \infty$.
\end{tm}
{\sc Proof.} 
In this proof it suffices to take $\Delta = 2 \sqrt{d-2}$ in (\ref{eq:5.5}) and
${\cal A}$ to be the class of all indicator functions of measurable convex sets in ${\mathbb R}^{d-2}$. 
Let $J$ be a random variable uniformly distributed over $\{1,\ldots, d\}$ and $(B_1, B_2)$ be a random vector uniformly distributed over the set
\begin{displaymath}
\{ (b_1, b_2) \in \{0,\ldots, q-1\}^2: b_1\neq b_2 \}.
\end{displaymath}
$J$ and $(B_1, B_2)$ are independent of each other and are also independent of all previously defined random quantities.
Define for $j=1,\ldots, d$,
\begin{displaymath}
\tilde{\pi}_j = \left\{ \begin{array}{ll}
\pi_j & \mbox{if $j\neq J$}, \\
\tau_{B_1, B_2} \circ \pi_j & \mbox{if $j= J$},
\end{array}
\right.
\end{displaymath}
where $\tau_{B_1, B_2}$ denotes the permutation of $\{0,\ldots, q-1\}$ that transposes $B_1$ and $B_2$
leaving all other elements fixed. We further define for $\ell = 1,\ldots, d-2$,
\begin{eqnarray}
\tilde{V}_\ell 
&=& \frac{1}{q^2 \sigma_\ell } \sum_{i=1}^{q^2 } 
\sum_{0\leq u_1,\cdots, u_d \leq 1: |u| =\ell +2}
\nu^*
[ \tilde{\pi}_{j_1} (a_{i,j_1}); \ldots; 
\tilde{\pi}_{j_{|u|}} (a_{i,j_{|u|}}) ],
\nonumber \\
\tilde{V} &=& (\tilde{V}_1,\ldots, \tilde{V}_{d-2})'.
\label{eq:5.47}
\end{eqnarray}
From symmetry, we observe that $(V, \tilde{V} )$ is an exchangeable 
pair of random vectors in that $(V, \tilde{V})$ and
$(\tilde{V}, V)$ possess the same $2(d-2)$-variate distribution.
We now write
\begin{eqnarray}
\tilde{V}_\ell -  V_\ell
&=& \tilde{S}_\ell - S_\ell,
\label{eq:3.75}
\end{eqnarray}
where
\begin{eqnarray*}
\tilde{S}_\ell &=&
\frac{1}{q^2 \sigma_\ell } \sum_{i=1}^{q^2 } 
\sum_{0\leq u_1,\ldots, u_d \leq 1: |u| =\ell +2}
{\cal I} \{ J\in \{j_1,\ldots, j_{|u|}\}, \pi_J (a_{i,J}) \in \{ B_1, B_2\}  \} 
\nonumber \\
&&\hspace{0.5cm}\times
\nu^*
[ \tilde{\pi}_{j_1} (a_{i,j_1}); \ldots; 
\tilde{\pi}_{j_{|u|}} (a_{i,j_{|u|}}) ],
\end{eqnarray*}
and
\begin{eqnarray*}
S_\ell &=&
\frac{1}{q^2 \sigma_\ell } \sum_{i=1}^{q^2 } 
\sum_{0\leq u_1,\ldots, u_d \leq 1: |u| =\ell +2}
{\cal I} \{ J\in \{j_1,\ldots, j_{|u|}\}, \pi_J (a_{i,J}) \in \{B_1, B_2\} \} 
\nonumber \\
&&\hspace{0.5cm}\times
\nu^*
[ \pi_{j_1} (a_{i,j_1}); \ldots; 
\pi_{j_{|u|}} (a_{i,j_{|u|}}) ].
\end{eqnarray*}
Let ${\cal W}$ be the $\sigma$-field generated by the random quantities
\begin{displaymath}
\{ 
\pi_j(a_{i,j}):
i=1,\ldots, q^2, j=1,\ldots, d \},
\end{displaymath}
$E^{{\cal W}}$ denote conditional expectation given ${\cal W}$ and $\psi_t (.)$ be as in (\ref{eq:5.1}).
From the exchangeability of $(V, \tilde{V})$, we have for $0<\varepsilon <1/2$,
\begin{eqnarray*}
0 &=& E \{ (\tilde{V}_i - V_i )[ \frac{\partial}{\partial v_i} \psi_{\varepsilon^2} ( V )+ \frac{\partial}{\partial v_i} \psi_{\varepsilon^2} (\tilde{V})] \} \\
&=& 2E \{ [ \frac{\partial}{\partial v_i} \psi_{\varepsilon^2} ( V ) ] E^{{\cal W}} (\tilde{V}_i - V_i ) \} +E \{ (\tilde{V}_i - V_i )
[ \frac{\partial}{\partial v_i} \psi_{\varepsilon^2} (\tilde{V} )- \frac{\partial}{ \partial v_i} \psi_{\varepsilon^2} (V)]\}.
\end{eqnarray*}
We observe from Proposition \ref{pn:a.1} (see Appendix) that 
\begin{eqnarray*}
E \{ (\tilde{V}_i - V_i )
[ \frac{\partial}{\partial v_i} \psi_{\varepsilon^2} (\tilde{V} )- \frac{\partial}{\partial v_i} \psi_{\varepsilon^2} ( V )]\}
&=&
-2E \{ [ \frac{\partial}{\partial v_i} \psi_{\varepsilon^2} ( V ) ] E^{{\cal W}} (\tilde{V}_i - V_i ) \} 
\nonumber \\
&=& \frac{ 4 (i+2) }{d (q-1)} E[ V_i \frac{\partial}{\partial v_i} \psi_{\varepsilon^2} ( V ) ].
\end{eqnarray*}
Now using Lemma \ref{la:5.1}, we have
\begin{eqnarray*}
&& E[ \chi_{\varepsilon^2} (V| h)]
\nonumber \\
&=& E[ \sum_{i=1}^{d-2} V_i \frac{\partial}{\partial v_i} \psi_{\varepsilon^2} (V ) 
- \sum_{i=1}^{d-2} \frac{\partial^2}{\partial v_i^2} \psi_{\varepsilon^2} (V) ] 
\nonumber \\
&=&
\sum_{i=1}^{d-2} \frac{d (q-1)}{ 4 (i+2) } E \{ (\tilde{V}_i - V_i )
[ \frac{\partial}{\partial v_i} \psi_{\varepsilon^2} (\tilde{V} )- \frac{\partial}{\partial v_i} \psi_{\varepsilon^2} ( V )]\}
- E [ \sum_{i=1}^{d-2} \frac{\partial^2}{\partial v_i^2} \psi_{\varepsilon^2} (V) ] 
\nonumber \\
&=& 
\sum_{i=1}^{d-2} \sum_{j=1}^{d-2} \frac{d (q-1)}{ 4 (i+2) } E \{ (\tilde{V}_i - V_i ) (\tilde{V}_j - V_j) 
\int_0^1 \frac{\partial^2}{\partial v_i\partial v_j} \psi_{\varepsilon^2} (V + t( \tilde{V}- V) ) dt \}
\nonumber \\
&&
- E [ \sum_{i=1}^{d-2} \frac{\partial^2}{\partial v_i^2} \psi_{\varepsilon^2} (V) ] 
\nonumber \\
&=&
\sum_{i=1}^{d-2} \Big\{ \frac{d (q-1)}{ 4 (i+2) } E \{ (\tilde{V}_i - V_i )^2 
\int_0^1 \frac{\partial^2}{\partial v_i^2} \psi_{\varepsilon^2} (V + t( \tilde{V}- V) ) dt \}
- E [ \frac{\partial^2}{\partial v_i^2} \psi_{\varepsilon^2} (V) ] \Big\}
\nonumber \\
&&
+ \sum_{i=1}^{d-2} \sum_{j\neq i} \frac{d (q-1)}{ 4 (i+2) } E \{ (\tilde{V}_i - V_i ) (\tilde{V}_j - V_j) 
\int_0^1 \frac{\partial^2}{\partial v_i\partial v_j} \psi_{\varepsilon^2} (V + t( \tilde{V}- V) ) dt \}
\nonumber \\
&=&
\sum_{i=1}^{d-2} \Big\{ [ \frac{d (q-1)}{ 4 (i+2) } E (\tilde{S}_i - S_i )^2 
-1 ]
E [ \frac{\partial^2}{\partial v_i^2} \psi_{\varepsilon^2} (V) ] 
\nonumber \\
&& \hspace{0.5cm}
- \frac{d (q-1)}{ 4 (i+2) } E [ (\tilde{S}_i - S_i )^2 ]
E [ \frac{\partial^2}{\partial v_i^2} \psi_{\varepsilon^2} (V) ] 
+ \frac{d (q-1)}{ 4 (i+2) } E [ (\tilde{S}_i - S_i )^2
\frac{\partial^2}{\partial v_i^2} \psi_{\varepsilon^2} (V) ] 
\nonumber \\
&& \hspace{0.5cm}
+ \frac{d (q-1)}{ 4 (i+2) } E \{ (\tilde{S}_i - S_i )^2 
 \int_0^1 [ \frac{\partial^2}{\partial v_i^2} \psi_{\varepsilon^2} (V + t( \tilde{V}- V) ) 
- \frac{\partial^2}{\partial v_i^2} \psi_{\varepsilon^2} (V) ] dt  \} \Big\}
\nonumber \\
&&
+ \sum_{i=1}^{d-2} \sum_{j\neq i} \Big\{
\frac{d (q-1)}{ 4 (i+2) } E \{ (\tilde{S}_i - S_i ) (\tilde{S}_j - S_j) 
\frac{\partial^2}{\partial v_i\partial v_j} \psi_{\varepsilon^2} (V ) \}
\nonumber \\
&& \hspace{0.5cm}
+ \frac{d (q-1)}{ 4 (i+2) } E \{ (\tilde{S}_i - S_i ) (\tilde{S}_j - S_j) 
\nonumber \\
&&\hspace{1cm}\times
\int_0^1 [ \frac{\partial^2}{\partial v_i\partial v_j} \psi_{\varepsilon^2} (V + t( \tilde{V}- V) ) 
- \frac{\partial^2}{\partial v_i\partial v_j} \psi_{\varepsilon^2} (V ) ]
dt \}
\Big\}.
\end{eqnarray*}
Hence it follows from Propositions \ref{pn:a.2} to \ref{pn:a.5} (in the Appendix) that
\begin{equation}
E[ \chi_{\varepsilon^2} (V| h)] = O(\frac{\log(1/\varepsilon)}{q^{1/2}}) + O(\frac{1}{\varepsilon q^{1/2}}),
\label{eq:5.7}
\end{equation}
as $q\rightarrow \infty$ uniformly over $h\in {\cal A}$ and $\varepsilon \in (0, 1/2)$. Using (\ref{eq:5.7}) and Lemma
\ref{la:5.2}, we have
\begin{equation}
\sup_{g\in {\cal A}} | E [ g( V) ] - \int_{{\mathbb R}^{d-2}} g(v) \Phi_{d-2} (dv) | = O( \varepsilon + 
\frac{\log(1/\varepsilon)}{q^{1/2}} + \frac{1}{\varepsilon q^{1/2}}),
\label{eq:5.8}
\end{equation}
as $q\rightarrow \infty$ uniformly over $\varepsilon\in (0, 1/2)$.
By taking $\varepsilon = q^{-1/4}$, we conclude that the left hand side of (\ref{eq:5.8}) tends to $0$ as $q\rightarrow \infty$.
This implies that $V$ converges to $\Phi_{d-2}$ in distribution as $q\rightarrow \infty$
and Theorem \ref{tm:5.1} is proved.\hfill $\Box$

\section{Acknowledgements}

I would like to thank Professor Rahul Mukerjee for his encouragement in the writing of this article.
He has also spent quite a lot of time reading this article and spotting a number of oversights.
For all these and more, I am very grateful to him.

\section{Appendix}

\begin{la} \label{la:a.6}
Let $f:[0,1)^d \rightarrow {\mathbb R}$ be smooth with a Lipschitz continuous mixed partial of order $d$.
Then for $(a_1,\ldots, a_d)', (x_1,\ldots, x_d)'\in [0,1)^d$, we have
\begin{eqnarray}
f(x_1,\ldots, x_d) &=& (x_1-a_1)\ldots (x_d-a_d) \frac{\partial^d f(a_1,\ldots, a_d) }{\partial x_1\ldots \partial x_d}  
\nonumber \\
&&
+ \int_{a_1}^{x_1}\cdots \int_{a_d}^{x_d} [ \frac{ \partial^d f(t_1,\ldots, t_d) }{\partial x_1\ldots \partial x_d}
- \frac{\partial^d f(a_1,\ldots, a_d) }{\partial x_1\ldots \partial x_d} ] dt_d \ldots dt_1
\nonumber \\
&& + h_{1; a_1,\ldots, a_d} (x_2,\ldots, x_d) + h_{2; a_1,\ldots, a_d} (x_1,x_3,\ldots, x_d) +\cdots
\nonumber \\
&&
+ h_{d;a_1,\ldots, a_d} (x_1,\ldots, x_{d-1}),
\label{eq:a.100}
\end{eqnarray}
where $h_{i, a_1,\ldots,a_d}: {\mathbb R}^{d-1}\rightarrow \mathbb R$, $i=1,\ldots,d$ are suitably chosen functions.
\end{la}
{\sc Proof.}
We shall use induction on $d$.
Clearly, (\ref{eq:a.100}) holds for $d=1$ since in this case, we have
\begin{displaymath}
f(x_1) = f(a_1) 
+ \int_{a_1}^{x_1} \frac{\partial f(t_1) }{\partial x_1} dt_1.
\end{displaymath}
Now we assume that (\ref{eq:a.100}) holds for some $d=k \geq 1$. Then
\begin{eqnarray*}
f(x_1,\ldots, x_{k+1}) 
&=& 
\int_{a_1}^{x_1}\cdots \int_{a_k}^{x_k} \frac{ \partial^k f(t_1,\ldots, t_k, x_{k+1}) }{\partial x_1\ldots \partial x_k}
dt_k \ldots dt_1
\nonumber \\
&& + h_{1; a_1,\ldots, a_k} (x_2,\ldots, x_k, x_{k+1}) + h_{2; a_1,\ldots, a_k} (x_1,x_3,\ldots, x_k, x_{k+1}) +\cdots
\nonumber \\
&&
+ h_{k;a_1,\ldots, a_k} (x_1,\ldots, x_{k-1}, x_{k+1})
\nonumber \\
&=& 
\int_{a_1}^{x_1}\cdots \int_{a_k}^{x_k} \frac{ \partial^k f(t_1,\ldots, t_k, a_{k+1}) }{\partial x_1\ldots \partial x_k}
dt_k\ldots dt_1
\nonumber \\
&&\hspace{0.5cm}
+\int_{a_1}^{x_1} \cdots \int_{a_{k+1}}^{x_{k+1}}  \frac{\partial^{k+1} f(t_1,\ldots,  t_{k+1}) }{\partial x_1\ldots \partial x_{k+1} } dt_{k+1}
\ldots dt_1
\nonumber \\
&& + h_{1; a_1,\ldots, a_k} (x_2,\ldots, x_k, x_{k+1}) + h_{2; a_1,\ldots, a_k} (x_1,x_3,\ldots, x_k, x_{k+1}) +\cdots
\nonumber \\
&&
+ h_{k;a_1,\ldots, a_k} (x_1,\ldots, x_{k-1}, x_{k+1}).
\end{eqnarray*}
This shows that (\ref{eq:a.100}) holds for $d=k+1$ and Lemma \ref{la:a.6} is proved.
\hfill $\Box$

\begin{la} \label{la:a.5}
Let $f:[0,1)^d \rightarrow \mathbb R$ be smooth with a Lipschitz continuous mixed partial of order $d$.
Then
\begin{equation}
\sum_{c_j=0}^{q-1}  \langle f, \psi_{{u_j}-1, t_j, c_j} \rangle =0, \hspace{0.5cm} \mbox{if $u_j \geq 1$,}
\label{eq:a.544}
\end{equation}
and
\begin{eqnarray}
\langle f, \prod_{l=1}^{|u|} \psi_{u_{j_l}-1, t_{j_l}, c_{j_l}} \rangle
&=&
q^{|u| -\sum_{l=1}^{|u|} 3 u_{j_l}/2 }
[\frac{\partial^{|u|} }{\partial x_{j_1} \ldots \partial x_{j_{|u|}} } f_{j_1,\ldots, j_{|u|}} 
(\frac{ t_{j_1} + 0.5}{q^{u_{j_1}-1} },\ldots, \frac{ t_{j_{|u|}} +0.5}{ q^{ u_{j_{|u|}-1} } } ) ]
\nonumber \\
&& \times
\prod_{l=1}^{|u|} ( \frac{ c_{j_l}}{q} + \frac{1}{2q} - \frac{1}{2} )
+O(1) q^{|u| -\sum_{l=1}^{|u|} 3 u_{j_l}/2 } \max_{1\leq l\leq |u|}  q^{-\beta (u_{j_l}-1)} 
\label{eq:a.545}
\end{eqnarray}
as $q\rightarrow \infty$ where
\begin{displaymath}
f_{j_1,\ldots, j_{|u|}} (x_{j_1},\ldots, x_{j_{|u|}}) =\int_{[0,1)^{d-|u|}} f(x) \prod_{1\leq i\leq d: i\not\in \{j_1,\ldots, j_{|u|}\} } dx_i.
\end{displaymath}
Also,
\begin{eqnarray*}
&&E\{ \nu_{u_1,\ldots, u_d} [\pi_j(a_{i,j}), \pi_{j;a_{i,j}} (b_{i,j,2}),\pi_{i,j,3} (0),\ldots, \pi_{i,j,u_j} (0):
1\leq j\leq d]^2 \}
\nonumber \\
&=&
( \sum_{t_{j_1}=0}^{q^{u_{j_1}-1}-1} \sum_{c_{j_1}=0}^{q-1} )
\cdots
( \sum_{t_{j_{|u|} }=0}^{q^{u_{j_{|u|} }-1}-1} 
\sum_{c_{j_{|u|} }=0}^{q-1} )
\langle f,  \prod_{l=1}^{|u|}  \psi_{u_{j_l} -1,t_{j_l},c_{j_l}} \rangle^2
\sum_{k=0}^{|u|} {|u| \choose k} \frac{1}{q^k } (1 - \frac{1}{q})^{|u|-k},
\end{eqnarray*}
and
\begin{eqnarray}
&&E\{ \nu_{u_1,\ldots, u_d} [\pi_j(a_{i,j}), \pi_{i,j,2} (0),\ldots, \pi_{i,j,u_j} (0):
1\leq j\leq d]^2 \}
\label{eq:a.54} \\
&=&
( \sum_{t_{j_1}=0}^{q^{u_{j_1}-1}-1} \sum_{c_{j_1}=0}^{q-1} )
\cdots
( \sum_{t_{j_{|u|} }=0}^{q^{u_{j_{|u|} }-1}-1} 
\sum_{c_{j_{|u|} }=0}^{q-1} )
\langle f,  \prod_{l=1}^{|u|}  \psi_{u_{j_l} -1,t_{j_l},c_{j_l}} \rangle^2
\sum_{k=0}^{|u|} {|u| \choose k} \frac{1}{q^k } (1 - \frac{1}{q})^{|u|-k}.
\nonumber
\end{eqnarray}
Finally
$E\{ \nu^* [\pi_{j_1} (a_{i, j_1});\ldots; \pi_{j_{|u|}} (a_{i, j_{|u|}}) ]^4 \} = O(1)$ as $q\rightarrow \infty$.
\end{la}
{\sc Proof.}
(\ref{eq:a.544}) easily follows from (\ref{eq:2.4}).
To prove (\ref{eq:a.545}), we observe that
\begin{eqnarray}
&& \langle f, \prod_{l=1}^{|u|} \psi_{u_{j_l}-1, t_{j_l}, c_{j_l}} \rangle
\nonumber \\
&=& \int_{[0,1)^d} f(x) [\prod_{l=1}^{|u|} \psi_{u_{j_l}-1, t_{j_l}, c_{j_l} } (x_{j_l}) ] dx
\nonumber \\
&=&
\int_{[0,1)^d} f(x) \Big\{\prod_{l=1}^{|u|} q^{u_{j_l}/2} \Big[ {\cal I}\{ x_{j_l} \in [ \frac{ q t_{j_l}+ c_{j_l} }{ q^{u_{j_l}}},
\frac{ q t_{j_l}+ c_{j_l} +1}{ q^{u_{j_l}}} ) \}
\nonumber \\
&&\hspace{0.5cm}
- \frac{1}{q} {\cal I}\{ x_{j_l} \in [ \frac{ t_{j_l} }{q^{u_{j_l}-1}},
\frac{ t_{j_l} +1}{q^{u_{j_l}-1}} ) \} \Big] \Big\} dx
\nonumber \\
&=&
q^{-\sum_{l=1}^{|u|} u_{j_l}/2 }
\int_{[0,1)^{|u|} } f_{j_1,\ldots, j_{|u|}}
(x_{j_1},\ldots, x_{j_{|u|}} ) \Big\{\prod_{l=1}^{|u|} q^{u_{j_l}} \Big[ {\cal I}\{ x_{j_l} \in 
\nonumber \\
&&\hspace{0.5cm}
[ \frac{ q t_{j_l}+ c_{j_l} }{ q^{u_{j_l}}},
\frac{ q t_{j_l}+ c_{j_l} +1}{ q^{u_{j_l}}} ) \}
- \frac{1}{q} {\cal I}\{ x_{j_l} \in [ \frac{ t_{j_l} }{q^{u_{j_l}-1}},
\frac{ t_{j_l} +1}{q^{u_{j_l}-1}} ) \} \Big] \Big\} dx_{j_1} \ldots dx_{j_{|u|}}.
\label{eq:a.70}
\end{eqnarray} 
We observe from Lemma \ref{la:a.6} that
the right hand side of (\ref{eq:a.70}) equals
\begin{eqnarray*}
&& q^{-\sum_{l=1}^{|u|} u_{j_l}/2 }
[\frac{\partial^{|u|} }{\partial x_{j_1} \ldots \partial x_{j_{|u|}} } f_{j_1,\ldots, j_{|u|}} 
(\frac{ t_{j_1} + 0.5}{q^{u_{j_1}-1} },\ldots, \frac{ t_{j_{|u|}} +0.5}{ q^{ u_{j_{|u|}-1} } } ) ]
\nonumber \\
&& \times
\prod_{l=1}^{|u|} \Big[ q^{u_{j_l}} \int_{(q t_{j_l} + c_{j_l})/q^{u_{j_l}} }^{ 
(q t_{j_l} + c_{j_l}+1)/q^{u_{j_l}} } (x_{j_l} - \frac{ t_{j_l} + 0.5}{ q^{u_{j_l} -1} } ) dx_{j_l} 
\nonumber \\
&& \hspace{0.5cm}
- q^{u_{j_l}-1} \int_{t_{j_l}/q^{u_{j_l}-1} }^{ 
( t_{j_l} + 1)/q^{u_{j_l}-1} } (x_{j_l} - \frac{ t_{j_l} + 0.5}{ q^{u_{j_l} -1} } ) dx_{j_l} \Big]
\nonumber \\
&& + q^{-\sum_{l=1}^{|u|} u_{j_l}/2 }
\int_{[0,1)^{|u|}} 
\int_{(t_{j_1}+0.5)/q^{u_{j_1}-1} }^{x_{j_1}} \cdots
\int_{(t_{j_{|u|} }+0.5)/q^{u_{j_{|u|} }-1} }^{x_{j_{|u|}} } 
\nonumber \\
&&\hspace{0.5cm}\times
[ \frac{ \partial^{|u|} f_{j_1,\ldots, j_{|u|} } (s_1,\ldots, s_{|u|} ) }{\partial x_{j_1} \ldots \partial x_{j_{|u|}} }
- \frac{ \partial^{|u|} 
f_{j_1,\ldots, j_{|u|}} 
(\frac{ t_{j_1} + 0.5}{q^{u_{j_1}-1} },\ldots, \frac{ t_{j_{|u|}} +0.5}{ q^{ u_{j_{|u|}-1} } } ) 
}{\partial x_{j_1} \ldots \partial x_{j_{|u|}} } ]
ds_{|u|} \ldots ds_1
\nonumber \\
&&\hspace{0.5cm}\times
 \Big\{\prod_{l=1}^{|u|} q^{u_{j_l}} \Big[ {\cal I}\{ x_{j_l} \in [ \frac{ q t_{j_l}+ c_{j_l} }{ q^{u_{j_l}}},
\frac{ q t_{j_l}+ c_{j_l} +1}{ q^{u_{j_l}}} ) \}
\nonumber \\
&&\hspace{0.5cm}
- \frac{1}{q} {\cal I}\{ x_{j_l} \in [ \frac{ t_{j_l} }{q^{u_{j_l}-1}},
\frac{ t_{j_l} +1}{q^{u_{j_l}-1}} ) \} \Big] \Big\} dx_{j_1}\ldots dx_{j_{|u|}}
\nonumber \\
&=&
q^{|u| -\sum_{l=1}^{|u|} 3 u_{j_l}/2 }
[\frac{\partial^{|u|} }{\partial x_{j_1} \ldots \partial x_{j_{|u|}} } f_{j_1,\ldots, j_{|u|}} 
(\frac{ t_{j_1} + 0.5}{q^{u_{j_1}-1} },\ldots, \frac{ t_{j_{|u|}} +0.5}{ q^{ u_{j_{|u|}-1} } } ) ]
\nonumber \\
&& \times
\prod_{l=1}^{|u|} [ \frac{ c_{j_l}}{q} - \frac{1}{2}(1 - \frac{1}{q}) ] 
+O(1) q^{|u| -\sum_{l=1}^{|u|} 3 u_{j_l}/2 } \max_{1\leq l\leq |u|}  q^{-\beta (u_{j_l}-1)}, 
\end{eqnarray*}
as $q\rightarrow \infty$.
Next we observe from (\ref{eq:3.85}) that
\begin{eqnarray*}
&& \nu_{u_1,\ldots, u_d} [\pi_j(a_{i,j}), \pi_{j;a_{i,j}} (b_{i,j,2}),\pi_{i,j,3} (0),\ldots, \pi_{i,j,u_j} (0):
1\leq j\leq d]
\nonumber \\
&=& (\sum_{t_{j_1}=0}^{q^{u_{j_1}-1}-1} \sum_{c_{j_1}=0}^{q-1} )
\cdots ( \sum_{t_{j_{|u|} }=0}^{q^{u_{j_{|u|} }-1}-1} \sum_{c_{j_{|u|} }=0}^{q-1} )
\langle f,  \prod_{l=1}^{|u|}  \psi_{u_{j_l} -1,t_{j_l},c_{j_l}} \rangle  
\nonumber \\
&&\hspace{0.5cm} \times
\prod_{l=1}^{|u|}  \psi_{u_{j_l}-1,t_{j_l},c_{j_l}} 
( \frac{ \pi_{j_l} (a_{i,j_l})}{ q}+ \frac{ \pi_{j_l;a_{i,j_l}} (b_{i,j_l,2})}{q^2} + \sum_{k=3}^{u_{j_l}}
\frac{ \pi_{i,j_l,k} (0)}{q^k}  ),
\end{eqnarray*}
and hence
\begin{eqnarray*}
&& E\{\nu_{u_1,\ldots, u_d} [\pi_j(a_{i,j}), \pi_{j;a_{i,j}} (b_{i,j,2}),\pi_{i,j,3} (0),\ldots, \pi_{i,j,u_j} (0):
1\leq j\leq d]^2 \}
\nonumber \\
&=& ( \sum_{t_{j_1}=0}^{q^{u_{j_1}-1}-1} \sum_{c_{j_1}=0}^{q-1} )
\cdots ( \sum_{t_{j_{|u|} }=0}^{q^{u_{j_{|u|} }-1}-1} \sum_{c_{j_{|u|} }=0}^{q-1} )
( \sum_{t_{j_1}'=0}^{q^{u_{j_1}-1}-1} \sum_{c_{j_1}'=0}^{q-1} )
\cdots ( \sum_{t_{j_{|u|} }'=0}^{q^{u_{j_{|u|} }-1}-1} \sum_{c_{j_{|u|} }'=0}^{q-1} )
\nonumber \\
&&\times
\langle f,  \prod_{l=1}^{|u|}  \psi_{u_{j_l} -1,t_{j_l},c_{j_l}} \rangle  
\langle f,  \prod_{l=1}^{|u|}  \psi_{u_{j_l} -1,t_{j_l}',c_{j_l}' } \rangle  
\nonumber \\
&&\times
\prod_{l=1}^{|u|} E \Big[ \psi_{u_{j_l}-1,t_{j_l},c_{j_l}} 
( \frac{ \pi_{j_l} (a_{i,j_l})}{ q}+ \frac{ \pi_{j_l;a_{i,j_l}} (b_{i,j_l,2})}{q^2} + \sum_{k=3}^{u_{j_l}}
\frac{ \pi_{i,j_l,k} (0)}{q^k}  )
\nonumber \\
&&\hspace{0.5cm} \times
\psi_{u_{j_l}-1,t_{j_l}',c_{j_l}'} 
( \frac{ \pi_{j_l} (a_{i,j_l})}{ q}+ \frac{ \pi_{j_l;a_{i,j_l}} (b_{i,j_l,2})}{q^2} + \sum_{k=3}^{u_{j_l}}
\frac{ \pi_{i,j_l,k} (0)}{q^k}  ) \Big]
\nonumber \\
&=& ( \sum_{t_{j_1}=0}^{q^{u_{j_1}-1}-1} \sum_{c_{j_1}=0}^{q-1} 
\sum_{c_{j_1}'=0}^{q-1} )
\cdots ( \sum_{t_{j_{|u|} }=0}^{q^{u_{j_{|u|} }-1}-1} \sum_{c_{j_{|u|} }=0}^{q-1} 
\sum_{c_{j_{|u|} }'=0}^{q-1} )
\nonumber \\
&&\times
\langle f,  \prod_{l=1}^{|u|}  \psi_{u_{j_l} -1,t_{j_l},c_{j_l}} \rangle  
\langle f,  \prod_{l=1}^{|u|}  \psi_{u_{j_l} -1,t_{j_l},c_{j_l}' } \rangle  
\nonumber \\
&&\times
\prod_{l=1}^{|u|} E \Big[ \psi_{u_{j_l}-1,t_{j_l},c_{j_l}} 
( \frac{ \pi_{j_l} (a_{i,j_l})}{ q}+ \frac{ \pi_{j_l;a_{i,j_l}} (b_{i,j_l,2})}{q^2} + \sum_{k=3}^{u_{j_l}}
\frac{ \pi_{i,j_l,k} (0)}{q^k}  )
\nonumber \\
&&\hspace{0.5cm} \times
\psi_{u_{j_l}-1,t_{j_l},c_{j_l}'} 
( \frac{ \pi_{j_l} (a_{i,j_l})}{ q}+ \frac{ \pi_{j_l;a_{i,j_l}} (b_{i,j_l,2})}{q^2} + \sum_{k=3}^{u_{j_l}}
\frac{ \pi_{i,j_l,k} (0)}{q^k}  ) \Big].
\end{eqnarray*}

We observe that the right hand side of the last equation can be expressed as a finite sum of terms of the 
following form
(up to permutations): for $\alpha=0,\ldots, |u|$,
\begin{eqnarray*}
&& ( \sum_{t_{j_1}=0}^{q^{u_{j_1}-1}-1} \sum_{c_{j_1}=0}^{q-1} 
\sum_{c_{j_1}'\neq c_{j_1} } )
\cdots ( \sum_{t_{j_\alpha }=0}^{q^{u_{j_\alpha} -1}-1} \sum_{c_{j_\alpha }=0}^{q-1} 
\sum_{c_{j_\alpha}' \neq c_{j_\alpha} } )
\nonumber \\
&&\times
( \sum_{t_{j_{\alpha+1} }=0}^{q^{u_{j_{\alpha+1} }-1}-1} \sum_{c_{j_{\alpha+1} }=0}^{q-1} )
\cdots ( \sum_{t_{j_{|u|} }=0}^{q^{u_{j_{|u|} }-1}-1} \sum_{c_{j_{|u|} }=0}^{q-1} )
\langle f,  \prod_{l=1}^{|u|}  \psi_{u_{j_l} -1,t_{j_l},c_{j_l}} \rangle  
\nonumber \\
&&\times
\langle f,  (\prod_{l_1=1}^\alpha  \psi_{u_{j_{l_1}} -1,t_{j_{l_1} },c_{j_{l_1} }' } )  
( \prod_{l_2=\alpha + 1}^{|u|}  \psi_{u_{j_{l_2}} -1,t_{j_{l_2} },c_{j_{l_2} } } ) \rangle  
\nonumber \\
&&\times
E\Big[ \prod_{l_1=1}^\alpha \psi_{u_{j_{l_1} }-1,t_{j_{l_1} },c_{j_{l_1} }} 
( \frac{ \pi_{j_{l_1} } (a_{i,j_{l_1}})}{ q}+ \frac{ \pi_{j_{l_1};a_{i,j_{l_1}}} (b_{i,j_{l_1},2})}{q^2} + \sum_{k=3}^{u_{j_{l_1}}}
\frac{ \pi_{i,j_{l_1},k} (0)}{q^k}  )
\nonumber \\
&&\hspace{0.5cm} \times
\psi_{u_{j_{l_1}}-1,t_{j_{l_1}},c_{j_{l_1} }'} 
( \frac{ \pi_{j_{l_1}} (a_{i,j_{l_1}})}{ q}+ \frac{ \pi_{j_{l_1};a_{i,j_{l_1}}} (b_{i,j_{l_1},2})}{q^2} + \sum_{k=3}^{u_{j_{l_1}}}
\frac{ \pi_{i,j_{l_1},k} (0)}{q^k}  ) \Big]
\nonumber \\
&&\times
E\Big[ \prod_{l_2=\alpha + 1}^{|u|} \psi_{u_{j_{l_2} }-1,t_{j_{l_2} },c_{j_{l_2} }} 
( \frac{ \pi_{j_{l_2} } (a_{i,j_{l_2}})}{ q}+ \frac{ \pi_{j_{l_2};a_{i,j_{l_2}}} (b_{i,j_{l_2},2})}{q^2} 
+ \sum_{k=3}^{u_{j_{l_2}}}
\frac{ \pi_{i,j_{l_2},k} (0)}{q^k}  )^2 \Big]
\nonumber \\
&=&
( \sum_{t_{j_1}=0}^{q^{u_{j_1}-1}-1} \sum_{c_{j_1}=0}^{q-1} 
\sum_{c_{j_1}'\neq c_{j_1}} )
\cdots
( \sum_{t_{j_\alpha}=0}^{q^{u_{j_\alpha}-1}-1} \sum_{c_{j_\alpha}=0}^{q-1} 
\sum_{c_{j_\alpha}'\neq c_{j_\alpha}} )
\nonumber \\
&&\hspace{0.5cm} \times
( \sum_{t_{j_{\alpha+1}}=0}^{q^{u_{j_{\alpha+1} }-1}-1} \sum_{c_{j_{\alpha+1} }=0}^{q-1} )
\cdots ( \sum_{t_{j_{|u|} }=0}^{q^{u_{j_{|u|} }-1}-1} 
\sum_{c_{j_{|u|} }=0}^{q-1} )
\langle f,  \prod_{l=1}^{|u|}  \psi_{u_{j_l} -1,t_{j_l},c_{j_l}} \rangle  
\nonumber \\
&&\hspace{0.5cm} \times
\langle f,  
( \prod_{l_1=1}^\alpha \psi_{u_{j_{l_1}} -1,t_{j_{l_1}},c_{j_{l_1} }'} )
( \prod_{l_2=\alpha+1}^{|u|}  \psi_{u_{j_{l_2} } -1,t_{j_{l_2} },c_{j_{l_2} }} ) \rangle  
E\Big[ \prod_{l_1=1}^\alpha \Big(
\nonumber \\
&&\hspace{1cm} 
- q^{u_{j_{l_1}}-1} {\cal I} \{ \frac{q t_{j_{l_1}} + c_{j_{l_1}} }{ q^{u_{j_{l_1}}}} \leq 
\frac{ \pi_{j_{l_1}} (a_{i,j_{l_1}})}{ q}+ \frac{ \pi_{j_{l_1};a_{i,j_{l_1}}} (b_{i,j_{l_1},2})}{q^2} 
\nonumber \\
&&\hspace{1.5cm}
+ \sum_{k=3}^{u_{j_{l_1}}}
\frac{ \pi_{i,j_{l_1},k} (0)}{q^k} < 
\frac{q t_{j_{l_1}} + c_{j_{l_1}} +1}{ q^{u_{j_{l_1}}}}  \}
\nonumber \\
&&\hspace{1cm}
- q^{u_{j_{l_1}}-1} {\cal I} \{ \frac{q t_{j_{l_1}} + c_{j_{l_1}}' }{ q^{u_{j_{l_1}}}} \leq 
\frac{ \pi_{j_{l_1}} (a_{i,j_{l_1}})}{ q}+ \frac{ \pi_{j_{l_1};a_{i,j_{l_1}}} (b_{i,j_{l_1},2})}{q^2} 
\nonumber \\
&&\hspace{1.5cm}
+ \sum_{k=3}^{u_{j_{l_1}}}
\frac{ \pi_{i,j_{l_1},k} (0)}{q^k} < 
\frac{q t_{j_{l_1}} + c_{j_{l_1}}' +1}{ q^{u_{j_{l_1}}}}  \}
\nonumber \\
&&
\hspace{1cm} + q^{u_{j_{l_1}}-2} {\cal I} \{ \frac{ t_{j_{l_1}} }{ q^{u_{j_{l_1}}-1}} \leq 
\frac{ \pi_{j_{l_1}} (a_{i,j_{l_1}})}{ q}+ \frac{ \pi_{j_{l_1};a_{i,j_{l_1}}} (b_{i,j_{l_1},2})}{q^2} 
\nonumber \\
&&\hspace{1.5cm}
+ \sum_{k=3}^{u_{j_{l_1}}}
\frac{ \pi_{i,j_{l_1},k} (0)}{q^k} < 
\frac{t_{j_{l_1}} + 1}{ q^{u_{j_{l_1}}-1}}  \} \Big) \Big]
\nonumber \\
&&\hspace{0.5cm} \times
E \Big[ \prod_{l_2= \alpha+1}^{|u|}  \Big(
q^{u_{j_{l_2}}} 
{\cal I} \{ \frac{q t_{j_{l_2}} + c_{j_{l_2}} }{ q^{u_{j_{l_2}}}} \leq 
\frac{ \pi_{j_{l_2}} (a_{i,j_{l_2}})}{ q}+ \frac{ \pi_{j_{l_2};a_{i,j_{l_2}}} (b_{i,j_{l_2},2})}{q^2} 
\nonumber \\
&&\hspace{1.5cm}
+ \sum_{k=3}^{u_{j_{l_2}}}
\frac{ \pi_{i,j_{l_2},k} (0)}{q^k} < 
\frac{q t_{j_{l_2}} + c_{j_{l_2}} +1}{ q^{u_{j_{l_2}}}}  \}
\nonumber \\
&&
\hspace{1cm} - 2 q^{u_{j_{l_2}}-1} {\cal I} \{ \frac{q t_{j_{l_2}} + c_{j_{l_2}} }{ q^{u_{j_{l_2}}}} \leq 
\frac{ \pi_{j_{l_2}} (a_{i,j_{l_2}})}{ q}+ \frac{ \pi_{j_{l_2};a_{i,j_{l_2}}} (b_{i,j_{l_2},2})}{q^2} 
\nonumber \\
&& \hspace{1.5cm}
+ \sum_{k=3}^{u_{j_{l_2}}}
\frac{ \pi_{i,j_{l_2},k} (0)}{q^k} < 
\frac{q t_{j_{l_2}} + c_{j_{l_2}} +1}{ q^{u_{j_{l_2}}}}  \}
\nonumber \\
&&
\hspace{1cm} + q^{u_{j_{l_2}}-2} {\cal I} \{ \frac{t_{j_{l_2}} }{ q^{u_{j_{l_2}}-1}} \leq 
\frac{ \pi_{j_{l_2}} (a_{i,j_{l_2}})}{ q}+ \frac{ \pi_{j_{l_2};a_{i,j_{l_2}}} (b_{i,j_{l_2},2})}{q^2} 
\nonumber \\
&&\hspace{1.5cm}
+ \sum_{k=3}^{u_{j_{l_2}}}
\frac{ \pi_{i,j_{l_2},k} (0)}{q^k} < 
\frac{t_{j_{l_2}} + 1}{ q^{u_{j_{l_2}}-1}}  \} \Big) \Big]
\nonumber \\
&=& 
( \sum_{t_{j_1}=0}^{q^{u_{j_1}-1}-1} \sum_{c_{j_1}=0}^{q-1} 
\sum_{c_{j_1}'\neq c_{j_1}} )
\cdots
( \sum_{t_{j_\alpha}=0}^{q^{u_{j_\alpha}-1}-1} \sum_{c_{j_\alpha}=0}^{q-1} 
\sum_{c_{j_\alpha}'\neq c_{j_\alpha}} )
\nonumber \\
&&\hspace{0.5cm} \times
( \sum_{t_{j_{\alpha+1}}=0}^{q^{u_{j_{\alpha+1} }-1}-1} \sum_{c_{j_{\alpha+1} }=0}^{q-1} )
\cdots ( \sum_{t_{j_{|u|} }=0}^{q^{u_{j_{|u|} }-1}-1} 
\sum_{c_{j_{|u|} }=0}^{q-1} )
\langle f,  \prod_{l=1}^{|u|}  \psi_{u_{j_l} -1,t_{j_l},c_{j_l}} \rangle  
\nonumber \\
&&\hspace{0.5cm} \times
\langle f,  
( \prod_{l_1=1}^\alpha \psi_{u_{j_{l_1}} -1,t_{j_{l_1}},c_{j_{l_1} }'} )
( \prod_{l_2=\alpha+1}^{|u|}  \psi_{u_{j_{l_2} } -1,t_{j_{l_2} },c_{j_{l_2} }} ) \rangle  
(-\frac{1}{q})^\alpha (1 - \frac{1}{q})^{|u|-\alpha}
\nonumber \\
&=& 
(-\frac{1}{q})^\alpha (1 - \frac{1}{q})^{|u|-\alpha}
[ \sum_{t_{j_1}=0}^{q^{u_{j_1}-1}-1} \sum_{c_{j_1}=0}^{q-1} 
\nonumber \\
&&\hspace{0.5cm}\times
(\sum_{c_{j_1}'=0}^{q-1} - {\cal I}\{ c_{j_1}'=c_{j_1} \} ) ]
\cdots
[ \sum_{t_{j_\alpha}=0}^{q^{u_{j_\alpha}-1}-1} \sum_{c_{j_\alpha}=0}^{q-1} 
(\sum_{c_{j_\alpha}'=0}^{q-1} - {\cal I}\{ c_{j_\alpha}' = c_{j_\alpha} \} ) ]
\nonumber \\
&&\hspace{0.5cm} \times
( \sum_{t_{j_{\alpha+1}}=0}^{q^{u_{j_{\alpha+1} }-1}-1} \sum_{c_{j_{\alpha+1} }=0}^{q-1} )
\cdots ( \sum_{t_{j_{|u|} }=0}^{q^{u_{j_{|u|} }-1}-1} 
\sum_{c_{j_{|u|} }=0}^{q-1} )
\langle f,  \prod_{l=1}^{|u|}  \psi_{u_{j_l} -1,t_{j_l},c_{j_l}} \rangle  
\nonumber \\
&&\hspace{0.5cm} \times
\langle f,  
( \prod_{l_1=1}^\alpha \psi_{u_{j_{l_1}} -1,t_{j_{l_1}},c_{j_{l_1} }'} )
( \prod_{l_2=\alpha+1}^{|u|}  \psi_{u_{j_{l_2} } -1,t_{j_{l_2} },c_{j_{l_2} }} ) \rangle  
\nonumber \\
&=&
\frac{1}{q^\alpha } (1 - \frac{1}{q})^{|u|-\alpha}
( \sum_{t_{j_1}=0}^{q^{u_{j_1}-1}-1} \sum_{c_{j_1}=0}^{q-1} )
\cdots
( \sum_{t_{j_{|u|} }=0}^{q^{u_{j_{|u|} }-1}-1} 
\sum_{c_{j_{|u|} }=0}^{q-1} )
\langle f,  \prod_{l=1}^{|u|}  \psi_{u_{j_l} -1,t_{j_l},c_{j_l}} \rangle^2.  
\end{eqnarray*}
Finally we conclude via symmetry that
\begin{eqnarray*}
&& E\{\nu_{u_1,\ldots, u_d} [\pi_j(a_{i,j}), \pi_{j;a_{i,j}} (b_{i,j,2}),\pi_{i,j,3} (0),\ldots, \pi_{i,j,u_j} (0):
1\leq j\leq d]^2 \}
\nonumber \\
&=&
\sum_{\alpha=0}^{|u|} {|u| \choose \alpha} \frac{1}{q^\alpha } (1 - \frac{1}{q})^{|u|-\alpha}
( \sum_{t_{j_1}=0}^{q^{u_{j_1}-1}-1} \sum_{c_{j_1}=0}^{q-1} )
\cdots
( \sum_{t_{j_{|u|} }=0}^{q^{u_{j_{|u|} }-1}-1} 
\sum_{c_{j_{|u|} }=0}^{q-1} )
\langle f,  \prod_{l=1}^{|u|}  \psi_{u_{j_l} -1,t_{j_l},c_{j_l}} \rangle^2.  
\end{eqnarray*}
The proof of (\ref{eq:a.54}) is similar and is omitted.
Finally using (\ref{eq:3.85}), (\ref{eq:3.36}) and that $u_{j_1}= \cdots = u_{j_{|u|}} = 1$, we have
\begin{eqnarray}
&& E\{ \nu^* [\pi_{j_1} (a_{i,j_1});\ldots; \pi_{j_{|u|}} (a_{i, j_{|u|}}) ]^4 \}
\nonumber \\
&=& ( \sum_{t_{j_1}=0}^{q^{u_{j_1}-1}-1} \sum_{c^{(1)}_{j_1}=0}^{q-1} 
\sum_{c^{(2)}_{j_1}=0}^{q-1} 
\sum_{c^{(3)}_{j_1}=0}^{q-1} 
\sum_{c^{(4)}_{j_1}=0}^{q-1} )
\cdots ( \sum_{t_{j_{|u|} }=0}^{q^{u_{j_{|u|} }-1}-1} \sum_{c^{(1)}_{j_{|u|} }=0}^{q-1} 
\sum_{c^{(2)}_{j_{|u|} }=0}^{q-1} 
\sum_{c^{(3)}_{j_{|u|} }=0}^{q-1} 
\sum_{c^{(4)}_{j_{|u|} }=0}^{q-1} )
\nonumber \\
&&\times
\langle f,  \prod_{l=1}^{|u|}  \psi_{u_{j_l} -1,t_{j_l}, c^{(1)}_{j_l}} \rangle  
\langle f,  \prod_{l=1}^{|u|}  \psi_{u_{j_l} -1,t_{j_l}, c^{(2)}_{j_l} } \rangle  
\langle f,  \prod_{l=1}^{|u|}  \psi_{u_{j_l} -1,t_{j_l}, c^{(3)}_{j_l}} \rangle  
\langle f,  \prod_{l=1}^{|u|}  \psi_{u_{j_l} -1,t_{j_l}, c^{(4)}_{j_l} } \rangle  
\nonumber \\
&&\times
\prod_{l=1}^{|u|} E \Big[ \psi_{u_{j_l}-1,t_{j_l},c^{(1)}_{j_l}} 
( \frac{ \pi_{j_l} (a_{i,j_l})}{ q} )
\psi_{u_{j_l}-1,t_{j_l},c^{(2)}_{j_l}} 
( \frac{ \pi_{j_l} (a_{i,j_l})}{ q} )
\nonumber \\
&&\hspace{0.5cm}\times
\psi_{u_{j_l}-1,t_{j_l},c^{(3)}_{j_l}} 
( \frac{ \pi_{j_l} (a_{i,j_l})}{ q} )
\psi_{u_{j_l}-1,t_{j_l},c^{(4)}_{j_l}} 
( \frac{ \pi_{j_l} (a_{i,j_l})}{ q} )
\Big]
\nonumber \\
&=& ( \sum_{c^{(1)}_{j_1}=0}^{q-1} 
\sum_{c^{(2)}_{j_1}=0}^{q-1} 
\sum_{c^{(3)}_{j_1}=0}^{q-1} 
\sum_{c^{(4)}_{j_1}=0}^{q-1} )
\cdots ( \sum_{c^{(1)}_{j_{|u|} }=0}^{q-1} 
\sum_{c^{(2)}_{j_{|u|} }=0}^{q-1} 
\sum_{c^{(3)}_{j_{|u|} }=0}^{q-1} 
\sum_{c^{(4)}_{j_{|u|} }=0}^{q-1} )
\langle f,  \prod_{l=1}^{|u|}  \psi_{0, 0, c^{(1)}_{j_l}} \rangle  
\nonumber \\
&&\times
\langle f,  \prod_{l=1}^{|u|}  \psi_{0, 0, c^{(2)}_{j_l} } \rangle  
\langle f,  \prod_{l=1}^{|u|}  \psi_{0, 0, c^{(3)}_{j_l}} \rangle  
\langle f,  \prod_{l=1}^{|u|}  \psi_{0, 0, c^{(4)}_{j_l} } \rangle  
\prod_{l=1}^{|u|} E [ \prod_{k=1}^4 \psi_{0, 0,c^{(k)}_{j_l}} 
( \frac{ \pi_{j_l} (a_{i,j_l})}{ q} ) ].
\label{eq:a.43}
\end{eqnarray}

It is convenient to define the following subsets of $\{1,\ldots, d\}$: for $\{l_1, l_2, l_3, l_4\} = \{1,2,3,4\}$,
\begin{eqnarray*}
\Xi_{\{1,2,3,4\} } &=& \{l\in \{1,\ldots, |u|\}: c^{(1)}_{j_l} = c^{(2)}_{j_l} = c^{(3)}_{j_l} = c^{(4)}_{j_l} \},
\nonumber \\
\Xi_{\{l_1, l_2, l_3\}, \{l_4\} } &=& \{ l\in \{1,\ldots, |u|\}: c^{(l_1)}_{j_l} = c^{(l_2)}_{j_l} = c^{(l_3)}_{j_l} \neq c^{(l_4)}_{j_l} \},
\nonumber \\
\Xi_{ \{ l_1, l_2\}, \{l_3\}, \{l_4\}} &=&
\{ l\in \{ 1,\ldots, |u|\}: c^{(l_1)}_{j_l} = c^{(l_2)}_{j_l}, 
\mbox{and $c^{(l_1)}_{j_l}, c^{(l_3)}_{j_l}, 
c^{(l_4)}_{j_l}$ are all distinct} \}, 
\nonumber \\
\Xi_{ \{ l_1, l_2\}, \{l_3, l_4\}} &=&
\{ l\in \{1,\ldots, |u|\}: c^{(l_1)}_{j_l} = c^{(l_2)}_{j_l} \neq c^{(l_3)}_{j_l} = c^{(l_4)}_{j_l} \},
\nonumber \\
\Xi_{ \{ l_1\}, \{l_2\}, \{l_3\}, \{l_4\}} &=&
\{ l\in \{1,\ldots, |u|\}: \mbox{$c^{(l_1)}_{j_l}, c^{(l_2)}_{j_l}, c^{(l_3)}_{j_l}, c^{(l_4)}_{j_l}$ are all distinct}  \}.
\end{eqnarray*}
The right hand side of (\ref{eq:a.43}) can be written as
a finite sum of terms of the following form: 
\begin{eqnarray}
&& [ \prod_{l\in \Xi_{\{1,2,3,4\}}} \sum_{c^{(1)}_{j_l}=0}^{q-1} ]
[ \prod_{l\in \Xi_{\{1,2,3\},\{4\}}} \sum_{c^{(1)}_{j_l}=0}^{q-1} 
\sum_{0\leq c^{(4)}_{j_l}\leq q-1: c^{(4)}_{j_l} \neq c^{(1)}_{j_l} } ] 
\nonumber \\
&&\times
[ \prod_{l\in \Xi_{\{1,2,4\},\{3\}}} \sum_{c^{(1)}_{j_l}=0}^{q-1} 
\sum_{0\leq c^{(3)}_{j_l}\leq q-1: c^{(3)}_{j_l} \neq c^{(1)}_{j_l} } ] 
[ \prod_{l\in \Xi_{\{1,3,4\},\{2\}}} \sum_{c^{(1)}_{j_l}=0}^{q-1} 
\sum_{0\leq c^{(2)}_{j_l}\leq q-1: c^{(2)}_{j_l} \neq c^{(1)}_{j_l} } ] 
\nonumber \\
&&\times
[ \prod_{l\in \Xi_{\{2,3,4\},\{1\}}} \sum_{c^{(2)}_{j_l}=0}^{q-1} 
\sum_{0\leq c^{(1)}_{j_l}\leq q-1: c^{(1)}_{j_l} \neq c^{(2)}_{j_l} } ] 
\nonumber \\
&&\times
[ \prod_{l\in \Xi_{\{1,2\}, \{3\}, \{4\} }} \sum_{c^{(1)}_{j_l}=0}^{q-1} 
\sum_{0\leq c^{(3)}_{j_l}\leq q-1: c^{(3)}_{j_l} \neq c^{(1)}_{j_l} }  
\sum_{0\leq c^{(4)}_{j_l}\leq q-1: c^{(4)}_{j_l} \neq c^{(1)}_{j_l}, c^{(3)}_{j_l} } ] 
\nonumber \\
&&\times
[ \prod_{l\in \Xi_{\{1,3\}, \{2\}, \{4\} }} \sum_{c^{(1)}_{j_l}=0}^{q-1} 
\sum_{0\leq c^{(2)}_{j_l}\leq q-1: c^{(2)}_{j_l} \neq c^{(1)}_{j_l} }  
\sum_{0\leq c^{(4)}_{j_l}\leq q-1: c^{(4)}_{j_l} \neq c^{(1)}_{j_l}, c^{(2)}_{j_l} } ] 
\nonumber \\
&&\times
[ \prod_{l\in \Xi_{\{1,4\}, \{2\}, \{3\} }} \sum_{c^{(1)}_{j_l}=0}^{q-1} 
\sum_{0\leq c^{(2)}_{j_l}\leq q-1: c^{(2)}_{j_l} \neq c^{(1)}_{j_l} }  
\sum_{0\leq c^{(3)}_{j_l}\leq q-1: c^{(3)}_{j_l} \neq c^{(1)}_{j_l}, c^{(3)}_{j_l} } ] 
 \nonumber \\
&&\times
[ \prod_{l\in \Xi_{\{2,3\}, \{1\}, \{4\} }} \sum_{c^{(2)}_{j_l}=0}^{q-1} 
\sum_{0\leq c^{(1)}_{j_l}\leq q-1: c^{(1)}_{j_l} \neq c^{(2)}_{j_l} }  
\sum_{0\leq c^{(4)}_{j_l}\leq q-1: c^{(4)}_{j_l} \neq c^{(1)}_{j_l}, c^{(2)}_{j_l} } ] 
 \nonumber \\
&&\times
[ \prod_{l\in \Xi_{\{2,4\}, \{1\}, \{3\} }} \sum_{c^{(2)}_{j_l}=0}^{q-1} 
\sum_{0\leq c^{(1)}_{j_l}\leq q-1: c^{(1)}_{j_l} \neq c^{(2)}_{j_l} }  
\sum_{0\leq c^{(3)}_{j_l}\leq q-1: c^{(3)}_{j_l} \neq c^{(1)}_{j_l}, c^{(2)}_{j_l} } ] 
\nonumber \\
&&\times
[ \prod_{l\in \Xi_{\{3,4\}, \{1\}, \{2\} }} \sum_{c^{(3)}_{j_l}=0}^{q-1} 
\sum_{0\leq c^{(1)}_{j_l}\leq q-1: c^{(1)}_{j_l} \neq c^{(3)}_{j_l} }  
\sum_{0\leq c^{(2)}_{j_l}\leq q-1: c^{(2)}_{j_l} \neq c^{(1)}_{j_l}, c^{(3)}_{j_l} } ] 
\nonumber \\
&&\times
[ \prod_{l\in \Xi_{\{1,2\}, \{3,4\} }} \sum_{c^{(1)}_{j_l}=0}^{q-1} 
\sum_{0\leq c^{(3)}_{j_l}\leq q-1: c^{(3)}_{j_l} \neq c^{(1)}_{j_l} } ] 
[ \prod_{l\in \Xi_{\{1,3\}, \{2,4\} }} \sum_{c^{(1)}_{j_l}=0}^{q-1} 
\sum_{0\leq c^{(2)}_{j_l}\leq q-1: c^{(2)}_{j_l} \neq c^{(1)}_{j_l} } ] 
\nonumber \\
&&\times
[ \prod_{l\in \Xi_{\{1,4\}, \{2,3\} }} \sum_{c^{(1)}_{j_l}=0}^{q-1} 
\sum_{0\leq c^{(2)}_{j_l}\leq q-1: c^{(2)}_{j_l} \neq c^{(1)}_{j_l} } ] 
[ \prod_{l\in \Xi_{\{1\}, \{2\}, \{3\}, \{4\} }} \sum_{c^{(1)}_{j_l}=0}^{q-1} 
\sum_{0\leq c^{(2)}_{j_l}\leq q-1: c^{(2)}_{j_l} \neq c^{(1)}_{j_l} }  
\nonumber \\
&&\hspace{0.5cm}\times
\sum_{0\leq c^{(3)}_{j_l}\leq q-1: c^{(3)}_{j_l} \neq c^{(1)}_{j_l}, c^{(2)}_{j_l} }  
\sum_{0\leq c^{(4)}_{j_l}\leq q-1: c^{(4)}_{j_l} \neq c^{(1)}_{j_l}, c^{(2)}_{j_l}, c^{(3)}_{j_l} } ]  
\langle f,  \prod_{l=1}^{|u|}  \psi_{0, 0, c^{(1)}_{j_l}} \rangle  
\nonumber \\
&& \times
\langle f,  \prod_{l=1}^{|u|}  \psi_{0, 0, c^{(2)}_{j_l} } \rangle  
\langle f,  \prod_{l=1}^{|u|}  \psi_{0, 0, c^{(3)}_{j_l}} \rangle  
\langle f,  \prod_{l=1}^{|u|}  \psi_{0, 0, c^{(4)}_{j_l} } \rangle  
\Big\{ \prod_{l\in \Xi_{\{1,2,3,4\} }} E \Big[ \psi_{0, 0,c^{(1)}_{j_l}} 
( \frac{ \pi_{j_l} (a_{i,j_l})}{ q} )^4
\Big] \Big\}
\nonumber \\
&& \times
\Big\{ \prod_{l\in \Xi_{\{1,2,3\}, \{4\} }} E \Big[ \psi_{0, 0,c^{(1)}_{j_l}} 
( \frac{ \pi_{j_l} (a_{i,j_l})}{ q} )^3
\psi_{0, 0,c^{(4)}_{j_l}} 
( \frac{ \pi_{j_l} (a_{i,j_l})}{ q} )
\Big] \Big\}
\nonumber \\
&& \times
\Big\{ \prod_{l\in \Xi_{\{1,2,4\}, \{3\} }} E \Big[ \psi_{0, 0,c^{(1)}_{j_l}} 
( \frac{ \pi_{j_l} (a_{i,j_l})}{ q} )^3
\psi_{0, 0,c^{(3)}_{j_l}} 
( \frac{ \pi_{j_l} (a_{i,j_l})}{ q} )
\Big] \Big\}
\nonumber \\
&& \times
\Big\{ \prod_{l\in \Xi_{\{1,3,4\}, \{2\} }} E \Big[ \psi_{0, 0,c^{(1)}_{j_l}} 
( \frac{ \pi_{j_l} (a_{i,j_l})}{ q} )^3
\psi_{0, 0,c^{(2)}_{j_l}} 
( \frac{ \pi_{j_l} (a_{i,j_l})}{ q} )
\Big] \Big\}
\nonumber \\
&& \times
\Big\{ \prod_{l\in \Xi_{\{2,3,4\}, \{1\} }} E \Big[ \psi_{0, 0,c^{(2)}_{j_l}} 
( \frac{ \pi_{j_l} (a_{i,j_l})}{ q} )^3
\psi_{0, 0,c^{(1)}_{j_l}} 
( \frac{ \pi_{j_l} (a_{i,j_l})}{ q} )
\Big] \Big\}
\nonumber \\
&& \times
\Big\{ \prod_{l\in \Xi_{\{1,2\}, \{3\}, \{4\} }} E \Big[ \psi_{0, 0,c^{(1)}_{j_l}} 
( \frac{ \pi_{j_l} (a_{i,j_l})}{ q} )^2
\psi_{0, 0,c^{(3)}_{j_l}} 
( \frac{ \pi_{j_l} (a_{i,j_l})}{ q} )
\psi_{0, 0,c^{(4)}_{j_l}} 
( \frac{ \pi_{j_l} (a_{i,j_l})}{ q} )
\Big] \Big\}
\nonumber \\
&& \times
\Big\{ \prod_{l\in \Xi_{\{1,3\}, \{2\}, \{4\} }} E \Big[ \psi_{0, 0,c^{(1)}_{j_l}} 
( \frac{ \pi_{j_l} (a_{i,j_l})}{ q} )^2
\psi_{0, 0,c^{(2)}_{j_l}} 
( \frac{ \pi_{j_l} (a_{i,j_l})}{ q} )
\psi_{0, 0,c^{(4)}_{j_l}} 
( \frac{ \pi_{j_l} (a_{i,j_l})}{ q} )
\Big] \Big\}
\nonumber \\
&& \times
\Big\{ \prod_{l\in \Xi_{\{1,4\}, \{2\}, \{3\} }} E \Big[ \psi_{0, 0,c^{(1)}_{j_l}} 
( \frac{ \pi_{j_l} (a_{i,j_l})}{ q} )^2
\psi_{0, 0,c^{(2)}_{j_l}} 
( \frac{ \pi_{j_l} (a_{i,j_l})}{ q} )
\psi_{0, 0,c^{(3)}_{j_l}} 
( \frac{ \pi_{j_l} (a_{i,j_l})}{ q} )
\Big] \Big\}
\nonumber \\
&& \times
\Big\{ \prod_{l\in \Xi_{\{2,3\}, \{1\}, \{4\} }} E \Big[ \psi_{0, 0,c^{(2)}_{j_l}} 
( \frac{ \pi_{j_l} (a_{i,j_l})}{ q} )^2
\psi_{0, 0,c^{(1)}_{j_l}} 
( \frac{ \pi_{j_l} (a_{i,j_l})}{ q} )
\psi_{0, 0,c^{(4)}_{j_l}} 
( \frac{ \pi_{j_l} (a_{i,j_l})}{ q} )
\Big] \Big\}
\nonumber \\
&& \times
\Big\{ \prod_{l\in \Xi_{\{2,4\}, \{1\}, \{3\} }} E \Big[ \psi_{0, 0,c^{(2)}_{j_l}} 
( \frac{ \pi_{j_l} (a_{i,j_l})}{ q} )^2
\psi_{0, 0,c^{(1)}_{j_l}} 
( \frac{ \pi_{j_l} (a_{i,j_l})}{ q} )
\psi_{0, 0,c^{(3)}_{j_l}} 
( \frac{ \pi_{j_l} (a_{i,j_l})}{ q} )
\Big] \Big\}
\nonumber \\
&& \times
\Big\{ \prod_{l\in \Xi_{\{3,4\}, \{1\}, \{2\} }} E \Big[ \psi_{0, 0,c^{(3)}_{j_l}} 
( \frac{ \pi_{j_l} (a_{i,j_l})}{ q} )^2
\psi_{0, 0,c^{(1)}_{j_l}} 
( \frac{ \pi_{j_l} (a_{i,j_l})}{ q} )
\psi_{0, 0,c^{(2)}_{j_l}} 
( \frac{ \pi_{j_l} (a_{i,j_l})}{ q} )
\Big] \Big\}
\nonumber \\
&& \times
\Big\{ \prod_{l\in \Xi_{\{1,2\}, \{3,4\} }} E \Big[ \psi_{0, 0,c^{(1)}_{j_l}} 
( \frac{ \pi_{j_l} (a_{i,j_l})}{ q} )^2
\psi_{0, 0,c^{(3)}_{j_l}} 
( \frac{ \pi_{j_l} (a_{i,j_l})}{ q} )^2
\Big] \Big\}
\nonumber \\
&& \times
\Big\{ \prod_{l\in \Xi_{\{1,3\}, \{2,4\} }} E \Big[ \psi_{0, 0,c^{(1)}_{j_l}} 
( \frac{ \pi_{j_l} (a_{i,j_l})}{ q} )^2
\psi_{0, 0,c^{(2)}_{j_l}} 
( \frac{ \pi_{j_l} (a_{i,j_l})}{ q} )^2
\Big] \Big\}
\nonumber \\
&& \times
\Big\{ \prod_{l\in \Xi_{\{1,4\}, \{2,3\} }} E \Big[ \psi_{0, 0,c^{(1)}_{j_l}} 
( \frac{ \pi_{j_l} (a_{i,j_l})}{ q} )^2
\psi_{0, 0,c^{(2)}_{j_l}} 
( \frac{ \pi_{j_l} (a_{i,j_l})}{ q} )^2
\Big] \Big\}
\nonumber \\
&& \times
\Big\{ \prod_{l\in \Xi_{\{1\}, \{2\}, \{3\}, \{4\} }} E \Big[ \psi_{0, 0,c^{(1)}_{j_l}} 
( \frac{ \pi_{j_l} (a_{i,j_l})}{ q} )
\psi_{0, 0,c^{(2)}_{j_l}} 
( \frac{ \pi_{j_l} (a_{i,j_l})}{ q} )
\nonumber \\
&&\hspace{0.5cm}\times
\psi_{0, 0,c^{(3)}_{j_l}} 
( \frac{ \pi_{j_l} (a_{i,j_l})}{ q} )
\psi_{0, 0,c^{(4)}_{j_l}} 
( \frac{ \pi_{j_l} (a_{i,j_l})}{ q} )
\Big] \Big\}.
\label{eq:a.44}
\end{eqnarray}
Now,
\begin{eqnarray*}
&& \prod_{l\in \Xi_{\{1,2,3,4\} }} E \Big[ \psi_{0, 0,c^{(1)}_{j_l}} 
( \frac{ \pi_{j_l} (a_{i,j_l})}{ q} )^4
\Big] 
\nonumber \\
&=&
\prod_{l\in \Xi_{\{1,2,3,4\} }} E \Big[  
[ q^{1/2} {\cal I} \{ \frac{ c^{(1)}_{j_l} }{ q } \leq \frac{ \pi_{j_l} (a_{i,j_l})}{ q}
< \frac{ c^{(1)}_{j_l} +1}{ q } \}
- q^{-1/2} ]^4 \Big] 
\nonumber \\
&=&
\prod_{l\in \Xi_{\{1,2,3,4\} }} q^2 E \Big[
[ {\cal I} \{ \frac{ c^{(1)}_{j_l} }{ q } \leq \frac{ \pi_{j_l} (a_{i,j_l})}{ q}
< \frac{ c^{(1)}_{j_l} +1}{ q } \}
- \frac{1}{ q} ]^4 \Big] 
\nonumber \\
&=&
\prod_{l\in \Xi_{\{1,2,3,4\} }} q^2 E \Big[  
(1 - \frac{4}{q} + \frac{6}{q^2} - \frac{4}{q^3 }) {\cal I} \{ \frac{ c^{(1)}_{j_l} }{ q } \leq \frac{ \pi_{j_l} (a_{i,j_l})}{ q}
< \frac{ c^{(1)}_{j_l} +1}{ q } \} + \frac{1}{q^4} \Big]
\nonumber \\
&=&
\prod_{l\in \Xi_{\{1,2,3,4\} }}  
(q - 4 + \frac{6}{q} - \frac{3}{q^2 }),
\nonumber \\
&& \prod_{l\in \Xi_{\{1,2,3\}, \{4\} }} E \Big[ \psi_{0, 0,c^{(1)}_{j_l}} 
( \frac{ \pi_{j_l} (a_{i,j_l})}{ q} )^3
\psi_{0, 0,c^{(4)}_{j_l}} 
( \frac{ \pi_{j_l} (a_{i,j_l})}{ q} )
\Big]
\nonumber \\
&=&
\prod_{l\in \Xi_{\{1,2,3\}, \{4\} }} q^2 E \Big[
[ {\cal I} \{ \frac{ c^{(1)}_{j_l} }{ q } \leq \frac{ \pi_{j_l} (a_{i,j_l})}{ q}
< \frac{ c^{(1)}_{j_l} +1}{ q } \}
- \frac{1}{ q} ]^3
\nonumber \\
&&\times
[ {\cal I} \{ \frac{ c^{(4)}_{j_l} }{ q } \leq \frac{ \pi_{j_l} (a_{i,j_l})}{ q}
< \frac{ c^{(4)}_{j_l} +1}{ q } \}
- \frac{1}{ q} ]
 \Big] 
\nonumber \\
&=&
\prod_{l\in \Xi_{\{1,2,3\}, \{4\} }} q^2 E \Big[
[ (1 - \frac{3}{q} + \frac{3}{q^2} ) {\cal I} \{ \frac{ c^{(1)}_{j_l} }{ q } \leq \frac{ \pi_{j_l} (a_{i,j_l})}{ q}
< \frac{ c^{(1)}_{j_l} +1}{ q } \}
- \frac{1}{ q^3} ]
\nonumber \\
&&\times
[ {\cal I} \{ \frac{ c^{(4)}_{j_l} }{ q } \leq \frac{ \pi_{j_l} (a_{i,j_l})}{ q}
< \frac{ c^{(4)}_{j_l} +1}{ q } \}
- \frac{1}{ q} ]
 \Big] 
\nonumber \\
&=&
\prod_{l\in \Xi_{\{1,2,3\}, \{4\} }} q^2 E \Big[
- (\frac{1}{q} - \frac{3}{q^2} + \frac{3}{q^3} ) {\cal I} \{ \frac{ c^{(1)}_{j_l} }{ q } \leq \frac{ \pi_{j_l} (a_{i,j_l})}{ q}
< \frac{ c^{(1)}_{j_l} +1}{ q } \}
+ \frac{1}{ q^4}
\nonumber \\
&&
-\frac{1}{q^3}  {\cal I} \{ \frac{ c^{(4)}_{j_l} }{ q } \leq \frac{ \pi_{j_l} (a_{i,j_l})}{ q}
< \frac{ c^{(4)}_{j_l} +1}{ q } \}
 \Big] 
\nonumber \\
&=&
\prod_{l\in \Xi_{\{1,2,3\}, \{4\} }} (-1 + \frac{3}{q} - \frac{3}{q^2}),
\nonumber \\
&& \prod_{l\in \Xi_{\{1,2\}, \{3\}, \{4\} }} E \Big[ \psi_{0, 0,c^{(1)}_{j_l}} 
( \frac{ \pi_{j_l} (a_{i,j_l})}{ q} )^2
\psi_{0, 0,c^{(3)}_{j_l}} 
( \frac{ \pi_{j_l} (a_{i,j_l})}{ q} )
\psi_{0, 0,c^{(4)}_{j_l}} 
( \frac{ \pi_{j_l} (a_{i,j_l})}{ q} )
\Big]
\nonumber \\
&=& 
\prod_{l\in \Xi_{\{1,2\}, \{3\}, \{4\} }} q^2 E \Big[
[ {\cal I} \{ \frac{ c^{(1)}_{j_l} }{ q } \leq \frac{ \pi_{j_l} (a_{i,j_l})}{ q}
< \frac{ c^{(1)}_{j_l} +1}{ q } \}
- \frac{1}{ q} ]^2
\nonumber \\
&&\times
[ {\cal I} \{ \frac{ c^{(3)}_{j_l} }{ q } \leq \frac{ \pi_{j_l} (a_{i,j_l})}{ q}
< \frac{ c^{(3)}_{j_l} +1}{ q } \}
- \frac{1}{ q} ]
[ {\cal I} \{ \frac{ c^{(4)}_{j_l} }{ q } \leq \frac{ \pi_{j_l} (a_{i,j_l})}{ q}
< \frac{ c^{(4)}_{j_l} +1}{ q } \}
- \frac{1}{ q} ] \Big]
\nonumber \\
&=&
\prod_{l\in \Xi_{\{1,2\}, \{3\}, \{4\} }} E \Big[
(1 - \frac{2}{q}) {\cal I} \{ \frac{ c^{(1)}_{j_l} }{ q } \leq \frac{ \pi_{j_l} (a_{i,j_l})}{ q}
< \frac{ c^{(1)}_{j_l} +1}{ q } \}
+ \frac{1}{q^2}
\nonumber \\
&&
-\frac{1}{q} {\cal I} \{ \frac{ c^{(3)}_{j_l} }{ q } \leq \frac{ \pi_{j_l} (a_{i,j_l})}{ q}
< \frac{ c^{(3)}_{j_l} +1}{ q } \}
-\frac{1}{q} {\cal I} \{ \frac{ c^{(4)}_{j_l} }{ q } \leq \frac{ \pi_{j_l} (a_{i,j_l})}{ q}
< \frac{ c^{(4)}_{j_l} +1}{ q } \}
\Big]
\nonumber \\
&=&
\prod_{l\in \Xi_{\{1,2\}, \{3\}, \{4\} }}  (\frac{1}{q} - \frac{3}{q^2}),
\nonumber \\
&& \prod_{l\in \Xi_{\{1,2\}, \{3,4\} }} E \Big[ \psi_{0, 0,c^{(1)}_{j_l}} 
( \frac{ \pi_{j_l} (a_{i,j_l})}{ q} )^2
\psi_{0, 0,c^{(3)}_{j_l}} 
( \frac{ \pi_{j_l} (a_{i,j_l})}{ q} )^2
\Big] \Big\}
\nonumber \\
&=& 
\prod_{l\in \Xi_{\{1,2\}, \{3,4\} }} q^2 E \Big[
[ {\cal I} \{ \frac{ c^{(1)}_{j_l} }{ q } \leq \frac{ \pi_{j_l} (a_{i,j_l})}{ q}
< \frac{ c^{(1)}_{j_l} +1}{ q } \}
- \frac{1}{ q} ]^2
\nonumber \\
&&\hspace{0.5cm}\times
[ {\cal I} \{ \frac{ c^{(3)}_{j_l} }{ q } \leq \frac{ \pi_{j_l} (a_{i,j_l})}{ q}
< \frac{ c^{(3)}_{j_l} +1}{ q } \}
- \frac{1}{ q} ]^2 \Big]
\nonumber \\
&=&
\prod_{l\in \Xi_{\{1,2\}, \{3,4\} }} E \Big[
(1 - \frac{2}{q})  {\cal I} \{ \frac{ c^{(1)}_{j_l} }{ q } \leq \frac{ \pi_{j_l} (a_{i,j_l})}{ q}
< \frac{ c^{(1)}_{j_l} +1}{ q } \}
\nonumber \\
&&\hspace{0.5cm}
+ (1 - \frac{2}{q})  {\cal I} \{ \frac{ c^{(3)}_{j_l} }{ q } \leq \frac{ \pi_{j_l} (a_{i,j_l})}{ q}
< \frac{ c^{(3)}_{j_l} +1}{ q } \}
+ \frac{1}{q^2} \Big]
\nonumber \\
&=&
\prod_{l\in \Xi_{\{1,2\}, \{3,4\} }} (\frac{2}{q} - \frac{3}{q^2}),
\nonumber \\
&& \prod_{l\in \Xi_{\{1\}, \{2\}, \{3\}, \{4\} }} E \Big[ \psi_{0, 0,c^{(1)}_{j_l}} 
( \frac{ \pi_{j_l} (a_{i,j_l})}{ q} )
\psi_{0, 0,c^{(2)}_{j_l}} 
( \frac{ \pi_{j_l} (a_{i,j_l})}{ q} )
\nonumber \\
&&\hspace{0.5cm}\times
\psi_{0, 0,c^{(3)}_{j_l}} 
( \frac{ \pi_{j_l} (a_{i,j_l})}{ q} )
\psi_{0, 0,c^{(4)}_{j_l}} 
( \frac{ \pi_{j_l} (a_{i,j_l})}{ q} )
\Big]
\nonumber \\
&=& \prod_{l\in \Xi_{\{1\}, \{2\}, \{3\}, \{4\} }} q^2 E \Big[ 
[ {\cal I} \{ \frac{ c^{(1)}_{j_l} }{ q } \leq \frac{ \pi_{j_l} (a_{i,j_l})}{ q}
< \frac{ c^{(1)}_{j_l} +1}{ q } \}
- \frac{1}{ q} ]
\nonumber \\
&& \hspace{0.5cm}\times
[ {\cal I} \{ \frac{ c^{(2)}_{j_l} }{ q } \leq \frac{ \pi_{j_l} (a_{i,j_l})}{ q}
< \frac{ c^{(2)}_{j_l} +1}{ q } \}
- \frac{1}{ q} ]
[ {\cal I} \{ \frac{ c^{(3)}_{j_l} }{ q } \leq \frac{ \pi_{j_l} (a_{i,j_l})}{ q}
< \frac{ c^{(3)}_{j_l} +1}{ q } \}
- \frac{1}{ q} ]
\nonumber \\
&& \hspace{0.5cm}\times
[ {\cal I} \{ \frac{ c^{(4)}_{j_l} }{ q } \leq \frac{ \pi_{j_l} (a_{i,j_l})}{ q}
< \frac{ c^{(4)}_{j_l} +1}{ q } \}
- \frac{1}{ q} ]
\Big]
\nonumber \\
&=& \prod_{l\in \Xi_{\{1\}, \{2\}, \{3\}, \{4\} }} (-\frac{3}{q^2}). 
\end{eqnarray*}

Writing $|A|$ to denote the cardinality of a finite set $A$, we observe that (\ref{eq:a.44}) is equal to
\begin{eqnarray*}
&& [ \prod_{l\in \Xi_{\{1,2,3,4\}}} \sum_{c^{(1)}_{j_l}=0}^{q-1} ]
[ \prod_{l\in \Xi_{\{1,2,3\},\{4\}}} \sum_{c^{(1)}_{j_l}=0}^{q-1} 
\sum_{0\leq c^{(4)}_{j_l}\leq q-1: c^{(4)}_{j_l} \neq c^{(1)}_{j_l} } ] 
\nonumber \\
&&\times
[ \prod_{l\in \Xi_{\{1,2,4\},\{3\}}} \sum_{c^{(1)}_{j_l}=0}^{q-1} 
\sum_{0\leq c^{(3)}_{j_l}\leq q-1: c^{(3)}_{j_l} \neq c^{(1)}_{j_l} } ] 
[ \prod_{l\in \Xi_{\{1,3,4\},\{2\}}} \sum_{c^{(1)}_{j_l}=0}^{q-1} 
\sum_{0\leq c^{(2)}_{j_l}\leq q-1: c^{(2)}_{j_l} \neq c^{(1)}_{j_l} } ] 
\nonumber \\
&&\times
[ \prod_{l\in \Xi_{\{2,3,4\},\{1\}}} \sum_{c^{(2)}_{j_l}=0}^{q-1} 
\sum_{0\leq c^{(1)}_{j_l}\leq q-1: c^{(1)}_{j_l} \neq c^{(2)}_{j_l} } ] 
\nonumber \\
&&\times
[ \prod_{l\in \Xi_{\{1,2\}, \{3\}, \{4\} }} \sum_{c^{(1)}_{j_l}=0}^{q-1} 
\sum_{0\leq c^{(3)}_{j_l}\leq q-1: c^{(3)}_{j_l} \neq c^{(1)}_{j_l} }  
\sum_{0\leq c^{(4)}_{j_l}\leq q-1: c^{(4)}_{j_l} \neq c^{(1)}_{j_l}, c^{(3)}_{j_l} } ] 
\nonumber \\
&&\times
[ \prod_{l\in \Xi_{\{1,3\}, \{2\}, \{4\} }} \sum_{c^{(1)}_{j_l}=0}^{q-1} 
\sum_{0\leq c^{(2)}_{j_l}\leq q-1: c^{(2)}_{j_l} \neq c^{(1)}_{j_l} }  
\sum_{0\leq c^{(4)}_{j_l}\leq q-1: c^{(4)}_{j_l} \neq c^{(1)}_{j_l}, c^{(2)}_{j_l} } ] 
\nonumber \\
&&\times
[ \prod_{l\in \Xi_{\{1,4\}, \{2\}, \{3\} }} \sum_{c^{(1)}_{j_l}=0}^{q-1} 
\sum_{0\leq c^{(2)}_{j_l}\leq q-1: c^{(2)}_{j_l} \neq c^{(1)}_{j_l} }  
\sum_{0\leq c^{(3)}_{j_l}\leq q-1: c^{(3)}_{j_l} \neq c^{(1)}_{j_l}, c^{(3)}_{j_l} } ] 
 \nonumber \\
&&\times
[ \prod_{l\in \Xi_{\{2,3\}, \{1\}, \{4\} }} \sum_{c^{(2)}_{j_l}=0}^{q-1} 
\sum_{0\leq c^{(1)}_{j_l}\leq q-1: c^{(1)}_{j_l} \neq c^{(2)}_{j_l} }  
\sum_{0\leq c^{(4)}_{j_l}\leq q-1: c^{(4)}_{j_l} \neq c^{(1)}_{j_l}, c^{(2)}_{j_l} } ] 
 \nonumber \\
&&\times
[ \prod_{l\in \Xi_{\{2,4\}, \{1\}, \{3\} }} \sum_{c^{(2)}_{j_l}=0}^{q-1} 
\sum_{0\leq c^{(1)}_{j_l}\leq q-1: c^{(1)}_{j_l} \neq c^{(2)}_{j_l} }  
\sum_{0\leq c^{(3)}_{j_l}\leq q-1: c^{(3)}_{j_l} \neq c^{(1)}_{j_l}, c^{(2)}_{j_l} } ] 
\nonumber \\
&&\times
[ \prod_{l\in \Xi_{\{3,4\}, \{1\}, \{2\} }} \sum_{c^{(3)}_{j_l}=0}^{q-1} 
\sum_{0\leq c^{(1)}_{j_l}\leq q-1: c^{(1)}_{j_l} \neq c^{(3)}_{j_l} }  
\sum_{0\leq c^{(2)}_{j_l}\leq q-1: c^{(2)}_{j_l} \neq c^{(1)}_{j_l}, c^{(3)}_{j_l} } ] 
\nonumber \\
&&\times
[ \prod_{l\in \Xi_{\{1,2\}, \{3,4\} }} \sum_{c^{(1)}_{j_l}=0}^{q-1} 
\sum_{0\leq c^{(3)}_{j_l}\leq q-1: c^{(3)}_{j_l} \neq c^{(1)}_{j_l} } ] 
[ \prod_{l\in \Xi_{\{1,3\}, \{2,4\} }} \sum_{c^{(1)}_{j_l}=0}^{q-1} 
\sum_{0\leq c^{(2)}_{j_l}\leq q-1: c^{(2)}_{j_l} \neq c^{(1)}_{j_l} } ] 
\nonumber \\
&&\times
[ \prod_{l\in \Xi_{\{1,4\}, \{2,3\} }} \sum_{c^{(1)}_{j_l}=0}^{q-1} 
\sum_{0\leq c^{(2)}_{j_l}\leq q-1: c^{(2)}_{j_l} \neq c^{(1)}_{j_l} } ] 
[ \prod_{l\in \Xi_{\{1\}, \{2\}, \{3\}, \{4\} }} \sum_{c^{(1)}_{j_l}=0}^{q-1} 
\sum_{0\leq c^{(2)}_{j_l}\leq q-1: c^{(2)}_{j_l} \neq c^{(1)}_{j_l} }  
\nonumber \\
&&\times
\sum_{0\leq c^{(3)}_{j_l}\leq q-1: c^{(3)}_{j_l} \neq c^{(1)}_{j_l}, c^{(2)}_{j_l} }  
\sum_{0\leq c^{(4)}_{j_l}\leq q-1: c^{(4)}_{j_l} \neq c^{(1)}_{j_l}, c^{(2)}_{j_l}, c^{(3)}_{j_l} } ]  
\langle f,  \prod_{l=1}^{|u|}  \psi_{0, 0, c^{(1)}_{j_l}} \rangle  
\nonumber \\
&& \times
\langle f,  \prod_{l=1}^{|u|}  \psi_{0, 0, c^{(2)}_{j_l} } \rangle  
\langle f,  \prod_{l=1}^{|u|}  \psi_{0, 0, c^{(3)}_{j_l}} \rangle  
\langle f,  \prod_{l=1}^{|u|}  \psi_{0, 0, c^{(4)}_{j_l} } \rangle  
(q -4 + \frac{6}{q} - \frac{3}{q^2})^{| \Xi_{\{1,2,3,4\}} |}
\nonumber \\
&&\times 
(-1 + \frac{3}{q} - \frac{3}{q^2})^{| \Xi_{\{1,2,3\}, \{4\}} | + |\Xi_{\{1,2,4\},\{3\} }| + \Xi_{\{1,3,4\},\{2\}}| + \Xi_{\{2,3,4\},\{1\}}|}
\nonumber \\
&&\times
(\frac{1}{q} -\frac{3}{q^2})^{| \Xi_{\{1,2\},\{3\}, \{4\}} | + |\Xi_{\{1,3\},\{2\}, \{4\} }|
+ \Xi_{\{1,4\},\{2\}, \{3\}}| + |\Xi_{\{2,3\},\{1\},\{4\}}| 
+ |\Xi_{\{2,4\},\{1\}, \{3\}}| + | \Xi_{\{3,4\}, \{1\},\{2\}}|}
\nonumber \\
&&\times 
(\frac{2}{q} -\frac{3}{q^2})^{|\Xi_{\{1,2\}, \{3,4\}} |+ |\Xi_{\{1,3\}, \{2,4\}}| + |\Xi_{\{1,4\},\{2,3\}}|}
(-\frac{3}{q^2})^{|\Xi_{\{1\},\{2\},\{3\},\{4\}}| }
\nonumber \\
&=&
O(1) 
q^{|\Xi_{\{1,2,3,4\}}|} q^{2(|\Xi_{\{1,2,3\},\{4\}}| + |\Xi_{\{1,2,4\},\{3\}}| + |\Xi_{\{1,3,4\},\{2\}}| + |\Xi_{\{2,3,4\},\{1\}}| )} 
\nonumber \\
&&\times
q^{3 ( |\Xi_{\{1,2\},\{3\}, \{4\}} |+ |\Xi_{\{1,3\},\{2\},\{4\}}| + |\Xi_{\{1,4\},\{2\},\{3\}}| + | \Xi_{\{2,3\},\{1\},\{4\}}| + 
|\Xi_{\{2,4\},\{1\}, \{3\}}| + |\Xi_{\{3,4\},\{1\},\{2\}}|)}
\nonumber \\
&&\times
q^{2 (| \Xi_{\{1,2\},\{3,4\}} | + \Xi_{\{1,3\},\{2,4\}}| + |\Xi_{\{1,4\},\{2,3\}}|)} q^{4 | \Xi_{\{1\},\{2\}, \{3\}, \{4\}} |}
q^{-2|u|}
q^{| \Xi_{\{1,2,3,4\}} |}
\nonumber \\
&&\times
(\frac{1}{q})^{| \Xi_{\{1,2\},\{3\}, \{4\}} |+ |\Xi_{\{1,3\},\{2\},\{4\}}| + |\Xi_{\{1,4\},\{2\},\{3\}}| + | \Xi_{\{2,3\},\{1\},\{4\}}| + 
|\Xi_{\{2,4\},\{1\}, \{3\}}| + |\Xi_{\{3,4\},\{1\},\{2\}}|)}
\nonumber \\
&&\times
(\frac{2}{q})^{| \Xi_{\{1,2\},\{3,4\}} | + \Xi_{\{1,3\},\{2,4\}}| + |\Xi_{\{1,4\},\{2,3\}}} 
(\frac{3}{q^2})^{|\Xi_{\{1\},\{2\},\{3\},\{4\}}|}
\nonumber \\
&=& O(1),
\end{eqnarray*}
as $q\rightarrow \infty$. The last equality uses (\ref{eq:a.545}). This proves 
Lemma \ref{la:a.5}.\hfill $\Box$

\begin{pn} \label{pn:a.1}
Let $V_\ell$ and $\tilde{V}_\ell, \ell=1,\ldots, d-2,$ be as in (\ref{eq:5.30}) and (\ref{eq:5.47}) respectively. Then
\begin{displaymath}
E^{{\cal W}} ( \tilde{V}_\ell - V_\ell )
= - \frac{2 (\ell+2) }{d (q-1) } V_\ell.
\end{displaymath}
\end{pn}
\begin{proof}[\sc Proof.]
First we observe from (\ref{eq:3.75}) that
\begin{eqnarray*}
E^{{\cal W}} (S_\ell ) &=&
E^{\cal W} \Big\{
\frac{1}{d q^2 \sigma_\ell } \sum_{i=1}^{q^2}
\sum_{k=1}^d
\sum_{0\leq u_1,\cdots,u_d\leq 1: |u| =\ell+2}
{\cal I}\{ k\in \{j_1,\ldots, j_{|u|} \} \}
\nonumber \\
&&\hspace{0.5cm} \times
{\cal I} \{ \pi_k (a_{i,k}) \in \{B_1, B_2\} \}
\nu^* [\pi_{j_1} (a_{i,j_1}); 
\cdots; \pi_{j_{|u|}} (a_{i,j_{|u|}}) ]
\Big\}
\nonumber \\
&=&
\frac{ 2 (\ell +2)}{d q^3 \sigma_\ell } \sum_{i=1}^{q^2}
\sum_{0\leq u_1,\cdots,u_d\leq 1: |u| =\ell+2}
\nu^* [\pi_{j_1} (a_{i,j_1}); 
\cdots; \pi_{j_{|u|}} (a_{i,j_{|u|}}) ],
\end{eqnarray*}
and 
\begin{eqnarray*}
E^{\cal W} (\tilde{S}_\ell ) &=&
E^{\cal W} \Big\{
\frac{1}{d q^2 \sigma_\ell } \sum_{i=1}^{q^2}
\sum_{k=1}^d
\sum_{0\leq u_1,\cdots,u_d\leq 1: |u| =\ell+2}
{\cal I}\{ k\in \{j_1,\ldots, j_{|u|} \} \}
\nonumber \\
&&\hspace{0.5cm} \times
{\cal I} \{ \pi_k (a_{i,k}) \in \{B_1, B_2\} \}
\nu^* [\tilde{\pi}_{j_1} (a_{i,j_1}); 
\cdots; \tilde{\pi}_{j_{|u|}} (a_{i,j_{|u|}}) ]
\Big\}
\nonumber \\
&=&
E^{\cal W} \Big\{
\frac{1}{d q^2 \sigma_\ell } \sum_{i=1}^{q^2}
\sum_{k=1}^d
\sum_{0\leq u_1,\cdots,u_d\leq 1: |u| =\ell +2}
\sum_{l=1}^{|u|}
{\cal I}\{ k= j_l\}
\nonumber \\
&&\hspace{0.5cm} \times
{\cal I} \{ \pi_{j_l} (a_{i, j_l}) \in \{B_1, B_2\} \}
\nu^* [\pi_{j_1} (a_{i,j_1});\ldots; \pi_{j_{l-1}} (a_{i, j_{l-1}}); 
\nonumber \\
&&\hspace{1cm}
\tau_{B_1, B_2} \circ \pi_{j_l} (a_{i, j_l}); \pi_{j_{l+1}} (a_{i, j_{l+1}}); \cdots; \tilde{\pi}_{j_{|u|}} (a_{i,j_{|u|}}) ]
\Big\}
\nonumber \\
&=&
E^{\cal W} \Big\{
\frac{1}{d q^2 \sigma_\ell } \sum_{i=1}^{q^2}
\sum_{k=1}^d
\sum_{0\leq u_1,\cdots,u_d\leq 1: |u| =\ell+2}
\sum_{l=1}^{|u|}
{\cal I}\{ k= j_l\}
\nonumber \\
&&\hspace{0.5cm} \times
{\cal I} \{ \pi_{j_l} (a_{i, j_l}) =B_1 \}
\nu^* [\pi_{j_1} (a_{i,j_1});\ldots; \pi_{j_{l-1}} (a_{i, j_{l-1}}); 
\nonumber \\
&&\hspace{1cm}
B_2; \pi_{j_{l+1}} (a_{i, j_{l+1}}); \cdots; \tilde{\pi}_{j_{|u|}} (a_{i,j_{|u|}}) ]
\Big\}
\nonumber \\
&&
+ E^{\cal W} \Big\{
\frac{1}{d q^2 \sigma_\ell } \sum_{i=1}^{q^2}
\sum_{k=1}^d
\sum_{0\leq u_1,\cdots,u_d\leq 1: |u| =\ell+2}
\sum_{l=1}^{|u|}
{\cal I}\{ k= j_l\}
\nonumber \\
&&\hspace{0.5cm} \times
{\cal I} \{ \pi_{j_l} (a_{i, j_l}) = B_2 \}
\nu^* [\pi_{j_1} (a_{i,j_1});\ldots; \pi_{j_{l-1}} (a_{i, j_{l-1}}); 
\nonumber \\
&&\hspace{1cm}
B_1; \pi_{j_{l+1}} (a_{i, j_{l+1}}); \cdots; \tilde{\pi}_{j_{|u|}} (a_{i,j_{|u|}}) ]
\Big\}
\nonumber \\
&=&
-E^{\cal W} \Big\{
\frac{1}{d q^2 (q-1) \sigma_\ell } \sum_{i=1}^{q^2}
\sum_{k=1}^d
\sum_{0\leq u_1,\cdots,u_d\leq 1: |u| =\ell+2}
\sum_{l=1}^{|u|}
{\cal I}\{ k= j_l\}
\nonumber \\
&&\hspace{0.5cm} \times
{\cal I} \{ \pi_{j_l} (a_{i, j_l}) =B_1 \}
\nu^* [\pi_{j_1} (a_{i,j_1});\ldots; \pi_{j_{|u|}} (a_{i, j_{|u|}}) ]
\Big\}
\nonumber \\
&&
- E^{\cal W} \Big\{
\frac{1}{d q^2 (q-1) \sigma_\ell } \sum_{i=1}^{q^2}
\sum_{k=1}^d
\sum_{0\leq u_1,\cdots,u_d\leq 1: |u| =\ell+2}
\sum_{l=1}^{|u|}
{\cal I}\{ k= j_l\}
\nonumber \\
&&\hspace{0.5cm} \times
{\cal I} \{ \pi_{j_l} (a_{i, j_l}) = B_2 \}
\nu^* [\pi_{j_1} (a_{i,j_1});\ldots; \pi_{j_{|u|}} (a_{i, j_{|u|}})  ]
\Big\}
\nonumber \\
&=&
- \frac{2 (\ell+2) }{d q^3 (q-1) \sigma_\ell } \sum_{i=1}^{q^2}
\sum_{0\leq u_1,\cdots,u_d\leq 1: |u| =\ell+2}
\nu^* [\pi_{j_1} (a_{i,j_1});\ldots; \pi_{j_{|u|}} (a_{i, j_{|u|}}) ].
\end{eqnarray*}
Thus we conclude that
\begin{eqnarray*}
E^{\cal W} (\tilde{V}_\ell - V_\ell ) 
&=& E^{\cal W} (\tilde{S}_\ell - S_\ell )
\nonumber \\
&=&
- \frac{2 (\ell +2)}{d q^2 (q-1) \sigma_\ell } \sum_{i=1}^{q^2}
\sum_{0\leq u_1,\cdots,u_d\leq 1: |u| =\ell +2}
\nu^* [\pi_{j_1} (a_{i,j_1});\ldots; \pi_{j_{|u|}} (a_{i, j_{|u|}}) ]
\nonumber \\
&=&
-\frac{2 (\ell+2) }{d (q-1) } V_\ell. \qedhere 
\end{eqnarray*}
\end{proof}

\begin{pn} \label{pn:a.2}
Let $S_i$ and $\tilde{S}_i$, $i=1,\ldots, d-2$, be as in (\ref{eq:3.75}).
Then
\begin{displaymath}
| \frac{ d(q-1)}{4 (i+2)} E[ (\tilde{S}_i - S_i)^2  ] -1 |
= O(1/q), \hspace{0.5cm}\forall i=1,\ldots, d-2,
\end{displaymath}
and
\begin{equation}
| \sum_{i=1}^{d-2} \{ [ \frac{ d(q-1)}{4 (i+2)} E(\tilde{S}_i - S_i)^2  -1] 
E [\frac{ \partial^2}{\partial v_i^2} \psi_{\varepsilon^2} (V) ] \} |
\leq O( \frac{\|h\|_\infty }{q}) \log(1/\varepsilon),
\label{eq:a.69}
\end{equation}
as $q\rightarrow \infty$ uniformly over $0<\varepsilon <1$.
\end{pn}
\begin{proof}[\sc Proof.]
For $\ell=1,\ldots, d-2$ and $k=1,2$, we write
\begin{eqnarray}
\tilde{S}_{\ell, k} &=& \frac{1}{q^2 \sigma_{\ell} } \sum_{i=1}^{q^2} \sum_{0\leq u_1,\ldots, u_d \leq 1: |u|=\ell +2}
{\cal I}\{J\in \{j_1,\ldots, j_{|u|} \}, \pi_J (a_{i,J} ) = B_k\}
\nonumber \\
&&\hspace{0.5cm}\times
\nu^* [\tilde{\pi}_{j_1} (a_{i, j_1});\ldots; \tilde{\pi}_{j_{|u|}} (a_{i, j_{|u|}}) ],
\nonumber \\
S_{\ell, k} &=& \frac{1}{q^2 \sigma_{\ell} } \sum_{i=1}^{q^2} \sum_{0\leq u_1,\ldots, u_d \leq 1: |u|=\ell +2}
{\cal I}\{J\in \{j_1,\ldots, j_{|u|} \}, \pi_J (a_{i,J} ) = B_k \}
\nonumber \\
&&\hspace{0.5cm}\times
\nu^* [\pi_{j_1} (a_{i, j_1});\ldots; \pi_{j_{|u|}} (a_{i, j_{|u|}}) ].
\label{eq:a.78}
\end{eqnarray}
Then for $i=1,\ldots, d-2$,
\begin{eqnarray}
\frac{ d(q-1)}{4 (i+2)} E[ (\tilde{S}_i - S_i)^2 ]
&=& \frac{ d(q-1)}{4 (i+2)} E[ (\tilde{S}_{i,1} + \tilde{S}_{i,2} - S_{i,1} - S_{i,2} )^2 ]
\nonumber \\
&=& \frac{ d(q-1)}{4 (i+2)} E ( \tilde{S}_{i,1}^2 + \tilde{S}_{i,2}^2 + S_{i,1}^2 + S_{i,2}^2 
+ 2 \tilde{S}_{i,1} \tilde{S}_{i,2} + 2 S_{i,1} S_{i,2} 
\nonumber \\
&&
-2 \tilde{S}_{i,1} S_{i,1} 
-2 \tilde{S}_{i,1} S_{i,2} 
-2 \tilde{S}_{i,2} S_{i,1} 
-2 \tilde{S}_{i,2} S_{i,2} )
\nonumber \\
&=& \frac{ d(q-1)}{4 (i+2)} E ( 4 S_{i,1}^2
+ 4 S_{i,1} S_{i,2} 
-4 \tilde{S}_{i,1} S_{i,1} 
-4 \tilde{S}_{i,1} S_{i,2} ).
\label{eq:a.48}
\end{eqnarray}
As in Proposition \ref{pn:3.1}, we have
for $i=1,\ldots, d-2$,
\begin{eqnarray*}
\sigma_i^2
&=&
\frac{1}{q^2 }
\sum_{0\leq u_1,\ldots, u_d \leq 1: |u|= i+2}
E \{ \nu^* 
[ \pi_{j_1} (a_{1,j_1});\ldots;
\pi_{j_{|u|}} (a_{1, j_{|u|}})]^2 \}
\nonumber \\
&&
+ \sum_{0\leq u_1,\ldots, u_d \leq 1: |u|=i+2}
\frac{ (-1)^{|u|-1} |u| }{q^2 (q-1)^{|u|-2} } 
E \{ \nu^* 
[ \pi_{j_1} (a_{1,j_1});\ldots, \pi_{j_{|u|}} (a_{1, j_{|u|}}) ]^2 \}
\nonumber \\
&&
+ \sum_{0\leq u_1,\ldots, u_d \leq 1: |u|= i+2}
\frac{ (-1)^{|u|} (q+1 -|u|) }{ q^2 (q-1)^{|u|-1} }
E \{ \nu^* 
[ \pi_{j_1} (a_{1,j_1}); \ldots; \pi_{j_{|u|}} (a_{1, j_{|u|}}) ]^2 \}
\nonumber \\
&=& \frac{1}{q^2 } [ 1 + O(\frac{1}{q})]
\sum_{0\leq u_1,\cdots, u_d \leq 1: |u|= i+2}
E \{ \nu^*
[ \pi_{j_1} (a_{1,j_1});\ldots;
\pi_{j_{|u|}} (a_{1, j_{|u|}})]^2 \},
\end{eqnarray*}
as $q\rightarrow\infty$.
Hence it follows from (\ref{eq:a.48}) and Lemma \ref{la:a.17} that
\begin{equation}
\frac{ d(q-1)}{4 (i+2)} E[ (\tilde{S}_i - S_i)^2 ]
= 1 + O(1/q), \hspace{0.5cm}\forall i=1,\ldots, d-2,
\label{eq:a.705}
\end{equation}
as $q\rightarrow \infty$. Finally, (\ref{eq:a.69}) is an immediate consequence of 
(\ref{eq:a.705}) and Lemma \ref{la:5.1}.
\end{proof}

\begin{la} \label{la:a.17}
Let $S_{\ell, k}$ and $\tilde{S}_{\ell, k}$, $\ell =1,\ldots, d-2$, $k=1,2$,  be as in (\ref{eq:a.78}). Then
\begin{eqnarray*}
&& E (\tilde{S}_{\ell, k}^2 ) = E (S_{\ell, k}^2 )
\nonumber \\
&=& 
\frac{\ell +2}{d q^3 \sigma_\ell^2 } \sum_{0\leq u_1,\ldots, u_d \leq 1: |u|= \ell+2}
E\{ \nu^* [\pi_{j_1} (a_{1, j_1});\ldots; \pi_{j_{|u|}} (a_{1, j_{|u|}}) ]^2\}
\nonumber \\
&& 
+ \frac{(-1)^{\ell -1} (\ell +2)}{d q^3 (q-1)^\ell \sigma_\ell^2 } \sum_{0\leq u_1,\ldots, u_d \leq 1: |u|=\ell+2}
E\{ \nu^* [\pi_{j_1} (a_{1, j_1});\ldots; \pi_{j_{|u|}} (a_{1, j_{|u|}}) ]^2\},
\nonumber \\
&& E(\tilde{S}_{\ell, 1} \tilde{S}_{\ell, 2} ) = E(S_{\ell, 1} S_{\ell, 2} )
\nonumber \\
&=& 
\frac{ (-1)^{\ell-1} \ell (\ell +2) }{d q^2 (q-1)^{\ell +2} \sigma_\ell^2 } \sum_{0\leq u_1,\ldots, u_d \leq 1: |u|=\ell+2}
E\{ \nu^* [\pi_{j_1} (a_{1, j_1});\ldots; \pi_{j_{|u|}} (a_{1, j_{|u|}}) ]^2 \},
\nonumber \\
&& E(\tilde{S}_{\ell, 1} S_{\ell, 2}) = E(S_{\ell, 1} \tilde{S}_{\ell, 2} )
\nonumber \\
&=&
 \frac{ (-1)^\ell \ell (\ell +2) }{ d q^2 (q-1)^{\ell+1}  \sigma_\ell^2 } \sum_{0\leq u_1,\ldots, u_d \leq 1: |u|=\ell+2}
E\{ \nu^* [\pi_{j_1} (a_{1, j_1});\ldots;
\pi_{j_{|u|}} (a_{1, j_{|u|}}) ]^2 \},
\nonumber \\
&& E(\tilde{S}_{\ell, 1} S_{\ell,1}) = E(\tilde{S}_{\ell,2} S_{\ell,2})
\nonumber \\
&=&
- \frac{ \ell+2 }{ d q^3 (q-1) \sigma_\ell^2 }  \sum_{0\leq u_1,\ldots, u_d \leq 1: |u|=\ell+2}
E\{ \nu^* [\pi_{j_1} (a_{1, j_1});\ldots;
\pi_{j_{|u|}} (a_{1, j_{|u|}}) ]^2 \}
\nonumber \\
&& 
+ \frac{ (-1)^\ell (\ell+2) }{ d q^3 (q-1)^{\ell+1} \sigma_\ell^2 } \sum_{0\leq u_1,\ldots, u_d \leq 1: |u|=\ell+2}
E\{ \nu^* [\pi_{j_1} (a_{1, j_1});\ldots;
\pi_{j_{|u|}} (a_{1, j_{|u|}}) ]^2 \}.
\end{eqnarray*}
\end{la}
{\sc Proof.}
We observe that for $\ell=1,\ldots, d-2$ and $k=1,2$,
\begin{eqnarray*}
&& E (S_{\ell, k}^2 ) 
\nonumber \\
&=&
\frac{ 1}{ q^4 \sigma_\ell^2 } E\Big\{ \sum_{i_1=1}^{q^2} \sum_{i_2=1}^{q^2} \sum_{0\leq u_1,\ldots, u_d \leq 1: |u|=\ell+2}
{\cal I}\{J\in \{j_1,\ldots, j_{|u|} \}\}
\nonumber \\
&& \hspace{0.5cm}\times
{\cal I}\{ \pi_J (a_{i_1,J} ) = \pi_J (a_{i_2,J}) = B_k\}
\nu^* [\pi_{j_1} (a_{i_1, j_1});\ldots; \pi_{j_{|u|}} (a_{i_1, j_{|u|}}) ]
\nonumber \\
&&\hspace{0.5cm}\times
\nu^* [\pi_{j_1} (a_{i_2, j_1});\ldots; \pi_{j_{|u|}} (a_{i_2, j_{|u|}}) ]
\Big\}
\nonumber \\
&=& 
\frac{ 1}{ q^4 \sigma_\ell^2 } E\Big\{ \sum_{i_1=1}^{q^2} \sum_{0\leq u_1,\ldots, u_d \leq 1: |u|=\ell+2}
{\cal I}\{J\in \{j_1,\ldots, j_{|u|} \}\}
\nonumber \\
&&\hspace{0.5cm}\times
{\cal I}\{ \pi_J (a_{i_1,J} ) = B_k\}
\nu^* [\pi_{j_1} (a_{i_1, j_1});\ldots; \pi_{j_{|u|}} (a_{i_1, j_{|u|}}) ]^2 \Big\}
\nonumber \\
&& 
+ \frac{ 1}{ q^4 \sigma_\ell^2 } E\Big\{ \sum_{i_1=1}^{q^2} \sum_{i_2\neq i_1} \sum_{0\leq u_1,\ldots, u_d \leq 1: |u|=\ell+2}
{\cal I}\{J\in \{j_1,\ldots, j_{|u|} \}\}
\nonumber \\
&& \hspace{0.5cm}\times
{\cal I}\{ \pi_J (a_{i_1,J} ) = \pi_J (a_{i_2,J}) = B_k\} (-\frac{1}{q-1})^{|u|-1}
\nonumber \\
&&\hspace{0.5cm}\times
\nu^* [\pi_{j_1} (a_{i_1, j_1});\ldots; \pi_{j_{|u|}} (a_{i_1, j_{|u|}}) ]^2 \Big\}
\nonumber \\
&=& 
\frac{1}{ d q^5 \sigma_\ell^2 } \sum_{i_1=1}^{q^2} \sum_{0\leq u_1,\ldots, u_d \leq 1: |u|=\ell+2}
\sum_{k=1}^d {\cal I}\{k\in \{j_1,\ldots, j_{|u|} \}\}
\nonumber \\
&&\hspace{0.5cm}\times
E\{ \nu^* [\pi_{j_1} (a_{i_1, j_1});\ldots; \pi_{j_{|u|}} (a_{i_1, j_{|u|}}) ]^2\}
\nonumber \\
&& 
+ \frac{1}{d q^5 \sigma_\ell^2 } \sum_{i_1=1}^{q^2} \sum_{i_2\neq i_1} \sum_{0\leq u_1,\ldots, u_d \leq 1: |u|=\ell+2}
\sum_{k=1}^d {\cal I}\{k \in \{j_1,\ldots, j_{|u|} \}\}
\nonumber \\
&& \hspace{0.5cm}\times
{\cal I}\{ a_{i_1,k} = a_{i_2,k} \} (-\frac{1}{q-1})^{|u|-1}
E\{ \nu^* [\pi_{j_1} (a_{i_1, j_1});\ldots; \pi_{j_{|u|}} (a_{i_1, j_{|u|}}) ]^2\}
\nonumber \\
&=& 
\frac{\ell+2}{d q^3 \sigma_\ell^2 } \sum_{0\leq u_1,\ldots, u_d \leq 1: |u|=\ell+2}
E\{ \nu^* [\pi_{j_1} (a_{1, j_1});\ldots; \pi_{j_{|u|}} (a_{1, j_{|u|}}) ]^2\}
\nonumber \\
&& 
+ \frac{\ell +2}{d q^3 (q-1) \sigma_\ell^2 } \sum_{0\leq u_1,\ldots, u_d \leq 1: |u|=\ell+2}
(-\frac{1}{q-1})^{|u|- 3}
\nonumber \\
&&\hspace{0.5cm}\times
E\{ \nu^* [\pi_{j_1} (a_{1, j_1});\ldots; \pi_{j_{|u|}} (a_{1, j_{|u|}}) ]^2\},
\nonumber \\
&& E( S_{\ell,1} S_{\ell, 2})
\nonumber \\
&=& 
\frac{ 1}{ q^4 \sigma_\ell^2 } E\Big\{ \sum_{i_1=1}^{q^2} \sum_{i_2=1}^{q^2} \sum_{0\leq u_1,\ldots, u_d \leq 1: |u|=\ell+2}
{\cal I}\{J\in \{j_1,\ldots, j_{|u|} \}\}
\nonumber \\
&& \hspace{0.5cm}\times
{\cal I}\{ \pi_J (a_{i_1,J} ) = B_1
\neq \pi_J (a_{i_2,J}) = B_2\}
\nonumber \\
&&\hspace{0.5cm}\times
\nu^* [\pi_{j_1} (a_{i_1, j_1});\ldots; \pi_{j_{|u|}} (a_{i_1, j_{|u|}}) ]
\nu^* [\pi_{j_1} (a_{i_2, j_1});\ldots; \pi_{j_{|u|}} (a_{i_2, j_{|u|}}) ]
\Big\}
\nonumber \\
&=& 
\frac{ 1 }{ d q^5 (q-1) \sigma_\ell^2 } \sum_{i_1=1}^{q^2} \sum_{i_2\neq i_1} \sum_{0\leq u_1,\ldots, u_d \leq 1: |u|=\ell+2}
\sum_{k=1}^d {\cal I}\{ k\in \{j_1,\ldots, j_{|u|} \}\}
\nonumber \\
&& \hspace{0.5cm}\times
{\cal I}\{ a_{i_1,k}
\neq a_{i_2, k} \}
E\{ \nu^* [\pi_{j_1} (a_{i_1, j_1});\ldots; \pi_{j_{|u|}} (a_{i_1, j_{|u|}}) ]
\nonumber \\
&&\hspace{0.5cm}\times
\nu^* [\pi_{j_1} (a_{i_2, j_1});\ldots; \pi_{j_{|u|}} (a_{i_2, j_{|u|}}) ]
\}
\nonumber \\
&=& 
\frac{ 1 }{ d q^5 (q-1) \sigma_\ell^2 } \sum_{i_1=1}^{q^2} \sum_{i_2\neq i_1} \sum_{0\leq u_1,\ldots, u_d \leq 1: |u|=\ell+2}
\sum_{k=1}^d {\cal I}\{ k\in \{j_1,\ldots, j_{|u|} \}\}
\nonumber \\
&&\hspace{0.5cm}\times
\sum_{1\leq l\leq |u|: j_l\neq k} {\cal I}\{ a_{i_1, j_l}
= a_{i_2, j_l} \}
E\{ \nu^* [\pi_{j_1} (a_{i_1, j_1});\ldots; \pi_{j_{|u|}} (a_{i_1, j_{|u|}}) ]
\nonumber \\
&& \hspace{0.5cm}\times
\nu^* [\pi_{j_1} (a_{i_2, j_1});\ldots; \pi_{j_{|u|}} (a_{i_2, j_{|u|}}) ]
\}
\nonumber \\
&& 
+ \frac{ 1 }{ d q^5 (q-1) \sigma_\ell^2 } \sum_{i_1=1}^{q^2} \sum_{i_2\neq i_1} \sum_{0\leq u_1,\ldots, u_d \leq 1: |u|=\ell+2}
\sum_{k=1}^d {\cal I}\{ k\in \{j_1,\ldots, j_{|u|} \}\}
\nonumber \\
&&\hspace{0.5cm}\times
{\cal I}\{ a_{i_1, j_l}
\neq a_{i_2, j_l}, \forall 1\leq l\leq |u| \}
E\{ \nu^* [\pi_{j_1} (a_{i_1, j_1});\ldots; \pi_{j_{|u|}} (a_{i_1, j_{|u|}}) ]
\nonumber \\
&&\hspace{0.5cm}\times
\nu^* [\pi_{j_1} (a_{i_2, j_1});\ldots; \pi_{j_{|u|}} (a_{i_2, j_{|u|}}) ]
\}
\nonumber \\
&=& 
\frac{ 1 }{ d q^5 (q-1) \sigma_\ell^2 } \sum_{i_1=1}^{q^2} \sum_{i_2\neq i_1} \sum_{0\leq u_1,\ldots, u_d \leq 1: |u|=\ell+2}
\sum_{k=1}^d {\cal I}\{ k\in \{j_1,\ldots, j_{|u|} \}\}
\nonumber \\
&&\hspace{0.5cm}\times
\sum_{1\leq l\leq |u|: j_l\neq k} {\cal I}\{ a_{i_1, j_l}
= a_{i_2, j_l} \} (-\frac{1}{q-1})^{|u|-1}
E\{ \nu^* [\pi_{j_1} (a_{i_1, j_1});\ldots; \pi_{j_{|u|}} (a_{i_1, j_{|u|}}) ]^2 \}
\nonumber \\
&& 
+ \frac{ 1 }{ d q^5 (q-1) \sigma_\ell^2 } \sum_{i_1=1}^{q^2} \sum_{i_2\neq i_1} \sum_{0\leq u_1,\ldots, u_d \leq 1: |u|=\ell+2}
\sum_{k=1}^d {\cal I}\{ k\in \{j_1,\ldots, j_{|u|} \}\}
\nonumber \\
&&\hspace{0.5cm}\times
{\cal I}\{ a_{i_1, j_l}
\neq a_{i_2, j_l}, \forall 1\leq l\leq |u| \} (-\frac{1}{q-1})^{|u|}
E\{ \nu^* [\pi_{j_1} (a_{i_1, j_1});\ldots; \pi_{j_{|u|}} (a_{i_1, j_{|u|}}) ]^2 \}
\nonumber \\
&=& 
\frac{ 1 }{ d q^5 \sigma_\ell^2 } \sum_{i_1=1}^{q^2} \sum_{0\leq u_1,\ldots, u_d \leq 1: |u|=\ell+2}
|u| (|u|-1)
(-\frac{1}{q-1})^{|u|-1}
\nonumber \\
&& \hspace{0.5cm}\times
E\{ \nu^* [\pi_{j_1} (a_{1, j_1});\ldots; \pi_{j_{|u|}} (a_{1, j_{|u|}}) ]^2 \}
\nonumber \\
&& 
+ \frac{ 1 }{ d q^5 (q-1) \sigma_\ell^2 } \sum_{i_1=1}^{q^2} \sum_{0\leq u_1,\ldots, u_d \leq 1: |u|=\ell+2}
[q^2 -1 - |u|(q-1)] 
\nonumber \\
&& \hspace{0.5cm}\times
|u| (-\frac{1}{q-1})^{|u|}
E\{ \nu^* [\pi_{j_1} (a_{1, j_1});\ldots; \pi_{j_{|u|}} (a_{1, j_{|u|}}) ]^2 \}
\nonumber \\
&=& 
\frac{ 1 }{ d q^3 (q-1)^2 \sigma_\ell^2 } \sum_{0\leq u_1,\ldots, u_d \leq 1: |u|=\ell+2}
|u| (|u|-1)
(-\frac{1}{q-1})^{|u|-3}
\nonumber \\
&& \hspace{0.5cm}\times
E\{ \nu^* [\pi_{j_1} (a_{1, j_1});\ldots; \pi_{j_{|u|}} (a_{1, j_{|u|}}) ]^2 \}
\nonumber \\
&& 
+ \frac{ 1 }{ d q^3 \sigma_\ell^2 } \sum_{0\leq u_1,\ldots, u_d \leq 1: |u|=\ell+2}
(q + 1 - |u|) |u| (-\frac{1}{q-1})^{|u|}
\nonumber \\
&& \hspace{0.5cm}\times
E\{ \nu^* [\pi_{j_1} (a_{1, j_1});\ldots; \pi_{j_{|u|}} (a_{1, j_{|u|}}) ]^2 \}
\nonumber \\
&=& 
\frac{ 1 }{ d q^3 (q-1)^2 \sigma_\ell^2 } \sum_{0\leq u_1,\ldots, u_d \leq 1: |u|=\ell+2}
|u| (|u|-1)
(-\frac{1}{q-1})^{|u|-3}
\nonumber \\
&& \hspace{0.5cm}\times
E\{ \nu^* [\pi_{j_1} (a_{1, j_1});\ldots; \pi_{j_{|u|}} (a_{1, j_{|u|}}) ]^2 \}
\nonumber \\
&& 
- \frac{ 1 }{ d q^3 (q-1)^2 \sigma_\ell^2 } \sum_{0\leq u_1,\ldots, u_d \leq 1: |u|=\ell+2}
(1 - \frac{ |u|-2}{ q-1} ) |u| (-\frac{1}{q-1})^{|u|-3}
\nonumber \\
&& \hspace{0.5cm}\times
E\{ \nu^* [\pi_{j_1} (a_{1, j_1});\ldots; \pi_{j_{|u|}} (a_{1, j_{|u|}}) ]^2 \}
\nonumber \\
&=&
\frac{ 1 }{ d q^3 (q-1)^2 \sigma_\ell^2 } \sum_{0\leq u_1,\ldots, u_d \leq 1: |u|=\ell+2}
( |u|^2 - 2 |u| + \frac{ |u| (|u|-2) }{q-1} ) 
\nonumber \\
&& \hspace{0.5cm}\times
(-\frac{1}{q-1})^{|u|-3}
E\{ \nu^* [\pi_{j_1} (a_{1, j_1});\ldots; \pi_{j_{|u|}} (a_{1, j_{|u|}}) ]^2 \},
\nonumber \\
&& E( \tilde{S}_{\ell,1} S_{\ell,2})
\nonumber \\
&=&
\frac{ 1 }{ q^4 \sigma_\ell^2 } E\Big\{ \sum_{i_1=1}^{q^2} \sum_{i_2=1}^{q^2} \sum_{0\leq u_1,\ldots, u_d \leq 1: |u|=\ell+2}
{\cal I}\{J\in \{j_1,\ldots, j_{|u|} \}\}
\nonumber \\
&& \hspace{0.5cm}\times
{\cal I}\{ \pi_J (a_{i_1,J} ) = B_1\}
{\cal I}\{ \pi_J (a_{i_2,J}) = B_2 \}
\nonumber \\
&&\hspace{0.5cm}\times
\nu^* [\tilde{\pi}_{j_1} (a_{i_1, j_1});\ldots; \tilde{\pi}_{j_{|u|}} (a_{i_1, j_{|u|}}) ]
\nu^* [\pi_{j_1} (a_{i_2, j_1});\ldots; \pi_{j_{|u|}} (a_{i_2, j_{|u|}}) ]
\Big\}
\nonumber \\
&=& 
\frac{ 1 }{ d q^4 \sigma_\ell^2 } E\Big\{ \sum_{i_1=1}^{q^2} \sum_{i_2\neq i_1} \sum_{0\leq u_1,\ldots, u_d \leq 1: |u|=\ell+2}
\sum_{k=1}^d {\cal I}\{k\in \{j_1,\ldots, j_{|u|} \}\}
\nonumber \\
&& \hspace{0.5cm}\times
{\cal I}\{ \pi_k (a_{i_1, k} ) = B_1\neq
\pi_k (a_{i_2, k}) = B_2 \}
\nonumber \\
&&\hspace{0.5cm}\times
\nu^* [\pi_{j_1} (a_{i_1, j_1});\ldots; \tau_{B_1, B_2} \circ \pi_k (a_{i_1, k});
\ldots; \pi_{j_{|u|}} (a_{i_1, j_{|u|}}) ]
\nonumber \\
&&\hspace{0.5cm}\times
\nu^* [\pi_{j_1} (a_{i_2, j_1});\ldots; \pi_{j_{|u|}} (a_{i_2, j_{|u|}}) ]
\Big\}
\nonumber \\
&=& 
\frac{ 1 }{ d q^5 (q-1) \sigma_\ell^2 } E\Big\{ \sum_{i_1=1}^{q^2} \sum_{i_2\neq i_1} \sum_{0\leq u_1,\ldots, u_d \leq 1: |u|=\ell+2}
\sum_{k=1}^d {\cal I}\{k\in \{j_1,\ldots, j_{|u|} \}\}
\nonumber \\
&& \hspace{0.5cm}\times
{\cal I}\{ a_{i_1, k} \neq
a_{i_2, k} \}
\nu^* [\pi_{j_1} (a_{i_1, j_1});\ldots; \pi_k (a_{i_2, k});
\ldots; \pi_{j_{|u|}} (a_{i_1, j_{|u|}}) ]
\nonumber \\
&&\hspace{0.5cm}\times
\nu^* [\pi_{j_1} (a_{i_2, j_1});\ldots; \pi_{j_{|u|}} (a_{i_2, j_{|u|}}) ]
\Big\}
\nonumber \\
&=& 
\frac{ 1 }{ d q^5 (q-1) \sigma_\ell^2 } E\Big\{ \sum_{i_1=1}^{q^2} \sum_{i_2\neq i_1} \sum_{0\leq u_1,\ldots, u_d \leq 1: |u|=\ell+2}
\sum_{k=1}^d 
{\cal I}\{k\in \{j_1,\ldots, j_{|u|} \}\}
\nonumber \\
&&\hspace{0.5cm}\times
\sum_{1\leq l\leq |u|: j_l\neq k} {\cal I}\{ a_{i_1, j_l} = a_{i_2, j_l}\} 
\nu^* [\pi_{j_1} (a_{i_1, j_1});\ldots; \pi_k (a_{i_2, k});
\ldots; \pi_{j_{|u|}} (a_{i_1, j_{|u|}}) ]
\nonumber \\
&&\hspace{0.5cm}\times
\nu^* [\pi_{j_1} (a_{i_2, j_1});\ldots; \pi_{j_{|u|}} (a_{i_2, j_{|u|}}) ]
\Big\}
\nonumber \\
&& 
+ \frac{ 1 }{ d q^5 (q-1) \sigma_\ell^2 } E\Big\{ \sum_{i_1=1}^{q^2} \sum_{i_2\neq i_1} \sum_{0\leq u_1,\ldots, u_d \leq 1: |u|=\ell+2}
\sum_{k=1}^d 
{\cal I}\{k\in \{j_1,\ldots, j_{|u|} \}\}
\nonumber \\
&& \hspace{0.5cm}\times
{\cal I}\{ a_{i_1, j_l} \neq
a_{i_2, j_l}, \forall 1\leq l\leq |u| \}
\nu^* [\pi_{j_1} (a_{i_1, j_1});\ldots; \pi_k (a_{i_2, k});
\ldots; \pi_{j_{|u|}} (a_{i_1, j_{|u|}}) ]
\nonumber \\
&&\hspace{0.5cm}\times
\nu^* [\pi_{j_1} (a_{i_2, j_1});\ldots; \pi_{j_{|u|}} (a_{i_2, j_{|u|}}) ]
\Big\}
\nonumber \\
&=& 
\frac{ 1 }{ d q^5 (q-1) \sigma_\ell^2 } \sum_{i_1=1}^{q^2} \sum_{0\leq u_1,\ldots, u_d \leq 1: |u|=\ell+2}
|u| (|u|-1) (q-1) 
\nonumber \\
&&\hspace{0.5cm}\times
(-\frac{1}{q-1})^{|u|-2}
E\{ \nu^* [\pi_{j_1} (a_{i_1, j_1});\ldots;
\pi_{j_{|u|}} (a_{i_1, j_{|u|}}) ]^2 \}
\nonumber \\
&& 
+ \frac{ 1 }{ d q^5 (q-1) \sigma_\ell^2 } \sum_{i_1=1}^{q^2} \sum_{0\leq u_1,\ldots, u_d \leq 1: |u|=\ell+2}
|u| [ q^2-1 - |u|(q-1)]
\nonumber \\
&&\hspace{0.5cm}\times
(-\frac{1}{q-1})^{|u|-1}
E\{ \nu^* [\pi_{j_1} (a_{i_1, j_1});\ldots; 
\pi_{j_{|u|}} (a_{i_1, j_{|u|}}) ]^2 \}
\nonumber \\
&=& 
- \frac{ 1 }{ d q^3 (q-1) \sigma_\ell^2 } \sum_{0\leq u_1,\ldots, u_d \leq 1: |u|=\ell+2}
|u| (|u|-1) (-\frac{1}{q-1})^{|u|- 3}
\nonumber \\
&&\hspace{0.5cm}\times
E\{ \nu^* [\pi_{j_1} (a_{1, j_1});\ldots;
\pi_{j_{|u|}} (a_{1, j_{|u|}}) ]^2 \}
\nonumber \\
&& 
+ \frac{ 1 }{ d q^3 (q-1) \sigma_\ell^2 } \sum_{0\leq u_1,\ldots, u_d \leq 1: |u|=\ell+2}
|u| (1- \frac{ |u|-2}{q-1})
(-\frac{1}{q-1})^{|u|- 3}
\nonumber \\
&&\hspace{0.5cm}\times
E\{ \nu^* [\pi_{j_1} (a_{1, j_1});\ldots; 
\pi_{j_{|u|}} (a_{1, j_{|u|}}) ]^2 \}
\nonumber \\
&=& 
- \frac{ 1 }{ d q^3 (q-1) \sigma_\ell^2 } \sum_{0\leq u_1,\ldots, u_d \leq 1: |u|=\ell+2}
|u| (|u|- 2 + \frac{ |u|-2}{q-1} )  
\nonumber \\
&&\hspace{0.5cm}\times
(-\frac{1}{q-1})^{|u|- 3}
E\{ \nu^* [\pi_{j_1} (a_{1, j_1});\ldots;
\pi_{j_{|u|}} (a_{1, j_{|u|}}) ]^2 \},
\nonumber \\
&& E( \tilde{S}_{\ell, 1} S_{\ell, 1} )
\nonumber \\
&=&
\frac{ 1}{ q^4 \sigma_\ell^2 } E\Big\{ \sum_{i_1=1}^{q^2} \sum_{i_2=1}^{q^2} \sum_{0\leq u_1,\ldots, u_d \leq 1: |u|=\ell+2}
{\cal I}\{J\in \{j_1,\ldots, j_{|u|} \}\}
\nonumber \\
&&\hspace{0.5cm}\times
\nu^* [\tilde{\pi}_{j_1} (a_{i_1, j_1});\ldots; \tilde{\pi}_{j_{|u|}} (a_{i_1, j_{|u|}}) ]
\nonumber \\
&&\hspace{0.5cm}\times
\nu^* [\pi_{j_1} (a_{i_2, j_1});\ldots; \pi_{j_{|u|}} (a_{i_2, j_{|u|}}) ]
{\cal I}\{ \pi_J (a_{i_1,J} ) = 
\pi_J (a_{i_2,J}) = B_1 \}
\Big\}
\nonumber \\
&=& 
\frac{ 1 }{ d q^4 \sigma_\ell^2 } E\Big\{ \sum_{i_1=1}^{q^2} \sum_{i_2=1}^{q^2} \sum_{0\leq u_1,\ldots, u_d \leq 1: |u|=\ell+2}
\sum_{k=1}^d {\cal I}\{ k\in \{j_1,\ldots, j_{|u|} \}\}
\nonumber \\
&&\hspace{0.5cm}\times
\nu^* [\pi_{j_1} (a_{i_1, j_1});\ldots; \tau_{B_1, B_2} \circ \pi_k (a_{i_1, k});
\ldots; \pi_{j_{|u|}} (a_{i_1, j_{|u|}}) ]
\nonumber \\
&&\hspace{0.5cm}\times
\nu^* [\pi_{j_1} (a_{i_2, j_1});\ldots; \pi_{j_{|u|}} (a_{i_2, j_{|u|}}) ]
{\cal I}\{ \pi_k (a_{i_1, k} ) = 
\pi_k (a_{i_2, k}) = B_1 \}
\Big\}
\nonumber \\
&=& 
\frac{ 1 }{ d q^4 \sigma_\ell^2 } E\Big\{ \sum_{i_1=1}^{q^2} \sum_{0\leq u_1,\ldots, u_d \leq 1: |u|=\ell+2}
\sum_{k=1}^d {\cal I}\{ k\in \{j_1,\ldots, j_{|u|} \}\}
\nonumber \\
&&\hspace{0.5cm}\times
{\cal I}\{ \pi_k (a_{i_1, k} ) = 
B_1 \}
\nu^* [\pi_{j_1} (a_{i_1, j_1});\ldots; \pi_{j_{|u|}} (a_{i_1, j_{|u|}}) ]
\nonumber \\
&&\hspace{0.5cm}\times
\nu^* [\pi_{j_1} (a_{i_1, j_1});\ldots; \tau_{B_1, B_2} \circ \pi_k (a_{i_1, k});
\ldots; \pi_{j_{|u|}} (a_{i_1, j_{|u|}}) ]
\Big\}
\nonumber \\
&& 
+ \frac{ 1 }{ d q^4 \sigma_\ell^2 } E\Big\{ \sum_{i_1=1}^{q^2} \sum_{i_2\neq i_1} \sum_{0\leq u_1,\ldots, u_d \leq 1: |u|=\ell+2}
\sum_{k=1}^d {\cal I}\{ k\in \{j_1,\ldots, j_{|u|} \}\}
\nonumber \\
&& \hspace{0.5cm}\times
{\cal I}\{ \pi_k (a_{i_1, k} ) = 
\pi_k (a_{i_2, k}) = B_1 \}
\nonumber \\
&&\hspace{0.5cm}\times
\nu^* [\pi_{j_1} (a_{i_1, j_1});\ldots; \tau_{B_1, B_2} \circ \pi_k (a_{i_1, k});
\ldots; \pi_{j_{|u|}} (a_{i_1, j_{|u|}}) ]
\nonumber \\
&&\hspace{0.5cm}\times
\nu^* [\pi_{j_1} (a_{i_2, j_1});\ldots; \pi_{j_{|u|}} (a_{i_2, j_{|u|}}) ]
\Big\}
\nonumber \\
&=& 
\frac{ 1 }{ d q^5 \sigma_\ell^2 }  \sum_{i_1=1}^{q^2} \sum_{0\leq u_1,\ldots, u_d \leq 1: |u|=\ell+2}
\sum_{k=1}^d {\cal I}\{ k\in \{j_1,\ldots, j_{|u|} \}\}
(-\frac{1}{q-1})
\nonumber \\
&&\hspace{0.5cm}\times
E\{ \nu^* [\pi_{j_1} (a_{i_1, j_1});\ldots;
\pi_{j_{|u|}} (a_{i_1, j_{|u|}}) ]^2 \}
\nonumber \\
&& 
+ \frac{ 1 }{ d q^4 \sigma_\ell^2 } \sum_{i_1=1}^{q^2} \sum_{i_2\neq i_1} \sum_{0\leq u_1,\ldots, u_d \leq 1: |u|=\ell+2}
\sum_{k=1}^d {\cal I}\{ k\in \{j_1,\ldots, j_{|u|} \}\}
\nonumber \\
&& \hspace{0.5cm}\times
{\cal I}\{ \pi_k (a_{i_1, k} ) = 
\pi_k (a_{i_2, k}) = B_1 \}
(-\frac{1}{q-1})^{|u|}
\nonumber \\
&&\hspace{0.5cm}\times
E\{ \nu^* [\pi_{j_1} (a_{i_1, j_1});\ldots;
\pi_{j_{|u|}} (a_{i_1, j_{|u|}}) ]^2 \}
\nonumber \\
&=& 
- \frac{ \ell +2 }{ d q^3 (q-1) \sigma_\ell^2 }  \sum_{0\leq u_1,\ldots, u_d \leq 1: |u|=\ell+2}
E\{ \nu^* [\pi_{j_1} (a_{1, j_1});\ldots;
\pi_{j_{|u|}} (a_{1, j_{|u|}}) ]^2 \}
\nonumber \\
&& 
- \frac{ \ell +2 }{ d q^3 (q-1)^2 \sigma_\ell^2 } \sum_{0\leq u_1,\ldots, u_d \leq 1: |u|=\ell+2}
(-\frac{1}{q-1})^{|u| - 3}
\nonumber \\
&&\hspace{0.5cm}\times
E\{ \nu^* [\pi_{j_1} (a_{1, j_1});\ldots;
\pi_{j_{|u|}} (a_{1, j_{|u|}}) ]^2 \}.
\end{eqnarray*}
This proves Lemma \ref{la:a.17}. \hfill $\Box$

\begin{pn} \label{pn:a.3}
Let $S_i$ and $\tilde{S}_i$, $i=1,\ldots, d-2$, be as in (\ref{eq:3.75}).
Then for $i=1,\ldots, d-2$,
\begin{displaymath}
\frac{ d(q-1)}{4 (i+2)} \Big| E[ (\tilde{S}_i - S_i)^2  \frac{\partial^2}{\partial v_i^2} \psi_{\varepsilon^2} (V) ]
- E[ (\tilde{S}_i - S_i)^2 ] E [ \frac{\partial^2}{\partial v_i^2} \psi_{\varepsilon^2} (V) ] \Big|
\leq O(\frac{\|h\|_\infty}{ q^{1/2}} ) \log(1/\varepsilon),
\end{displaymath}
as $q\rightarrow\infty$ uniformly over $0<\varepsilon<1$.
\end{pn}
{\sc Proof.}
Let $\tilde{S}_{i,k}$ and $S_{i,k}$, $i=1,\ldots, d-2$, $k=1,2$, be as in (\ref{eq:a.78}).
Then it follows from Lemma \ref{la:a.17} that
\begin{eqnarray}
&& \frac{ d(q-1)}{4 (i+2)} \Big| E[ (\tilde{S}_i - S_i)^2  \frac{\partial^2}{\partial v_i^2} \psi_{\varepsilon^2} (V) ]
- E[ (\tilde{S}_i - S_i)^2 ] E [ \frac{\partial^2}{\partial v_i^2} \psi_{\varepsilon^2} (V) ] \Big|
\nonumber \\
&=& \frac{ d(q-1)}{4(i+2)}
| E \{ [ \frac{\partial^2}{\partial v_i^2} \psi_{\varepsilon^2} (V) ]
E^{\cal W} [ (\tilde{S}_i - S_i)^2 - E (\tilde{S}_i - S_i)^2 ] \} |
\nonumber \\
&\leq & \frac{ d(q-1)}{4(i+2)}
\{ \sup_{v\in {\mathbb R}^{d-2}} |
\frac{\partial^2}{\partial v_i^2} \psi_{\varepsilon^2} (v) | \}
E \Big| E^{\cal W} [ (\tilde{S}_{i, 1} +\tilde{S}_{i,2} - S_{i, 1} 
- S_{i,2})^2 
\nonumber \\
&&\hspace{0.5cm}
- E (\tilde{S}_{i, 1} 
+\tilde{S}_{i,2} - S_{i,1} - S_{i,2} )^2 ] \Big|
\nonumber \\
&\leq & \frac{d(q-1)}{4(i+2)}
\{ \sup_{v\in {\mathbb R}^{d-2}} |
\frac{\partial^2}{\partial v_i^2} \psi_{\varepsilon^2} (v) | \}
\Big\{
E | E^{\cal W} [\tilde{S}_{i, 1}^2 - E (\tilde{S}_{i,1}^2 ) ]|
\nonumber \\
&&\hspace{0.5cm}
+ E | E^{\cal W} [\tilde{S}_{i, 2}^2 - E (\tilde{S}_{i,2}^2 ) ]|
+ E | E^{\cal W} [S_{i, 1}^2 - E (S_{i,1}^2 ) ]|
+ E | E^{\cal W} [S_{i, 2}^2 - E (S_{i,2}^2 ) ]|
\nonumber \\
&&\hspace{0.5cm}
+ 2 E | E^{\cal W} [\tilde{S}_{i, 1} \tilde{S}_{i,2} - E(
\tilde{S}_{i, 1} \tilde{S}_{i,2} )
 ]|
+ 2 E | E^{\cal W} [\tilde{S}_{i, 1} S_{i,1} - E(
\tilde{S}_{i, 1} S_{i,1} )
 ]|
\nonumber \\
&&\hspace{0.5cm}
+ 2 E | E^{\cal W} [\tilde{S}_{i, 1} S_{i,2} - E(
\tilde{S}_{i, 1} S_{i,2} )
 ]|
+ 2 E | E^{\cal W} [\tilde{S}_{i, 2} S_{i,1} - E(
\tilde{S}_{i, 2} S_{i,1} )
 ]|
\nonumber \\
&&\hspace{0.5cm}
+ 2 E | E^{\cal W} [\tilde{S}_{i, 2} S_{i,2} - E(
\tilde{S}_{i, 2} S_{i,2} )
 ]|
+ 2 E | E^{\cal W} [S_{i, 1} S_{i,2} - E(
S_{i, 1} S_{i,2} )
 ]| \Big\}
\nonumber \\
&\leq & \frac{ d(q-1)}{2 (i+2)}
 \{ \sup_{v\in {\mathbb R}^{d-2}} |
\frac{\partial^2}{\partial v_i^2} \psi_{\varepsilon^2} (v) | \}
\Big\{
E | E^{\cal W} [\tilde{S}_{i, 1}^2 - E (\tilde{S}_{i,1}^2 ) ]|
\nonumber \\
&&\hspace{0.5cm}
+ E | E^{\cal W} [S_{i, 1}^2 - E (S_{i,1}^2 ) ]|
+ E | E^{\cal W} ( \tilde{S}_{i, 1} \tilde{S}_{i,2} ) |
+ 2 E | E^{\cal W} ( \tilde{S}_{i, 1} S_{i,1} ) |
\nonumber \\
&&\hspace{0.5cm}
+ 2 E | E^{\cal W} ( \tilde{S}_{i, 1} S_{i,2} )|
+ E | E^{\cal W} ( S_{i, 1} S_{i,2} )|
+O(q^{-2})
\Big\},
\label{eq:a.67}
\end{eqnarray}
as $q\rightarrow\infty$ uniformly over $0<\varepsilon<1$.
Finally we conclude from Lemma \ref{la:5.1} and Lemmas \ref{la:a.18} to \ref{la:a.22} that
\begin{displaymath}
\frac{ d(q-1)}{4 (i+2)} \Big| E[ (\tilde{S}_i - S_i)^2  \frac{\partial^2}{\partial v_i^2} \psi_{\varepsilon^2} (V) ]
- E[ (\tilde{S}_i - S_i)^2 ] E [ \frac{\partial^2}{\partial v_i^2} \psi_{\varepsilon^2} (V) ] \Big|
\leq O(\frac{\|h\|_\infty}{ q^{1/2}} ) \log(1/\varepsilon),
\end{displaymath}
as $q\rightarrow\infty$ uniformly over $0<\varepsilon<1$. This proves Proposition \ref{pn:a.3}.
\hfill $\Box$

\begin{la} \label{la:a.18}
With the notation of (\ref{eq:a.67}), for $\ell=1,\ldots, d-2$,
\begin{eqnarray*}
\frac{d (q-1)}{2 (\ell+2)} E| E^{\cal W} [S_{\ell,1}^2 - E ( S_{\ell,1}^2) ] | &=& O(q^{-1/2}),
\nonumber \\
\frac{d (q-1)}{2 (\ell+2)} E| E^{\cal W} [\tilde{S}_{\ell,1}^2 - E ( \tilde{S}_{\ell,1}^2) ] | &=& O(q^{-1/2}),
\hspace{0.5cm}\mbox{as $q\rightarrow \infty$.}
\end{eqnarray*}
\end{la}
{\sc Proof.}
First we observe  from (\ref{eq:a.78}) that for $\ell=1,\ldots, d-2$,
\begin{eqnarray*}
&& E^{\cal W} ( S_{\ell, 1}^2 )
\nonumber \\
&=&
E^{\cal W} \Big\{ \frac{1}{ q^4 \sigma_\ell^2} \sum_{i_1=1}^{q^2}
\sum_{i_2=1}^{q^2} \sum_{0\leq u^{(1)}_1,\ldots, u^{(1)}_d\leq 1: |u^{(1)}|=\ell +2}
\sum_{0\leq u^{(2)}_1,\ldots, u^{(2)}_d\leq 1: |u^{(2)}|=\ell +2}
\nonumber \\
&&
\times
{\cal I} \{ J\in \{j_{1,1},\ldots, j_{1,|u^{(1)}|} \} \cap \{ j_{2,1},\ldots, j_{2,|u^{(2)}|} \} \}
{\cal I} \{ \pi_J (a_{i_1, J} ) = B_1 = \pi_J (a_{i_2, J}) \}
\nonumber \\
&&
\times
\nu^* [\pi_{j_{1,1}} (a_{i_1, j_{1,1}});\ldots; \pi_{j_{1,|u^{(1)}|}} (a_{i_1, j_{1, |u^{(1)}|}} ) ]
\nu^* [\pi_{j_{2,1}} (a_{i_2, j_{2,1}});\ldots; \pi_{j_{2,|u^{(2)}|}} (a_{i_2, j_{2, |u^{(2)}|}} ) ]
\Big\}
\nonumber \\
&=&
\frac{1}{d q^5 \sigma_\ell^2} \sum_{i_1=1}^{q^2}
\sum_{i_2=1}^{q^2} \sum_{0\leq u^{(1)}_1,\ldots, u^{(1)}_d\leq 1: |u^{(1)}|=\ell +2}
\sum_{0\leq u^{(2)}_1,\ldots, u^{(2)}_d\leq 1: |u^{(2)}|=\ell +2}
\nonumber \\
&&
\times
\sum_{k=1}^d {\cal I} \{ k\in \{j_{1,1},\ldots, j_{1,|u^{(1)}|} \} \cap \{ j_{2,1},\ldots, j_{2,|u^{(2)}|} \} \}
{\cal I} \{ a_{i_1, k} = a_{i_2, k} \}
\nonumber \\
&&
\times
\nu^* [\pi_{j_{1,1}} (a_{i_1, j_{1,1}});\ldots; \pi_{j_{1,|u^{(1)}|}} (a_{i_1, j_{1, |u^{(1)}|}} ) ]
\nu^* [\pi_{j_{2,1}} (a_{i_2, j_{2,1}});\ldots; \pi_{j_{2,|u^{(2)}|}} (a_{i_2, j_{2, |u^{(2)}|}} ) ].
\end{eqnarray*}
Here for $l=1,\ldots, 4$, given $u^{(l)}= (u^{(l)}_1,\ldots, u^{(l)}_d )'$, we write
$k\in \{ j_{l,1}, \ldots, j_{l, |u^{(l)}| } \}$
if and only if $u^{(l)}_k \geq 1$.
Hence
\begin{eqnarray}
&& E \{ [ E^{\cal W} (S_{\ell, 1}^2 ) ]^2 \}
\nonumber \\
&=&
E\Big\{
\frac{1}{d^2 q^{10} \sigma_\ell^4} \sum_{i_1=1}^{q^2}
\sum_{i_2=1}^{q^2} \sum_{i_3=1}^{q^2} \sum_{i_4=1}^{q^2}
\sum_{0\leq u^{(1)}_1,\ldots, u^{(1)}_d\leq 1: |u^{(1)}|= \ell+2}
\nonumber \\
&&\hspace{0.5cm}\times
\sum_{0\leq u^{(2)}_1,\ldots, u^{(2)}_d\leq 1: |u^{(2)}|= \ell+2}
\sum_{0\leq u^{(3)}_1,\ldots, u^{(3)}_d\leq 1: |u^{(3)}|= \ell+2}
\sum_{0\leq u^{(4)}_1,\ldots, u^{(4)}_d\leq 1: |u^{(4)}|= \ell+2}
\nonumber \\
&&\hspace{0.5cm}\times
\sum_{k_1=1}^d \sum_{k_3=1}^d {\cal I} \{ k_1\in \{j_{1,1},\ldots, j_{1,|u^{(1)}|} \} \cap \{ j_{2,1},\ldots, j_{2,|u^{(2)}|} \} \}
\nonumber \\
&&\hspace{0.5cm}\times {\cal I} \{ k_3\in \{j_{3,1},\ldots, j_{3,|u^{(3)}|} \} \cap \{ j_{4,1},\ldots, j_{4,|u^{(4)}|} \} \}
{\cal I} \{ a_{i_1, k_1} = a_{i_2, k_1} \}
{\cal I} \{ a_{i_3, k_3} = a_{i_4, k_3} \}
\nonumber \\
&&\hspace{0.5cm}\times
\nu^* [\pi_{j_{1,1}} (a_{i_1, j_{1,1}});\ldots; \pi_{j_{1,|u^{(1)}|}} (a_{i_1, j_{1, |u^{(1)}|}} ) ]
\nu^* [\pi_{j_{2,1}} (a_{i_2, j_{2,1}});\ldots; \pi_{j_{2,|u^{(2)}|}} (a_{i_2, j_{2, |u^{(2)}|}} ) ]
\nonumber \\
&&\hspace{0.5cm}\times
\nu^* [\pi_{j_{3,1}} (a_{i_3, j_{3,1}});\ldots; \pi_{j_{3,|u^{(3)}|}} (a_{i_3, j_{3, |u^{(3)}|}} ) ]
\nu^* [\pi_{j_{4,1}} (a_{i_4, j_{4,1}});\ldots; \pi_{j_{4,|u^{(4)}|}} (a_{i_4, j_{4, |u^{(4)}|}} ) ]
\Big\}
\nonumber \\
&=& R_{\{1,2,3,4\}} + R_{\{1,2,3\},\{4\}} + R_{\{1,2,4\},\{3\}} + R_{\{1,3,4\},\{2\}} + R_{\{2,3,4\},\{1\}} 
\nonumber \\
&&
+ R_{\{1,2\}, \{3\},\{4\}} + R_{\{1,3\}, \{2\}, \{4\}} + R_{\{1,4\},\{2\}, \{3\}} + R_{\{2,3\}, \{1\}, \{4\}} + R_{\{2,4\}, \{1\}, \{3\}}
\nonumber \\
&&
+ R_{\{3,4\}, \{1\}, \{2\}} + R_{\{1,2\},\{3,4\}} + R_{\{1,3\},\{2,4\}} + R_{\{1,4\}, \{2,3\}} + R_{\{1\},\{2\},\{3\},\{4\}},
\label{eq:a.61}
\end{eqnarray}
where given a partition, say $A_1,\ldots, A_p$, of $\{1,2,3,4\}$, (that is $\cup_{i=1}^p A_i = \{1,2,3,4\}$ and $A_i\cap A_j
= \emptyset$ whenever $i\neq j$), we define
\begin{eqnarray}
&& R_{A_1,\ldots, A_p}
\nonumber \\
&=&
E\Big\{
\frac{1}{d^2 q^{10} \sigma_\ell^4} {\sum}^*
\sum_{0\leq u^{(1)}_1,\ldots, u^{(1)}_d\leq 1: |u^{(1)}|= \ell+2}
\sum_{0\leq u^{(2)}_1,\ldots, u^{(2)}_d\leq 1: |u^{(2)}|= \ell+2}
\sum_{0\leq u^{(3)}_1,\ldots, u^{(3)}_d\leq 1: |u^{(3)}|= \ell+2}
\nonumber \\
&&\times
\sum_{0\leq u^{(4)}_1,\ldots, u^{(4)}_d\leq 1: |u^{(4)}|= \ell+2}
\sum_{k_1=1}^d \sum_{k_3=1}^d {\cal I} \{ k_1\in \{j_{1,1},\ldots, j_{1,|u^{(1)}|} \} \cap \{ j_{2,1},\ldots, j_{2,|u^{(2)}|} \} \}
\nonumber \\
&&\times {\cal I} \{ k_3\in \{j_{3,1},\ldots, j_{3,|u^{(3)}|} \} \cap \{ j_{4,1},\ldots, j_{4,|u^{(4)}|} \} \}
{\cal I} \{ a_{i_1, k_1} = a_{i_2, k_1} \}
{\cal I} \{ a_{i_3, k_3} = a_{i_4, k_3} \}
\nonumber \\
&&\times
\nu^* [\pi_{j_{1,1}} (a_{i_1, j_{1,1}});\ldots; \pi_{j_{1,|u^{(1)}|}} (a_{i_1, j_{1, |u^{(1)}|}} ) ]
\nu^* [\pi_{j_{2,1}} (a_{i_2, j_{2,1}});\ldots; \pi_{j_{2,|u^{(2)}|}} (a_{i_2, j_{2, |u^{(2)}|}} ) ]
\nonumber \\
&&\times
\nu^* [\pi_{j_{3,1}} (a_{i_3, j_{3,1}});\ldots; \pi_{j_{3,|u^{(3)}|}} (a_{i_3, j_{3, |u^{(3)}|}} ) ]
\nu^* [\pi_{j_{4,1}} (a_{i_4, j_{4,1}});\ldots; \pi_{j_{4,|u^{(4)}|}} (a_{i_4, j_{4, |u^{(4)}|}} ) ]
\Big\},
\label{eq:a.49}
\end{eqnarray}
and $\Sigma^*$ denotes summation over $1\leq i_1,i_2,i_3, i_4\leq q^2$ such that if
$k,l\in A_j$, then $i_k=i_l$, and if $k,l$ are in different $A_j$'s, then $i_k\neq i_l$.
In order to evaluate the terms on the right hand side of (\ref{eq:a.61}), it is convenient to further define the following subsets of $\{1,\ldots, d\}$:
for $\{l_1, l_2, l_3, l_4\} = \{1,2,3,4\}$,
\begin{eqnarray*}
\Theta_{\{1,2,3,4\}} &=& \Big\{l\in \cap_{\alpha =1}^4 \{ j_{l_\alpha, 1},\ldots, j_{l_\alpha, |u^{(l_\alpha)}|} \} \Big\},
\nonumber \\
\Theta_{\{l_1,l_2,l_3\}} &=& \Big\{ l\in \cap_{\alpha=1}^3 \{ j_{l_\alpha, 1},\ldots, j_{l_\alpha, |u^{(l_\alpha)}|} \} \Big\}
\backslash
\{ j_{l_4, 1},\ldots, j_{l_4, |u^{(l_4)}|} \},
\nonumber \\
\Theta_{\{l_1,l_2\}} &=& \Big\{ l\in \cap_{\alpha =1}^2 \{ j_{l_\alpha, 1},\ldots, j_{l_\alpha, |u^{(l_\alpha )}|} \} \Big\}
\backslash
\cup_{\beta=3}^4 \{ j_{l_\beta, 1},\ldots, j_{l_\beta, |u^{(l_\beta )}|} \},
\nonumber \\
\Theta_{\{l_1\}} &=& \Big\{ l\in \{ j_{l_1, 1},\ldots, j_{l_1, |u^{(l_1 )}|} \} \Big\}
\backslash
\cup_{\beta=2}^4 \{ j_{l_\beta, 1},\ldots, j_{l_\beta, |u^{(l_\beta )}|} \}.
\end{eqnarray*}
Now we observe from Lemma \ref{la:a.5} that as $q\rightarrow \infty$,
\begin{eqnarray*}
&& R_{\{1,2,3,4\}} 
\nonumber \\
&=&
E\Big\{
\frac{1}{d^2 q^{10} \sigma_\ell^4} \sum_{i_1=i_2=i_3=i_4=1}^{q^2}
\sum_{0\leq u^{(1)}_1,\ldots, u^{(1)}_d\leq 1: |u^{(1)}|=\ell+2}
\nonumber \\
&&\times
\sum_{0\leq u^{(2)}_1,\ldots, u^{(2)}_d\leq 1: |u^{(2)}|= \ell+2}
\sum_{0\leq u^{(3)}_1,\ldots, u^{(3)}_d\leq 1: |u^{(3)}|= \ell+2}
\sum_{0\leq u^{(4)}_1,\ldots, u^{(4)}_d\leq 1: |u^{(4)}|= \ell+2}
\nonumber \\
&&\times
\sum_{k_1=1}^d \sum_{k_3=1}^d {\cal I} \{ k_1\in \{j_{1,1},\ldots, j_{1,|u^{(1)}|} \} \cap \{ j_{2,1},\ldots, j_{2,|u^{(2)}|} \} \}
\nonumber \\
&&\times {\cal I} \{ k_3\in \{j_{3,1},\ldots, j_{3,|u^{(3)}|} \} \cap \{ j_{4,1},\ldots, j_{4,|u^{(4)}|} \} \}
\nonumber \\
&&\times
\nu^* [\pi_{j_{1,1}} (a_{i_1, j_{1,1}});\ldots; \pi_{j_{1,|u^{(1)}|}} (a_{i_1, j_{1, |u^{(1)}|}} ) ]
\nu^* [\pi_{j_{2,1}} (a_{i_2, j_{2,1}});\ldots; \pi_{j_{2,|u^{(2)}|}} (a_{i_2, j_{2, |u^{(2)}|}} ) ]
\nonumber \\
&&\times
\nu^* [\pi_{j_{3,1}} (a_{i_3, j_{3,1}});\ldots; \pi_{j_{3,|u^{(3)}|}} (a_{i_3, j_{3, |u^{(3)}|}} ) ]
\nu^* [\pi_{j_{4,1}} (a_{i_4, j_{4,1}});\ldots; \pi_{j_{4,|u^{(4)}|}} (a_{i_4, j_{4, |u^{(4)}|}} ) ]
\Big\}
\nonumber \\
&=& O(\frac{1}{q^4}),
\nonumber \\
&& R_{\{1,2,3\},\{4\}} 
\nonumber \\
&=& R_{\{1,2,4\},\{3\}} \hspace{0.1cm} = \hspace{0.1cm} R_{\{1,3,4\},\{2\}} \hspace{0.1cm} = \hspace{0.1cm} R_{\{2,3,4\},\{1\}} 
\nonumber \\
&=&
E\Big\{
\frac{1}{d^2 q^{10} \sigma_\ell^4} \sum_{i_1=i_2=i_3=1}^{q^2}
\sum_{1\leq i_4\leq q^2: i_4\neq i_1}
\sum_{0\leq u^{(1)}_1,\ldots, u^{(1)}_d\leq 1: |u^{(1)}|=\ell+2}
\nonumber \\
&&\times
\sum_{0\leq u^{(2)}_1,\ldots, u^{(2)}_d\leq 1: |u^{(2)}|=\ell+2}
\sum_{0\leq u^{(3)}_1,\ldots, u^{(3)}_d\leq 1: |u^{(3)}|=\ell+2}
\sum_{0\leq u^{(4)}_1,\ldots, u^{(4)}_d\leq 1: |u^{(4)}|=\ell+2}
\nonumber \\
&&\times
\sum_{k_1=1}^d \sum_{k_3=1}^d {\cal I} \{ k_1\in \{j_{1,1},\ldots, j_{1,|u^{(1)}|} \} \cap \{ j_{2,1},\ldots, j_{2,|u^{(2)}|} \} \}
\nonumber \\
&&\times {\cal I} \{ k_3\in \{j_{3,1},\ldots, j_{3,|u^{(3)}|} \} \cap \{ j_{4,1},\ldots, j_{4,|u^{(4)}|} \} \}
{\cal I} \{ a_{i_3, k_3} = a_{i_4, k_3} \}
\nonumber \\
&&\times
\nu^* [\pi_{j_{1,1}} (a_{i_1, j_{1,1}});\ldots; \pi_{j_{1,|u^{(1)}|}} (a_{i_1, j_{1, |u^{(1)}|}} ) ]
\nu^* [\pi_{j_{2,1}} (a_{i_2, j_{2,1}});\ldots; \pi_{j_{2,|u^{(2)}|}} (a_{i_2, j_{2, |u^{(2)}|}} ) ]
\nonumber \\
&&\times
\nu^* [\pi_{j_{3,1}} (a_{i_3, j_{3,1}});\ldots; \pi_{j_{3,|u^{(3)}|}} (a_{i_3, j_{3, |u^{(3)}|}} ) ]
\nu^* [\pi_{j_{4,1}} (a_{i_4, j_{4,1}});\ldots; \pi_{j_{4,|u^{(4)}|}} (a_{i_4, j_{4, |u^{(4)}|}} ) ]
\Big\}
\nonumber \\
&=& O(\frac{1}{q^3}),
\nonumber \\
&& R_{\{1,2\},\{3\},\{4\}} 
\hspace{0.1cm} = \hspace{0.1cm} R_{\{3,4\},\{1\},\{2\}}
\nonumber \\
&=&
E\Big\{
\frac{1}{d^2 q^{10} \sigma_\ell^4} \sum_{i_1=i_2=1}^{q^2}
\sum_{1\leq i_3\leq q^2: i_3\neq i_1} \sum_{1\leq i_4\leq q^2: i_4\neq i_1, i_3}
\sum_{0\leq u^{(1)}_1,\ldots, u^{(1)}_d\leq 1: |u^{(1)}|=\ell+2}
\nonumber \\
&&\times
\sum_{0\leq u^{(2)}_1,\ldots, u^{(2)}_d\leq 1: |u^{(2)}|=\ell+2}
\sum_{0\leq u^{(3)}_1,\ldots, u^{(3)}_d\leq 1: |u^{(3)}|=\ell+2}
\sum_{0\leq u^{(4)}_1,\ldots, u^{(4)}_d\leq 1: |u^{(4)}|=\ell+2}
\nonumber \\
&&\times
\sum_{k_1=1}^d \sum_{k_3=1}^d {\cal I} \{ k_1\in \{j_{1,1},\ldots, j_{1,|u^{(1)}|} \} \cap \{ j_{2,1},\ldots, j_{2,|u^{(2)}|} \} \}
\nonumber \\
&&\times {\cal I} \{ k_3\in \{j_{3,1},\ldots, j_{3,|u^{(3)}|} \} \cap \{ j_{4,1},\ldots, j_{4,|u^{(4)}|} \} \}
{\cal I} \{ a_{i_3, k_3} = a_{i_4, k_3} \}
\nonumber \\
&&\times
\nu^* [\pi_{j_{1,1}} (a_{i_1, j_{1,1}});\ldots; \pi_{j_{1,|u^{(1)}|}} (a_{i_1, j_{1, |u^{(1)}|}} ) ]
\nu^* [\pi_{j_{2,1}} (a_{i_2, j_{2,1}});\ldots; \pi_{j_{2,|u^{(2)}|}} (a_{i_2, j_{2, |u^{(2)}|}} ) ]
\nonumber \\
&&\times
\nu^* [\pi_{j_{3,1}} (a_{i_3, j_{3,1}});\ldots; \pi_{j_{3,|u^{(3)}|}} (a_{i_3, j_{3, |u^{(3)}|}} ) ]
\nu^* [\pi_{j_{4,1}} (a_{i_4, j_{4,1}});\ldots; \pi_{j_{4,|u^{(4)}|}} (a_{i_4, j_{4, |u^{(4)}|}} ) ]
\Big\}
\nonumber \\
&=& \frac{O(1)}{q^{10} \sigma_\ell^4} \sum_{i_1=1}^{q^2}
\sum_{1\leq i_3\leq q^2: i_3\neq i_1} \sum_{1\leq i_4\leq q^2: i_4\neq i_1, i_3}
\sum_{0\leq u^{(1)}_1,\ldots, u^{(1)}_d\leq 1: |u^{(1)}|=\ell+2}
\sum_{0\leq u^{(2)}_1,\ldots, u^{(2)}_d\leq 1: |u^{(2)}|=\ell+2}
\nonumber \\
&&\times
\sum_{0\leq u^{(3)}_1,\ldots, u^{(3)}_d\leq 1: |u^{(3)}|=\ell+2}
\sum_{0\leq u^{(4)}_1,\ldots, u^{(4)}_d\leq 1: |u^{(4)}|=\ell+2}
\nonumber \\
&&\times
\sum_{k_3=1}^d 
{\cal I} \{ k_3\in \{j_{3,1},\ldots, j_{3,|u^{(3)}|} \} \cap \{ j_{4,1},\ldots, j_{4,|u^{(4)}|} \} \}
{\cal I} \{ a_{i_3, k_3} = a_{i_4, k_3} \}
0^{ |\Theta_{\{4\}}| }
\nonumber \\
&&\times
(\frac{1}{q})^{|\Theta_{\{1,2,3,4\}}|+ |\Theta_{\{1,2,4\}}| + |\Theta_{\{1,3,4\}}| + |\Theta_{\{2,3,4\}}| + |\Theta_{\{1,4\}}|+|\Theta_{\{2,4\}}|
+ |\Theta_{\{3,4\}}| -1}
\nonumber \\
&=&
\frac{O(1)}{q^{10} \sigma_\ell^4} \sum_{i_1=1}^{q^2}
\sum_{1\leq i_3\leq q^2: i_3\neq i_1} \sum_{1\leq i_4\leq q^2: i_4\neq i_1, i_3}
\sum_{0\leq u^{(1)}_1,\ldots, u^{(1)}_d\leq 1: |u^{(1)}|=\ell+2}
\sum_{0\leq u^{(2)}_1,\ldots, u^{(2)}_d\leq 1: |u^{(2)}|=\ell+2}
\nonumber \\
&&\times
\sum_{0\leq u^{(3)}_1,\ldots, u^{(3)}_d\leq 1: |u^{(3)}|=\ell+2}
\sum_{0\leq u^{(4)}_1,\ldots, u^{(4)}_d\leq 1: |u^{(4)}|=\ell+2}
\nonumber \\
&&\times
\sum_{k_3=1}^d 
{\cal I} \{ k_3\in \{j_{3,1},\ldots, j_{3,|u^{(3)}|} \} \cap \{ j_{4,1},\ldots, j_{4,|u^{(4)}|} \} \}
{\cal I} \{ a_{i_3, k_3} = a_{i_4, k_3} \} (\frac{1}{q})^{|u^{(4)}|-1}
\nonumber \\
&=& O(\frac{1}{q^3}),
\end{eqnarray*}
as $q\rightarrow \infty$. In a similar manner, we have
\begin{eqnarray*}
&& R_{\{1,3\},\{2\}, \{4\}} \hspace{0.1cm}= \hspace{0.1cm} R_{\{1,4\},\{2\},\{3\}} \hspace{0.1cm} = \hspace{0.1cm} R_{\{2,3\},\{1\},\{4\}} 
\hspace{0.1cm}= \hspace{0.1cm} R_{\{2,4\}, \{1\},\{3\}}
\nonumber \\
&=&
E\Big\{
\frac{1}{d^2 q^{10} \sigma_\ell^4} \sum_{i_1=i_3=1}^{q^2}
\sum_{1\leq i_2\leq q^2: i_2\neq i_1} \sum_{1\leq i_4\leq q^2: i_4\neq i_1, i_2}
\sum_{0\leq u^{(1)}_1,\ldots, u^{(1)}_d\leq 1: |u^{(1)}|=\ell+2}
\nonumber \\
&&\times
\sum_{0\leq u^{(2)}_1,\ldots, u^{(2)}_d\leq 1: |u^{(2)}|=\ell+2}
\sum_{0\leq u^{(3)}_1,\ldots, u^{(3)}_d\leq 1: |u^{(3)}|=\ell+2}
\sum_{0\leq u^{(4)}_1,\ldots, u^{(4)}_d\leq 1: |u^{(4)}|=\ell+2}
\nonumber \\
&&\times
\sum_{k_1=1}^d \sum_{k_3=1}^d {\cal I} \{ k_1\in \{j_{1,1},\ldots, j_{1,|u^{(1)}|} \} \cap \{ j_{2,1},\ldots, j_{2,|u^{(2)}|} \} \}
\nonumber \\
&&\times {\cal I} \{ k_3\in \{j_{3,1},\ldots, j_{3,|u^{(3)}|} \} \cap \{ j_{4,1},\ldots, j_{4,|u^{(4)}|} \} \}
{\cal I} \{ a_{i_1, k_1} = a_{i_2, k_1} \}
{\cal I} \{ a_{i_3, k_3} = a_{i_4, k_3} \}
\nonumber \\
&&\times
\nu^* [\pi_{j_{1,1}} (a_{i_1, j_{1,1}});\ldots; \pi_{j_{1,|u^{(1)}|}} (a_{i_1, j_{1, |u^{(1)}|}} ) ]
\nu^* [\pi_{j_{2,1}} (a_{i_2, j_{2,1}});\ldots; \pi_{j_{2,|u^{(2)}|}} (a_{i_2, j_{2, |u^{(2)}|}} ) ]
\nonumber \\
&&\times
\nu^* [\pi_{j_{3,1}} (a_{i_3, j_{3,1}});\ldots; \pi_{j_{3,|u^{(3)}|}} (a_{i_3, j_{3, |u^{(3)}|}} ) ]
\nu^* [\pi_{j_{4,1}} (a_{i_4, j_{4,1}});\ldots; \pi_{j_{4,|u^{(4)}|}} (a_{i_4, j_{4, |u^{(4)}|}} ) ]
\Big\}
\nonumber \\
&=& O(\frac{1}{q^3}),
\nonumber \\
&& R_{\{1,3\},\{2,4\}} \hspace{0.1cm}=\hspace{0.1cm} R_{\{1,4\},\{2,3\}} 
\nonumber \\
&=&
E\Big\{
\frac{1}{d^2 q^{10} \sigma_\ell^4} \sum_{i_1=i_3=1}^{q^2}
\sum_{1\leq i_2=i_4\leq q^2: i_2\neq i_1}
\sum_{0\leq u^{(1)}_1,\ldots, u^{(1)}_d\leq 1: |u^{(1)}|=\ell+2}
\nonumber \\
&&\times
\sum_{0\leq u^{(2)}_1,\ldots, u^{(2)}_d\leq 1: |u^{(2)}|=\ell+2}
\sum_{0\leq u^{(3)}_1,\ldots, u^{(3)}_d\leq 1: |u^{(3)}|=\ell+2}
\sum_{0\leq u^{(4)}_1,\ldots, u^{(4)}_d\leq 1: |u^{(4)}|=\ell+2}
\nonumber \\
&&\times
\sum_{k_1=1}^d \sum_{k_3=1}^d {\cal I} \{ k_1\in \{j_{1,1},\ldots, j_{1,|u^{(1)}|} \} \cap \{ j_{2,1},\ldots, j_{2,|u^{(2)}|} \} \}
\nonumber \\
&& \times
{\cal I} \{ k_3\in \{j_{3,1},\ldots, j_{3,|u^{(3)}|} \} \cap \{ j_{4,1},\ldots, j_{4,|u^{(4)}|} \} \}
{\cal I} \{ a_{i_1, k_1} = a_{i_2, k_1} \}
{\cal I} \{ a_{i_3, k_3} = a_{i_4, k_3} \}
\nonumber \\
&&\times
\nu^* [\pi_{j_{1,1}} (a_{i_1, j_{1,1}});\ldots; \pi_{j_{1,|u^{(1)}|}} (a_{i_1, j_{1, |u^{(1)}|}} ) ]
\nu^* [\pi_{j_{2,1}} (a_{i_2, j_{2,1}});\ldots; \pi_{j_{2,|u^{(2)}|}} (a_{i_2, j_{2, |u^{(2)}|}} ) ]
\nonumber \\
&&\times
\nu^* [\pi_{j_{3,1}} (a_{i_3, j_{3,1}});\ldots; \pi_{j_{3,|u^{(3)}|}} (a_{i_3, j_{3, |u^{(3)}|}} ) ]
\nu^* [\pi_{j_{4,1}} (a_{i_4, j_{4,1}});\ldots; \pi_{j_{4,|u^{(4)}|}} (a_{i_4, j_{4, |u^{(4)}|}} ) ]
\Big\}
\nonumber \\
&=& O(\frac{1}{q^3}).
\end{eqnarray*}

Next we have
\begin{eqnarray*}
&& R_{\{1\}, \{2\}, \{3\}, \{4\}} 
\nonumber \\
&=&
E\Big\{
\frac{1}{d^2 q^{10} \sigma_\ell^4} \sum_{i_1=1}^{q^2} \sum_{1\leq i_2\leq q^2: i_2\neq i_1}
\sum_{1\leq i_3\leq q^2: i_3\neq i_1, i_2} \sum_{1\leq i_4\leq q^2: i_4\neq i_1, i_2, i_3}
\nonumber \\
&&\times
\sum_{0\leq u^{(1)}_1,\ldots, u^{(1)}_d\leq 1: |u^{(1)}|=\ell+2}
\sum_{0\leq u^{(2)}_1,\ldots, u^{(2)}_d\leq 1: |u^{(2)}|=\ell+2}
\nonumber \\
&&\times
\sum_{0\leq u^{(3)}_1,\ldots, u^{(3)}_d\leq 1: |u^{(3)}|=\ell+2}
\sum_{0\leq u^{(4)}_1,\ldots, u^{(4)}_d\leq 1: |u^{(4)}|=\ell+2}
\nonumber \\
&&\times
\sum_{k_1=1}^d \sum_{k_3=1}^d {\cal I} \{ k_1\in \{j_{1,1},\ldots, j_{1,|u^{(1)}|} \} \cap \{ j_{2,1},\ldots, j_{2,|u^{(2)}|} \} \}
\nonumber \\
&&\times {\cal I} \{ k_3\in \{j_{3,1},\ldots, j_{3,|u^{(3)}|} \} \cap \{ j_{4,1},\ldots, j_{4,|u^{(4)}|} \} \}
{\cal I} \{ a_{i_1, k_1} = a_{i_2, k_1} \}
{\cal I} \{ a_{i_3, k_3} = a_{i_4, k_3} \}
\nonumber \\
&&\times
\nu^* [\pi_{j_{1,1}} (a_{i_1, j_{1,1}});\ldots; \pi_{j_{1,|u^{(1)}|}} (a_{i_1, j_{1, |u^{(1)}|}} ) ]
\nu^* [\pi_{j_{2,1}} (a_{i_2, j_{2,1}});\ldots; \pi_{j_{2,|u^{(2)}|}} (a_{i_2, j_{2, |u^{(2)}|}} ) ]
\nonumber \\
&&\times
\nu^* [\pi_{j_{3,1}} (a_{i_3, j_{3,1}});\ldots; \pi_{j_{3,|u^{(3)}|}} (a_{i_3, j_{3, |u^{(3)}|}} ) ]
\nu^* [\pi_{j_{4,1}} (a_{i_4, j_{4,1}});\ldots; \pi_{j_{4,|u^{(4)}|}} (a_{i_4, j_{4, |u^{(4)}|}} ) ]
\Big\}
\nonumber \\
&=&
\frac{ O(1)}{q^{10} \sigma_\ell^4} \sum_{i_1=1}^{q^2} \sum_{1\leq i_2\leq q^2: i_2\neq i_1}
\sum_{1\leq i_3\leq q^2: i_3\neq i_1, i_2} \sum_{1\leq i_4\leq q^2: i_4\neq i_1, i_2, i_3}
\sum_{0\leq u^{(1)}_1,\ldots, u^{(1)}_d\leq 1: |u^{(1)}|=\ell+2}
\nonumber \\
&&\times
\sum_{0\leq u^{(2)}_1,\ldots, u^{(2)}_d\leq 1: |u^{(2)}|=\ell+2}
\sum_{0\leq u^{(3)}_1,\ldots, u^{(3)}_d\leq 1: |u^{(3)}|=\ell+2}
\sum_{0\leq u^{(4)}_1,\ldots, u^{(4)}_d\leq 1: |u^{(4)}|=\ell+2}
\nonumber \\
&&\times
\sum_{k_1=1}^d \sum_{k_3=1}^d {\cal I} \{ k_1\in \{j_{1,1},\ldots, j_{1,|u^{(1)}|} \} \cap \{ j_{2,1},\ldots, j_{2,|u^{(2)}|} \} \}
\nonumber \\
&&\times {\cal I} \{ k_3\in \{j_{3,1},\ldots, j_{3,|u^{(3)}|} \} \cap \{ j_{4,1},\ldots, j_{4,|u^{(4)}|} \} \}
{\cal I} \{ a_{i_1, k_1} = a_{i_2, k_1} \}
{\cal I} \{ a_{i_3, k_3} = a_{i_4, k_3} \}
\nonumber \\
&&\times
0^{ |\Theta_{\{1\}}| + |\Theta_{\{2\}}| + |\Theta_{\{3\}}| + |\Theta_{\{4\}}| }
(\frac{1}{q^2})^{ |\Theta_{\{1,2,3,4\} }| + |\Theta_{\{1,2,3\}}|+  |\Theta_{\{1,2,4\}}| + |\Theta_{\{1,3,4\}}| + |\Theta_{\{2,3,4\}}|} 
\nonumber \\
&&\times (\frac{1}{q})^{ |\Theta_{\{1,2\}}| + |\Theta_{\{1,3\}}| + |\Theta_{\{1,4\}}| + |\Theta_{\{2,3\}}| + |\Theta_{\{2,4\}}|
+ |\Theta_{\{3,4\}}| -2} {\cal I}\{ |\Theta_{\{1,2\}}| \geq |\Theta_{\{3,4\}} |\}
\nonumber \\
&=&
\frac{ O(1)}{q^{10} \sigma_\ell^4} \sum_{i_1=1}^{q^2} \sum_{1\leq i_2\leq q^2: i_2\neq i_1}
\sum_{1\leq i_3\leq q^2: i_3\neq i_1, i_2} \sum_{1\leq i_4\leq q^2: i_4\neq i_1, i_2, i_3}
\sum_{0\leq u^{(1)}_1,\ldots, u^{(1)}_d\leq 1: |u^{(1)}|=\ell+2}
\nonumber \\
&&\times
\sum_{0\leq u^{(2)}_1,\ldots, u^{(2)}_d\leq 1: |u^{(2)}|=\ell+2}
\sum_{0\leq u^{(3)}_1,\ldots, u^{(3)}_d\leq 1: |u^{(3)}|=\ell+2}
\sum_{0\leq u^{(4)}_1,\ldots, u^{(4)}_d\leq 1: |u^{(4)}|=\ell+2}
\nonumber \\
&&\times
\sum_{k_1=1}^d \sum_{k_3=1}^d 
{\cal I} \{ a_{i_1, k_1} = a_{i_2, k_1} \}
{\cal I} \{ a_{i_3, k_3} = a_{i_4, k_3} \} (\frac{1}{q})^{|u^{(3)}| + |u^{(4)}| -2}
\nonumber \\
&=& O(\frac{1}{q^4}).
\end{eqnarray*}
The second last equality can be obtained using the heuristic that $a_{i_1, j} \neq a_{i_2, j}$ when $i_1\neq i_2$ etc.\  and $k_1 \neq k_3$.
However the above bound remains valid when $k_1= k_3$ or when $a_{i_1, j}= a_{i_2, j}$ since this additional constraint introduces
a factor of $q$ while reduces the number of $i_1,\ldots, i_4$ by a factor of $1/q$.
Finally,
\begin{eqnarray*}
&& R_{\{1,2\},\{3,4\}} 
\nonumber \\
&=&
E\Big\{
\frac{1}{d^2 q^{10} \sigma_\ell^4} \sum_{i_1=i_2=1}^{q^2}
\sum_{1\leq i_3=i_4\leq q^2: i_4\neq i_1}
\sum_{0\leq u^{(1)}_1,\ldots, u^{(1)}_d\leq 1: |u^{(1)}|=\ell+2}
\nonumber \\
&&\times
\sum_{0\leq u^{(2)}_1,\ldots, u^{(2)}_d\leq 1: |u^{(2)}|=\ell+2}
\sum_{0\leq u^{(3)}_1,\ldots, u^{(3)}_d\leq 1: |u^{(3)}|=\ell+2}
\sum_{0\leq u^{(4)}_1,\ldots, u^{(4)}_d\leq 1: |u^{(4)}|=\ell+2}
\nonumber \\
&&\times
\sum_{k_1=1}^d \sum_{k_3=1}^d {\cal I} \{ k_1\in \{j_{1,1},\ldots, j_{1,|u^{(1)}|} \} \cap \{ j_{2,1},\ldots, j_{2,|u^{(2)}|} \} \}
\nonumber \\
&&\times {\cal I} \{ k_3\in \{j_{3,1},\ldots, j_{3,|u^{(3)}|} \} \cap \{ j_{4,1},\ldots, j_{4,|u^{(4)}|} \} \}
\nonumber \\
&&\times
\nu^* [\pi_{j_{1,1}} (a_{i_1, j_{1,1}});\ldots; \pi_{j_{1,|u^{(1)}|}} (a_{i_1, j_{1, |u^{(1)}|}} ) ]
\nu^* [\pi_{j_{2,1}} (a_{i_2, j_{2,1}});\ldots; \pi_{j_{2,|u^{(2)}|}} (a_{i_2, j_{2, |u^{(2)}|}} ) ]
\nonumber \\
&&\times
\nu^* [\pi_{j_{3,1}} (a_{i_3, j_{3,1}});\ldots; \pi_{j_{3,|u^{(3)}|}} (a_{i_3, j_{3, |u^{(3)}|}} ) ]
\nu^* [\pi_{j_{4,1}} (a_{i_4, j_{4,1}});\ldots; \pi_{j_{4,|u^{(4)}|}} (a_{i_4, j_{4, |u^{(4)}|}} ) ]
\Big\}
\nonumber \\
&=&
\frac{1}{d^2 q^{10} \sigma_\ell^4} \sum_{i_1=1}^{q^2}
\sum_{1\leq i_3\leq q^2: i_3\neq i_1}
\sum_{0\leq u^{(1)}_1,\ldots, u^{(1)}_d\leq 1: |u^{(1)}|=\ell+2}
\nonumber \\
&&\times
\sum_{0\leq u^{(2)}_1,\ldots, u^{(2)}_d\leq 1: |u^{(2)}|=\ell+2}
\sum_{0\leq u^{(3)}_1,\ldots, u^{(3)}_d\leq 1: |u^{(3)}|=\ell+2}
\sum_{0\leq u^{(4)}_1,\ldots, u^{(4)}_d\leq 1: |u^{(4)}|=\ell+2}
\nonumber \\
&&\times
\sum_{k_1=1}^d \sum_{k_3=1}^d {\cal I} \{ k_1\in \{j_{1,1},\ldots, j_{1,|u^{(1)}|} \} \cap \{ j_{2,1},\ldots, j_{2,|u^{(2)}|} \} \}
\nonumber \\
&&\times {\cal I} \{ k_3\in \{j_{3,1},\ldots, j_{3,|u^{(3)}|} \} \cap \{ j_{4,1},\ldots, j_{4,|u^{(4)}|} \} \}
\nonumber \\
&&\times
E \Big\{ \nu^* [\pi_{j_{1,1}} (a_{i_1, j_{1,1}});\ldots; \pi_{j_{1,|u^{(1)}|}} (a_{i_1, j_{1, |u^{(1)}|}} ) ]
\nu^* [\pi_{j_{2,1}} (a_{i_1, j_{2,1}});\ldots; \pi_{j_{2,|u^{(2)}|}} (a_{i_1, j_{2, |u^{(2)}|}} ) ]
\nonumber \\
&&\times
\nu^* [\pi_{j_{3,1}} (a_{i_3, j_{3,1}});\ldots; \pi_{j_{3,|u^{(3)}|}} (a_{i_3, j_{3, |u^{(3)}|}} ) ]
\nu^* [\pi_{j_{4,1}} (a_{i_3, j_{4,1}});\ldots; \pi_{j_{4,|u^{(4)}|}} (a_{i_3, j_{4, |u^{(4)}|}} ) ]
\Big\}
\nonumber \\
&=&
\frac{1}{d^2 q^{10} \sigma_\ell^4} \sum_{i_1=1}^{q^2}
\sum_{1\leq i_3\leq q^2: i_3\neq i_1}
\sum_{0\leq u^{(1)}_1,\ldots, u^{(1)}_d\leq 1: |u^{(1)}|=\ell+2}
\sum_{0\leq u^{(3)}_1,\ldots, u^{(3)}_d\leq 1: |u^{(3)}|=\ell+2}
|u^{(1)}| |u^{(3)}|
\nonumber \\
&&\times
E \Big\{ \nu^* [\pi_{j_{1,1}} (a_{i_1, j_{1,1}});\ldots; \pi_{j_{1,|u^{(1)}|}} (a_{i_1, j_{1, |u^{(1)}|}} ) ]^2
\nonumber \\
&&\times
\nu^* [\pi_{j_{3,1}} (a_{i_3, j_{3,1}});\ldots; \pi_{j_{3,|u^{(3)}|}} (a_{i_3, j_{3, |u^{(3)}|}} ) ]^2
\Big\} + O(\frac{1}{q^3})
\nonumber \\
&=&
\frac{1}{d^2 q^6 \sigma_\ell^4}
\sum_{0\leq u^{(1)}_1,\ldots, u^{(1)}_d\leq 1: |u^{(1)}|=\ell+2}
\sum_{0\leq u^{(3)}_1,\ldots, u^{(3)}_d\leq 1: |u^{(3)}|=\ell+2}
|u^{(1)}| |u^{(3)}| 
\nonumber \\
&&\times
E \{ \nu^* [\pi_{j_{1,1}} (a_{1, j_{1,1}});\ldots; \pi_{j_{1,|u^{(1)}|}} (a_{1, j_{1, |u^{(1)}|}} ) ]^2 \}
\nonumber \\
&&\times
E \{ \nu^* [\pi_{j_{3,1}} (a_{1, j_{3,1}});\ldots; \pi_{j_{3,|u^{(3)}|}} (a_{1, j_{3, |u^{(3)}|}} ) ]^2
\} + O(\frac{1}{q^3})
\nonumber \\
&=& [ E ( S_{\ell, 1}^2 ) ]^2 + O(\frac{1}{q^3}),
\end{eqnarray*}
as $q\rightarrow\infty$. The last equality uses Lemma \ref{la:a.17}.
Consequently it follows from (\ref{eq:a.61}) that
\begin{eqnarray*}
\frac{ d(q-1) }{ 2 (\ell+2)}  E | E^{\cal W} [ S_{\ell,1}^2  - E  (S_{\ell, 1}^2 ) ] |
&\leq & \frac{ d (q-1)}{2 (\ell+2)} \Big\{ E \{ [ E^{\cal W} ( S_{\ell, 1}^2 ) - E ( S_{\ell, 1}^2 ) ]^2 \} \Big\}^{1/2}
\nonumber \\
&=& \frac{ d(q-1)}{2 (\ell+2)} \Big\{ E \{ [ E^{\cal W}  (S_{\ell, 1}^2) ]^2 \} - [ E (S_{\ell, 1}^2 ) ]^2 \Big\}^{1/2}
\nonumber \\
&=& O(\frac{1}{q^{1/2}}),
\end{eqnarray*}
as $q\rightarrow\infty$. 
Next we observe that
\begin{eqnarray*}
&& E^{\cal W}  (\tilde{S}_{\ell, 1}^2 )
\nonumber \\
&=&
E^{\cal W} \Big\{ \frac{1}{ q^4 \sigma_\ell^2} \sum_{i_1=1}^{q^2}
\sum_{i_2=1}^{q^2} \sum_{0\leq u^{(1)}_1,\ldots, u^{(1)}_d\leq 1: |u^{(1)}|=\ell+2}
\sum_{0\leq u^{(2)}_1,\ldots, u^{(2)}_d\leq 1: |u^{(2)}|=\ell+2}
\nonumber \\
&&\times
{\cal I} \{ J\in \{j_{1,1},\ldots, j_{1,|u^{(1)}|} \} \cap \{ j_{2,1},\ldots, j_{2,|u^{(2)}|} \} \}
{\cal I} \{ \pi_J (a_{i_1, J} ) = B_1 = \pi_J (a_{i_2, J}) \}
\nonumber \\
&&\times
\nu^* [\pi_{j_{1,1}} (a_{i_1, j_{1,1}});\ldots; \tau_{B_1,B_2}\circ \pi_J (a_{i_1, J});\ldots;
\pi_{j_{1,|u^{(1)}|}} (a_{i_1, j_{1, |u^{(1)}|}} ) ]
\nonumber \\
&&\times
\nu^* [\pi_{j_{2,1}} (a_{i_2, j_{2,1}});\ldots; \tau_{B_1,B_2}\circ \pi_J (a_{i_2, J}); \ldots;
\pi_{j_{2,|u^{(2)}|}} (a_{i_2, j_{2, |u^{(2)}|}} ) ]
\Big\}
\nonumber \\
&=&
\frac{1}{ d q^5 (q-1) \sigma_\ell^2} \sum_{i_1=1}^{q^2}
\sum_{i_2=1}^{q^2} \sum_{0\leq u^{(1)}_1,\ldots, u^{(1)}_d\leq 1: |u^{(1)}|=\ell+2}
\sum_{0\leq u^{(2)}_1,\ldots, u^{(2)}_d\leq 1: |u^{(2)}|=\ell+2}
\nonumber \\
&&\times
\sum_{k=1}^d {\cal I} \{ k\in \{j_{1,1},\ldots, j_{1,|u^{(1)}|} \} \cap \{ j_{2,1},\ldots, j_{2,|u^{(2)}|} \} \}
{\cal I} \{ a_{i_1, k} = a_{i_2, k} \}
\nonumber \\
&&\times
\sum_{0\leq \tilde{c}_k\leq q-1: \tilde{c}_k\neq \pi_k (a_{i_1, k})}
\nu^* [\pi_{j_{1,1}} (a_{i_1, j_{1,1}});\ldots; \tilde{c}_k;\ldots;
\pi_{j_{1,|u^{(1)}|}} (a_{i_1, j_{1, |u^{(1)}|}} ) ]
\nonumber \\
&&\times
\nu^* [\pi_{j_{2,1}} (a_{i_2, j_{2,1}});\ldots; \tilde{c}_k; \ldots;
\pi_{j_{2,|u^{(2)}|}} (a_{i_2, j_{2, |u^{(2)}|}} ) ],
\end{eqnarray*}
and
\begin{eqnarray*}
&& E \{ [ E^{\cal W}  (\tilde{S}_{\ell, 1}^2 ) ]^2 \}
\nonumber \\
&=&
E\Big\{ \frac{1}{ d^2 q^{10} (q-1)^2 \sigma_\ell^4 } \sum_{i_1=1}^{q^2}
\sum_{i_2=1}^{q^2} \sum_{i_3=1}^{q^2} \sum_{i_4=1}^{q^2} 
\sum_{0\leq u^{(1)}_1,\ldots, u^{(1)}_d\leq 1: |u^{(1)}|=\ell+2}
\nonumber \\
&&\times
\sum_{0\leq u^{(2)}_1,\ldots, u^{(2)}_d\leq 1: |u^{(2)}|=\ell+2}
\sum_{0\leq u^{(3)}_1,\ldots, u^{(3)}_d\leq 1: |u^{(3)}|=\ell+2}
\sum_{0\leq u^{(4)}_1,\ldots, u^{(4)}_d\leq 1: |u^{(4)}|=\ell+2}
\nonumber \\
&&\times
\sum_{k_1 =1}^d {\cal I} \{ k_1\in \{j_{1,1},\ldots, j_{1,|u^{(1)}|} \} \cap \{ j_{2,1},\ldots, j_{2,|u^{(2)}|} \} \}
{\cal I} \{ a_{i_1, k_1} = a_{i_2, k_1} \}
\nonumber \\
&&\times
\sum_{k_3 =1}^d {\cal I} \{ k_3\in \{j_{3,1},\ldots, j_{3,|u^{(3)}|} \} \cap \{ j_{4,1},\ldots, j_{4,|u^{(4)}|} \} \}
{\cal I} \{ a_{i_3, k_3} = a_{i_4, k_3} \}
\nonumber \\
&&\times
\sum_{0\leq \tilde{c}_{k_1}\leq q-1: \tilde{c}_{k_1}\neq \pi_{k_1} (a_{i_1, k_1})}
\sum_{0\leq \tilde{c}_{k_3}\leq q-1: \tilde{c}_{k_3}\neq \pi_{k_3} (a_{i_3, k_3})}
\nonumber \\
&&\times
\nu^* [\pi_{j_{1,1}} (a_{i_1, j_{1,1}});\ldots; \tilde{c}_{k_1};\ldots;
\pi_{j_{1,|u^{(1)}|}} (a_{i_1, j_{1, |u^{(1)}|}} ) ]
\nonumber \\
&&\times
\nu^* [\pi_{j_{2,1}} (a_{i_2, j_{2,1}});\ldots; \tilde{c}_{k_1}; \ldots;
\pi_{j_{2,|u^{(2)}|}} (a_{i_2, j_{2, |u^{(2)}|}} ) ]
\nonumber \\
&&\times
\nu^* [\pi_{j_{3,1}} (a_{i_3, j_{3,1}});\ldots; \tilde{c}_{k_3};\ldots;
\pi_{j_{3,|u^{(3)}|}} (a_{i_3, j_{3, |u^{(3)}|}} ) ]
\nonumber \\
&&\times
\nu^* [\pi_{j_{4,1}} (a_{i_4, j_{4,1}});\ldots; \tilde{c}_{k_3}; \ldots;
\pi_{j_{4,|u^{(4)}|}} (a_{i_4, j_{4, |u^{(4)}|}} ) ]
\Big\}
\nonumber \\
&=& \tilde{R}_{\{1,2,3,4\}} + \tilde{R}_{\{1,2,3\},\{4\}} + \tilde{R}_{\{1,2,4\},\{3\}} + \tilde{R}_{\{1,3,4\},\{2\}} + \tilde{R}_{\{2,3,4\},\{1\}} 
\nonumber \\
&&
+ \tilde{R}_{\{1,2\}, \{3\},\{4\}} + \tilde{R}_{\{1,3\}, \{2\}, \{4\}} + \tilde{R}_{\{1,4\},\{2\}, \{3\}} + \tilde{R}_{\{2,3\}, \{1\}, \{4\}} 
+ \tilde{R}_{\{2,4\}, \{1\}, \{3\}}
\nonumber \\
&&
+ \tilde{R}_{\{3,4\}, \{1\}, \{2\}} + \tilde{R}_{\{1,2\},\{3,4\}} + \tilde{R}_{\{1,3\},\{2,4\}} + \tilde{R}_{\{1,4\}, \{2,3\}} 
+ \tilde{R}_{\{1\},\{2\},\{3\},\{4\}},
\end{eqnarray*}
where given a partition, say $A_1,\ldots, A_p$, of $\{1,2,3,4\}$, we define
\begin{eqnarray*}
 \tilde{R}_{A_1,\ldots, A_p}
&=&
E\Big\{ \frac{1}{ d^2 q^{10} (q-1)^2 \sigma_\ell^4 } {\sum}^* 
\sum_{0\leq u^{(1)}_1,\ldots, u^{(1)}_d\leq 1: |u^{(1)}|=\ell+2}
\sum_{0\leq u^{(2)}_1,\ldots, u^{(2)}_d\leq 1: |u^{(2)}|=\ell+2}
\nonumber \\
&&\times
\sum_{0\leq u^{(3)}_1,\ldots, u^{(3)}_d\leq 1: |u^{(3)}|=\ell+2}
\sum_{0\leq u^{(4)}_1,\ldots, u^{(4)}_d\leq 1: |u^{(4)}|=\ell+2}
\nonumber \\
&&\times
\sum_{k_1 =1}^d {\cal I} \{ k_1\in \{j_{1,1},\ldots, j_{1,|u^{(1)}|} \} \cap \{ j_{2,1},\ldots, j_{2,|u^{(2)}|} \} \}
{\cal I} \{ a_{i_1, k_1} = a_{i_2, k_1} \}
\nonumber \\
&&\times
\sum_{k_3 =1}^d {\cal I} \{ k_3\in \{j_{3,1},\ldots, j_{3,|u^{(3)}|} \} \cap \{ j_{4,1},\ldots, j_{4,|u^{(4)}|} \} \}
{\cal I} \{ a_{i_3, k_3} = a_{i_4, k_3} \}
\nonumber \\
&&\times
\sum_{0\leq \tilde{c}_{k_1}\leq q-1: \tilde{c}_{k_1}\neq \pi_{k_1} (a_{i_1, k_1})}
\sum_{0\leq \tilde{c}_{k_3}\leq q-1: \tilde{c}_{k_3}\neq \pi_{k_3} (a_{i_3, k_3})}
\nonumber \\
&&\times
\nu^* [\pi_{j_{1,1}} (a_{i_1, j_{1,1}});\ldots; \tilde{c}_{k_1};\ldots;
\pi_{j_{1,|u^{(1)}|}} (a_{i_1, j_{1, |u^{(1)}|}} ) ]
\nonumber \\
&&\times
\nu^* [\pi_{j_{2,1}} (a_{i_2, j_{2,1}});\ldots; \tilde{c}_{k_1}; \ldots;
\pi_{j_{2,|u^{(2)}|}} (a_{i_2, j_{2, |u^{(2)}|}} ) ]
\nonumber \\
&&\times
\nu^* [\pi_{j_{3,1}} (a_{i_3, j_{3,1}});\ldots; \tilde{c}_{k_3};\ldots;
\pi_{j_{3,|u^{(3)}|}} (a_{i_3, j_{3, |u^{(3)}|}} ) ]
\nonumber \\
&&\times
\nu^* [\pi_{j_{4,1}} (a_{i_4, j_{4,1}});\ldots; \tilde{c}_{k_3}; \ldots;
\pi_{j_{4,|u^{(4)}|}} (a_{i_4, j_{4, |u^{(4)}|}} ) ]
\Big\},
\end{eqnarray*}
and $\Sigma^*$ is as in (\ref{eq:a.49}). In a similar way, we observe that
\begin{eqnarray*}
&& \tilde{R}_{\{1,2,3,4\}} 
\nonumber \\
&=&
E\Big\{ \frac{1}{ d^2 q^{10} (q-1)^2 \sigma_\ell^4 } \sum_{i_1=i_2=i_3=i_4=1}^{q^2}
\sum_{0\leq u^{(1)}_1,\ldots, u^{(1)}_d\leq 1: |u^{(1)}|=\ell+2}
\nonumber \\
&&\times
\sum_{0\leq u^{(2)}_1,\ldots, u^{(2)}_d\leq 1: |u^{(2)}|=\ell+2}
\sum_{0\leq u^{(3)}_1,\ldots, u^{(3)}_d\leq 1: |u^{(3)}|=\ell+2}
\sum_{0\leq u^{(4)}_1,\ldots, u^{(4)}_d\leq 1: |u^{(4)}|=\ell+2}
\nonumber \\
&&\times
\sum_{k_1 =1}^d {\cal I} \{ k_1\in \{j_{1,1},\ldots, j_{1,|u^{(1)}|} \} \cap \{ j_{2,1},\ldots, j_{2,|u^{(2)}|} \} \}
\sum_{0\leq \tilde{c}_{k_1}\leq q-1: \tilde{c}_{k_1}\neq \pi_{k_1} (a_{i_1, k_1})}
\nonumber \\
&&\times
\sum_{k_3 =1}^d {\cal I} \{ k_3\in \{j_{3,1},\ldots, j_{3,|u^{(3)}|} \} \cap \{ j_{4,1},\ldots, j_{4,|u^{(4)}|} \} \}
\sum_{0\leq \tilde{c}_{k_3}\leq q-1: \tilde{c}_{k_3}\neq \pi_{k_3} (a_{i_3, k_3})}
\nonumber \\
&&\times
\nu^* [\pi_{j_{1,1}} (a_{i_1, j_{1,1}});\ldots; \tilde{c}_{k_1};\ldots;
\pi_{j_{1,|u^{(1)}|}} (a_{i_1, j_{1, |u^{(1)}|}} ) ]
\nonumber \\
&&\times
\nu^* [\pi_{j_{2,1}} (a_{i_2, j_{2,1}});\ldots; \tilde{c}_{k_1}; \ldots;
\pi_{j_{2,|u^{(2)}|}} (a_{i_2, j_{2, |u^{(2)}|}} ) ]
\nonumber \\
&&\times
\nu^* [\pi_{j_{3,1}} (a_{i_3, j_{3,1}});\ldots; \tilde{c}_{k_3};\ldots;
\pi_{j_{3,|u^{(3)}|}} (a_{i_3, j_{3, |u^{(3)}|}} ) ]
\nonumber \\
&&\times
\nu^* [\pi_{j_{4,1}} (a_{i_4, j_{4,1}});\ldots; \tilde{c}_{k_3}; \ldots;
\pi_{j_{4,|u^{(4)}|}} (a_{i_4, j_{4, |u^{(4)}|}} ) ]
\Big\}
\nonumber \\
&=& O(\frac{1}{q^4}),
\nonumber \\
&& \tilde{R}_{\{1,2,3\},\{4\}} \hspace{0.1cm}= \hspace{0.1cm}
\tilde{R}_{\{1,2,4\},\{3\}} \hspace{0.1cm}= \hspace{0.1cm} \tilde{R}_{\{1,3,4\},\{2\}} 
\hspace{0.1cm} = \hspace{0.1cm} \tilde{R}_{\{2,3,4\},\{1\}}
\nonumber \\
&=&
E\Big\{ \frac{1}{ d^2 q^{10} (q-1)^2 \sigma_\ell^4 } \sum_{i_1=i_2=i_3=1}^{q^2}
\sum_{1\leq i_4\leq q^2: i_4\neq i_1} 
\sum_{0\leq u^{(1)}_1,\ldots, u^{(1)}_d\leq 1: |u^{(1)}|=\ell+2}
\nonumber \\
&&\times
\sum_{0\leq u^{(2)}_1,\ldots, u^{(2)}_d\leq 1: |u^{(2)}|=\ell+2}
\sum_{0\leq u^{(3)}_1,\ldots, u^{(3)}_d\leq 1: |u^{(3)}|=\ell+2}
\sum_{0\leq u^{(4)}_1,\ldots, u^{(4)}_d\leq 1: |u^{(4)}|=\ell+2}
\nonumber \\
&&\times
\sum_{k_1 =1}^d {\cal I} \{ k_1\in \{j_{1,1},\ldots, j_{1,|u^{(1)}|} \} \cap \{ j_{2,1},\ldots, j_{2,|u^{(2)}|} \} \}
\nonumber \\
&&\times
\sum_{k_3 =1}^d {\cal I} \{ k_3\in \{j_{3,1},\ldots, j_{3,|u^{(3)}|} \} \cap \{ j_{4,1},\ldots, j_{4,|u^{(4)}|} \} \}
{\cal I} \{ a_{i_3, k_3} = a_{i_4, k_3} \}
\nonumber \\
&&\times
\sum_{0\leq \tilde{c}_{k_1}\leq q-1: \tilde{c}_{k_1}\neq \pi_{k_1} (a_{i_1, k_1})}
\sum_{0\leq \tilde{c}_{k_3}\leq q-1: \tilde{c}_{k_3}\neq \pi_{k_3} (a_{i_3, k_3})}
\nonumber \\
&&\times
\nu^* [\pi_{j_{1,1}} (a_{i_1, j_{1,1}});\ldots; \tilde{c}_{k_1};\ldots;
\pi_{j_{1,|u^{(1)}|}} (a_{i_1, j_{1, |u^{(1)}|}} ) ]
\nonumber \\
&&\times
\nu^* [\pi_{j_{2,1}} (a_{i_2, j_{2,1}});\ldots; \tilde{c}_{k_1}; \ldots;
\pi_{j_{2,|u^{(2)}|}} (a_{i_2, j_{2, |u^{(2)}|}} ) ]
\nonumber \\
&&\times
\nu^* [\pi_{j_{3,1}} (a_{i_3, j_{3,1}});\ldots; \tilde{c}_{k_3};\ldots;
\pi_{j_{3,|u^{(3)}|}} (a_{i_3, j_{3, |u^{(3)}|}} ) ]
\nonumber \\
&&\times
\nu^* [\pi_{j_{4,1}} (a_{i_4, j_{4,1}});\ldots; \tilde{c}_{k_3}; \ldots;
\pi_{j_{4,|u^{(4)}|}} (a_{i_4, j_{4, |u^{(4)}|}} ) ]
\Big\}
\nonumber \\
&=& O(\frac{1}{q^3}),
\nonumber \\
&& \tilde{R}_{\{1,2\},\{3,4\}} 
\nonumber \\
&=&
E\Big\{ \frac{1}{ d^2 q^{10} (q-1)^2 \sigma_\ell^4 } \sum_{i_1=i_2=1}^{q^2}
\sum_{1\leq i_3= i_4 \leq q^2: i_3\neq i_1} 
\sum_{0\leq u^{(1)}_1,\ldots, u^{(1)}_d\leq 1: |u^{(1)}|=\ell+2}
\nonumber \\
&&\times
\sum_{0\leq u^{(2)}_1,\ldots, u^{(2)}_d\leq 1: |u^{(2)}|=\ell+2}
\sum_{0\leq u^{(3)}_1,\ldots, u^{(3)}_d\leq 1: |u^{(3)}|=\ell+2}
\sum_{0\leq u^{(4)}_1,\ldots, u^{(4)}_d\leq 1: |u^{(4)}|=\ell+2}
\nonumber \\
&&\times
\sum_{k_1 =1}^d {\cal I} \{ k_1\in \{j_{1,1},\ldots, j_{1,|u^{(1)}|} \} \cap \{ j_{2,1},\ldots, j_{2,|u^{(2)}|} \} \}
\nonumber \\
&&\times
\sum_{k_3 =1}^d {\cal I} \{ k_3\in \{j_{3,1},\ldots, j_{3,|u^{(3)}|} \} \cap \{ j_{4,1},\ldots, j_{4,|u^{(4)}|} \} \}
\nonumber \\
&&\times
\sum_{0\leq \tilde{c}_{k_1}\leq q-1: \tilde{c}_{k_1}\neq \pi_{k_1} (a_{i_1, k_1})}
\sum_{0\leq \tilde{c}_{k_3}\leq q-1: \tilde{c}_{k_3}\neq \pi_{k_3} (a_{i_3, k_3})}
\nonumber \\
&&\times
\nu^* [\pi_{j_{1,1}} (a_{i_1, j_{1,1}});\ldots; \tilde{c}_{k_1};\ldots;
\pi_{j_{1,|u^{(1)}|}} (a_{i_1, j_{1, |u^{(1)}|}} ) ]
\nonumber \\
&&\times
\nu^* [\pi_{j_{2,1}} (a_{i_2, j_{2,1}});\ldots; \tilde{c}_{k_1}; \ldots;
\pi_{j_{2,|u^{(2)}|}} (a_{i_2, j_{2, |u^{(2)}|}} ) ]
\nonumber \\
&&\times
\nu^* [\pi_{j_{3,1}} (a_{i_3, j_{3,1}});\ldots; \tilde{c}_{k_3};\ldots;
\pi_{j_{3,|u^{(3)}|}} (a_{i_3, j_{3, |u^{(3)}|}} ) ]
\nonumber \\
&&\times
\nu^* [\pi_{j_{4,1}} (a_{i_4, j_{4,1}});\ldots; \tilde{c}_{k_3}; \ldots;
\pi_{j_{4,|u^{(4)}|}} (a_{i_4, j_{4, |u^{(4)}|}} ) ]
\Big\}
\nonumber \\
&=& [E (\tilde{S}_{\ell, 1}^2) ]^2 + O(\frac{1}{q^3}),
\nonumber \\
&& \tilde{R}_{\{1,3\},\{2,4\}} \hspace{0.1cm} = \hspace{0.1cm} \tilde{R}_{\{1,4\}, \{2,3\}}
\nonumber \\
&=&
E\Big\{ \frac{1}{ d^2 q^{10} (q-1)^2 \sigma_\ell^4 } \sum_{i_1=i_3=1}^{q^2}
\sum_{1\leq i_2= i_4 \leq q^2: i_2\neq i_1} 
\sum_{0\leq u^{(1)}_1,\ldots, u^{(1)}_d\leq 1: |u^{(1)}|=\ell+2}
\nonumber \\
&&\times
\sum_{0\leq u^{(2)}_1,\ldots, u^{(2)}_d\leq 1: |u^{(2)}|=\ell+2}
\sum_{0\leq u^{(3)}_1,\ldots, u^{(3)}_d\leq 1: |u^{(3)}|=\ell+2}
\sum_{0\leq u^{(4)}_1,\ldots, u^{(4)}_d\leq 1: |u^{(4)}|=\ell+2}
\nonumber \\
&&\times
\sum_{k_1 =1}^d {\cal I} \{ k_1\in \{j_{1,1},\ldots, j_{1,|u^{(1)}|} \} \cap \{ j_{2,1},\ldots, j_{2,|u^{(2)}|} \} \}
{\cal I} \{ a_{i_1, k_1} = a_{i_2, k_1} \}
\nonumber \\
&&\times
\sum_{k_3 =1}^d {\cal I} \{ k_3\in \{j_{3,1},\ldots, j_{3,|u^{(3)}|} \} \cap \{ j_{4,1},\ldots, j_{4,|u^{(4)}|} \} \}
{\cal I} \{ a_{i_3, k_3} = a_{i_4, k_3} \}
\nonumber \\
&&\times
\sum_{0\leq \tilde{c}_{k_1}\leq q-1: \tilde{c}_{k_1}\neq \pi_{k_1} (a_{i_1, k_1})}
\sum_{0\leq \tilde{c}_{k_3}\leq q-1: \tilde{c}_{k_3}\neq \pi_{k_3} (a_{i_3, k_3})}
\nonumber \\
&&\times
\nu^* [\pi_{j_{1,1}} (a_{i_1, j_{1,1}});\ldots; \tilde{c}_{k_1};\ldots;
\pi_{j_{1,|u^{(1)}|}} (a_{i_1, j_{1, |u^{(1)}|}} ) ]
\nonumber \\
&&\times
\nu^* [\pi_{j_{2,1}} (a_{i_2, j_{2,1}});\ldots; \tilde{c}_{k_1}; \ldots;
\pi_{j_{2,|u^{(2)}|}} (a_{i_2, j_{2, |u^{(2)}|}} ) ]
\nonumber \\
&&\times
\nu^* [\pi_{j_{3,1}} (a_{i_3, j_{3,1}});\ldots; \tilde{c}_{k_3};\ldots;
\pi_{j_{3,|u^{(3)}|}} (a_{i_3, j_{3, |u^{(3)}|}} ) ]
\nonumber \\
&&\times
\nu^* [\pi_{j_{4,1}} (a_{i_4, j_{4,1}});\ldots; \tilde{c}_{k_3}; \ldots;
\pi_{j_{4,|u^{(4)}|}} (a_{i_4, j_{4, |u^{(4)}|}} ) ]
\Big\}
\nonumber \\
&=& O(\frac{1}{q^3}),
\nonumber \\
&& \tilde{R}_{\{1,2\},\{3\},\{4\}} \hspace{0.1cm} = \hspace{0.1cm} \tilde{R}_{\{3,4\}, \{1\}, \{2\}}
\nonumber \\
&=&
E\Big\{ \frac{1}{ d^2 q^{10} (q-1)^2 \sigma_\ell^4 } \sum_{i_1=i_2=1}^{q^2}
\sum_{1\leq i_3\leq q^2: i_3\neq i_1} \sum_{1 \leq i_4\leq q^2: i_4 \neq i_1, i_3} 
\sum_{0\leq u^{(1)}_1,\ldots, u^{(1)}_d\leq 1: |u^{(1)}|=\ell+2}
\nonumber \\
&&\times
\sum_{0\leq u^{(2)}_1,\ldots, u^{(2)}_d\leq 1: |u^{(2)}|=\ell+2}
\sum_{0\leq u^{(3)}_1,\ldots, u^{(3)}_d\leq 1: |u^{(3)}|=\ell+2}
\sum_{0\leq u^{(4)}_1,\ldots, u^{(4)}_d\leq 1: |u^{(4)}|=\ell+2}
\nonumber \\
&&\times
\sum_{k_1 =1}^d {\cal I} \{ k_1\in \{j_{1,1},\ldots, j_{1,|u^{(1)}|} \} \cap \{ j_{2,1},\ldots, j_{2,|u^{(2)}|} \} \}
\nonumber \\
&&\times
\sum_{k_3 =1}^d {\cal I} \{ k_3\in \{j_{3,1},\ldots, j_{3,|u^{(3)}|} \} \cap \{ j_{4,1},\ldots, j_{4,|u^{(4)}|} \} \}
{\cal I} \{ a_{i_3, k_3} = a_{i_4, k_3} \}
\nonumber \\
&&\times
\sum_{0\leq \tilde{c}_{k_1}\leq q-1: \tilde{c}_{k_1}\neq \pi_{k_1} (a_{i_1, k_1})}
\sum_{0\leq \tilde{c}_{k_3}\leq q-1: \tilde{c}_{k_3}\neq \pi_{k_3} (a_{i_3, k_3})}
\nonumber \\
&&\times
\nu^* [\pi_{j_{1,1}} (a_{i_1, j_{1,1}});\ldots; \tilde{c}_{k_1};\ldots;
\pi_{j_{1,|u^{(1)}|}} (a_{i_1, j_{1, |u^{(1)}|}} ) ]
\nonumber \\
&&\times
\nu^* [\pi_{j_{2,1}} (a_{i_2, j_{2,1}});\ldots; \tilde{c}_{k_1}; \ldots;
\pi_{j_{2,|u^{(2)}|}} (a_{i_2, j_{2, |u^{(2)}|}} ) ]
\nonumber \\
&&\times
\nu^* [\pi_{j_{3,1}} (a_{i_3, j_{3,1}});\ldots; \tilde{c}_{k_3};\ldots;
\pi_{j_{3,|u^{(3)}|}} (a_{i_3, j_{3, |u^{(3)}|}} ) ]
\nonumber \\
&&\times
\nu^* [\pi_{j_{4,1}} (a_{i_4, j_{4,1}});\ldots; \tilde{c}_{k_3}; \ldots;
\pi_{j_{4,|u^{(4)}|}} (a_{i_4, j_{4, |u^{(4)}|}} ) ]
\Big\}
\nonumber \\
&=& O(\frac{1}{q^3}),
\nonumber \\
&& \tilde{R}_{\{1\},\{2\},\{3\},\{4\}} 
\nonumber \\
&=&
E\Big\{ \frac{1}{ d^2 q^{10} (q-1)^2 \sigma_\ell^4 } \sum_{i_1=1}^{q^2}
\sum_{1\leq i_2\leq q^2: i_2\neq i_1} \sum_{1\leq i_3\leq q^2: i_3\neq i_1, i_2} 
\sum_{1\leq i_4\leq q^2: i_4\neq i_1, i_2, i_3} 
\nonumber \\
&&\times
\sum_{0\leq u^{(1)}_1,\ldots, u^{(1)}_d\leq 1: |u^{(1)}|=\ell+2}
\sum_{0\leq u^{(2)}_1,\ldots, u^{(2)}_d\leq 1: |u^{(2)}|=\ell+2}
\sum_{0\leq u^{(3)}_1,\ldots, u^{(3)}_d\leq 1: |u^{(3)}|=\ell+2}
\nonumber \\
&&\times
\sum_{0\leq u^{(4)}_1,\ldots, u^{(4)}_d\leq 1: |u^{(4)}|=\ell+2}
\sum_{k_1 =1}^d {\cal I} \{ k_1\in \{j_{1,1},\ldots, j_{1,|u^{(1)}|} \} \cap \{ j_{2,1},\ldots, j_{2,|u^{(2)}|} \} \}
\nonumber \\
&&\times
\sum_{k_3 =1}^d {\cal I} \{ k_3\in \{j_{3,1},\ldots, j_{3,|u^{(3)}|} \} \cap \{ j_{4,1},\ldots, j_{4,|u^{(4)}|} \} \}
{\cal I} \{ a_{i_1, k_1} = a_{i_2, k_1} \}
\nonumber \\
&&\times
{\cal I} \{ a_{i_3, k_3} = a_{i_4, k_3} \}
\sum_{0\leq \tilde{c}_{k_1}\leq q-1: \tilde{c}_{k_1}\neq \pi_{k_1} (a_{i_1, k_1})}
\sum_{0\leq \tilde{c}_{k_3}\leq q-1: \tilde{c}_{k_3}\neq \pi_{k_3} (a_{i_3, k_3})}
\nonumber \\
&&\times
\nu^* [\pi_{j_{1,1}} (a_{i_1, j_{1,1}});\ldots; \tilde{c}_{k_1};\ldots;
\pi_{j_{1,|u^{(1)}|}} (a_{i_1, j_{1, |u^{(1)}|}} ) ]
\nonumber \\
&&\times
\nu^* [\pi_{j_{2,1}} (a_{i_2, j_{2,1}});\ldots; \tilde{c}_{k_1}; \ldots;
\pi_{j_{2,|u^{(2)}|}} (a_{i_2, j_{2, |u^{(2)}|}} ) ]
\nonumber \\
&&\times
\nu^* [\pi_{j_{3,1}} (a_{i_3, j_{3,1}});\ldots; \tilde{c}_{k_3};\ldots;
\pi_{j_{3,|u^{(3)}|}} (a_{i_3, j_{3, |u^{(3)}|}} ) ]
\nonumber \\
&&\times
\nu^* [\pi_{j_{4,1}} (a_{i_4, j_{4,1}});\ldots; \tilde{c}_{k_3}; \ldots;
\pi_{j_{4,|u^{(4)}|}} (a_{i_4, j_{4, |u^{(4)}|}} ) ]
\Big\}
\nonumber \\
&=& O(\frac{1}{q^4}),
\end{eqnarray*}
and
\begin{eqnarray*}
&& \tilde{R}_{\{1,3\},\{2\},\{4\}} \hspace{0.1cm}= \hspace{0.1cm} \tilde{R}_{\{1,4\},\{2\},\{3\}} 
\hspace{0.1cm}= \hspace{0.1cm} \tilde{R}_{\{2,3\}, \{1\},\{4\}} \hspace{0.1cm}= \hspace{0.1cm} \tilde{R}_{\{2,4\}, \{1\}, \{3\}} 
\nonumber \\
&=&
E\Big\{ \frac{1}{ d^2 q^{10} (q-1)^2 \sigma_\ell^4 } \sum_{i_1=i_3=1}^{q^2}
\sum_{1\leq i_2\leq q^2: i_2\neq i_1} \sum_{1 \leq i_4\leq q^2: i_4 \neq i_1, i_2} 
\sum_{0\leq u^{(1)}_1,\ldots, u^{(1)}_d\leq 1: |u^{(1)}|=\ell+2}
\nonumber \\
&&\times
\sum_{0\leq u^{(2)}_1,\ldots, u^{(2)}_d\leq 1: |u^{(2)}|=\ell+2}
\sum_{0\leq u^{(3)}_1,\ldots, u^{(3)}_d\leq 1: |u^{(3)}|=\ell+2}
\sum_{0\leq u^{(4)}_1,\ldots, u^{(4)}_d\leq 1: |u^{(4)}|=\ell+2}
\nonumber \\
&&\times
\sum_{k_1 =1}^d {\cal I} \{ k_1\in \{j_{1,1},\ldots, j_{1,|u^{(1)}|} \} \cap \{ j_{2,1},\ldots, j_{2,|u^{(2)}|} \} \}
{\cal I} \{ a_{i_1, k_1} = a_{i_2, k_1} \}
\nonumber \\
&&\times
\sum_{k_3 =1}^d {\cal I} \{ k_3\in \{j_{3,1},\ldots, j_{3,|u^{(3)}|} \} \cap \{ j_{4,1},\ldots, j_{4,|u^{(4)}|} \} \}
{\cal I} \{ a_{i_3, k_3} = a_{i_4, k_3} \}
\nonumber \\
&&\times
\sum_{0\leq \tilde{c}_{k_1}\leq q-1: \tilde{c}_{k_1}\neq \pi_{k_1} (a_{i_1, k_1})}
\sum_{0\leq \tilde{c}_{k_3}\leq q-1: \tilde{c}_{k_3}\neq \pi_{k_3} (a_{i_3, k_3})}
\nonumber \\
&&\times
\nu^* [\pi_{j_{1,1}} (a_{i_1, j_{1,1}});\ldots; \tilde{c}_{k_1};\ldots;
\pi_{j_{1,|u^{(1)}|}} (a_{i_1, j_{1, |u^{(1)}|}} ) ]
\nonumber \\
&&\times
\nu^* [\pi_{j_{2,1}} (a_{i_2, j_{2,1}});\ldots; \tilde{c}_{k_1}; \ldots;
\pi_{j_{2,|u^{(2)}|}} (a_{i_2, j_{2, |u^{(2)}|}} ) ]
\nonumber \\
&&\times
\nu^* [\pi_{j_{3,1}} (a_{i_3, j_{3,1}});\ldots; \tilde{c}_{k_3};\ldots;
\pi_{j_{3,|u^{(3)}|}} (a_{i_3, j_{3, |u^{(3)}|}} ) ]
\nonumber \\
&&\times
\nu^* [\pi_{j_{4,1}} (a_{i_4, j_{4,1}});\ldots; \tilde{c}_{k_3}; \ldots;
\pi_{j_{4,|u^{(4)}|}} (a_{i_4, j_{4, |u^{(4)}|}} ) ]
\Big\}
\nonumber \\
&=& O(\frac{1}{q^4}),
\end{eqnarray*}
as $q\rightarrow \infty$.
Consequently,
\begin{eqnarray*}
\frac{ d(q-1) }{2 (\ell+2)}  E | E^{\cal W} [ \tilde{S}_{\ell, 1}^2  - E (\tilde{S}_{\ell, 1}^2 ) ]|
&\leq & \frac{ d (q-1)}{2 (\ell+2)} \Big\{ E \{ [ E^{\cal W}  (\tilde{S}_{\ell, 1}^2 ) - E (\tilde{S}_{\ell, 1}^2 ) ]^2 \} \Big\}^{1/2}
\nonumber \\
&=& \frac{ d(q-1)}{2 (\ell+2)} \Big\{ E \{ [ E^{\cal W}  (\tilde{S}_{\ell, 1}^2 ) ]^2 \} - [ E (\tilde{S}_{\ell, 1}^2 ) ]^2 \Big\}^{1/2}
\nonumber \\
&=& O(\frac{1}{q^{1/2}}),
\end{eqnarray*}
as $q\rightarrow\infty$. 
This proves Lemma \ref{la:a.18}.\hfill $\Box$

\begin{la} \label{la:a.19}
With the notation of (\ref{eq:a.67}), for $1\leq \ell_1, \ell_2\leq d-2$,
\begin{displaymath}
\frac{d (q-1)}{2 (\ell_1 +2)} E | E^{\cal W} (S_{\ell_1, 1} S_{\ell_2, 2}) | = O(\frac{1}{q }),
\hspace{0.5cm}\mbox{as $q\rightarrow\infty$}.
\end{displaymath}
\end{la}
{\sc Proof.}
We observe that
\begin{eqnarray*}
&& E^{\cal W} (S_{\ell_1, 1} S_{\ell_2, 2})
\nonumber \\
&=&
E^{\cal W} \Big\{ \frac{1}{q^4 \sigma_{\ell_1} \sigma_{\ell_2} } \sum_{i_1=1}^{q^2} \sum_{i_2=1}^{q^2} 
\sum_{0\leq u^{(1)}_1,\ldots, u^{(1)}_d \leq 1: |u^{(1)}|=\ell_1 +2}
\sum_{0\leq u^{(2)}_1,\ldots, u^{(2)}_d \leq 1: |u^{(2)}|=\ell_2 +2}
\nonumber \\
&&\times
{\cal I}\{J\in \{j_{1,1},\ldots, j_{1,|u^{(1)}|} \}
\cap \{j_{2,1},\ldots, j_{2,|u^{(2)}|} \} \}
{\cal I}\{ \pi_J (a_{i_1,J} ) = B_1,
\pi_J (a_{i_2,J} ) = B_2\}
\nonumber \\
&&\times
\nu^* [\pi_{j_{1,1}} (a_{i_1, j_{1,1}});\ldots; \pi_{j_{1,|u^{(1)}|}} (a_{i_1, j_{1,|u^{(1)}|}}) ]
\nu^* [\pi_{j_{2,1}} (a_{i_2, j_{2,1}});\ldots; \pi_{j_{2, |u^{(2)}|}} (a_{i_2, j_{2,|u^{(2)}|}}) ]
\Big\}
\nonumber \\
&=&
E^{\cal W} \Big\{ \frac{1}{d q^4 \sigma_{\ell_1} \sigma_{\ell_2} } \sum_{i_1=1}^{q^2} \sum_{i_2\neq i_1} 
\sum_{0\leq u^{(1)}_1,\ldots, u^{(1)}_d \leq 1: |u^{(1)}|=\ell_1 +2}
\sum_{0\leq u^{(2)}_1,\ldots, u^{(2)}_d \leq 1: |u^{(2)}|=\ell_2 +2}
\nonumber \\
&&\times
\sum_{k=1}^d {\cal I}\{ k\in \{j_{1,1},\ldots, j_{1,|u^{(1)}|} \}
\cap \{j_{2,1},\ldots, j_{2,|u^{(2)}|} \} \}
\nonumber \\
&&\times
{\cal I}\{ \pi_k (a_{i_1, k} ) = B_1
\neq \pi_k (a_{i_2, k} ) = B_2\}
\nonumber \\
&&\times
\nu^* [\pi_{j_{1,1}} (a_{i_1, j_{1,1}});\ldots; \pi_{j_{1,|u^{(1)}|}} (a_{i_1, j_{1,|u^{(1)}|}}) ]
\nu^* [\pi_{j_{2,1}} (a_{i_2, j_{2,1}});\ldots; \pi_{j_{2, |u^{(2)}|}} (a_{i_2, j_{2,|u^{(2)}|}}) ]
\Big\}
\nonumber \\
&=&
\frac{1}{d q^5 (q-1) \sigma_{\ell_1} \sigma_{\ell_2} } \sum_{i_1=1}^{q^2} \sum_{i_2\neq i_1} 
\sum_{0\leq u^{(1)}_1,\ldots, u^{(1)}_d \leq 1: |u^{(1)}|=\ell_1 +2}
\sum_{0\leq u^{(2)}_1,\ldots, u^{(2)}_d \leq 1: |u^{(2)}|=\ell_2 +2}
\sum_{k=1}^d 
\nonumber \\
&&\times
{\cal I}\{ k\in \{j_{1,1},\ldots, j_{1,|u^{(1)}|} \}
\cap \{j_{2,1},\ldots, j_{2,|u^{(2)}|} \} \}
{\cal I}\{ a_{i_1, k}
\neq a_{i_2, k} \}
\nonumber \\
&&\times
\nu^* [\pi_{j_{1,1}} (a_{i_1, j_{1,1}});\ldots; \pi_{j_{1,|u^{(1)}|}} (a_{i_1, j_{1,|u^{(1)}|}}) ]
\nu^* [\pi_{j_{2,1}} (a_{i_2, j_{2,1}});\ldots; \pi_{j_{2, |u^{(2)}|}} (a_{i_2, j_{2,|u^{(2)}|}}) ],
\end{eqnarray*}
and
\begin{eqnarray*}
&& E \{ [ E^{\cal W} (S_{\ell_1, 1} S_{\ell_2, 2}) ]^2 \}
\nonumber \\
&=&
E \Big\{ \frac{1}{d^2 q^{10} (q-1)^2 \sigma_{\ell_1}^2 \sigma_{\ell_2}^2 } \sum_{i_1=1}^{q^2} \sum_{i_2\neq i_1} 
\sum_{i_3=1}^{q^2} \sum_{i_4\neq i_3} 
\sum_{0\leq u^{(1)}_1,\ldots, u^{(1)}_d \leq 1: |u^{(1)}|=\ell_1 +2}
\nonumber \\
&&\hspace{0.5cm}\times
\sum_{0\leq u^{(2)}_1,\ldots, u^{(2)}_d \leq 1: |u^{(2)}|=\ell_2 +2}
\sum_{0\leq u^{(3)}_1,\ldots, u^{(3)}_d \leq 1: |u^{(3)}|=\ell_1 +2}
\sum_{0\leq u^{(4)}_1,\ldots, u^{(4)}_d \leq 1: |u^{(4)}|=\ell_2 +2}
\nonumber \\
&&\hspace{0.5cm}\times
\sum_{k_1=1}^d {\cal I}\{ k_1 \in \{j_{1,1},\ldots, j_{1,|u^{(1)}|} \}
\cap \{j_{2,1},\ldots, j_{2,|u^{(2)}|} \} \}
\nonumber \\
&&\hspace{0.5cm}\times
\sum_{k_3=1}^d {\cal I}\{ k_3 \in \{j_{3,1},\ldots, j_{3,|u^{(3)}|} \}
\cap \{j_{4,1},\ldots, j_{4,|u^{(4)}|} \} \}
\nonumber \\
&&\hspace{0.5cm}\times
{\cal I}\{ a_{i_1, k_1}
\neq a_{i_2, k_1 } \}
\nu^* [\pi_{j_{1,1}} (a_{i_1, j_{1,1}});\ldots; \pi_{j_{1,|u^{(1)}|}} (a_{i_1, j_{1,|u^{(1)}|}}) ]
\nonumber \\
&&\hspace{0.5cm}\times
\nu^* [\pi_{j_{2,1}} (a_{i_2, j_{2,1}});\ldots; \pi_{j_{2, |u^{(2)}|}} (a_{i_2, j_{2,|u^{(2)}|}}) ]
\nonumber \\
&&\hspace{0.5cm}\times
{\cal I}\{ a_{i_3, k_3}
\neq a_{i_4, k_3 } \}
\nu^* [\pi_{j_{3,1}} (a_{i_3, j_{3,1}});\ldots; \pi_{j_{3,|u^{(3)}|}} (a_{i_3, j_{3,|u^{(3)}|}}) ]
\nonumber \\
&&\hspace{0.5cm}\times
\nu^* [\pi_{j_{4,1}} (a_{i_4, j_{4,1}});\ldots; \pi_{j_{4, |u^{(4)}|}} (a_{i_4, j_{4,|u^{(4)}|}}) ]
\Big\}
\nonumber \\
&=& R^{(1,2)}_{\{1,3\},\{2\},\{4\}} + R^{(1,2)}_{\{1,4\},\{2\},\{3\}} + R^{(1,2)}_{\{2,3\},\{1\},\{4\}} 
+ R^{(1,2)}_{\{2,4\},\{1\},\{3\}} 
\nonumber \\
&&+ R^{(1,2)}_{\{1,3\},\{2,4\}} + R^{(1,2)}_{\{1,4\},\{2,3\}} 
+ R^{(1,2)}_{\{1\},\{2\},\{3\},\{4\}},
\end{eqnarray*}
where 
\begin{eqnarray*}
&& R^{(1,2)}_{\{1,3\},\{2\},\{4\}} \hspace{0.1cm} = \hspace{0.1cm} R^{(1,2)}_{\{1,4\},\{2\}, \{3\}} 
\hspace{0.1cm} =\hspace{0.1cm} R^{(1,2)}_{\{2,3\}, \{1\}, \{4\}} 
\hspace{0.1cm}= \hspace{0.1cm} R^{(1,2)}_{\{2,4\},\{1\}, \{3\}}
\nonumber \\
&=&
E \Big\{ \frac{1}{d^2 q^{10} (q-1)^2 \sigma_{\ell_1}^2 \sigma_{\ell_2}^2 } \sum_{i_1=1}^{q^2} \sum_{1\leq i_2\leq q^2: i_2\neq i_1} 
\sum_{1\leq i_4\leq q^2: i_4\neq i_1, i_2} 
\sum_{0\leq u^{(1)}_1,\ldots, u^{(1)}_d \leq 1: |u^{(1)}|=\ell_1 +2}
\nonumber \\
&&\times
\sum_{0\leq u^{(2)}_1,\ldots, u^{(2)}_d \leq 1: |u^{(2)}|=\ell_2 +2}
\sum_{0\leq u^{(3)}_1,\ldots, u^{(3)}_d \leq 1: |u^{(3)}|=\ell_1 +2}
\sum_{0\leq u^{(4)}_1,\ldots, u^{(4)}_d \leq 1: |u^{(4)}|=\ell_2 +2}
\nonumber \\
&&\times
\sum_{k_1=1}^d {\cal I}\{ k_1 \in \{j_{1,1},\ldots, j_{1,|u^{(1)}|} \}
\cap \{j_{2,1},\ldots, j_{2,|u^{(2)}|} \} \}
{\cal I}\{ a_{i_1, k_1}
\neq a_{i_2, k_1 } \}
\nonumber \\
&&\times
\sum_{k_3=1}^d {\cal I}\{ k_3 \in \{j_{3,1},\ldots, j_{3,|u^{(3)}|} \}
\cap \{j_{4,1},\ldots, j_{4,|u^{(4)}|} \} \}
{\cal I}\{ a_{i_1, k_3}
\neq a_{i_4, k_3 } \}
\nonumber \\
&&\times
\nu^* [\pi_{j_{1,1}} (a_{i_1, j_{1,1}});\ldots; \pi_{j_{1,|u^{(1)}|}} (a_{i_1, j_{1,|u^{(1)}|}}) ]
\nu^* [\pi_{j_{2,1}} (a_{i_2, j_{2,1}});\ldots; \pi_{j_{2, |u^{(2)}|}} (a_{i_2, j_{2,|u^{(2)}|}}) ]
\nonumber \\
&&\times
\nu^* [\pi_{j_{3,1}} (a_{i_1, j_{3,1}});\ldots; \pi_{j_{3,|u^{(3)}|}} (a_{i_1, j_{3,|u^{(3)}|}}) ]
\nu^* [\pi_{j_{4,1}} (a_{i_4, j_{4,1}});\ldots; \pi_{j_{4, |u^{(4)}|}} (a_{i_4, j_{4,|u^{(4)}|}}) ]
\Big\}
\nonumber \\
&=& O(\frac{1}{q^5}),
\nonumber \\
&& R^{(1,2)}_{\{1,3\},\{2,4\}} \hspace{0.1cm}= \hspace{0.1cm} R^{(1,2)}_{\{1,4\}, \{2,3\}}
\nonumber \\
&=&
E \Big\{ \frac{1}{d^2 q^{10} (q-1)^2 \sigma_{\ell_1}^2 \sigma_{\ell_2}^2 } \sum_{i_1=1}^{q^2} \sum_{1\leq i_2\leq q^2: i_2\neq i_1} 
\sum_{0\leq u^{(1)}_1,\ldots, u^{(1)}_d \leq 1: |u^{(1)}|=\ell_1 +2}
\nonumber \\
&&\times
\sum_{0\leq u^{(2)}_1,\ldots, u^{(2)}_d \leq 1: |u^{(2)}|=\ell_2 +2}
\sum_{0\leq u^{(3)}_1,\ldots, u^{(3)}_d \leq 1: |u^{(3)}|=\ell_1 +2}
\sum_{0\leq u^{(4)}_1,\ldots, u^{(4)}_d \leq 1: |u^{(4)}|=\ell_2 +2}
\nonumber \\
&&\times
\sum_{k_1=1}^d {\cal I}\{ k_1 \in \{j_{1,1},\ldots, j_{1,|u^{(1)}|} \}
\cap \{j_{2,1},\ldots, j_{2,|u^{(2)}|} \} \}
{\cal I}\{ a_{i_1, k_1}
\neq a_{i_2, k_1 } \}
\nonumber \\
&&\times
\sum_{k_3=1}^d {\cal I}\{ k_3 \in \{j_{3,1},\ldots, j_{3,|u^{(3)}|} \}
\cap \{j_{4,1},\ldots, j_{4,|u^{(4)}|} \} \}
{\cal I}\{ a_{i_1, k_3}
\neq a_{i_2, k_3 } \}
\nonumber \\
&&\times
\nu^* [\pi_{j_{1,1}} (a_{i_1, j_{1,1}});\ldots; \pi_{j_{1,|u^{(1)}|}} (a_{i_1, j_{1,|u^{(1)}|}}) ]
\nu^* [\pi_{j_{2,1}} (a_{i_2, j_{2,1}});\ldots; \pi_{j_{2, |u^{(2)}|}} (a_{i_2, j_{2,|u^{(2)}|}}) ]
\nonumber \\
&&\times
\nu^* [\pi_{j_{3,1}} (a_{i_1, j_{3,1}});\ldots; \pi_{j_{3,|u^{(3)}|}} (a_{i_1, j_{3,|u^{(3)}|}}) ]
\nu^* [\pi_{j_{4,1}} (a_{i_2, j_{4,1}});\ldots; \pi_{j_{4, |u^{(4)}|}} (a_{i_2, j_{4,|u^{(4)}|}}) ]
\Big\}
\nonumber \\
&=& O(\frac{1}{q^4}),
\nonumber \\
&& R^{(1,2)}_{\{1\},\{2\},\{3\},\{4\}}
\nonumber \\
&=&
E \Big\{ \frac{1}{d^2 q^{10} (q-1)^2 \sigma_{\ell_1}^2 \sigma_{\ell_2}^2 } \sum_{i_1=1}^{q^2} \sum_{1\leq i_2\leq q^2: i_2\neq i_1} 
\sum_{1\leq i_3\leq q^2: i_3\neq i_1, i_2} 
\sum_{1\leq i_4\leq q^2: i_4\neq i_1,i_2,i_3} 
\nonumber \\
&&\times
\sum_{0\leq u^{(1)}_1,\ldots, u^{(1)}_d \leq 1: |u^{(1)}|=\ell_1 +2}
\sum_{0\leq u^{(2)}_1,\ldots, u^{(2)}_d \leq 1: |u^{(2)}|=\ell_2 +2}
\nonumber \\
&&\times
\sum_{0\leq u^{(3)}_1,\ldots, u^{(3)}_d \leq 1: |u^{(3)}|=\ell_1 +2}
\sum_{0\leq u^{(4)}_1,\ldots, u^{(4)}_d \leq 1: |u^{(4)}|=\ell_2 +2}
\nonumber \\
&&\times
\sum_{k_1=1}^d {\cal I}\{ k_1 \in \{j_{1,1},\ldots, j_{1,|u^{(1)}|} \}
\cap \{j_{2,1},\ldots, j_{2,|u^{(2)}|} \} \}
{\cal I}\{ a_{i_1, k_1}
\neq a_{i_2, k_1 } \}
\nonumber \\
&&\times
\sum_{k_3=1}^d {\cal I}\{ k_3 \in \{j_{3,1},\ldots, j_{3,|u^{(3)}|} \}
\cap \{j_{4,1},\ldots, j_{4,|u^{(4)}|} \} \}
{\cal I}\{ a_{i_3, k_3}
\neq a_{i_4, k_3 } \}
\nonumber \\
&&\times
\nu^* [\pi_{j_{1,1}} (a_{i_1, j_{1,1}});\ldots; \pi_{j_{1,|u^{(1)}|}} (a_{i_1, j_{1,|u^{(1)}|}}) ]
\nu^* [\pi_{j_{2,1}} (a_{i_2, j_{2,1}});\ldots; \pi_{j_{2, |u^{(2)}|}} (a_{i_2, j_{2,|u^{(2)}|}}) ]
\nonumber \\
&&\times
\nu^* [\pi_{j_{3,1}} (a_{i_3, j_{3,1}});\ldots; \pi_{j_{3,|u^{(3)}|}} (a_{i_3, j_{3,|u^{(3)}|}}) ]
\nu^* [\pi_{j_{4,1}} (a_{i_4, j_{4,1}});\ldots; \pi_{j_{4, |u^{(4)}|}} (a_{i_4, j_{4,|u^{(4)}|}}) ]
\Big\}
\nonumber \\
&=& O(\frac{1}{q^6}),
\end{eqnarray*}
as $q\rightarrow \infty$. 
Thus we conclude that
\begin{displaymath}
\frac{d (q-1)}{2 (\ell_1 +2)} E | E^{\cal W} (S_{\ell_1, 1} S_{\ell_2, 2}) | \leq
\frac{d (q-1)}{2 (\ell_1 +2)}
\{ E [ E^{\cal W} (S_{\ell_1, 1} S_{\ell_2, 2}) ]^2 \}^{1/2}
\nonumber \\
= O(\frac{1}{q}),
\end{displaymath}
as $q\rightarrow\infty$. This proves Lemma  \ref{la:a.19}. \hfill $\Box$

\begin{la} \label{la:a.20}
With the notation of (\ref{eq:a.67}),
for $1\leq \ell_1, \ell_2\leq d-2$,
\begin{displaymath}
\frac{d (q-1)}{\ell_1 +2} E | E^{\cal W} ( \tilde{S}_{\ell_1,1} S_{\ell_2,2}) | = O(\frac{1}{q }),
\hspace{0.5cm}\mbox{as $q\rightarrow\infty$}.
\end{displaymath}
\end{la}
{\sc Proof.}
We observe that
\begin{eqnarray*}
&& E^{\cal W} (\tilde{S}_{\ell_1,1} S_{\ell_2, 2})
\nonumber \\
&=&
E^{\cal W} \Big\{ \frac{1}{q^4 \sigma_{\ell_1} \sigma_{\ell_2} } \sum_{i_1=1}^{q^2} \sum_{i_2\neq i_1} 
\sum_{0\leq u^{(1)}_1,\ldots, u^{(1)}_d \leq 1: |u^{(1)}|=\ell_1 +2}
\sum_{0\leq u^{(2)}_1,\ldots, u^{(2)}_d \leq 1: |u^{(2)}|=\ell_2 +2}
\nonumber \\
&&\times
{\cal I}\{J\in \{j_{1,1},\ldots, j_{1,|u^{(1)}|} \}
\cap \{j_{2,1},\ldots, j_{2,|u^{(2)}|} \} \}
{\cal I}\{ \pi_J (a_{i_1,J} ) = B_1,
\pi_J (a_{i_2,J} ) = B_2\}
\nonumber \\
&&\times
\nu^* [\tilde{\pi}_{j_{1,1}} (a_{i_1, j_{1,1}});\ldots; \tilde{\pi}_{j_{1,|u^{(1)}|}} (a_{i_1, j_{1,|u^{(1)}|}}) ]
\nu^* [\pi_{j_{2,1}} (a_{i_2, j_{2,1}});\ldots; \pi_{j_{2, |u^{(2)}|}} (a_{i_2, j_{2,|u^{(2)}|}}) ]
\Big\}
\nonumber \\
&=&
E^{\cal W} \Big\{ \frac{1}{d q^4 \sigma_{\ell_1} \sigma_{\ell_2} } \sum_{i_1=1}^{q^2} \sum_{i_2\neq i_1} 
\sum_{0\leq u^{(1)}_1,\ldots, u^{(1)}_d \leq 1: |u^{(1)}|= \ell_1 +2}
\sum_{0\leq u^{(2)}_1,\ldots, u^{(2)}_d \leq 1: |u^{(2)}|= \ell_2 +2}
\sum_{k=1}^d 
\nonumber \\
&&\times
{\cal I}\{ k\in \{j_{1,1},\ldots, j_{1,|u^{(1)}|} \}
\cap \{j_{2,1},\ldots, j_{2,|u^{(2)}|} \} \}
{\cal I}\{ \pi_k (a_{i_1, k} ) = B_1,
\pi_k (a_{i_2, k} ) = B_2\}
\nonumber \\
&&\times
\nu^* [\pi_{j_{1,1}} (a_{i_1, j_{1,1}});\ldots; \tau_{B_1,B_2}\circ \pi_k( a_{i_1, k});\ldots;
\pi_{j_{1,|u^{(1)}|}} (a_{i_1, j_{1,|u^{(1)}|}}) ]
\nonumber \\
&&\times
\nu^* [\pi_{j_{2,1}} (a_{i_2, j_{2,1}});\ldots; \pi_{j_{2, |u^{(2)}|}} (a_{i_2, j_{2,|u^{(2)}|}}) ]
\Big\}
\nonumber \\
&=&
\frac{1}{d q^5 (q-1) \sigma_{\ell_1} \sigma_{\ell_2} } \sum_{i_1=1}^{q^2} \sum_{i_2\neq i_1} 
\sum_{0\leq u^{(1)}_1,\ldots, u^{(1)}_d \leq 1: |u^{(1)}|=\ell_1 +2}
\sum_{0\leq u^{(2)}_1,\ldots, u^{(2)}_d \leq 1: |u^{(2)}|=\ell_2 +2}
\nonumber \\
&&\hspace{0.5cm}\times
\sum_{k=1}^d {\cal I}\{ k\in \{j_{1,1},\ldots, j_{1,|u^{(1)}|} \}
\cap \{j_{2,1},\ldots, j_{2,|u^{(2)}|} \} \}
{\cal I}\{ a_{i_1, k}
\neq a_{i_2, k} \}
\nonumber \\
&&\hspace{0.5cm}\times
\nu^* [\pi_{j_{1,1}} (a_{i_1, j_{1,1}});\ldots; \pi_k( a_{i_2, k});\ldots;
\pi_{j_{1,|u^{(1)}|}} (a_{i_1, j_{1,|u^{(1)}|}}) ]
\nonumber \\
&&\hspace{0.5cm}\times
\nu^* [\pi_{j_{2,1}} (a_{i_2, j_{2,1}});\ldots; \pi_{j_{2, |u^{(2)}|}} (a_{i_2, j_{2,|u^{(2)}|}}) ],
\end{eqnarray*}
and
\begin{eqnarray*}
&& E \{ [ E^{\cal W} (\tilde{S}_{\ell_1,1} S_{\ell_2,2}) ]^2\}
\nonumber \\
&=&
E\Big\{ \frac{1}{d^2 q^{10} (q-1)^2 \sigma_{\ell_1}^2 \sigma_{\ell_2}^2 } 
\sum_{i_1=1}^{q^2} \sum_{i_2\neq i_1} 
\sum_{i_3=1}^{q^2} \sum_{i_4\neq i_3} 
\sum_{0\leq u^{(1)}_1,\ldots, u^{(1)}_d \leq 1: |u^{(1)}|=\ell_1 +2}
\nonumber \\
&&\hspace{0.5cm}\times
\sum_{0\leq u^{(2)}_1,\ldots, u^{(2)}_d \leq 1: |u^{(2)}|=\ell_2 +2}
\sum_{0\leq u^{(3)}_1,\ldots, u^{(3)}_d \leq 1: |u^{(3)}|=\ell_1 +2}
\sum_{0\leq u^{(4)}_1,\ldots, u^{(4)}_d \leq 1: |u^{(4)}|=\ell_2 +2}
\nonumber \\
&&\hspace{0.5cm}\times
\sum_{k_1=1}^d {\cal I}\{ k_1\in \{j_{1,1},\ldots, j_{1,|u^{(1)}|} \}
\cap \{j_{2,1},\ldots, j_{2,|u^{(2)}|} \} \}
{\cal I}\{ a_{i_1, k_1}
\neq a_{i_2, k_1} \}
\nonumber \\
&&\hspace{0.5cm}\times
\sum_{k_3=1}^d {\cal I}\{ k_3\in \{j_{3,1},\ldots, j_{3,|u^{(3)}|} \}
\cap \{j_{4,1},\ldots, j_{4,|u^{(4)}|} \} \}
{\cal I}\{ a_{i_3, k_3}
\neq a_{i_4, k_3} \}
\nonumber \\
&&\hspace{0.5cm}\times
\nu^* [\pi_{j_{1,1}} (a_{i_1, j_{1,1}});\ldots; \pi_{k_1} ( a_{i_2, k_1});\ldots;
\pi_{j_{1,|u^{(1)}|}} (a_{i_1, j_{1,|u^{(1)}|}}) ]
\nonumber \\
&&\hspace{0.5cm}\times
\nu^* [\pi_{j_{2,1}} (a_{i_2, j_{2,1}});\ldots; \pi_{j_{2, |u^{(2)}|}} (a_{i_2, j_{2,|u^{(2)}|}}) ]
\nonumber \\
&& \hspace{0.5cm}\times
\nu^* [\pi_{j_{3,1}} (a_{i_3, j_{3,1}});\ldots; \pi_{k_3} ( a_{i_4, k_3});\ldots;
\pi_{j_{3,|u^{(3)}|}} (a_{i_3, j_{3,|u^{(3)}|}}) ]
\nonumber \\
&&\hspace{0.5cm}\times
\nu^* [\pi_{j_{4,1}} (a_{i_4, j_{4,1}});\ldots; \pi_{j_{4, |u^{(4)}|}} (a_{i_4, j_{4,|u^{(4)}|}}) ]
\Big\}
\nonumber \\
&=& R^{(\tilde{1},2)}_{\{1,3\},\{2\},\{4\}} + R^{(\tilde{1},2)}_{\{1,4\},\{2\},\{3\}} + R^{(\tilde{1},2)}_{\{2,3\},\{1\},\{4\}} 
+ R^{(\tilde{1},2)}_{\{2,4\},\{1\},\{3\}} 
\nonumber \\
&&+ R^{(\tilde{1},2)}_{\{1,3\},\{2,4\}} + R^{(\tilde{1},2)}_{\{1,4\},\{2,3\}} 
+ R^{(\tilde{1},2)}_{\{1\},\{2\},\{3\},\{4\}},
\end{eqnarray*}
where
\begin{eqnarray*}
&& R^{(\tilde{1},2)}_{\{1,3\},\{2\},\{4\}} \hspace{0.1cm}= \hspace{0.1cm}
 R^{(\tilde{1},2)}_{\{1,4\},\{2\},\{3\}} \hspace{0.1cm} = \hspace{0.1cm} R^{(\tilde{1},2)}_{\{2,3\},\{1\},\{4\}} 
\hspace{0.1cm} = \hspace{0.1cm} R^{(\tilde{1},2)}_{\{2,4\},\{1\},\{3\}}
\nonumber \\
&=&
E\Big\{ \frac{1}{d^2 q^{10} (q-1)^2 \sigma_{\ell_1}^2 \sigma_{\ell_2}^2 } 
\sum_{i_1=1}^{q^2} \sum_{1\leq i_2\leq q^2: i_2\neq i_1} 
\sum_{1\leq i_4\leq q^2: i_4 \neq i_2, i_3} 
\sum_{0\leq u^{(1)}_1,\ldots, u^{(1)}_d \leq 1: |u^{(1)}|=\ell_1 +2}
\nonumber \\
&&\times
\sum_{0\leq u^{(2)}_1,\ldots, u^{(2)}_d \leq 1: |u^{(2)}|=\ell_2 +2}
\sum_{0\leq u^{(3)}_1,\ldots, u^{(3)}_d \leq 1: |u^{(3)}|=\ell_1 +2}
\sum_{0\leq u^{(4)}_1,\ldots, u^{(4)}_d \leq 1: |u^{(4)}|=\ell_2 +2}
\nonumber \\
&&\times
\sum_{k_1=1}^d {\cal I}\{ k_1\in \{j_{1,1},\ldots, j_{1,|u^{(1)}|} \}
\cap \{j_{2,1},\ldots, j_{2,|u^{(2)}|} \} \}
\nonumber \\
&&\times
\sum_{k_3=1}^d {\cal I}\{ k_3\in \{j_{3,1},\ldots, j_{3,|u^{(3)}|} \}
\cap \{j_{4,1},\ldots, j_{4,|u^{(4)}|} \} \}
\nonumber \\
&&\times
\nu^* [\pi_{j_{1,1}} (a_{i_1, j_{1,1}});\ldots; \pi_{k_1} ( a_{i_2, k_1});\ldots;
\pi_{j_{1,|u^{(1)}|}} (a_{i_1, j_{1,|u^{(1)}|}}) ]
\nonumber \\
&&\times
\nu^* [\pi_{j_{2,1}} (a_{i_2, j_{2,1}});\ldots; \pi_{j_{2, |u^{(2)}|}} (a_{i_2, j_{2,|u^{(2)}|}}) ]
{\cal I}\{ a_{i_1, k_1}
\neq a_{i_2, k_1} \}
\nonumber \\
&&\times
\nu^* [\pi_{j_{3,1}} (a_{i_1, j_{3,1}});\ldots; \pi_{k_3} ( a_{i_4, k_3});\ldots;
\pi_{j_{3,|u^{(3)}|}} (a_{i_1, j_{3,|u^{(3)}|}}) ]
\nonumber \\
&&\times
\nu^* [\pi_{j_{4,1}} (a_{i_4, j_{4,1}});\ldots; \pi_{j_{4, |u^{(4)}|}} (a_{i_4, j_{4,|u^{(4)}|}}) ]
{\cal I}\{ a_{i_1, k_3}
\neq a_{i_4, k_3} \}
\Big\}
\nonumber \\
&=& O(\frac{1}{q^4}),
\nonumber \\
&& R^{(\tilde{1},2)}_{\{1,3\},\{2,4\}} 
\hspace{0.1cm} = \hspace{0.1cm} R^{(\tilde{1},2)}_{\{1,4\},\{2,3\}} 
\nonumber \\
&=&
E\Big\{ \frac{1}{d^2 q^{10} (q-1)^2 \sigma_{\ell_1}^2 \sigma_{\ell_2}^2 } 
\sum_{i_1=1}^{q^2} \sum_{1\leq i_2\leq q^2: i_2\neq i_1} 
\sum_{0\leq u^{(1)}_1,\ldots, u^{(1)}_d \leq 1: |u^{(1)}|=\ell_1 +2}
\nonumber \\
&&\times
\sum_{0\leq u^{(2)}_1,\ldots, u^{(2)}_d \leq 1: |u^{(2)}|=\ell_2 +2}
\sum_{0\leq u^{(3)}_1,\ldots, u^{(3)}_d \leq 1: |u^{(3)}|=\ell_1 +2}
\sum_{0\leq u^{(4)}_1,\ldots, u^{(4)}_d \leq 1: |u^{(4)}|=\ell_2 +2}
\nonumber \\
&&\times
\sum_{k_1=1}^d {\cal I}\{ k_1\in \{j_{1,1},\ldots, j_{1,|u^{(1)}|} \}
\cap \{j_{2,1},\ldots, j_{2,|u^{(2)}|} \} \}
\nonumber \\
&&\times
\sum_{k_3=1}^d {\cal I}\{ k_3\in \{j_{3,1},\ldots, j_{3,|u^{(3)}|} \}
\cap \{j_{4,1},\ldots, j_{4,|u^{(4)}|} \} \}
\nonumber \\
&&\times
\nu^* [\pi_{j_{1,1}} (a_{i_1, j_{1,1}});\ldots; \pi_{k_1} ( a_{i_2, k_1});\ldots;
\pi_{j_{1,|u^{(1)}|}} (a_{i_1, j_{1,|u^{(1)}|}}) ]
\nonumber \\
&&\times
\nu^* [\pi_{j_{2,1}} (a_{i_2, j_{2,1}});\ldots; \pi_{j_{2, |u^{(2)}|}} (a_{i_2, j_{2,|u^{(2)}|}}) ]
{\cal I}\{ a_{i_1, k_1}
\neq a_{i_2, k_1} \}
\nonumber \\
&&\times
\nu^* [\pi_{j_{3,1}} (a_{i_1, j_{3,1}});\ldots; \pi_{k_3} ( a_{i_2, k_3});\ldots;
\pi_{j_{3,|u^{(3)}|}} (a_{i_1, j_{3,|u^{(3)}|}}) ]
\nonumber \\
&&\times
\nu^* [\pi_{j_{4,1}} (a_{i_2, j_{4,1}});\ldots; \pi_{j_{4, |u^{(4)}|}} (a_{i_2, j_{4,|u^{(4)}|}}) ]
{\cal I}\{ a_{i_1, k_3}
\neq a_{i_2, k_3} \}
\Big\}
\nonumber \\
&=& O(\frac{1}{q^4}),
\end{eqnarray*}
and
\begin{eqnarray*}
R^{(\tilde{1},2)}_{\{1\},\{2\},\{3\},\{4\}} 
&=&
E\Big\{ \frac{1}{d^2 q^{10} (q-1)^2 \sigma_{\ell_1}^2 \sigma_{\ell_2}^2 } 
\sum_{i_1=1}^{q^2} \sum_{1\leq i_2\leq q^2: i_2\neq i_1} 
\sum_{1\leq i_3\leq q^2: i_3\neq i_1,i_2}
\nonumber \\
&&\times
\sum_{1\leq i_4\leq q^2: i_4 \neq i_1, i_2, i_3} 
\sum_{0\leq u^{(1)}_1,\ldots, u^{(1)}_d \leq 1: |u^{(1)}|=\ell_1 +2}
\sum_{0\leq u^{(2)}_1,\ldots, u^{(2)}_d \leq 1: |u^{(2)}|=\ell_2 +2}
\nonumber \\
&&\times
\sum_{0\leq u^{(3)}_1,\ldots, u^{(3)}_d \leq 1: |u^{(3)}|=\ell_1 +2}
\sum_{0\leq u^{(4)}_1,\ldots, u^{(4)}_d \leq 1: |u^{(4)}|=\ell_2 +2}
\nonumber \\
&&\times
\sum_{k_1=1}^d {\cal I}\{ k_1\in \{j_{1,1},\ldots, j_{1,|u^{(1)}|} \}
\cap \{j_{2,1},\ldots, j_{2,|u^{(2)}|} \} \}
\nonumber \\
&&\times
\sum_{k_3=1}^d {\cal I}\{ k_3\in \{j_{3,1},\ldots, j_{3,|u^{(3)}|} \}
\cap \{j_{4,1},\ldots, j_{4,|u^{(4)}|} \} \}
\nonumber \\
&&\times
\nu^* [\pi_{j_{1,1}} (a_{i_1, j_{1,1}});\ldots; \pi_{k_1} ( a_{i_2, k_1});\ldots;
\pi_{j_{1,|u^{(1)}|}} (a_{i_1, j_{1,|u^{(1)}|}}) ]
\nonumber \\
&&\times
\nu^* [\pi_{j_{2,1}} (a_{i_2, j_{2,1}});\ldots; \pi_{j_{2, |u^{(2)}|}} (a_{i_2, j_{2,|u^{(2)}|}}) ]
{\cal I}\{ a_{i_1, k_1}
\neq a_{i_2, k_1} \}
\nonumber \\
&&\times
\nu^* [\pi_{j_{3,1}} (a_{i_3, j_{3,1}});\ldots; \pi_{k_3} ( a_{i_4, k_3});\ldots;
\pi_{j_{3,|u^{(3)}|}} (a_{i_3, j_{3,|u^{(3)}|}}) ]
\nonumber \\
&&\times
\nu^* [\pi_{j_{4,1}} (a_{i_4, j_{4,1}});\ldots; \pi_{j_{4, |u^{(4)}|}} (a_{i_4, j_{4,|u^{(4)}|}}) ]
{\cal I}\{ a_{i_3, k_3}
\neq a_{i_4, k_3} \}
\Big\}
\nonumber \\
&=& O(\frac{1}{q^4}),
\end{eqnarray*}
as $q\rightarrow\infty$.
Thus we conclude that
\begin{displaymath}
\frac{d (q-1)}{\ell_1 +2} E | E^{\cal W} (\tilde{S}_{\ell_1,1} S_{\ell_2,2}) | \leq
\frac{d (q-1)}{\ell_1 +2}
\{ E [ E^{\cal W} (\tilde{S}_{\ell_1,1} S_{\ell_2,2} ) ]^2 \}^{1/2}
= O(\frac{1}{q }),
\end{displaymath}
as $q\rightarrow\infty$. This proves Lemma  \ref{la:a.20}. \hfill $\Box$

\begin{la} \label{la:a.21}
With the notation of (\ref{eq:a.67}),
for $1\leq \ell_1, \ell_2 \leq d-2$,
\begin{displaymath}
\frac{d (q-1)}{\ell_1 +2} E | E^{\cal W} ( \tilde{S}_{\ell_1,1} S_{\ell_2,1}) | = O(\frac{1}{q}),
\hspace{0.5cm}\mbox{as $q\rightarrow\infty$}.
\end{displaymath}
\end{la}
{\sc Proof.}
We observe that
\begin{eqnarray*}
&& E^{\cal W} (\tilde{S}_{\ell_1,1} S_{\ell_2,1})
\nonumber \\
&=&
E^{\cal W} \Big\{ \frac{1}{q^4 \sigma_{\ell_1} \sigma_{\ell_2} } \sum_{i_1=1}^{q^2} \sum_{i_2=1}^{q^2} 
\sum_{0\leq u^{(1)}_1,\ldots, u^{(1)}_d \leq 1: |u^{(1)}|=\ell_1 +2}
\sum_{0\leq u^{(2)}_1,\ldots, u^{(2)}_d \leq 1: |u^{(2)}|=\ell_2 +2}
\nonumber \\
&&\times
{\cal I}\{J\in \{j_{1,1},\ldots, j_{1,|u^{(1)}|} \}
\cap \{j_{2,1},\ldots, j_{2,|u^{(2)}|} \} \}
{\cal I}\{ \pi_J (a_{i_1,J} ) = B_1
= \pi_J (a_{i_2,J} ) \}
\nonumber \\
&&\times
\nu^* [\tilde{\pi}_{j_{1,1}} (a_{i_1, j_{1,1}});\ldots; \tilde{\pi}_{j_{1,|u^{(1)}|}} (a_{i_1, j_{1,|u^{(1)}|}}) ]
\nu^* [\pi_{j_{2,1}} (a_{i_2, j_{2,1}});\ldots; \pi_{j_{2, |u^{(2)}|}} (a_{i_2, j_{2,|u^{(2)}|}}) ]
\Big\}
\nonumber \\
&=&
E^{\cal W} \Big\{ \frac{1}{d q^4 \sigma_{\ell_1} \sigma_{\ell_2} } \sum_{i_1=1}^{q^2} \sum_{i_2=1}^{q^2} 
\sum_{0\leq u^{(1)}_1,\ldots, u^{(1)}_d \leq 1: |u^{(1)}|=\ell_1 +2}
\sum_{0\leq u^{(2)}_1,\ldots, u^{(2)}_d \leq 1: |u^{(2)}|=\ell_2 +2}
\sum_{k=1}^d
\nonumber \\
&&\times
{\cal I}\{ k\in \{j_{1,1},\ldots, j_{1,|u^{(1)}|} \}
\cap \{j_{2,1},\ldots, j_{2,|u^{(2)}|} \} \}
{\cal I}\{ \pi_k (a_{i_1, k} ) = B_1
= \pi_k (a_{i_2, k} ) \}
\nonumber \\
&&\hspace{0.5cm}\times
\nu^* [\pi_{j_{1,1}} (a_{i_1, j_{1,1}});\ldots; \tau_{B_1,B_2}\circ \pi_k (a_{i_1, k});\ldots; \pi_{j_{1,|u^{(1)}|}} (a_{i_1, j_{1,|u^{(1)}|}}) ]
\nonumber \\
&&\hspace{0.5cm}\times
\nu^* [\pi_{j_{2,1}} (a_{i_2, j_{2,1}});\ldots; \pi_{j_{2, |u^{(2)}|}} (a_{i_2, j_{2,|u^{(2)}|}}) ]
\Big\}
\nonumber \\
&=&
- \frac{1}{d q^5 (q-1) \sigma_{\ell_1} \sigma_{\ell_2} } \sum_{i_1=1}^{q^2} \sum_{i_2=1}^{q^2} 
\sum_{0\leq u^{(1)}_1,\ldots, u^{(1)}_d \leq 1: |u^{(1)}|=\ell_1 +2}
\sum_{0\leq u^{(2)}_1,\ldots, u^{(2)}_d \leq 1: |u^{(2)}|=\ell_2 +2}
\nonumber \\
&&\times
\sum_{k=1}^d {\cal I}\{ k\in \{j_{1,1},\ldots, j_{1,|u^{(1)}|} \}
\cap \{j_{2,1},\ldots, j_{2,|u^{(2)}|} \} \}
{\cal I}\{ a_{i_1, k} 
= a_{i_2, k}  \}
\nonumber \\
&&\times
\nu^* [\pi_{j_{1,1}} (a_{i_1, j_{1,1}});\ldots; \pi_{j_{1,|u^{(1)}|}} (a_{i_1, j_{1,|u^{(1)}|}}) ]
\nu^* [\pi_{j_{2,1}} (a_{i_2, j_{2,1}});\ldots; \pi_{j_{2, |u^{(2)}|}} (a_{i_2, j_{2,|u^{(2)}|}}) ],
\end{eqnarray*}
and
\begin{eqnarray*}
&& E \{ [ E^{\cal W} (\tilde{S}_{\ell_1,1} S_{\ell_2,1}) ]^2\}
\nonumber \\
&=&
E\Big\{ \frac{1}{d^2 q^{10} (q-1)^2 \sigma_{\ell_1}^2 \sigma_{\ell_2}^2 } 
\sum_{i_1=1}^{q^2} \sum_{i_2=1}^{q^2} 
\sum_{i_3=1}^{q^2} \sum_{i_4=1}^{q^2} 
\sum_{0\leq u^{(1)}_1,\ldots, u^{(1)}_d \leq 1: |u^{(1)}|=\ell_1 +2}
\nonumber \\
&&\times
\sum_{0\leq u^{(2)}_1,\ldots, u^{(2)}_d \leq 1: |u^{(2)}|=\ell_2 +2}
\sum_{0\leq u^{(3)}_1,\ldots, u^{(3)}_d \leq 1: |u^{(3)}|=\ell_1 +2}
\sum_{0\leq u^{(4)}_1,\ldots, u^{(4)}_d \leq 1: |u^{(4)}|=\ell_2 +2}
\nonumber \\
&&\times
\sum_{k_1=1}^d {\cal I}\{ k_1\in \{j_{1,1},\ldots, j_{1,|u^{(1)}|} \}
\cap \{j_{2,1},\ldots, j_{2,|u^{(2)}|} \} \}
{\cal I}\{ a_{i_1, k_1}
= a_{i_2, k_1} \}
\nonumber \\
&&\times
\sum_{k_3=1}^d {\cal I}\{ k_3\in \{j_{3,1},\ldots, j_{3,|u^{(3)}|} \}
\cap \{j_{4,1},\ldots, j_{4,|u^{(4)}|} \} \}
{\cal I}\{ a_{i_3, k_3}
= a_{i_4, k_3} \}
\nonumber \\
&&\times
\nu^* [\pi_{j_{1,1}} (a_{i_1, j_{1,1}});\ldots;
\pi_{j_{1,|u^{(1)}|}} (a_{i_1, j_{1,|u^{(1)}|}}) ]
\nu^* [\pi_{j_{2,1}} (a_{i_2, j_{2,1}});\ldots; \pi_{j_{2, |u^{(2)}|}} (a_{i_2, j_{2,|u^{(2)}|}}) ]
\nonumber \\
&& \times
\nu^* [\pi_{j_{3,1}} (a_{i_3, j_{3,1}});\ldots; 
\pi_{j_{3,|u^{(3)}|}} (a_{i_3, j_{3,|u^{(3)}|}}) ]
\nu^* [\pi_{j_{4,1}} (a_{i_4, j_{4,1}});\ldots; \pi_{j_{4, |u^{(4)}|}} (a_{i_4, j_{4,|u^{(4)}|}}) ]
\Big\}
\nonumber \\
&=&
R^{(\tilde{1},1)}_{\{1,2,3,4\}} + R^{(\tilde{1},1)}_{\{1,2,3\},\{4\}} + R^{(\tilde{1},1)}_{\{1,2,4\},\{3\}}
+ R^{(\tilde{1},1)}_{\{1,3,4\},\{2\}} + R^{(\tilde{1},1)}_{\{2,3,4\},\{1\}}
\nonumber \\
&& + R^{(\tilde{1},1)}_{\{1,2\},\{3\},\{4\}} + R^{(\tilde{1},1)}_{\{1,3\},\{2\},\{4\}} + R^{(\tilde{1},1)}_{\{1,4\},\{2\},\{3\}}
+ R^{(\tilde{1},1)}_{\{2,3\},\{1\},\{4\}} + R^{(\tilde{1},1)}_{\{2,4\},\{1\},\{3\}}
\nonumber \\
&& + R^{(\tilde{1},1)}_{\{3,4\},\{1\},\{2\}} + R^{(\tilde{1},1)}_{\{1,2\},\{3,4\}} + R^{(\tilde{1},1)}_{\{1,3\},\{2,4\}}
+ R^{(\tilde{1},1)}_{\{1,4\},\{2,3\}} + R^{(\tilde{1},1)}_{\{1\},\{2\},\{3\},\{4\}},
\end{eqnarray*}
where
\begin{eqnarray*}
&& R^{(\tilde{1},1)}_{\{1,2,3,4\}} 
\nonumber \\
&=&
E\Big\{ \frac{1}{d^2 q^{10} (q-1)^2 \sigma_{\ell_1}^2 \sigma_{\ell_2}^2 } 
\sum_{i_1=1}^{q^2} 
\sum_{0\leq u^{(1)}_1,\ldots, u^{(1)}_d \leq 1: |u^{(1)}|=\ell_1 +2}
\nonumber \\
&&\times
\sum_{0\leq u^{(2)}_1,\ldots, u^{(2)}_d \leq 1: |u^{(2)}|=\ell_2 +2}
\sum_{0\leq u^{(3)}_1,\ldots, u^{(3)}_d \leq 1: |u^{(3)}|=\ell_1 +2}
\sum_{0\leq u^{(4)}_1,\ldots, u^{(4)}_d \leq 1: |u^{(4)}|=\ell_2 +2}
\nonumber \\
&&\times
\sum_{k_1=1}^d {\cal I}\{ k_1\in \{j_{1,1},\ldots, j_{1,|u^{(1)}|} \}
\cap \{j_{2,1},\ldots, j_{2,|u^{(2)}|} \} \}
\nonumber \\
&&\times
\sum_{k_3=1}^d {\cal I}\{ k_3\in \{j_{3,1},\ldots, j_{3,|u^{(3)}|} \}
\cap \{j_{4,1},\ldots, j_{4,|u^{(4)}|} \} \}
\nonumber \\
&&\times
\nu^* [\pi_{j_{1,1}} (a_{i_1, j_{1,1}});\ldots;
\pi_{j_{1,|u^{(1)}|}} (a_{i_1, j_{1,|u^{(1)}|}}) ]
\nu^* [\pi_{j_{2,1}} (a_{i_1, j_{2,1}});\ldots; \pi_{j_{2, |u^{(2)}|}} (a_{i_1, j_{2,|u^{(2)}|}}) ]
\nonumber \\
&&\times
\nu^* [\pi_{j_{3,1}} (a_{i_1, j_{3,1}});\ldots; 
\pi_{j_{3,|u^{(3)}|}} (a_{i_1, j_{3,|u^{(3)}|}}) ]
\nu^* [\pi_{j_{4,1}} (a_{i_1, j_{4,1}});\ldots; \pi_{j_{4, |u^{(4)}|}} (a_{i_1, j_{4,|u^{(4)}|}}) ]
\Big\}
\nonumber \\
&=& O(\frac{1}{q^6}),
\end{eqnarray*}
\begin{eqnarray*}
&& R^{(\tilde{1},1)}_{\{1,2,3\},\{4\}} 
\hspace{0.1cm} = \hspace{0.1cm}
R^{(\tilde{1},1)}_{\{1,2,4\},\{3\}}
\hspace{0.1cm} = \hspace{0.1cm} R^{(\tilde{1},1)}_{\{1,3,4\},\{2\}} \hspace{0.1cm}= \hspace{0.1cm} R^{(\tilde{1},1)}_{\{2,3,4\},\{1\}}
\nonumber \\
&=&
E\Big\{ \frac{1}{d^2 q^{10} (q-1)^2 \sigma_{\ell_1}^2 \sigma_{\ell_2}^2 } 
\sum_{i_1=1}^{q^2} \sum_{1\leq i_4\leq q^2: i_4\neq i_1}
\sum_{0\leq u^{(1)}_1,\ldots, u^{(1)}_d \leq 1: |u^{(1)}|=\ell_1 +2}
\nonumber \\
&&\times
\sum_{0\leq u^{(2)}_1,\ldots, u^{(2)}_d \leq 1: |u^{(2)}|=\ell_2 +2}
\sum_{0\leq u^{(3)}_1,\ldots, u^{(3)}_d \leq 1: |u^{(3)}|=\ell_1 +2}
\sum_{0\leq u^{(4)}_1,\ldots, u^{(4)}_d \leq 1: |u^{(4)}|=\ell_2 +2}
\nonumber \\
&&\times
\sum_{k_1=1}^d {\cal I}\{ k_1\in \{j_{1,1},\ldots, j_{1,|u^{(1)}|} \}
\cap \{j_{2,1},\ldots, j_{2,|u^{(2)}|} \} \}
\nonumber \\
&&\times
\sum_{k_3=1}^d {\cal I}\{ k_3\in \{j_{3,1},\ldots, j_{3,|u^{(3)}|} \}
\cap \{j_{4,1},\ldots, j_{4,|u^{(4)}|} \} \} {\cal I}\{ a_{i_1, k_3}
= a_{i_4, k_3} \}
\nonumber \\
&&\times
\nu^* [\pi_{j_{1,1}} (a_{i_1, j_{1,1}});\ldots;
\pi_{j_{1,|u^{(1)}|}} (a_{i_1, j_{1,|u^{(1)}|}}) ]
\nu^* [\pi_{j_{2,1}} (a_{i_1, j_{2,1}});\ldots; \pi_{j_{2, |u^{(2)}|}} (a_{i_1, j_{2,|u^{(2)}|}}) ]
\nonumber \\
&&\times
\nu^* [\pi_{j_{3,1}} (a_{i_1, j_{3,1}});\ldots; 
\pi_{j_{3,|u^{(3)}|}} (a_{i_1, j_{3,|u^{(3)}|}}) ]
\nu^* [\pi_{j_{4,1}} (a_{i_4, j_{4,1}});\ldots; \pi_{j_{4, |u^{(4)}|}} (a_{i_4, j_{4,|u^{(4)}|}}) ]
\Big\}
\nonumber \\
&=& O(\frac{1}{q^5}),
\nonumber \\
&& R^{(\tilde{1},1)}_{\{1,2\},\{3\},\{4\}} 
\hspace{0.1cm} = \hspace{0.1cm}
R^{(\tilde{1},1)}_{\{3,4\},\{1\},\{2\}} 
\nonumber \\
&=&
E\Big\{ \frac{1}{d^2 q^{10} (q-1)^2 \sigma_{\ell_1} \sigma_{\ell_2}^2 } 
\sum_{i_1=1}^{q^2} \sum_{1\leq i_3\leq q^2:i_3\neq i_1} \sum_{1\leq i_4\leq q^2: i_4\neq i_1, i_3}
\sum_{0\leq u^{(1)}_1,\ldots, u^{(1)}_d \leq 1: |u^{(1)}|=\ell_1+2}
\nonumber \\
&&\times
\sum_{0\leq u^{(2)}_1,\ldots, u^{(2)}_d \leq 1: |u^{(2)}|=\ell_2 +2}
\sum_{0\leq u^{(3)}_1,\ldots, u^{(3)}_d \leq 1: |u^{(3)}|=\ell_1 +2}
\sum_{0\leq u^{(4)}_1,\ldots, u^{(4)}_d \leq 1: |u^{(4)}|=\ell_2 +2}
\nonumber \\
&&\times
\sum_{k_1=1}^d {\cal I}\{ k_1\in \{j_{1,1},\ldots, j_{1,|u^{(1)}|} \}
\cap \{j_{2,1},\ldots, j_{2,|u^{(2)}|} \} \}
\nonumber \\
&&\times
\sum_{k_3=1}^d {\cal I}\{ k_3\in \{j_{3,1},\ldots, j_{3,|u^{(3)}|} \}
\cap \{j_{4,1},\ldots, j_{4,|u^{(4)}|} \} \} {\cal I}\{ a_{i_3, k_3}
= a_{i_4, k_3} \}
\nonumber \\
&&\times
\nu^* [\pi_{j_{1,1}} (a_{i_1, j_{1,1}});\ldots;
\pi_{j_{1,|u^{(1)}|}} (a_{i_1, j_{1,|u^{(1)}|}}) ]
\nu^* [\pi_{j_{2,1}} (a_{i_1, j_{2,1}});\ldots; \pi_{j_{2, |u^{(2)}|}} (a_{i_1, j_{2,|u^{(2)}|}}) ]
\nonumber \\
&&\times
\nu^* [\pi_{j_{3,1}} (a_{i_3, j_{3,1}});\ldots; 
\pi_{j_{3,|u^{(3)}|}} (a_{i_3, j_{3,|u^{(3)}|}}) ]
\nu^* [\pi_{j_{4,1}} (a_{i_4, j_{4,1}});\ldots; \pi_{j_{4, |u^{(4)}|}} (a_{i_4, j_{4,|u^{(4)}|}}) ]
\Big\}
\nonumber \\
&=& O(\frac{1}{q^5}),
\nonumber \\
&& R^{(\tilde{1},1)}_{\{1,3\},\{2\},\{4\}} 
\hspace{0.1cm} = \hspace{0.1cm}
R^{(\tilde{1},1)}_{\{1,4\},\{2\},\{3\}}
\hspace{0.1cm} = \hspace{0.1cm} R^{(\tilde{1},1)}_{\{2,3\},\{1\},\{4\}} \hspace{0.1cm} = \hspace{0.1cm} R^{(\tilde{1},1)}_{\{2,4\},\{1\},\{3\}}
\nonumber \\
&=&
E\Big\{ \frac{1}{d^2 q^{10} (q-1)^2 \sigma_{\ell_1}^2 \sigma_{\ell_2}^2 } 
\sum_{i_1=1}^{q^2} \sum_{1\leq i_2\leq q^2:i_2\neq i_1} \sum_{1\leq i_4\leq q^2: i_4\neq i_1, i_2}
\sum_{0\leq u^{(1)}_1,\ldots, u^{(1)}_d \leq 1: |u^{(1)}|=\ell_1 +2}
\nonumber \\
&&\times
\sum_{0\leq u^{(2)}_1,\ldots, u^{(2)}_d \leq 1: |u^{(2)}|=\ell_2 +2}
\sum_{0\leq u^{(3)}_1,\ldots, u^{(3)}_d \leq 1: |u^{(3)}|=\ell_1 +2}
\sum_{0\leq u^{(4)}_1,\ldots, u^{(4)}_d \leq 1: |u^{(4)}|=\ell_2 +2}
\nonumber \\
&&\times
\sum_{k_1=1}^d {\cal I}\{ k_1\in \{j_{1,1},\ldots, j_{1,|u^{(1)}|} \}
\cap \{j_{2,1},\ldots, j_{2,|u^{(2)}|} \} \} {\cal I}\{ a_{i_1, k_1}
= a_{i_2, k_1} \}
\nonumber \\
&&\times
\sum_{k_3=1}^d {\cal I}\{ k_3\in \{j_{3,1},\ldots, j_{3,|u^{(3)}|} \}
\cap \{j_{4,1},\ldots, j_{4,|u^{(4)}|} \} \} {\cal I}\{ a_{i_1, k_3}
= a_{i_4, k_3} \}
\nonumber \\
&&\times
\nu^* [\pi_{j_{1,1}} (a_{i_1, j_{1,1}});\ldots;
\pi_{j_{1,|u^{(1)}|}} (a_{i_1, j_{1,|u^{(1)}|}}) ]
\nu^* [\pi_{j_{2,1}} (a_{i_2, j_{2,1}});\ldots; \pi_{j_{2, |u^{(2)}|}} (a_{i_2, j_{2,|u^{(2)}|}}) ]
\nonumber \\
&&\times
\nu^* [\pi_{j_{3,1}} (a_{i_1, j_{3,1}});\ldots; 
\pi_{j_{3,|u^{(3)}|}} (a_{i_1, j_{3,|u^{(3)}|}}) ]
\nu^* [\pi_{j_{4,1}} (a_{i_4, j_{4,1}});\ldots; \pi_{j_{4, |u^{(4)}|}} (a_{i_4, j_{4,|u^{(4)}|}}) ]
\Big\}
\nonumber \\
&=& O(\frac{1}{q^6}),
\nonumber \\
&& R^{(\tilde{1},1)}_{\{1,2\},\{3,4\}} 
\nonumber \\
&=&
E\Big\{ \frac{1}{d^2 q^{10} (q-1)^2 \sigma_{\ell_1}^2 \sigma_{\ell_2}^2 } 
\sum_{i_1=1}^{q^2} \sum_{1\leq i_3\leq q^2:i_3\neq i_1} 
\sum_{0\leq u^{(1)}_1,\ldots, u^{(1)}_d \leq 1: |u^{(1)}|=\ell_1 +2}
\nonumber \\
&&\times
\sum_{0\leq u^{(2)}_1,\ldots, u^{(2)}_d \leq 1: |u^{(2)}|=\ell_2 +2}
\sum_{0\leq u^{(3)}_1,\ldots, u^{(3)}_d \leq 1: |u^{(3)}|=\ell_1 +2}
\sum_{0\leq u^{(4)}_1,\ldots, u^{(4)}_d \leq 1: |u^{(4)}|=\ell_2 +2}
\nonumber \\
&&\times
\sum_{k_1=1}^d {\cal I}\{ k_1\in \{j_{1,1},\ldots, j_{1,|u^{(1)}|} \}
\cap \{j_{2,1},\ldots, j_{2,|u^{(2)}|} \} \}
\nonumber \\
&&\times
\sum_{k_3=1}^d {\cal I}\{ k_3\in \{j_{3,1},\ldots, j_{3,|u^{(3)}|} \}
\cap \{j_{4,1},\ldots, j_{4,|u^{(4)}|} \} \}
\nonumber \\
&&\times
\nu^* [\pi_{j_{1,1}} (a_{i_1, j_{1,1}});\ldots;
\pi_{j_{1,|u^{(1)}|}} (a_{i_1, j_{1,|u^{(1)}|}}) ]
\nu^* [\pi_{j_{2,1}} (a_{i_1, j_{2,1}});\ldots; \pi_{j_{2, |u^{(2)}|}} (a_{i_1, j_{2,|u^{(2)}|}}) ]
\nonumber \\
&&\times
\nu^* [\pi_{j_{3,1}} (a_{i_3, j_{3,1}});\ldots; 
\pi_{j_{3,|u^{(3)}|}} (a_{i_3, j_{3,|u^{(3)}|}}) ]
\nu^* [\pi_{j_{4,1}} (a_{i_3, j_{4,1}});\ldots; \pi_{j_{4, |u^{(4)}|}} (a_{i_3, j_{4,|u^{(4)}|}}) ]
\Big\}
\nonumber \\
&=& O(\frac{1}{q^4}),
\nonumber \\
&& R^{(\tilde{1},1)}_{\{1,3\},\{2,4\}} 
\hspace{0.1cm} = \hspace{0.1cm}
R^{(\tilde{1},1)}_{\{1,4\},\{3,4\}}
\nonumber \\
&=&
E\Big\{ \frac{1}{d^2 q^{10} (q-1)^2 \sigma_{\ell_1}^2 \sigma_{\ell_2}^2 } 
\sum_{i_1=1}^{q^2} \sum_{1\leq i_2\leq q^2:i_2 \neq i_1} 
\sum_{0\leq u^{(1)}_1,\ldots, u^{(1)}_d \leq 1: |u^{(1)}|=\ell_1 +2}
\nonumber \\
&&\times
\sum_{0\leq u^{(2)}_1,\ldots, u^{(2)}_d \leq 1: |u^{(2)}|=\ell_2 +2}
\sum_{0\leq u^{(3)}_1,\ldots, u^{(3)}_d \leq 1: |u^{(3)}|=\ell_1 +2}
\sum_{0\leq u^{(4)}_1,\ldots, u^{(4)}_d \leq 1: |u^{(4)}|=\ell_2 +2}
\nonumber \\
&&\times
\sum_{k_1=1}^d {\cal I}\{ k_1\in \{j_{1,1},\ldots, j_{1,|u^{(1)}|} \} 
\cap \{j_{2,1},\ldots, j_{2,|u^{(2)}|} \} \} {\cal I}\{ a_{i_1, k_1}
= a_{i_2, k_1} \}
\nonumber \\
&&\times
\sum_{k_3=1}^d {\cal I}\{ k_3\in \{j_{3,1},\ldots, j_{3,|u^{(3)}|} \}
\cap \{j_{4,1},\ldots, j_{4,|u^{(4)}|} \} \} {\cal I}\{ a_{i_1, k_3}
= a_{i_2, k_3} \}
\nonumber \\
&&\times
\nu^* [\pi_{j_{1,1}} (a_{i_1, j_{1,1}});\ldots;
\pi_{j_{1,|u^{(1)}|}} (a_{i_1, j_{1,|u^{(1)}|}}) ]
\nu^* [\pi_{j_{2,1}} (a_{i_2, j_{2,1}});\ldots; \pi_{j_{2, |u^{(2)}|}} (a_{i_2, j_{2,|u^{(2)}|}}) ]
\nonumber \\
&&\times
\nu^* [\pi_{j_{3,1}} (a_{i_1, j_{3,1}});\ldots; 
\pi_{j_{3,|u^{(3)}|}} (a_{i_1, j_{3,|u^{(3)}|}}) ]
\nu^* [\pi_{j_{4,1}} (a_{i_2, j_{4,1}});\ldots; \pi_{j_{4, |u^{(4)}|}} (a_{i_2, j_{4,|u^{(4)}|}}) ]
\Big\}
\nonumber \\
&=& O(\frac{1}{q^5}),
\nonumber \\
&& R^{(\tilde{1},1)}_{\{1\},\{2\},\{3\},\{4\}} 
\nonumber \\
&=&
E\Big\{ \frac{1}{d^2 q^{10} (q-1)^2 \sigma_{\ell_1}^2 \sigma_{\ell_2}^2 } 
\sum_{i_1=1}^{q^2} \sum_{1\leq i_2\leq q^2:i_2\neq i_1} 
\sum_{1\leq i_3\leq q^2:i_3\neq i_1, i_2} 
\sum_{1\leq i_4\leq q^2: i_4\neq i_1, i_2, i_3}
\nonumber \\
&&\times
\sum_{0\leq u^{(1)}_1,\ldots, u^{(1)}_d \leq 1: |u^{(1)}|=\ell_1 +2}
\sum_{0\leq u^{(2)}_1,\ldots, u^{(2)}_d \leq 1: |u^{(2)}|=\ell_2 +2}
\nonumber \\
&&\times
\sum_{0\leq u^{(3)}_1,\ldots, u^{(3)}_d \leq 1: |u^{(3)}|=\ell_1 +2}
\sum_{0\leq u^{(4)}_1,\ldots, u^{(4)}_d \leq 1: |u^{(4)}|=\ell_2 +2}
\nonumber \\
&&\times
\sum_{k_1=1}^d {\cal I}\{ k_1\in \{j_{1,1},\ldots, j_{1,|u^{(1)}|} \}
\cap \{j_{2,1},\ldots, j_{2,|u^{(2)}|} \} \} {\cal I}\{ a_{i_1, k_1}
= a_{i_2, k_1} \}
\nonumber \\
&&\times
\sum_{k_3=1}^d {\cal I}\{ k_3\in \{j_{3,1},\ldots, j_{3,|u^{(3)}|} \}
\cap \{j_{4,1},\ldots, j_{4,|u^{(4)}|} \} \} {\cal I}\{ a_{i_3, k_3}
= a_{i_4, k_3} \}
\nonumber \\
&&\times
\nu^* [\pi_{j_{1,1}} (a_{i_1, j_{1,1}});\ldots;
\pi_{j_{1,|u^{(1)}|}} (a_{i_1, j_{1,|u^{(1)}|}}) ]
\nu^* [\pi_{j_{2,1}} (a_{i_2, j_{2,1}});\ldots; \pi_{j_{2, |u^{(2)}|}} (a_{i_2, j_{2,|u^{(2)}|}}) ]
\nonumber \\
&&\times
\nu^* [\pi_{j_{3,1}} (a_{i_3, j_{3,1}});\ldots; 
\pi_{j_{3,|u^{(3)}|}} (a_{i_3, j_{3,|u^{(3)}|}}) ]
\nu^* [\pi_{j_{4,1}} (a_{i_4, j_{4,1}});\ldots; \pi_{j_{4, |u^{(4)}|}} (a_{i_4, j_{4,|u^{(4)}|}}) ]
\Big\}
\nonumber \\
&=& O(\frac{1}{q^6}),
\end{eqnarray*}
as $q\rightarrow\infty$.
Thus we conclude that
\begin{displaymath}
\frac{d (q-1)}{\ell_1 +2} E | E^{\cal W} (\tilde{S}_{\ell_1,1} S_{\ell_2,1} ) | \leq
\frac{d (q-1)}{\ell_1 +2}
\{ E [ E^{\cal W} (\tilde{S}_{\ell_1,1} S_{\ell_2,1} ) ]^2 \}^{1/2}
= O(\frac{1}{q}),
\end{displaymath}
as $q\rightarrow\infty$. This proves Lemma  \ref{la:a.21}. \hfill $\Box$

\begin{la} \label{la:a.22}
With the notation of (\ref{eq:a.67}),
for $1\leq \ell_1, \ell_2 \leq d-2$,
\begin{displaymath}
\frac{d (q-1)}{\ell_1 +2} E | E^{\cal W} ( \tilde{S}_{\ell_1,1} \tilde{S}_{\ell_2,2} ) | = O(\frac{1}{q}),
\hspace{0.5cm}\mbox{as $q\rightarrow\infty$}.
\end{displaymath}
\end{la}
{\sc Proof.}
We observe that
\begin{eqnarray*}
&& E^{\cal W} (\tilde{S}_{\ell_1,1} \tilde{S}_{\ell_2,2} )
\nonumber \\
&=&
E^{\cal W} \Big\{ \frac{1}{q^4 \sigma_{\ell_1} \sigma_{\ell_2} } \sum_{i_1=1}^{q^2} \sum_{i_2\neq i_1} 
\sum_{0\leq u^{(1)}_1,\ldots, u^{(1)}_d \leq 1: |u^{(1)}|=\ell_1 +2}
\sum_{0\leq u^{(2)}_1,\ldots, u^{(2)}_d \leq 1: |u^{(2)}|=\ell_2 +2}
\nonumber \\
&&\times
{\cal I}\{J\in \{j_{1,1},\ldots, j_{1,|u^{(1)}|} \}
\cap \{j_{2,1},\ldots, j_{2,|u^{(2)}|} \} \}
{\cal I}\{ \pi_J (a_{i_1,J} ) = B_1,
\pi_J (a_{i_2,J} ) = B_2\}
\nonumber \\
&&\times
\nu^* [\tilde{\pi}_{j_{1,1}} (a_{i_1, j_{1,1}});\ldots; \tilde{\pi}_{j_{1,|u^{(1)}|}} (a_{i_1, j_{1,|u^{(1)}|}}) ]
\nu^* [\tilde{\pi}_{j_{2,1}} (a_{i_2, j_{2,1}});\ldots; \tilde{\pi}_{j_{2, |u^{(2)}|}} (a_{i_2, j_{2,|u^{(2)}|}}) ]
\Big\}
\nonumber \\
&=&
\frac{1}{d q^5 (q-1) \sigma_{\ell_1} \sigma_{\ell_2} } \sum_{i_1=1}^{q^2} \sum_{i_2\neq i_1} 
\sum_{0\leq u^{(1)}_1,\ldots, u^{(1)}_d \leq 1: |u^{(1)}|=\ell_1 +2}
\sum_{0\leq u^{(2)}_1,\ldots, u^{(2)}_d \leq 1: |u^{(2)}|=\ell_2 +2}
\nonumber \\
&&\hspace{0.5cm}\times
\sum_{k=1}^d {\cal I}\{ k\in \{j_{1,1},\ldots, j_{1,|u^{(1)}|} \}
\cap \{j_{2,1},\ldots, j_{2,|u^{(2)}|} \} \}
{\cal I}\{ a_{i_1,k} 
\neq a_{i_2,k}  \}
\nonumber \\
&&\hspace{0.5cm}\times
\nu^* [\pi_{j_{1,1}} (a_{i_1, j_{1,1}});\ldots; \pi_k (a_{i_2, k});\ldots; \pi_{j_{1,|u^{(1)}|}} (a_{i_1, j_{1,|u^{(1)}|}}) ]
\nonumber \\
&&\hspace{0.5cm}\times
\nu^* [\pi_{j_{2,1}} (a_{i_2, j_{2,1}});\ldots; \pi_k (a_{i_1, k});\ldots; \pi_{j_{2, |u^{(2)}|}} (a_{i_2, j_{2,|u^{(2)}|}}) ],
\end{eqnarray*}
and
\begin{eqnarray*}
&& E \{ [ E^{\cal W} (\tilde{S}_{\ell_1,1} \tilde{S}_{\ell_2,2} ) ]^2 \}
\nonumber \\
&=&
E \Big\{ \frac{1}{d^2 q^{10} (q-1)^2 \sigma_{\ell_1}^2 \sigma_{\ell_2}^2 } \sum_{i_1=1}^{q^2} \sum_{i_2\neq i_1} 
\sum_{i_3=1}^{q^2} \sum_{i_4\neq i_3} \sum_{0\leq u^{(1)}_1,\ldots, u^{(1)}_d \leq 1: |u^{(1)}|=\ell_1 +2}
\nonumber \\
&&\hspace{0.5cm}\times
\sum_{0\leq u^{(2)}_1,\ldots, u^{(2)}_d \leq 1: |u^{(2)}|=\ell_2 +2}
\sum_{0\leq u^{(3)}_1,\ldots, u^{(3)}_d \leq 1: |u^{(3)}|=\ell_1 +2}
\sum_{0\leq u^{(4)}_1,\ldots, u^{(4)}_d \leq 1: |u^{(4)}|=\ell_2 +2}
\nonumber \\
&&\hspace{0.5cm}\times
\sum_{k_1=1}^d {\cal I}\{ k_1 \in \{j_{1,1},\ldots, j_{1,|u^{(1)}|} \}
\cap \{j_{2,1},\ldots, j_{2,|u^{(2)}|} \} \}
{\cal I}\{ a_{i_1, k_1}
\neq a_{i_2, k_1 } \}
\nonumber \\
&&\hspace{0.5cm}\times
\sum_{k_3=1}^d {\cal I}\{ k_3 \in \{j_{3,1},\ldots, j_{3,|u^{(3)}|} \}
\cap \{j_{4,1},\ldots, j_{4,|u^{(4)}|} \} \}
{\cal I}\{ a_{i_3, k_3}
\neq a_{i_4, k_3 } \}
\nonumber \\
&&\hspace{0.5cm}\times
\nu^* [\pi_{j_{1,1}} (a_{i_1, j_{1,1}});\ldots; \pi_{k_1} (a_{i_2, k_1});\ldots; \pi_{j_{1,|u^{(1)}|}} (a_{i_1, j_{1,|u^{(1)}|}}) ]
\nonumber \\
&&\hspace{0.5cm}\times
\nu^* [\pi_{j_{2,1}} (a_{i_2, j_{2,1}});\ldots; \pi_{k_1} (a_{i_1, k_1});\ldots; \pi_{j_{2, |u^{(2)}|}} (a_{i_2, j_{2,|u^{(2)}|}}) ]
\nonumber \\
&&\hspace{0.5cm}\times
\nu^* [\pi_{j_{3,1}} (a_{i_3, j_{3,1}});\ldots; \pi_{k_3} (a_{i_4, k_3});\ldots; \pi_{j_{3,|u^{(3)}|}} (a_{i_3, j_{3,|u^{(3)}|}}) ]
\nonumber \\
&&\hspace{0.5cm}\times
\nu^* [\pi_{j_{4,1}} (a_{i_4, j_{4,1}});\ldots; \pi_{k_3} (a_{i_3, k_3});\ldots; \pi_{j_{4, |u^{(4)}|}} (a_{i_4, j_{4,|u^{(4)}|}}) ]
\Big\}
\nonumber \\
&=& \tilde{R}^{(1,2)}_{\{1,3\},\{2\},\{4\}} + \tilde{R}^{(1,2)}_{\{1,4\},\{2\},\{3\}} + \tilde{R}^{(1,2)}_{\{2,3\},\{1\},\{4\}} 
+ \tilde{R}^{(1,2)}_{\{2,4\},\{1\},\{3\}} 
\nonumber \\
&&+ \tilde{R}^{(1,2)}_{\{1,3\},\{2,4\}} + \tilde{R}^{(1,2)}_{\{1,4\},\{2,3\}} 
+ \tilde{R}^{(1,2)}_{\{1\},\{2\},\{3\},\{4\}},
\end{eqnarray*}
where
\begin{eqnarray*}
&& \tilde{R}^{(1,2)}_{\{1,3\},\{2\},\{4\}} 
\hspace{0.1cm} = \hspace{0.1cm} \tilde{R}^{(1,2)}_{\{1,4\},\{2\},\{3\}} \hspace{0.1cm} = 
\hspace{0.1cm} \tilde{R}^{(1,2)}_{\{2,3\},\{1\},\{4\}} 
\hspace{0.1cm} = \hspace{0.1cm} \tilde{R}^{(1,2)}_{\{2,4\},\{1\},\{3\}}
\nonumber \\
&=&
E \Big\{ \frac{1}{d^2 q^{10} (q-1)^2 \sigma_{\ell_1}^2 \sigma_{\ell_2}^2 } \sum_{i_1=1}^{q^2} \sum_{1\leq i_2\leq q^2: i_2\neq i_1} 
\sum_{1\leq i_4\leq q^2: i_4\neq i_1, i_2} 
\nonumber \\
&&\times
\sum_{0\leq u^{(1)}_1,\ldots, u^{(1)}_d \leq 1: |u^{(1)}|=\ell_1 +2}
\sum_{0\leq u^{(2)}_1,\ldots, u^{(2)}_d \leq 1: |u^{(2)}|=\ell_2 +2}
\nonumber \\
&&\times
\sum_{0\leq u^{(3)}_1,\ldots, u^{(3)}_d \leq 1: |u^{(3)}|=\ell_1 +2}
\sum_{0\leq u^{(4)}_1,\ldots, u^{(4)}_d \leq 1: |u^{(4)}|=\ell_2 +2}
\nonumber \\
&&\times
\sum_{k_1=1}^d {\cal I}\{ k_1 \in \{j_{1,1},\ldots, j_{1,|u^{(1)}|} \}
\cap \{j_{2,1},\ldots, j_{2,|u^{(2)}|} \} \} {\cal I}\{ a_{i_1, k_1}
\neq a_{i_2, k_1 } \}
\nonumber \\
&&\times
\sum_{k_3=1}^d {\cal I}\{ k_3 \in \{j_{3,1},\ldots, j_{3,|u^{(3)}|} \}
\cap \{j_{4,1},\ldots, j_{4,|u^{(4)}|} \} \} {\cal I}\{ a_{i_1, k_3}
\neq a_{i_4, k_3 } \}
\nonumber \\
&&\times
\nu^* [\pi_{j_{1,1}} (a_{i_1, j_{1,1}});\ldots; \pi_{k_1} (a_{i_2, k_1});\ldots; \pi_{j_{1,|u^{(1)}|}} (a_{i_1, j_{1,|u^{(1)}|}}) ]
\nonumber \\
&&\times
\nu^* [\pi_{j_{2,1}} (a_{i_2, j_{2,1}});\ldots; \pi_{k_1} (a_{i_1, k_1});\ldots; \pi_{j_{2, |u^{(2)}|}} (a_{i_2, j_{2,|u^{(2)}|}}) ]
\nonumber \\
&&\times
\nu^* [\pi_{j_{3,1}} (a_{i_1, j_{3,1}});\ldots; \pi_{k_3} (a_{i_4, k_3});\ldots; \pi_{j_{3,|u^{(3)}|}} (a_{i_1, j_{3,|u^{(3)}|}}) ]
\nonumber \\
&&\times
\nu^* [\pi_{j_{4,1}} (a_{i_4, j_{4,1}});\ldots; \pi_{k_3} (a_{i_1, k_3});\ldots; \pi_{j_{4, |u^{(4)}|}} (a_{i_4, j_{4,|u^{(4)}|}}) ]
\Big\}
\nonumber \\
&=& O(\frac{1}{q^5}),
\nonumber \\
&& \tilde{R}^{(1,2)}_{\{1,3\},\{2,4\}} 
\hspace{0.1cm} = \hspace{0.1cm} \tilde{R}^{(1,2)}_{\{1,4\},\{2,3\}}
\nonumber \\
&=&
E \Big\{ \frac{1}{d^2 q^{10} (q-1)^2 \sigma_{\ell_1}^2 \sigma_{\ell_2}^2 } \sum_{i_1=1}^{q^2} \sum_{1\leq i_2\leq q^2: i_2\neq i_1} 
\sum_{0\leq u^{(1)}_1,\ldots, u^{(1)}_d \leq 1: |u^{(1)}|=\ell_1 +2}
\nonumber \\
&&\times
\sum_{0\leq u^{(2)}_1,\ldots, u^{(2)}_d \leq 1: |u^{(2)}|=\ell_2 +2}
\sum_{0\leq u^{(3)}_1,\ldots, u^{(3)}_d \leq 1: |u^{(3)}|=\ell_1 +2}
\sum_{0\leq u^{(4)}_1,\ldots, u^{(4)}_d \leq 1: |u^{(4)}|=\ell_2 +2}
\nonumber \\
&&\times
\sum_{k_1=1}^d {\cal I}\{ k_1 \in \{j_{1,1},\ldots, j_{1,|u^{(1)}|} \}
\cap \{j_{2,1},\ldots, j_{2,|u^{(2)}|} \} \} {\cal I}\{ a_{i_1, k_1}
\neq a_{i_2, k_1 } \}
\nonumber \\
&&\times
\sum_{k_3=1}^d {\cal I}\{ k_3 \in \{j_{3,1},\ldots, j_{3,|u^{(3)}|} \}
\cap \{j_{4,1},\ldots, j_{4,|u^{(4)}|} \} \} {\cal I}\{ a_{i_1, k_3}
\neq a_{i_2, k_3 } \}
\nonumber \\
&&\times
\nu^* [\pi_{j_{1,1}} (a_{i_1, j_{1,1}});\ldots; \pi_{k_1} (a_{i_2, k_1});\ldots; \pi_{j_{1,|u^{(1)}|}} (a_{i_1, j_{1,|u^{(1)}|}}) ]
\nonumber \\
&&\times
\nu^* [\pi_{j_{2,1}} (a_{i_2, j_{2,1}});\ldots; \pi_{k_1} (a_{i_1, k_1});\ldots; \pi_{j_{2, |u^{(2)}|}} (a_{i_2, j_{2,|u^{(2)}|}}) ]
\nonumber \\
&&\times
\nu^* [\pi_{j_{3,1}} (a_{i_1, j_{3,1}});\ldots; \pi_{k_3} (a_{i_2, k_3});\ldots; \pi_{j_{3,|u^{(3)}|}} (a_{i_1, j_{3,|u^{(3)}|}}) ]
\nonumber \\
&&\times
\nu^* [\pi_{j_{4,1}} (a_{i_2, j_{4,1}});\ldots; \pi_{k_3} (a_{i_1, k_3});\ldots; \pi_{j_{4, |u^{(4)}|}} (a_{i_2, j_{4,|u^{(4)}|}}) ]
\Big\}
\nonumber \\
&=& O(\frac{1}{q^4}),
\nonumber \\
&& \tilde{R}^{(1,2)}_{\{1\},\{2\},\{3\},\{4\}} 
\nonumber \\
&=&
E \Big\{ \frac{1}{d^2 q^{10} (q-1)^2 \sigma_{\ell_1}^2 \sigma_{\ell_2}^2 } \sum_{i_1=1}^{q^2} \sum_{1\leq i_2\leq q^2: i_2\neq i_1} 
\sum_{1\leq i_3\leq q^2: i_3\neq i_1, i_2}
\nonumber \\
&&\times
\sum_{1\leq i_4\leq q^2: i_4\neq i_1, i_2, i_3} 
\sum_{0\leq u^{(1)}_1,\ldots, u^{(1)}_d \leq 1: |u^{(1)}|=\ell_1 +2}
\sum_{0\leq u^{(2)}_1,\ldots, u^{(2)}_d \leq 1: |u^{(2)}|=\ell_2 +2}
\nonumber \\
&&\times
\sum_{0\leq u^{(3)}_1,\ldots, u^{(3)}_d \leq 1: |u^{(3)}|=\ell_1 +2}
\sum_{0\leq u^{(4)}_1,\ldots, u^{(4)}_d \leq 1: |u^{(4)}|=\ell_2 +2}
\nonumber \\
&&\times
\sum_{k_1=1}^d {\cal I}\{ k_1 \in \{j_{1,1},\ldots, j_{1,|u^{(1)}|} \}
\cap \{j_{2,1},\ldots, j_{2,|u^{(2)}|} \} \} {\cal I}\{ a_{i_1, k_1}
\neq a_{i_2, k_1 } \}
\nonumber \\
&&\times
\sum_{k_3=1}^d {\cal I}\{ k_3 \in \{j_{3,1},\ldots, j_{3,|u^{(3)}|} \}
\cap \{j_{4,1},\ldots, j_{4,|u^{(4)}|} \} \} {\cal I}\{ a_{i_3, k_3}
\neq a_{i_4, k_3 } \}
\nonumber \\
&&\times
\nu^* [\pi_{j_{1,1}} (a_{i_1, j_{1,1}});\ldots; \pi_{k_1} (a_{i_2, k_1});\ldots; \pi_{j_{1,|u^{(1)}|}} (a_{i_1, j_{1,|u^{(1)}|}}) ]
\nonumber \\
&&\times
\nu^* [\pi_{j_{2,1}} (a_{i_2, j_{2,1}});\ldots; \pi_{k_1} (a_{i_1, k_1});\ldots; \pi_{j_{2, |u^{(2)}|}} (a_{i_2, j_{2,|u^{(2)}|}}) ]
\nonumber \\
&&\times
\nu^* [\pi_{j_{3,1}} (a_{i_3, j_{3,1}});\ldots; \pi_{k_3} (a_{i_4, k_3});\ldots; \pi_{j_{3,|u^{(3)}|}} (a_{i_3, j_{3,|u^{(3)}|}}) ]
\nonumber \\
&&\times
\nu^* [\pi_{j_{4,1}} (a_{i_4, j_{4,1}});\ldots; \pi_{k_3} (a_{i_3, k_3});\ldots; \pi_{j_{4, |u^{(4)}|}} (a_{i_4, j_{4,|u^{(4)}|}}) ]
\Big\}
\nonumber \\
&=& O(\frac{1}{q^6}),
\end{eqnarray*}
as $q\rightarrow\infty$.
Thus we conclude that
\begin{displaymath}
\frac{d (q-1)}{\ell_1 +2} E | E^{\cal W} (\tilde{S}_{\ell_1,1} \tilde{S}_{\ell_2,2} ) | \leq
\frac{d (q-1)}{\ell_1 +2}
\{ E [ E^{\cal W} (\tilde{S}_{\ell_1,1} \tilde{S}_{\ell_2,2} ) ]^2 \}^{1/2}
= O(\frac{1}{q}),
\end{displaymath}
as $q\rightarrow\infty$. This proves Lemma  \ref{la:a.22}. \hfill $\Box$

\begin{pn} \label{pn:a.4}
Let $S_i$ and $\tilde{S}_i$, $i=1,\ldots, d-2$, be as in (\ref{eq:3.75}).
Then for $1\leq i\neq j\leq d-2$,
\begin{displaymath}
| \frac{ d(q-1)}{4(i+2)} 
E \{ (\tilde{S}_i -S_i)^2
\int_0^1 [ \frac{\partial^2}{\partial v_i^2} \psi_{\varepsilon^2} (V + t(\tilde{V}-V))
- \frac{\partial^2}{\partial v_i^2} \psi_{\varepsilon^2} (V) ] dt \} |
\leq O(\frac{\|h\|_\infty}{ \varepsilon q^{1/2}}),
\end{displaymath}
and
\begin{eqnarray*}
&& | \frac{ d(q-1)}{4(i+2)} 
E \{ (\tilde{S}_i -S_i) (\tilde{S}_j - S_j)
\int_0^1 [ \frac{\partial^2}{\partial v_i\partial v_j} \psi_{\varepsilon^2} (V + t(\tilde{V}-V))
- \frac{\partial^2}{\partial v_i\partial v_j} \psi_{\varepsilon^2} (V) ] dt \} |
\nonumber \\
& \leq & O(\frac{\|h\|_\infty}{ \varepsilon q^{1/2}}),
\end{eqnarray*}
as $q\rightarrow \infty$ uniformly over $0<\varepsilon <1$.
\end{pn}
{\sc Proof.}
Using Taylor series and Lemma \ref{la:5.1}, we observe that for $0<\varepsilon <1$ and
$i=1,\ldots, d-2$,
\begin{eqnarray*}
&& | \frac{ d(q-1)}{4(i+2)} 
E \{ (\tilde{S}_i -S_i)^2
\int_0^1 [ \frac{\partial^2}{\partial v_i^2} \psi_{\varepsilon^2} (V + t(\tilde{V}-V))
- \frac{\partial^2}{\partial v_i^2} \psi_{\varepsilon^2} (V) ] dt \} |
\nonumber \\
&\leq & \sum_{j=1}^{d-2} \frac{ d(q-1)}{4(i+2)} 
E [ (\tilde{S}_i -S_i)^2 |\tilde{S}_j -S_j| \sup_{v\in {\mathbb R}^{d-2} }
|\frac{\partial^3}{\partial v_i^2 \partial v_j} \psi_{\varepsilon^2} (v) | ]
\nonumber \\
&\leq & \frac{c \|h\|_\infty }{\varepsilon}
\sum_{j=1}^{d-2} \frac{ d(q-1)}{4(i+2)} 
\{ E [ (\tilde{S}_i -S_i)^4 ]\}^{1/2} \{ E[ (\tilde{S}_j -S_j)^2 ] \}^{1/2}
\nonumber \\
&\leq & \frac{2^{1/2} c \|h\|_\infty }{\varepsilon}
\sum_{j=1}^{d-2} \frac{ d(q-1)}{4(i+2)} 
[ E ( S_i^4 ) ]^{1/2} \{ E[ (\tilde{S}_j -S_j)^2 ] \}^{1/2},
\end{eqnarray*}
where $c$ is a constant depending only on $d$. In a similar fashion, we have for $1\leq i\neq j\leq d-2$,
\begin{eqnarray*}
&& | \frac{ d(q-1)}{4(i+2)} 
E \{ (\tilde{S}_i -S_i) (\tilde{S}_j - S_j)
\int_0^1 [ \frac{\partial^2}{\partial v_i\partial v_j} \psi_{\varepsilon^2} (V + t(\tilde{V}-V))
- \frac{\partial^2}{\partial v_i\partial v_j} \psi_{\varepsilon^2} (V) ] dt \} |
\nonumber \\
& \leq & \sum_{k=1}^{d-2}
\frac{ d(q-1)}{4(i+2)}
E [ |\tilde{S}_i -S_i| |\tilde{S}_j - S_j| |\tilde{S}_k -S_k| 
\sup_{v\in {\mathbb R}^{d-2}} | \frac{\partial^3}{\partial v_i\partial v_j\partial v_k} \psi_{\varepsilon^2} (v) | ]
\nonumber \\
&\leq & \frac{ c\|h\|_\infty}{\varepsilon} \sum_{k=1}^{d-2} \frac{ d(q-1)}{4(i+2)}
\{E [ (\tilde{S}_i -S_i)^4] \}^{1/4} \{E [(\tilde{S}_j - S_j)^4 ] \}^{1/4} \{ E [(\tilde{S}_k -S_k)^2] \}^{1/2}
\nonumber \\
&\leq &
\frac{ 2^{1/2} c\|h\|_\infty}{\varepsilon} \sum_{k=1}^{d-2} \frac{ d(q-1)}{4(i+2)}
[ E (S_i^4 ) ]^{1/4} [ E ( S_j^4 )]^{1/4} \{ E [(\tilde{S}_k -S_k)^2] \}^{1/2}.
\end{eqnarray*}
Proposition \ref{pn:a.4} now follows from Proposition \ref{pn:a.2} and Lemma \ref{la:a.23}.
\hfill $\Box$

\begin{la} \label{la:a.23}
Let $S_\ell, \ell = 1,\ldots, d-2,$ be as in (\ref{eq:3.75}).
Then  $E (S_\ell^4 ) = O(1/q^2)$ as $q\rightarrow\infty$.
\end{la}
{\sc Proof.}
We observe from (\ref{eq:a.78}) that $S_\ell^4 = S_{\ell,1}^4 + S_{\ell,2}^4$ and hence
\begin{eqnarray}
&& E (S_\ell^4 ) 
\hspace{0.1cm} = \hspace{0.1cm} 2 E (S_{\ell,1}^4 )
\nonumber \\
&=&
\frac{2}{ q^8 \sigma_\ell^4}  E\Big\{  \sum_{i_1=1}^{q^2} \sum_{i_2=1}^{q^2}
\sum_{i_3=1}^{q^2} \sum_{i_4=1}^{q^2}
\sum_{0\leq u_1^{(1)},\ldots, u_d^{(1)}\leq 1: |u^{(1)}|=\ell+2} 
\nonumber \\
&&\times
\sum_{0\leq u_1^{(2)},\ldots, u_d^{(2)}\leq 1: |u^{(2)}|=\ell+2} 
\sum_{0\leq u_1^{(3)},\ldots, u_d^{(3)}\leq 1: |u^{(3)}|=\ell+2} 
\sum_{0\leq u_1^{(4)},\ldots, u_d^{(4)}\leq 1: |u^{(4)}|=\ell+2} 
\nonumber \\
&&\times
{\cal I}\{ J\in \cap_{l =1}^4 \{ j_{l,1},\ldots, j_{ l,|u^{(l)}|} \}\}
{\cal I}\{ B_1 = \pi_J (a_{i_l, J}), l=1,\ldots, 4 \}
\nonumber \\
&&\times
\nu^* [\pi_{j_{1,1}} (a_{i_1, j_{1,1}}); \ldots; \pi_{j_{1,|u^{(1)}|}} (a_{i_1, j_{1, |u^{(1)}|}}) ]
\nu^* [\pi_{j_{2,1}} (a_{i_2, j_{2,1}}); \ldots; \pi_{j_{2,|u^{(2)}|}} (a_{i_2, j_{2, |u^{(2)}|}}) ]
\nonumber \\
&&\times
\nu^* [\pi_{j_{3,1}} (a_{i_3, j_{3,1}}); \ldots; \pi_{j_{3,|u^{(3)}|}} (a_{i_3, j_{3, |u^{(3)}|}}) ]
\nu^* [\pi_{j_{4,1}} (a_{i_4, j_{4,1}}); \ldots; \pi_{j_{4,|u^{(4)}|}} (a_{i_4, j_{4, |u^{(4)}|}}) ]
\Big\}
\nonumber \\
&=&
\frac{2}{ d q^9 \sigma_\ell^4}  E\Big\{  \sum_{i_1=1}^{q^2} \sum_{i_2=1}^{q^2}
\sum_{i_3=1}^{q^2} \sum_{i_4=1}^{q^2}
\sum_{0\leq u_1^{(1)},\ldots, u_d^{(1)}\leq 1: |u^{(1)}|=\ell+2} 
\nonumber \\
&&\times
\sum_{0\leq u_1^{(2)},\ldots, u_d^{(2)}\leq 1: |u^{(2)}|=\ell+2} 
\sum_{0\leq u_1^{(3)},\ldots, u_d^{(3)}\leq 1: |u^{(3)}|=\ell+2} 
\sum_{0\leq u_1^{(4)},\ldots, u_d^{(4)}\leq 1: |u^{(4)}|=\ell+2} 
\nonumber \\
&&\times
\sum_{k=1}^d {\cal I}\{ k\in \cap_{l =1}^4 \{ j_{l,1},\ldots, j_{ l,|u^{(l)}|} \}\}
{\cal I}\{ a_{i_1, k} = a_{i_2, k} = a_{i_3, k} = a_{i_4, k} \}
\nonumber \\
&&\times
\nu^* [\pi_{j_{1,1}} (a_{i_1, j_{1,1}}); \ldots; \pi_{j_{1,|u^{(1)}|}} (a_{i_1, j_{1, |u^{(1)}|}}) ]
\nu^* [\pi_{j_{2,1}} (a_{i_2, j_{2,1}}); \ldots; \pi_{j_{2,|u^{(2)}|}} (a_{i_2, j_{2, |u^{(2)}|}}) ]
\nonumber \\
&&\times
\nu^* [\pi_{j_{3,1}} (a_{i_3, j_{3,1}}); \ldots; \pi_{j_{3,|u^{(3)}|}} (a_{i_3, j_{3, |u^{(3)}|}}) ]
\nu^* [\pi_{j_{4,1}} (a_{i_4, j_{4,1}}); \ldots; \pi_{j_{4,|u^{(4)}|}} (a_{i_4, j_{4, |u^{(4)}|}}) ]
\Big\}
\nonumber \\
&=& R^*_{\{1,2,3,4\}} + R^*_{\{1,2,3\},\{4\}} + R^*_{\{1,2,4\},\{3\}} + R^*_{\{1,3,4\},\{2\}} 
\nonumber \\
&&
+ R^*_{\{2,3,4\},\{1\}}
+ R^*_{\{1,2\},\{3\},\{4\}} + R^*_{\{1,3\},\{2\},\{4\}} + R^*_{\{1,4\},\{2\},\{3\}} 
\nonumber \\
&&
+ R^*_{\{2,3\},\{1\},\{4\}}
+ R^*_{\{2,4\},\{1\},\{3\}}
+ R^*_{\{3,4\},\{1\},\{2\}} + R^*_{\{1,2\},\{3,4\}} 
\nonumber \\
&&
+ R^*_{\{1,3\},\{2,4\}} + R^*_{\{1,4\},\{2,3\}} + R^*_{\{1\},\{2\},\{3\},\{4\}},
\label{eq:a.62}
\end{eqnarray}
where
\begin{eqnarray*}
&& R^*_{\{1,2,3,4\}} 
\nonumber \\
&=&
\frac{2}{ d q^9 \sigma_\ell^4}  E\Big\{  \sum_{i_1=i_2=i_3=i_4=1}^{q^2} 
\sum_{0\leq u_1^{(1)},\ldots, u_d^{(1)}\leq 1: |u^{(1)}|=\ell+2} 
\sum_{0\leq u_1^{(2)},\ldots, u_d^{(2)}\leq 1: |u^{(2)}|=\ell+2} 
\nonumber \\
&&\times
\sum_{0\leq u_1^{(3)},\ldots, u_d^{(3)}\leq 1: |u^{(3)}|=\ell+2} 
\sum_{0\leq u_1^{(4)},\ldots, u_d^{(4)}\leq 1: |u^{(4)}|=\ell+2} 
\sum_{k=1}^d {\cal I}\{ k\in \cap_{l =1}^4 \{ j_{l,1},\ldots, j_{ l,|u^{(l)}|} \}\}
\nonumber \\
&&\times
\nu^* [\pi_{j_{1,1}} (a_{i_1, j_{1,1}}); \ldots; \pi_{j_{1,|u^{(1)}|}} (a_{i_1, j_{1, |u^{(1)}|}}) ]
\nu^* [\pi_{j_{2,1}} (a_{i_2, j_{2,1}}); \ldots; \pi_{j_{2,|u^{(2)}|}} (a_{i_2, j_{2, |u^{(2)}|}}) ]
\nonumber \\
&&\times
\nu^* [\pi_{j_{3,1}} (a_{i_3, j_{3,1}}); \ldots; \pi_{j_{3,|u^{(3)}|}} (a_{i_3, j_{3, |u^{(3)}|}}) ]
\nu^* [\pi_{j_{4,1}} (a_{i_4, j_{4,1}}); \ldots; \pi_{j_{4,|u^{(4)}|}} (a_{i_4, j_{4, |u^{(4)}|}}) ]
\Big\}
\nonumber \\
&=& O(\frac{1}{q^3}),
\nonumber \\
&& R^*_{\{1,2,3\},\{4\}} 
\hspace{0.1cm} = \hspace{0.1cm} R^*_{\{1,2,4\},\{3\}} \hspace{0.1cm} = \hspace{0.1cm} R^*_{\{1,3,4\},\{2\}} 
\hspace{0.1cm} = \hspace{0.1cm} R^*_{\{2,3,4\},\{1\}}
\nonumber \\
&=&
\frac{2}{ d q^9 \sigma_\ell^4}  E\Big\{  \sum_{i_1=i_2=i_3=1}^{q^2} 
\sum_{1\leq i_4\leq q^2: i_4\neq i_1}
\sum_{0\leq u_1^{(1)},\ldots, u_d^{(1)}\leq 1: |u^{(1)}|=\ell+2} 
\nonumber \\
&&\times
\sum_{0\leq u_1^{(2)},\ldots, u_d^{(2)}\leq 1: |u^{(2)}|=\ell+2} 
\sum_{0\leq u_1^{(3)},\ldots, u_d^{(3)}\leq 1: |u^{(3)}|=\ell+2} 
\sum_{0\leq u_1^{(4)},\ldots, u_d^{(4)}\leq 1: |u^{(4)}|=\ell+2} 
\nonumber \\
&&\times
\sum_{k=1}^d {\cal I}\{ k\in \cap_{l =1}^4 \{ j_{l,1},\ldots, j_{ l,|u^{(l)}|} \}\}
{\cal I}\{ a_{i_1, k} = a_{i_4, k} \}
\nonumber \\
&&\times
\nu^* [\pi_{j_{1,1}} (a_{i_1, j_{1,1}}); \ldots; \pi_{j_{1,|u^{(1)}|}} (a_{i_1, j_{1, |u^{(1)}|}}) ]
\nu^* [\pi_{j_{2,1}} (a_{i_2, j_{2,1}}); \ldots; \pi_{j_{2,|u^{(2)}|}} (a_{i_2, j_{2, |u^{(2)}|}}) ]
\nonumber \\
&&\times
\nu^* [\pi_{j_{3,1}} (a_{i_3, j_{3,1}}); \ldots; \pi_{j_{3,|u^{(3)}|}} (a_{i_3, j_{3, |u^{(3)}|}}) ]
\nu^* [\pi_{j_{4,1}} (a_{i_4, j_{4,1}}); \ldots; \pi_{j_{4,|u^{(4)}|}} (a_{i_4, j_{4, |u^{(4)}|}}) ]
\Big\}
\nonumber \\
&=& O(\frac{1}{q^4}),
\nonumber \\
&& R^*_{\{1,2\}, \{3,4\}} 
\hspace{0.1cm} = \hspace{0.1cm} R^*_{\{1,3\},\{2,4\}} \hspace{0.1cm} = \hspace{0.1cm} R^*_{\{1,4\},\{2,3\}}
\nonumber \\
&=&
\frac{2}{ d q^9 \sigma_\ell^4}  E\Big\{  \sum_{i_1=i_2=1}^{q^2} 
\sum_{1\leq i_3=i_4 \leq q^2: i_3\neq i_1}
\sum_{0\leq u_1^{(1)},\ldots, u_d^{(1)}\leq 1: |u^{(1)}|=\ell+2} 
\nonumber \\
&&\times
\sum_{0\leq u_1^{(2)},\ldots, u_d^{(2)}\leq 1: |u^{(2)}|=\ell+2} 
\sum_{0\leq u_1^{(3)},\ldots, u_d^{(3)}\leq 1: |u^{(3)}|=\ell+2} 
\sum_{0\leq u_1^{(4)},\ldots, u_d^{(4)}\leq 1: |u^{(4)}|=\ell+2} 
\nonumber \\
&&\times
\sum_{k=1}^d {\cal I}\{ k\in \cap_{l =1}^4 \{ j_{l,1},\ldots, j_{ l,|u^{(l)}|} \}\}
{\cal I}\{ a_{i_1, k} = a_{i_3, k} \}
\nonumber \\
&&\times
\nu^* [\pi_{j_{1,1}} (a_{i_1, j_{1,1}}); \ldots; \pi_{j_{1,|u^{(1)}|}} (a_{i_1, j_{1, |u^{(1)}|}}) ]
\nu^* [\pi_{j_{2,1}} (a_{i_1, j_{2,1}}); \ldots; \pi_{j_{2,|u^{(2)}|}} (a_{i_1, j_{2, |u^{(2)}|}}) ]
\nonumber \\
&&\times
\nu^* [\pi_{j_{3,1}} (a_{i_3, j_{3,1}}); \ldots; \pi_{j_{3,|u^{(3)}|}} (a_{i_3, j_{3, |u^{(3)}|}}) ]
\nu^* [\pi_{j_{4,1}} (a_{i_3, j_{4,1}}); \ldots; \pi_{j_{4,|u^{(4)}|}} (a_{i_3, j_{4, |u^{(4)}|}}) ]
\Big\}
\nonumber \\
&=& O(\frac{1}{q^2}),
\nonumber \\
&& R^*_{\{1,2\}, \{3\}, \{4\}} 
\nonumber \\
&=&
\frac{2}{ d q^9 \sigma_\ell^4}  E\Big\{  \sum_{i_1=i_2=1}^{q^2} 
\sum_{1\leq i_3 \leq q^2: i_3\neq i_1} \sum_{1\leq i_4\leq q^2: i_4\neq i_1, i_3}
\sum_{0\leq u_1^{(1)},\ldots, u_d^{(1)}\leq 1: |u^{(1)}|=\ell+2} 
\nonumber \\
&&\times
\sum_{0\leq u_1^{(2)},\ldots, u_d^{(2)}\leq 1: |u^{(2)}|=\ell+2} 
\sum_{0\leq u_1^{(3)},\ldots, u_d^{(3)}\leq 1: |u^{(3)}|=\ell+2} 
\sum_{0\leq u_1^{(4)},\ldots, u_d^{(4)}\leq 1: |u^{(4)}|=\ell+2} 
\nonumber \\
&&\times
\sum_{k=1}^d {\cal I}\{ k\in \cap_{l =1}^4 \{ j_{l,1},\ldots, j_{ l,|u^{(l)}|} \}\}
{\cal I}\{ a_{i_1, k} = a_{i_3, k} = a_{i_4, k} \}
\nonumber \\
&&\times
\nu^* [\pi_{j_{1,1}} (a_{i_1, j_{1,1}}); \ldots; \pi_{j_{1,|u^{(1)}|}} (a_{i_1, j_{1, |u^{(1)}|}}) ]
\nu^* [\pi_{j_{2,1}} (a_{i_1, j_{2,1}}); \ldots; \pi_{j_{2,|u^{(2)}|}} (a_{i_1, j_{2, |u^{(2)}|}}) ]
\nonumber \\
&&\times
\nu^* [\pi_{j_{3,1}} (a_{i_3, j_{3,1}}); \ldots; \pi_{j_{3,|u^{(3)}|}} (a_{i_3, j_{3, |u^{(3)}|}}) ]
\nu^* [\pi_{j_{4,1}} (a_{i_4, j_{4,1}}); \ldots; \pi_{j_{4,|u^{(4)}|}} (a_{i_4, j_{4, |u^{(4)}|}}) ]
\Big\}
\nonumber \\
&=& O(\frac{1}{q^3}),
\nonumber \\
&& R^*_{\{1\}, \{2\}, \{3\}, \{4\}} 
\nonumber \\
&=&
\frac{2}{ d q^9 \sigma_\ell^4}  E\Big\{  \sum_{i_1=1}^{q^2} 
\sum_{1\leq i_2\leq q^2: i_2\neq i_1}
\sum_{1\leq i_3 \leq q^2: i_3\neq i_1, i_2} \sum_{1\leq i_4\leq q^2: i_4\neq i_1, _2, i_3}
\nonumber \\
&&\times
\sum_{0\leq u_1^{(1)},\ldots, u_d^{(1)}\leq 1: |u^{(1)}|=\ell+2} 
\sum_{0\leq u_1^{(2)},\ldots, u_d^{(2)}\leq 1: |u^{(2)}|=\ell+2} 
\nonumber \\
&&\times
\sum_{0\leq u_1^{(3)},\ldots, u_d^{(3)}\leq 1: |u^{(3)}|=\ell+2} 
\sum_{0\leq u_1^{(4)},\ldots, u_d^{(4)}\leq 1: |u^{(4)}|=\ell+2} 
\nonumber \\
&&\times
\sum_{k=1}^d {\cal I}\{ k\in \cap_{l =1}^4 \{ j_{l,1},\ldots, j_{ l,|u^{(l)}|} \}\}
{\cal I}\{ a_{i_1, k} = a_{i_2, k} = a_{i_3, k} = a_{i_4, k} \}
\nonumber \\
&&\times
\nu^* [\pi_{j_{1,1}} (a_{i_1, j_{1,1}}); \ldots; \pi_{j_{1,|u^{(1)}|}} (a_{i_1, j_{1, |u^{(1)}|}}) ]
\nu^* [\pi_{j_{2,1}} (a_{i_2, j_{2,1}}); \ldots; \pi_{j_{2,|u^{(2)}|}} (a_{i_2, j_{2, |u^{(2)}|}}) ]
\nonumber \\
&&\times
\nu^* [\pi_{j_{3,1}} (a_{i_3, j_{3,1}}); \ldots; \pi_{j_{3,|u^{(3)}|}} (a_{i_3, j_{3, |u^{(3)}|}}) ]
\nu^* [\pi_{j_{4,1}} (a_{i_4, j_{4,1}}); \ldots; \pi_{j_{4,|u^{(4)}|}} (a_{i_4, j_{4, |u^{(4)}|}}) ]
\Big\}
\nonumber \\
&=& O(\frac{1}{q^4}),
\end{eqnarray*}
as $q\rightarrow \infty$.
Consequently we conclude from (\ref{eq:a.62}) that $E  (S_\ell^4 ) = O(1/q^2)$ as $q\rightarrow\infty$.
This proves Lemma \ref{la:a.23}.
\hfill $\Box$

\begin{pn} \label{pn:a.5}
Let $S_i$ and $\tilde{S}_i$, $i=1,\ldots, d-2$, be as in (\ref{eq:3.75}).
Then for $1\leq i\neq j\leq d-2$,
\begin{displaymath}
| \frac{ d(q-1)}{4(i+2)} 
E [ (\tilde{S}_i -S_i) (\tilde{S}_j - S_j)
\frac{\partial^2}{\partial v_i \partial v_j} \psi_{\varepsilon^2} (V) ] |
= O( \frac{\|h\|_\infty}{ q^{1/2}}) \log(1/\varepsilon),
\end{displaymath}
as $q\rightarrow \infty$ uniformly over $0<\varepsilon <1$.
\end{pn}
{\sc Proof.} Let $S_{i,k}$ and $\tilde{S}_{i,k}$, $i=1,\ldots, d-2$ and $k=1,2$, be as in (\ref{eq:a.78}).
Then $S_i= S_{i,1}+ S_{i,2}$ and $\tilde{S}_i = \tilde{S}_{i,1} + \tilde{S}_{i,2}$.
For $1\leq i\neq j\leq d-2$, we observe that
\begin{eqnarray*}
&& | \frac{ d(q-1)}{4(i+2)} 
E [ (\tilde{S}_i -S_i) (\tilde{S}_j - S_j)
\frac{\partial^2}{\partial v_i \partial v_j} \psi_{\varepsilon^2} (V) ] |
\nonumber \\
&= & | \frac{d(q-1)}{4 (i+2)} E \{ \frac{\partial^2}{\partial v_i \partial v_j} \psi_{\varepsilon^2} (V)
 E^{\cal W} [(\tilde{S}_i - S_i) (\tilde{S}_j -S_j) ] \} |
\nonumber \\
& = & \frac{d(q-1)}{4 (i+2)} \{ \sup_{v\in {\mathbb R}^{d-2}} |\frac{\partial^2}{\partial v_i \partial v_j} \psi_{\varepsilon^2} (v)| \}
\nonumber \\
&&\times
E | E^{\cal W} [(\tilde{S}_{i,1} + \tilde{S}_{i,2} - S_{i,1} -S_{i,2}) 
(\tilde{S}_{j,1} + \tilde{S}_{j,2} -S_{j,1} -S_{j,2}) ] |
\nonumber \\
& = & \frac{d(q-1)}{4 (i+2)} 
\{ \sup_{v\in {\mathbb R}^{d-2}} | \frac{\partial^2}{\partial v_i \partial v_j} \psi_{\varepsilon^2} (v) | \}
E | E^{\cal W} ( \tilde{S}_{i,1} \tilde{S}_{j,1} + \tilde{S}_{i,1} \tilde{S}_{j,2} -\tilde{S}_{i,1} S_{j,1} - \tilde{S}_{i,1} S_{j,2}
\nonumber \\
&&
 + \tilde{S}_{i,2} \tilde{S}_{j,1} + \tilde{S}_{i,2} \tilde{S}_{j,2} -\tilde{S}_{i,2} S_{j,1} - \tilde{S}_{i,2} S_{j,2}
- S_{i,1} \tilde{S}_{j,1} - S_{i,1} \tilde{S}_{j,2} + S_{i,1} S_{j,1} + S_{i,1} S_{j,2}
\nonumber \\
&&
 - S_{i,2} \tilde{S}_{j,1} - S_{i,2} \tilde{S}_{j,2} + S_{i,2} S_{j,1} + S_{i,2} S_{j,2} ) |
\nonumber \\
&\leq & \frac{d(q-1)}{2 (i+2)}
\{ \sup_{v\in {\mathbb R}^{d-2}} | \frac{\partial^2}{\partial v_i \partial v_j} \psi_{\varepsilon^2} (v) | \} \Big\{
E | E^{\cal W} ( \tilde{S}_{i,1} \tilde{S}_{j,1} )|   + E| E^{\cal W} ( \tilde{S}_{i,1} \tilde{S}_{j,2} )| 
\nonumber \\
&& 
+ E| E^{\cal W} ( \tilde{S}_{i,1} S_{j,1} )| 
+ E| E^{\cal W} ( \tilde{S}_{i,1} S_{j,2} )| 
+ E| E^{\cal W} ( S_{i,1} \tilde{S}_{j,1} )| 
\nonumber \\
&&
+ E| E^{\cal W} ( S_{i,1} \tilde{S}_{j,2} )| 
+ E| E^{\cal W} ( S_{i,1} S_{j,1} )| 
+ E| E^{\cal W} ( S_{i,1} S_{j,2} )| \Big\}.
\end{eqnarray*}
Now it follows from Lemma \ref{la:5.1}, Lemmas \ref{la:a.19} to \ref{la:a.22} and Lemma \ref{la:a.24} that
\begin{displaymath}
| \frac{ d(q-1)}{4(i+2)} 
E [ (\tilde{S}_i -S_i) (\tilde{S}_j - S_j)
\frac{\partial^2}{\partial v_i \partial v_j} \psi_{\varepsilon^2} (V) ] |
\nonumber \\
\leq O( \frac{\|h\|_\infty}{ q^{1/2}}) \log(1/\varepsilon),
\end{displaymath}
as $q\rightarrow \infty$ uniformly over $0<\varepsilon <1$.
This proves Proposition \ref{pn:a.5}. \hfill $\Box$

\begin{la} \label{la:a.24}
Let $S_{\ell, 1}$ and $\tilde{S}_{\ell, 1}$, $\ell = 1,\ldots, d-2,$ be as in (\ref{eq:a.78}).
Then  for $1\leq \ell_1\neq \ell_2\leq d-2$, 
\begin{eqnarray*}
\frac{d(q-1)}{2 (\ell_1 +2)} E| E^{\cal W} (S_{\ell_1, 1} S_{\ell_2, 1} )| & = & O( q^{-1/2}), 
\nonumber \\
\frac{d(q-1)}{2 (\ell_1 +2)} E| E^{\cal W} ( \tilde{S}_{\ell_1, 1} \tilde{S}_{\ell_2, 1} )| & = & O(q^{-1/2}), 
\hspace{0.5cm}\mbox{as $q\rightarrow\infty$}.
\end{eqnarray*}
\end{la}
{\sc Proof.} We observe from (\ref{eq:a.78}) that for $1\leq \ell_1 <\ell_2 \leq d-2$,
\begin{eqnarray*}
&& E^{\cal W} (S_{\ell_1, 1} S_{\ell_2, 1} )
\nonumber \\
&=& E^{\cal W} \Big\{ \frac{ 1}{q^4 \sigma_{\ell_1} \sigma_{\ell_2} } 
\sum_{i_1=1}^{q^2} \sum_{i_2=1}^{q^2} \sum_{0\leq u_1^{(1)},\ldots, u_d^{(1)}\leq 1: |u^{(1)}|=\ell_1+2}
\sum_{0\leq u_1^{(2)},\ldots, u_d^{(2)}\leq 1: |u^{(2)}|=\ell_2+2}
\nonumber \\
&& \times {\cal I} \{ J\in \{j_{1,1},\ldots, j_{1,|u^{(1)}|} \} \cap \{j_{2,1},\ldots, j_{2,|u^{(2)}|} \}
{\cal I} \{ \pi_J (a_{i_1, J} ) = B_1 = \pi_J (a_{i_2, J}) \}
\nonumber \\
&&
\times
\nu^* [\pi_{j_{1,1}} (a_{i_1, j_{1,1}});\ldots; \pi_{j_{1,|u^{(1)}|}} (a_{i_1, j_{1, |u^{(1)}|}} ) ]
\nu^* [\pi_{j_{2,1}} (a_{i_2, j_{2,1}});\ldots; \pi_{j_{2,|u^{(2)}|}} (a_{i_2, j_{2, |u^{(2)}|}} ) ]
\Big\}
\nonumber \\
&=& \frac{ 1}{d q^5 \sigma_{\ell_1} \sigma_{\ell_2} } 
\sum_{i_1=1}^{q^2} \sum_{i_2=1}^{q^2} \sum_{0\leq u_1^{(1)},\ldots, u_d^{(1)}\leq 1: |u^{(1)}|=\ell_1+2}
\sum_{0\leq u_1^{(2)},\ldots, u_d^{(2)}\leq 1: |u^{(2)}|=\ell_2+2}
\nonumber \\
&& \times \sum_{k=1}^d {\cal I} \{ k\in \{j_{1,1},\ldots, j_{1,|u^{(1)}|} \} \cap \{j_{2,1},\ldots, j_{2,|u^{(2)}|} \}
{\cal I} \{ a_{i_1, k}  = a_{i_2, k} \}
\nonumber \\
&&
\times
\nu^* [\pi_{j_{1,1}} (a_{i_1, j_{1,1}});\ldots; \pi_{j_{1,|u^{(1)}|}} (a_{i_1, j_{1, |u^{(1)}|}} ) ]
\nu^* [\pi_{j_{2,1}} (a_{i_2, j_{2,1}});\ldots; \pi_{j_{2,|u^{(2)}|}} (a_{i_2, j_{2, |u^{(2)}|}} ) ]
\Big\}.
\end{eqnarray*}
Hence
\begin{eqnarray*}
&& E\{ [ E^{\cal W} (S_{\ell_1, 1} S_{\ell_2, 1} ) ]^2\}
\nonumber \\
&=& E\Big\{ \frac{ 1}{d^2 q^{10} \sigma_{\ell_1}^2 \sigma_{\ell_2}^2 } 
\sum_{i_1=1}^{q^2} \sum_{i_2=1}^{q^2} \sum_{i_3=1}^{q^2} \sum_{i_4=1}^{q^2}
\sum_{0\leq u_1^{(1)},\ldots, u_d^{(1)}\leq 1: |u^{(1)}|=\ell_1+2}
\sum_{0\leq u_1^{(2)},\ldots, u_d^{(2)}\leq 1: |u^{(2)}|=\ell_2+2}
\nonumber \\
&& \times \sum_{0\leq u_1^{(3)},\ldots, u_d^{(3)}\leq 1: |u^{(3)}|=\ell_1+2}
\sum_{0\leq u_1^{(4)},\ldots, u_d^{(4)}\leq 1: |u^{(4)}|=\ell_2+2}
\nonumber \\
&&
 \times \sum_{k_1 =1}^d {\cal I} \{ k_1 \in \{j_{1,1},\ldots, j_{1,|u^{(1)}|} \} \cap \{j_{2,1},\ldots, j_{2,|u^{(2)}|} \}
{\cal I} \{ a_{i_1, k_1}  = a_{i_2, k_1} \}
\nonumber \\
&&\times
\sum_{k_3 =1}^d {\cal I} \{ k_3 \in \{j_{3,1},\ldots, j_{3,|u^{(3)}|} \} \cap \{j_{4,1},\ldots, j_{4,|u^{(4)}|} \}
{\cal I} \{ a_{i_3, k_3}  = a_{i_4, k_3} \}
\nonumber \\
&&
\times
\nu^* [\pi_{j_{1,1}} (a_{i_1, j_{1,1}});\ldots; \pi_{j_{1,|u^{(1)}|}} (a_{i_1, j_{1, |u^{(1)}|}} ) ]
\nu^* [\pi_{j_{2,1}} (a_{i_2, j_{2,1}});\ldots; \pi_{j_{2,|u^{(2)}|}} (a_{i_2, j_{2, |u^{(2)}|}} ) ]
\nonumber \\
&&\times
\nu^* [\pi_{j_{3,1}} (a_{i_3, j_{3,1}});\ldots; \pi_{j_{3,|u^{(3)}|}} (a_{i_3, j_{3, |u^{(3)}|}} ) ]
\nu^* [\pi_{j_{4,1}} (a_{i_4, j_{4,1}});\ldots; \pi_{j_{4,|u^{(4)}|}} (a_{i_4, j_{4, |u^{(4)}|}} ) ]
\Big\}
\nonumber \\
&=& R^\dag_{\{1,2,3,4\}} + R^\dag_{\{1,2,3\},\{4\}} + R^\dag_{\{1,2,4\},\{3\}} + R^\dag_{\{1,3,4\},\{2\}}
+ R^\dag_{\{2,3,4\},\{1\}} 
\nonumber \\
&& + R^\dag_{\{1,2\}, \{3\},\{4\}} + R^\dag_{\{1,3\},\{2\},\{4\}} + R^\dag_{\{1,4\},\{2\},\{3\}}
+ R^\dag_{\{2,3\},\{1\},\{4\}} + R^\dag_{\{2,4\},\{1\},\{3\}} 
\nonumber \\
&& + R^\dag_{\{3,4\},\{1\},\{2\}} + R^\dag_{\{1,2\},\{3,4\}} + R^\dag_{\{1,3\},\{2,4\}}
+R^\dag_{\{1,4\},\{2, 3\}} + R^\dag_{\{1\},\{2\},\{3\},\{4\}},
\end{eqnarray*}
where
\begin{eqnarray*}
&& R^\dag_{\{1,2,3,4\}} 
\nonumber \\
&=&E\Big\{ \frac{ 1}{d^2 q^{10} \sigma_{\ell_1}^2 \sigma_{\ell_2}^2 } 
\sum_{i_1=i_2=i_3=i_4=1}^{q^2} 
\sum_{0\leq u_1^{(1)},\ldots, u_d^{(1)}\leq 1: |u^{(1)}|=\ell_1+2}
\sum_{0\leq u_1^{(2)},\ldots, u_d^{(2)}\leq 1: |u^{(2)}|=\ell_2+2}
\nonumber \\
&& \times \sum_{0\leq u_1^{(3)},\ldots, u_d^{(3)}\leq 1: |u^{(3)}|=\ell_1+2}
\sum_{0\leq u_1^{(4)},\ldots, u_d^{(4)}\leq 1: |u^{(4)}|=\ell_2+2}
\nonumber \\
&&
 \times \sum_{k_1 =1}^d {\cal I} \{ k_1 \in \{j_{1,1},\ldots, j_{1,|u^{(1)}|} \} \cap \{j_{2,1},\ldots, j_{2,|u^{(2)}|} \}
\nonumber \\
&&\times
\sum_{k_3 =1}^d {\cal I} \{ k_3 \in \{j_{3,1},\ldots, j_{3,|u^{(3)}|} \} \cap \{j_{4,1},\ldots, j_{4,|u^{(4)}|} \}
\nonumber \\
&&
\times
\nu^* [\pi_{j_{1,1}} (a_{i_1, j_{1,1}});\ldots; \pi_{j_{1,|u^{(1)}|}} (a_{i_1, j_{1, |u^{(1)}|}} ) ]
\nu^* [\pi_{j_{2,1}} (a_{i_2, j_{2,1}});\ldots; \pi_{j_{2,|u^{(2)}|}} (a_{i_2, j_{2, |u^{(2)}|}} ) ]
\nonumber \\
&&\times
\nu^* [\pi_{j_{3,1}} (a_{i_3, j_{3,1}});\ldots; \pi_{j_{3,|u^{(3)}|}} (a_{i_3, j_{3, |u^{(3)}|}} ) ]
\nu^* [\pi_{j_{4,1}} (a_{i_4, j_{4,1}});\ldots; \pi_{j_{4,|u^{(4)}|}} (a_{i_4, j_{4, |u^{(4)}|}} ) ]
\Big\}
\nonumber \\
&=& O(\frac{1}{q^4}),
\nonumber \\
&& R^\dag_{\{1,2,3\},\{4\}} \hspace{0.1cm} = \hspace{0.1cm} R^\dag_{\{1,2,4\},\{3\}} 
\hspace{0.1cm} = \hspace{0.1cm} R^\dag_{\{1,3,4\},\{2\}}
\hspace{0.1cm} = \hspace{0.1cm}  R^\dag_{\{2,3,4\},\{1\}}  
\nonumber \\
&=& E\Big\{ \frac{ 1}{d^2 q^{10} \sigma_{\ell_1}^2 \sigma_{\ell_2}^2 } 
\sum_{i_1=i_2=i_3=1}^{q^2} \sum_{1\leq i_4\leq q^2: i_4\neq i_1}
\sum_{0\leq u_1^{(1)},\ldots, u_d^{(1)}\leq 1: |u^{(1)}|=\ell_1+2}
\sum_{0\leq u_1^{(2)},\ldots, u_d^{(2)}\leq 1: |u^{(2)}|=\ell_2+2}
\nonumber \\
&& \times \sum_{0\leq u_1^{(3)},\ldots, u_d^{(3)}\leq 1: |u^{(3)}|=\ell_1+2}
\sum_{0\leq u_1^{(4)},\ldots, u_d^{(4)}\leq 1: |u^{(4)}|=\ell_2+2}
\nonumber \\
&&
 \times \sum_{k_1 =1}^d {\cal I} \{ k_1 \in \{j_{1,1},\ldots, j_{1,|u^{(1)}|} \} \cap \{j_{2,1},\ldots, j_{2,|u^{(2)}|} \}
\nonumber \\
&&\times
\sum_{k_3 =1}^d {\cal I} \{ k_3 \in \{j_{3,1},\ldots, j_{3,|u^{(3)}|} \} \cap \{j_{4,1},\ldots, j_{4,|u^{(4)}|} \}
{\cal I} \{ a_{i_1, k_3}  = a_{i_4, k_3} \}
\nonumber \\
&&
\times
\nu^* [\pi_{j_{1,1}} (a_{i_1, j_{1,1}});\ldots; \pi_{j_{1,|u^{(1)}|}} (a_{i_1, j_{1, |u^{(1)}|}} ) ]
\nu^* [\pi_{j_{2,1}} (a_{i_2, j_{2,1}});\ldots; \pi_{j_{2,|u^{(2)}|}} (a_{i_2, j_{2, |u^{(2)}|}} ) ]
\nonumber \\
&&\times
\nu^* [\pi_{j_{3,1}} (a_{i_3, j_{3,1}});\ldots; \pi_{j_{3,|u^{(3)}|}} (a_{i_3, j_{3, |u^{(3)}|}} ) ]
\nu^* [\pi_{j_{4,1}} (a_{i_4, j_{4,1}});\ldots; \pi_{j_{4,|u^{(4)}|}} (a_{i_4, j_{4, |u^{(4)}|}} ) ]
\Big\}
\nonumber \\
&=& O(\frac{1}{q^3}),
\nonumber \\
&& R^\dag_{\{1,2\},\{3\},\{4\}} \hspace{0.1cm} = \hspace{0.1cm} R^\dag_{\{3,4\},\{1\},\{2\}}
\nonumber \\
&=& E\Big\{ \frac{ 1}{d^2 q^{10} \sigma_{\ell_1}^2 \sigma_{\ell_2}^2 } 
\sum_{i_1=i_2=1}^{q^2} \sum_{1\leq i_3\leq q^2: i_3\neq i_1} \sum_{1\leq i_4\leq q^2: i_4\neq i_1, i_3}
\sum_{0\leq u_1^{(1)},\ldots, u_d^{(1)}\leq 1: |u^{(1)}|=\ell_1+2}
\nonumber \\
&& \times
\sum_{0\leq u_1^{(2)},\ldots, u_d^{(2)}\leq 1: |u^{(2)}|=\ell_2+2}
\sum_{0\leq u_1^{(3)},\ldots, u_d^{(3)}\leq 1: |u^{(3)}|=\ell_1+2}
\sum_{0\leq u_1^{(4)},\ldots, u_d^{(4)}\leq 1: |u^{(4)}|=\ell_2+2}
\nonumber \\
&&
 \times \sum_{k_1 =1}^d {\cal I} \{ k_1 \in \{j_{1,1},\ldots, j_{1,|u^{(1)}|} \} \cap \{j_{2,1},\ldots, j_{2,|u^{(2)}|} \}
\nonumber \\
&&\times
\sum_{k_3 =1}^d {\cal I} \{ k_3 \in \{j_{3,1},\ldots, j_{3,|u^{(3)}|} \} \cap \{j_{4,1},\ldots, j_{4,|u^{(4)}|} \}
{\cal I} \{ a_{i_3, k_3}  = a_{i_4, k_3} \}
\nonumber \\
&&
\times
\nu^* [\pi_{j_{1,1}} (a_{i_1, j_{1,1}});\ldots; \pi_{j_{1,|u^{(1)}|}} (a_{i_1, j_{1, |u^{(1)}|}} ) ]
\nu^* [\pi_{j_{2,1}} (a_{i_2, j_{2,1}});\ldots; \pi_{j_{2,|u^{(2)}|}} (a_{i_2, j_{2, |u^{(2)}|}} ) ]
\nonumber \\
&&\times
\nu^* [\pi_{j_{3,1}} (a_{i_3, j_{3,1}});\ldots; \pi_{j_{3,|u^{(3)}|}} (a_{i_3, j_{3, |u^{(3)}|}} ) ]
\nu^* [\pi_{j_{4,1}} (a_{i_4, j_{4,1}});\ldots; \pi_{j_{4,|u^{(4)}|}} (a_{i_4, j_{4, |u^{(4)}|}} ) ]
\Big\}
\nonumber \\
&=& O(\frac{1}{q^3}),
\nonumber \\
&& R^\dag_{\{1,3\},\{2\},\{4\}} \hspace{0.1cm}= \hspace{0.1cm}  R^\dag_{\{1,4\},\{2\},\{3\}}
\hspace{0.1cm} = \hspace{0.1cm}  R^\dag_{\{2,3\},\{1\},\{4\}} \hspace{0.1cm} = \hspace{0.1cm} R^\dag_{\{2,4\},\{1\},\{3\}} 
\nonumber \\
&=& E\Big\{ \frac{ 1}{d^2 q^{10} \sigma_{\ell_1}^2 \sigma_{\ell_2}^2 } 
\sum_{i_1=i_3=1}^{q^2} \sum_{1\leq i_2\leq q^2: i_2\neq i_1} \sum_{1\leq i_4\leq q^2: i_4\neq i_1, i_2}
\sum_{0\leq u_1^{(1)},\ldots, u_d^{(1)}\leq 1: |u^{(1)}|=\ell_1+2}
\nonumber \\
&& \times
\sum_{0\leq u_1^{(2)},\ldots, u_d^{(2)}\leq 1: |u^{(2)}|=\ell_2+2}
\sum_{0\leq u_1^{(3)},\ldots, u_d^{(3)}\leq 1: |u^{(3)}|=\ell_1+2}
\sum_{0\leq u_1^{(4)},\ldots, u_d^{(4)}\leq 1: |u^{(4)}|=\ell_2+2}
\nonumber \\
&&
 \times \sum_{k_1 =1}^d {\cal I} \{ k_1 \in \{j_{1,1},\ldots, j_{1,|u^{(1)}|} \} \cap \{j_{2,1},\ldots, j_{2,|u^{(2)}|} \}
{\cal I} \{ a_{i_1, k_1}  = a_{i_2, k_1} \}
\nonumber \\
&&\times
\sum_{k_3 =1}^d {\cal I} \{ k_3 \in \{j_{3,1},\ldots, j_{3,|u^{(3)}|} \} \cap \{j_{4,1},\ldots, j_{4,|u^{(4)}|} \}
{\cal I} \{ a_{i_3, k_3}  = a_{i_4, k_3} \}
\nonumber \\
&&
\times
\nu^* [\pi_{j_{1,1}} (a_{i_1, j_{1,1}});\ldots; \pi_{j_{1,|u^{(1)}|}} (a_{i_1, j_{1, |u^{(1)}|}} ) ]
\nu^* [\pi_{j_{2,1}} (a_{i_2, j_{2,1}});\ldots; \pi_{j_{2,|u^{(2)}|}} (a_{i_2, j_{2, |u^{(2)}|}} ) ]
\nonumber \\
&&\times
\nu^* [\pi_{j_{3,1}} (a_{i_3, j_{3,1}});\ldots; \pi_{j_{3,|u^{(3)}|}} (a_{i_3, j_{3, |u^{(3)}|}} ) ]
\nu^* [\pi_{j_{4,1}} (a_{i_4, j_{4,1}});\ldots; \pi_{j_{4,|u^{(4)}|}} (a_{i_4, j_{4, |u^{(4)}|}} ) ]
\Big\}
\nonumber \\
&=& O(\frac{1}{q^4}),
\nonumber \\
&& R^\dag_{\{1,3\},\{2,4\}} \hspace{0.1cm} = \hspace{0.1cm} R^\dag_{\{1,4\},\{2, 3\}} 
\nonumber \\
&=& E\Big\{ \frac{ 1}{d^2 q^{10} \sigma_{\ell_1}^2 \sigma_{\ell_2}^2 } 
\sum_{i_1=i_3=1}^{q^2} \sum_{1\leq i_2=i_4\leq q^2: i_2\neq i_1}
\sum_{0\leq u_1^{(1)},\ldots, u_d^{(1)}\leq 1: |u^{(1)}|=\ell_1+2}
\nonumber \\
&& \times
\sum_{0\leq u_1^{(2)},\ldots, u_d^{(2)}\leq 1: |u^{(2)}|=\ell_2+2}
\sum_{0\leq u_1^{(3)},\ldots, u_d^{(3)}\leq 1: |u^{(3)}|=\ell_1+2}
\sum_{0\leq u_1^{(4)},\ldots, u_d^{(4)}\leq 1: |u^{(4)}|=\ell_2+2}
\nonumber \\
&&
 \times \sum_{k_1 =1}^d {\cal I} \{ k_1 \in \{j_{1,1},\ldots, j_{1,|u^{(1)}|} \} \cap \{j_{2,1},\ldots, j_{2,|u^{(2)}|} \}
{\cal I} \{ a_{i_1, k_1}  = a_{i_2, k_1} \}
\nonumber \\
&&\times
\sum_{k_3 =1}^d {\cal I} \{ k_3 \in \{j_{3,1},\ldots, j_{3,|u^{(3)}|} \} \cap \{j_{4,1},\ldots, j_{4,|u^{(4)}|} \}
{\cal I} \{ a_{i_1, k_3}  = a_{i_2, k_3} \}
\nonumber \\
&&
\times
\nu^* [\pi_{j_{1,1}} (a_{i_1, j_{1,1}});\ldots; \pi_{j_{1,|u^{(1)}|}} (a_{i_1, j_{1, |u^{(1)}|}} ) ]
\nu^* [\pi_{j_{2,1}} (a_{i_2, j_{2,1}});\ldots; \pi_{j_{2,|u^{(2)}|}} (a_{i_2, j_{2, |u^{(2)}|}} ) ]
\nonumber \\
&&\times
\nu^* [\pi_{j_{3,1}} (a_{i_3, j_{3,1}});\ldots; \pi_{j_{3,|u^{(3)}|}} (a_{i_3, j_{3, |u^{(3)}|}} ) ]
\nu^* [\pi_{j_{4,1}} (a_{i_4, j_{4,1}});\ldots; \pi_{j_{4,|u^{(4)}|}} (a_{i_4, j_{4, |u^{(4)}|}} ) ]
\Big\}
\nonumber \\
&=& O(\frac{1}{q^3}),
\nonumber \\
&& R^\dag_{\{1\},\{2\},\{3\},\{4\}}
\nonumber \\
&=& E\Big\{ \frac{ 1}{d^2 q^{10} \sigma_{\ell_1}^2 \sigma_{\ell_2}^2 } 
\sum_{i_1=1}^{q^2} \sum_{1\leq i_2\leq q^2:i_2\neq i_1} \sum_{1\leq i_3\leq q^2: i_3\neq i_1,i_2} 
\sum_{1\leq i_4\leq q^2: i_4\neq i_1,i_2,i_3}
\sum_{0\leq u_1^{(1)},\ldots, u_d^{(1)}\leq 1: |u^{(1)}|=\ell_1+2}
\nonumber \\
&& \times
\sum_{0\leq u_1^{(2)},\ldots, u_d^{(2)}\leq 1: |u^{(2)}|=\ell_2+2}
\sum_{0\leq u_1^{(3)},\ldots, u_d^{(3)}\leq 1: |u^{(3)}|=\ell_1+2}
\sum_{0\leq u_1^{(4)},\ldots, u_d^{(4)}\leq 1: |u^{(4)}|=\ell_2+2}
\nonumber \\
&&
 \times \sum_{k_1 =1}^d {\cal I} \{ k_1 \in \{j_{1,1},\ldots, j_{1,|u^{(1)}|} \} \cap \{j_{2,1},\ldots, j_{2,|u^{(2)}|} \}
{\cal I} \{ a_{i_1, k_1}  = a_{i_2, k_1} \}
\nonumber \\
&&\times
\sum_{k_3 =1}^d {\cal I} \{ k_3 \in \{j_{3,1},\ldots, j_{3,|u^{(3)}|} \} \cap \{j_{4,1},\ldots, j_{4,|u^{(4)}|} \}
{\cal I} \{ a_{i_3, k_3}  = a_{i_4, k_3} \}
\nonumber \\
&&
\times
\nu^* [\pi_{j_{1,1}} (a_{i_1, j_{1,1}});\ldots; \pi_{j_{1,|u^{(1)}|}} (a_{i_1, j_{1, |u^{(1)}|}} ) ]
\nu^* [\pi_{j_{2,1}} (a_{i_2, j_{2,1}});\ldots; \pi_{j_{2,|u^{(2)}|}} (a_{i_2, j_{2, |u^{(2)}|}} ) ]
\nonumber \\
&&\times
\nu^* [\pi_{j_{3,1}} (a_{i_3, j_{3,1}});\ldots; \pi_{j_{3,|u^{(3)}|}} (a_{i_3, j_{3, |u^{(3)}|}} ) ]
\nu^* [\pi_{j_{4,1}} (a_{i_4, j_{4,1}});\ldots; \pi_{j_{4,|u^{(4)}|}} (a_{i_4, j_{4, |u^{(4)}|}} ) ]
\Big\}
\nonumber \\
&=& O(\frac{1}{q^4}),
\nonumber \\
&& R^\dag_{\{1,2\},\{3,4\}}
\nonumber \\
&=&
E\Big\{ \frac{ 1}{d^2 q^{10} \sigma_{\ell_1}^2 \sigma_{\ell_2}^2 } 
\sum_{i_1=1}^{q^2}  \sum_{1\leq i_3\leq q^2: i_3\neq i_1} 
\sum_{0\leq u_1^{(1)},\ldots, u_d^{(1)}\leq 1: |u^{(1)}|=\ell_1+2}
\nonumber \\
&& \times
\sum_{0\leq u_1^{(2)},\ldots, u_d^{(2)}\leq 1: |u^{(2)}|=\ell_2+2}
\sum_{0\leq u_1^{(3)},\ldots, u_d^{(3)}\leq 1: |u^{(3)}|=\ell_1+2}
\sum_{0\leq u_1^{(4)},\ldots, u_d^{(4)}\leq 1: |u^{(4)}|=\ell_2+2}
\nonumber \\
&&
 \times \sum_{k_1 =1}^d {\cal I} \{ k_1 \in \{j_{1,1},\ldots, j_{1,|u^{(1)}|} \} \cap \{j_{2,1},\ldots, j_{2,|u^{(2)}|} \}
\nonumber \\
&&\times
\sum_{k_3 =1}^d {\cal I} \{ k_3 \in \{j_{3,1},\ldots, j_{3,|u^{(3)}|} \} \cap \{j_{4,1},\ldots, j_{4,|u^{(4)}|} \}
\nonumber \\
&&
\times
\nu^* [\pi_{j_{1,1}} (a_{i_1, j_{1,1}});\ldots; \pi_{j_{1,|u^{(1)}|}} (a_{i_1, j_{1, |u^{(1)}|}} ) ]
\nu^* [\pi_{j_{2,1}} (a_{i_1, j_{2,1}});\ldots; \pi_{j_{2,|u^{(2)}|}} (a_{i_1, j_{2, |u^{(2)}|}} ) ]
\nonumber \\
&&\times
\nu^* [\pi_{j_{3,1}} (a_{i_3, j_{3,1}});\ldots; \pi_{j_{3,|u^{(3)}|}} (a_{i_3, j_{3, |u^{(3)}|}} ) ]
\nu^* [\pi_{j_{4,1}} (a_{i_3, j_{4,1}});\ldots; \pi_{j_{4,|u^{(4)}|}} (a_{i_3, j_{4, |u^{(4)}|}} ) ]
\Big\}
\nonumber \\
&=& O(\frac{1}{q^3}),
\end{eqnarray*}
as $q\rightarrow \infty$. The last equality ues the assumption that $\ell_1\neq \ell_2$.
Thus we conclude that
\begin{eqnarray*}
\frac{d(q-1)}{2(\ell_1+2) } E| E^{\cal W} (S_{\ell_1,1} S_{\ell_2,1} ) | \leq 
\frac{d(q-1)}{2(\ell_1+2)} \{ E[ E^{\cal W} (S_{\ell_1,1} S_{\ell_2,1} ) ]^2\}^{1/2} = O(\frac{1}{q^{1/2}}),
\end{eqnarray*}
as $q\rightarrow\infty$. Next we observe that from (\ref{eq:a.78}) that for $1\leq \ell_1< \ell_2\leq d-2$,
\begin{eqnarray*}
&& E^{\cal W} (\tilde{S}_{\ell_1, 1} \tilde{S}_{\ell_2, 1} )
\nonumber \\
&=& E^{\cal W} \Big\{ \frac{1}{ q^4 \sigma_{\ell_1} \sigma_{\ell_2}} \sum_{i_1=1}^{q^2}
\sum_{i_2=1}^{q^2} \sum_{0\leq u^{(1)}_1,\ldots, u^{(1)}_d\leq 1: |u^{(1)}|=\ell_1 +2}
\sum_{0\leq u^{(2)}_1,\ldots, u^{(2)}_d\leq 1: |u^{(2)}|=\ell_2 +2}
\nonumber \\
&&\times
{\cal I} \{ J\in \{j_{1,1},\ldots, j_{1,|u^{(1)}|} \} \cap \{ j_{2,1},\ldots, j_{2,|u^{(2)}|} \} \}
{\cal I} \{ \pi_J (a_{i_1, J} ) = B_1 = \pi_J (a_{i_2, J}) \}
\nonumber \\
&&\times
\nu^* [\pi_{j_{1,1}} (a_{i_1, j_{1,1}});\ldots; \tau_{B_1,B_2}\circ \pi_J (a_{i_1, J});\ldots;
\pi_{j_{1,|u^{(1)}|}} (a_{i_1, j_{1, |u^{(1)}|}} ) ]
\nonumber \\
&&\times
\nu^* [\pi_{j_{2,1}} (a_{i_2, j_{2,1}});\ldots; \tau_{B_1,B_2}\circ \pi_J (a_{i_2, J}); \ldots;
\pi_{j_{2,|u^{(2)}|}} (a_{i_2, j_{2, |u^{(2)}|}} ) ]
\Big\}
\nonumber \\
&=&
\frac{1}{ d q^5 (q-1) \sigma_{\ell_1} \sigma_{\ell_2}} \sum_{i_1=1}^{q^2}
\sum_{i_2=1}^{q^2} \sum_{0\leq u^{(1)}_1,\ldots, u^{(1)}_d\leq 1: |u^{(1)}|=\ell_1 +2}
\sum_{0\leq u^{(2)}_1,\ldots, u^{(2)}_d\leq 1: |u^{(2)}|=\ell_2 +2}
\nonumber \\
&&\times
\sum_{k=1}^d {\cal I} \{ k\in \{j_{1,1},\ldots, j_{1,|u^{(1)}|} \} \cap \{ j_{2,1},\ldots, j_{2,|u^{(2)}|} \} \}
{\cal I} \{ a_{i_1, k} = a_{i_2, k} \}
\nonumber \\
&&\times
\sum_{0\leq \tilde{c}_k\leq q-1: \tilde{c}_k\neq \pi_k (a_{i_1, k})}
\nu^* [\pi_{j_{1,1}} (a_{i_1, j_{1,1}});\ldots; \tilde{c}_k;\ldots;
\pi_{j_{1,|u^{(1)}|}} (a_{i_1, j_{1, |u^{(1)}|}} ) ]
\nonumber \\
&&\times
\nu^* [\pi_{j_{2,1}} (a_{i_2, j_{2,1}});\ldots; \tilde{c}_k; \ldots;
\pi_{j_{2,|u^{(2)}|}} (a_{i_2, j_{2, |u^{(2)}|}} ) ],
\end{eqnarray*}
and
\begin{eqnarray*}
&& E \{ [ E^{\cal W}  (\tilde{S}_{\ell_1, 1} \tilde{S}_{\ell_2, 1} ) ]^2 \}
\nonumber \\
&=& E\Big\{ \frac{1}{ d^2 q^{10} (q-1)^2 \sigma_{\ell_1}^2 \sigma_{\ell_2}^2 } \sum_{i_1=1}^{q^2}
\sum_{i_2=1}^{q^2} \sum_{i_3=1}^{q^2} \sum_{i_4=1}^{q^2} 
\sum_{0\leq u^{(1)}_1,\ldots, u^{(1)}_d\leq 1: |u^{(1)}|=\ell_1 +2}
\nonumber \\
&&\times
\sum_{0\leq u^{(2)}_1,\ldots, u^{(2)}_d\leq 1: |u^{(2)}|=\ell_2 +2}
\sum_{0\leq u^{(3)}_1,\ldots, u^{(3)}_d\leq 1: |u^{(3)}|=\ell_1 +2}
\sum_{0\leq u^{(4)}_1,\ldots, u^{(4)}_d\leq 1: |u^{(4)}|=\ell_2 +2}
\nonumber \\
&&\times
\sum_{k_1 =1}^d {\cal I} \{ k_1 \in \{j_{1,1},\ldots, j_{1,|u^{(1)}|} \} \cap \{ j_{2,1},\ldots, j_{2,|u^{(2)}|} \} \}
{\cal I} \{ a_{i_1, k_1} = a_{i_2, k_1} \}
\nonumber \\
&&\times
\sum_{k_3=1}^d {\cal I} \{ k_3 \in \{j_{3,1},\ldots, j_{3,|u^{(3)}|} \} \cap \{ j_{4,1},\ldots, j_{4,|u^{(4)}|} \} \}
{\cal I} \{ a_{i_3, k_3 } = a_{i_4, k_3 } \}
\nonumber \\
&&\times
\sum_{0\leq \tilde{c}_{k_1}\leq q-1: \tilde{c}_{k_1} \neq \pi_{k_1} (a_{i_1, k_1})}
\sum_{0\leq \tilde{c}_{k_3}\leq q-1: \tilde{c}_{k_3} \neq \pi_{k_3} (a_{i_3, k_3})}
\nonumber \\
&&\times
\nu^* [\pi_{j_{1,1}} (a_{i_1, j_{1,1}});\ldots; \tilde{c}_{k_1};\ldots;
\pi_{j_{1,|u^{(1)}|}} (a_{i_1, j_{1, |u^{(1)}|}} ) ]
\nonumber \\
&&\times
\nu^* [\pi_{j_{2,1}} (a_{i_2, j_{2,1}});\ldots; \tilde{c}_{k_1}; \ldots;
\pi_{j_{2,|u^{(2)}|}} (a_{i_2, j_{2, |u^{(2)}|}} ) ] 
\nonumber \\
&&\times
\nu^* [\pi_{j_{3,1}} (a_{i_3, j_{3,1}});\ldots; \tilde{c}_{k_3};\ldots;
\pi_{j_{3,|u^{(3)}|}} (a_{i_3, j_{3, |u^{(3)}|}} ) ]
\nonumber \\
&&\times
\nu^* [\pi_{j_{4,1}} (a_{i_4, j_{4,1}});\ldots; \tilde{c}_{k_3}; \ldots;
\pi_{j_{4,|u^{(4)}|}} (a_{i_4, j_{4, |u^{(4)}|}} ) ] \Big\}
\nonumber \\
&=& R^\ddag_{\{1,2,3,4\}} + R^\ddag_{\{1,2,3\},\{4\}} + R^\ddag_{\{1,2,4\},\{3\}} + R^\ddag_{\{1,3,4\},\{2\}}
+ R^\ddag_{\{2,3,4\},\{1\}} 
\nonumber \\
&& + R^\ddag_{\{1,2\}, \{3\},\{4\}} + R^\ddag_{\{1,3\},\{2\},\{4\}} + R^\ddag_{\{1,4\},\{2\},\{3\}}
+ R^\dag_{\{2,3\},\{1\},\{4\}} + R^\dag_{\{2,4\},\{1\},\{3\}} 
\nonumber \\
&& + R^\ddag_{\{3,4\},\{1\},\{2\}} + R^\ddag_{\{1,2\},\{3,4\}} + R^\ddag_{\{1,3\},\{2,4\}}
+R^\ddag_{\{1,4\},\{2, 3\}} + R^\ddag_{\{1\},\{2\},\{3\},\{4\}},
\end{eqnarray*}
where
\begin{eqnarray*}
&& R^\ddag_{\{1,2,3,4\}} \nonumber \\
&=& E\Big\{ \frac{1}{ d^2 q^{10} (q-1)^2 \sigma_{\ell_1}^2 \sigma_{\ell_2}^2 } 
\sum_{i_=i_2=i_3=i_4=1}^{q^2} 
\sum_{0\leq u^{(1)}_1,\ldots, u^{(1)}_d\leq 1: |u^{(1)}|=\ell_1 +2}
\nonumber \\
&&\times
\sum_{0\leq u^{(2)}_1,\ldots, u^{(2)}_d\leq 1: |u^{(2)}|=\ell_2 +2}
\sum_{0\leq u^{(3)}_1,\ldots, u^{(3)}_d\leq 1: |u^{(3)}|=\ell_1 +2}
\sum_{0\leq u^{(4)}_1,\ldots, u^{(4)}_d\leq 1: |u^{(4)}|=\ell_2 +2}
\nonumber \\
&&\times
\sum_{k_1 =1}^d {\cal I} \{ k_1 \in \{j_{1,1},\ldots, j_{1,|u^{(1)}|} \} \cap \{ j_{2,1},\ldots, j_{2,|u^{(2)}|} \} \}
\nonumber \\
&&\times
\sum_{k_3=1}^d {\cal I} \{ k_3 \in \{j_{3,1},\ldots, j_{3,|u^{(3)}|} \} \cap \{ j_{4,1},\ldots, j_{4,|u^{(4)}|} \} \}
\nonumber \\
&&\times
\sum_{0\leq \tilde{c}_{k_1}\leq q-1: \tilde{c}_{k_1} \neq \pi_{k_1} (a_{i_1, k_1})}
\sum_{0\leq \tilde{c}_{k_3}\leq q-1: \tilde{c}_{k_3} \neq \pi_{k_3} (a_{i_3, k_3})}
\nonumber \\
&&\times
\nu^* [\pi_{j_{1,1}} (a_{i_1, j_{1,1}});\ldots; \tilde{c}_{k_1};\ldots;
\pi_{j_{1,|u^{(1)}|}} (a_{i_1, j_{1, |u^{(1)}|}} ) ]
\nonumber \\
&&\times
\nu^* [\pi_{j_{2,1}} (a_{i_2, j_{2,1}});\ldots; \tilde{c}_{k_1}; \ldots;
\pi_{j_{2,|u^{(2)}|}} (a_{i_2, j_{2, |u^{(2)}|}} ) ] 
\nonumber \\
&&\times
\nu^* [\pi_{j_{3,1}} (a_{i_3, j_{3,1}});\ldots; \tilde{c}_{k_3};\ldots;
\pi_{j_{3,|u^{(3)}|}} (a_{i_3, j_{3, |u^{(3)}|}} ) ]
\nonumber \\
&&\times
\nu^* [\pi_{j_{4,1}} (a_{i_4, j_{4,1}});\ldots; \tilde{c}_{k_3}; \ldots;
\pi_{j_{4,|u^{(4)}|}} (a_{i_4, j_{4, |u^{(4)}|}} ) ] \Big\}
\nonumber \\
&=& O(\frac{1}{q^4}),
\nonumber \\
&& R^\ddag_{\{1,2,3\},\{4\}} \hspace{0.1cm}= \hspace{0.1cm} R^\ddag_{\{1,2,4\},\{3\}} 
\hspace{0.1cm} = \hspace{0.1cm} R^\ddag_{\{1,3,4\},\{2\}}
\hspace{0.1cm} = \hspace{0.1cm} R^\ddag_{\{2,3,4\},\{1\}} 
\nonumber \\
&=& E\Big\{ \frac{1}{ d^2 q^{10} (q-1)^2 \sigma_{\ell_1}^2 \sigma_{\ell_2}^2 } \sum_{i_1=i_2=i_3=1}^{q^2}
 \sum_{1\leq i_4\leq q^2:i_4\neq i_1} 
\sum_{0\leq u^{(1)}_1,\ldots, u^{(1)}_d\leq 1: |u^{(1)}|=\ell_1 +2}
\nonumber \\
&&\times
\sum_{0\leq u^{(2)}_1,\ldots, u^{(2)}_d\leq 1: |u^{(2)}|=\ell_2 +2}
\sum_{0\leq u^{(3)}_1,\ldots, u^{(3)}_d\leq 1: |u^{(3)}|=\ell_1 +2}
\sum_{0\leq u^{(4)}_1,\ldots, u^{(4)}_d\leq 1: |u^{(4)}|=\ell_2 +2}
\nonumber \\
&&\times
\sum_{k_1 =1}^d {\cal I} \{ k_1 \in \{j_{1,1},\ldots, j_{1,|u^{(1)}|} \} \cap \{ j_{2,1},\ldots, j_{2,|u^{(2)}|} \} \}
\nonumber \\
&&\times
\sum_{k_3=1}^d {\cal I} \{ k_3 \in \{j_{3,1},\ldots, j_{3,|u^{(3)}|} \} \cap \{ j_{4,1},\ldots, j_{4,|u^{(4)}|} \} \}
{\cal I} \{ a_{i_3, k_3 } = a_{i_4, k_3 } \}
\nonumber \\
&&\times
\sum_{0\leq \tilde{c}_{k_1}\leq q-1: \tilde{c}_{k_1} \neq \pi_{k_1} (a_{i_1, k_1})}
\sum_{0\leq \tilde{c}_{k_3}\leq q-1: \tilde{c}_{k_3} \neq \pi_{k_3} (a_{i_3, k_3})}
\nonumber \\
&&\times
\nu^* [\pi_{j_{1,1}} (a_{i_1, j_{1,1}});\ldots; \tilde{c}_{k_1};\ldots;
\pi_{j_{1,|u^{(1)}|}} (a_{i_1, j_{1, |u^{(1)}|}} ) ]
\nonumber \\
&&\times
\nu^* [\pi_{j_{2,1}} (a_{i_2, j_{2,1}});\ldots; \tilde{c}_{k_1}; \ldots;
\pi_{j_{2,|u^{(2)}|}} (a_{i_2, j_{2, |u^{(2)}|}} ) ] 
\nonumber \\
&&\times
\nu^* [\pi_{j_{3,1}} (a_{i_3, j_{3,1}});\ldots; \tilde{c}_{k_3};\ldots;
\pi_{j_{3,|u^{(3)}|}} (a_{i_3, j_{3, |u^{(3)}|}} ) ]
\nonumber \\
&&\times
\nu^* [\pi_{j_{4,1}} (a_{i_4, j_{4,1}});\ldots; \tilde{c}_{k_3}; \ldots;
\pi_{j_{4,|u^{(4)}|}} (a_{i_4, j_{4, |u^{(4)}|}} ) ] \Big\}
\nonumber \\
&=& O(\frac{1}{q^2}),
\nonumber \\
&& R^\ddag_{\{1,2\}, \{3\},\{4\}} \hspace{0.1cm} = \hspace{0.1cm} R^\ddag_{\{3,4\}, \{1\},\{2\}}
\nonumber \\
&=& E\Big\{ \frac{1}{ d^2 q^{10} (q-1)^2 \sigma_{\ell_1}^2 \sigma_{\ell_2}^2 } \sum_{i_1=i_2=1}^{q^2}
\sum_{1\leq i_3\leq q^2:i_3\neq i_1} \sum_{1\leq i_4\leq q^2: i_4\neq i_1, i_3} 
\sum_{0\leq u^{(1)}_1,\ldots, u^{(1)}_d\leq 1: |u^{(1)}|=\ell_1 +2}
\nonumber \\
&&\times
\sum_{0\leq u^{(2)}_1,\ldots, u^{(2)}_d\leq 1: |u^{(2)}|=\ell_2 +2}
\sum_{0\leq u^{(3)}_1,\ldots, u^{(3)}_d\leq 1: |u^{(3)}|=\ell_1 +2}
\sum_{0\leq u^{(4)}_1,\ldots, u^{(4)}_d\leq 1: |u^{(4)}|=\ell_2 +2}
\nonumber \\
&&\times
\sum_{k_1 =1}^d {\cal I} \{ k_1 \in \{j_{1,1},\ldots, j_{1,|u^{(1)}|} \} \cap \{ j_{2,1},\ldots, j_{2,|u^{(2)}|} \} \}
\nonumber \\
&&\times
\sum_{k_3=1}^d {\cal I} \{ k_3 \in \{j_{3,1},\ldots, j_{3,|u^{(3)}|} \} \cap \{ j_{4,1},\ldots, j_{4,|u^{(4)}|} \} \}
{\cal I} \{ a_{i_3, k_3 } = a_{i_4, k_3 } \}
\nonumber \\
&&\times
\sum_{0\leq \tilde{c}_{k_1}\leq q-1: \tilde{c}_{k_1} \neq \pi_{k_1} (a_{i_1, k_1})}
\sum_{0\leq \tilde{c}_{k_3}\leq q-1: \tilde{c}_{k_3} \neq \pi_{k_3} (a_{i_3, k_3})}
\nonumber \\
&&\times
\nu^* [\pi_{j_{1,1}} (a_{i_1, j_{1,1}});\ldots; \tilde{c}_{k_1};\ldots;
\pi_{j_{1,|u^{(1)}|}} (a_{i_1, j_{1, |u^{(1)}|}} ) ]
\nonumber \\
&&\times
\nu^* [\pi_{j_{2,1}} (a_{i_2, j_{2,1}});\ldots; \tilde{c}_{k_1}; \ldots;
\pi_{j_{2,|u^{(2)}|}} (a_{i_2, j_{2, |u^{(2)}|}} ) ] 
\nonumber \\
&&\times
\nu^* [\pi_{j_{3,1}} (a_{i_3, j_{3,1}});\ldots; \tilde{c}_{k_3};\ldots;
\pi_{j_{3,|u^{(3)}|}} (a_{i_3, j_{3, |u^{(3)}|}} ) ]
\nonumber \\
&&\times
\nu^* [\pi_{j_{4,1}} (a_{i_4, j_{4,1}});\ldots; \tilde{c}_{k_3}; \ldots;
\pi_{j_{4,|u^{(4)}|}} (a_{i_4, j_{4, |u^{(4)}|}} ) ] \Big\}
\nonumber \\
&=& O(\frac{1}{q^3}),
\nonumber \\
&& R^\ddag_{\{1,3\},\{2\},\{4\}} \hspace{0.1cm}=\hspace{0.1cm} R^\ddag_{\{1,4\},\{2\},\{3\}}
\hspace{0.1cm}=\hspace{0.1cm} R^\dag_{\{2,3\},\{1\},\{4\}} \hspace{0.1cm}=\hspace{0.1cm} R^\dag_{\{2,4\},\{1\},\{3\}} 
\nonumber \\
&=& E\Big\{ \frac{1}{ d^2 q^{10} (q-1)^2 \sigma_{\ell_1}^2 \sigma_{\ell_2}^2 } \sum_{i_1=i_3=1}^{q^2}
\sum_{1\leq i_2\leq q^2:i_2\neq i_1} \sum_{1\leq i_4\leq q^2: i_4\neq i_1, i_2} 
\sum_{0\leq u^{(1)}_1,\ldots, u^{(1)}_d\leq 1: |u^{(1)}|=\ell_1 +2}
\nonumber \\
&&\times
\sum_{0\leq u^{(2)}_1,\ldots, u^{(2)}_d\leq 1: |u^{(2)}|=\ell_2 +2}
\sum_{0\leq u^{(3)}_1,\ldots, u^{(3)}_d\leq 1: |u^{(3)}|=\ell_1 +2}
\sum_{0\leq u^{(4)}_1,\ldots, u^{(4)}_d\leq 1: |u^{(4)}|=\ell_2 +2}
\nonumber \\
&&\times
\sum_{k_1 =1}^d {\cal I} \{ k_1 \in \{j_{1,1},\ldots, j_{1,|u^{(1)}|} \} \cap \{ j_{2,1},\ldots, j_{2,|u^{(2)}|} \} \}
{\cal I} \{ a_{i_1, k_1 } = a_{i_2, k_1 } \}
\nonumber \\
&&\times
\sum_{k_3=1}^d {\cal I} \{ k_3 \in \{j_{3,1},\ldots, j_{3,|u^{(3)}|} \} \cap \{ j_{4,1},\ldots, j_{4,|u^{(4)}|} \} \}
{\cal I} \{ a_{i_3, k_3 } = a_{i_4, k_3 } \}
\nonumber \\
&&\times
\sum_{0\leq \tilde{c}_{k_1}\leq q-1: \tilde{c}_{k_1} \neq \pi_{k_1} (a_{i_1, k_1})}
\sum_{0\leq \tilde{c}_{k_3}\leq q-1: \tilde{c}_{k_3} \neq \pi_{k_3} (a_{i_3, k_3})}
\nonumber \\
&&\times
\nu^* [\pi_{j_{1,1}} (a_{i_1, j_{1,1}});\ldots; \tilde{c}_{k_1};\ldots;
\pi_{j_{1,|u^{(1)}|}} (a_{i_1, j_{1, |u^{(1)}|}} ) ]
\nonumber \\
&&\times
\nu^* [\pi_{j_{2,1}} (a_{i_2, j_{2,1}});\ldots; \tilde{c}_{k_1}; \ldots;
\pi_{j_{2,|u^{(2)}|}} (a_{i_2, j_{2, |u^{(2)}|}} ) ] 
\nonumber \\
&&\times
\nu^* [\pi_{j_{3,1}} (a_{i_3, j_{3,1}});\ldots; \tilde{c}_{k_3};\ldots;
\pi_{j_{3,|u^{(3)}|}} (a_{i_3, j_{3, |u^{(3)}|}} ) ]
\nonumber \\
&&\times
\nu^* [\pi_{j_{4,1}} (a_{i_4, j_{4,1}});\ldots; \tilde{c}_{k_3}; \ldots;
\pi_{j_{4,|u^{(4)}|}} (a_{i_4, j_{4, |u^{(4)}|}} ) ] \Big\}
\nonumber \\
&=& O(\frac{1}{q^4}),
\nonumber \\
&& R^\ddag_{\{1,3\},\{2,4\}}
\hspace{0.1cm} = \hspace{0.1cm} R^\ddag_{\{1,4\},\{2, 3\}} 
\nonumber \\
&=& E\Big\{ \frac{1}{ d^2 q^{10} (q-1)^2 \sigma_{\ell_1}^2 \sigma_{\ell_2}^2 } \sum_{i_1=i_3=1}^{q^2}
\sum_{1\leq i_2=i_4\leq q^2:i_2\neq i_3} 
\sum_{0\leq u^{(1)}_1,\ldots, u^{(1)}_d\leq 1: |u^{(1)}|=\ell_1 +2}
\nonumber \\
&&\times
\sum_{0\leq u^{(2)}_1,\ldots, u^{(2)}_d\leq 1: |u^{(2)}|=\ell_2 +2}
\sum_{0\leq u^{(3)}_1,\ldots, u^{(3)}_d\leq 1: |u^{(3)}|=\ell_1 +2}
\sum_{0\leq u^{(4)}_1,\ldots, u^{(4)}_d\leq 1: |u^{(4)}|=\ell_2 +2}
\nonumber \\
&&\times
\sum_{k_1 =1}^d {\cal I} \{ k_1 \in \{j_{1,1},\ldots, j_{1,|u^{(1)}|} \} \cap \{ j_{2,1},\ldots, j_{2,|u^{(2)}|} \} \}
{\cal I} \{ a_{i_1, k_1} = a_{i_2, k_1} \}
\nonumber \\
&&\times
\sum_{k_3=1}^d {\cal I} \{ k_3 \in \{j_{3,1},\ldots, j_{3,|u^{(3)}|} \} \cap \{ j_{4,1},\ldots, j_{4,|u^{(4)}|} \} \}
{\cal I} \{ a_{i_1, k_3 } = a_{i_2, k_3 } \}
\nonumber \\
&&\times
\sum_{0\leq \tilde{c}_{k_1}\leq q-1: \tilde{c}_{k_1} \neq \pi_{k_1} (a_{i_1, k_1})}
\sum_{0\leq \tilde{c}_{k_3}\leq q-1: \tilde{c}_{k_3} \neq \pi_{k_3} (a_{i_3, k_3})}
\nonumber \\
&&\times
\nu^* [\pi_{j_{1,1}} (a_{i_1, j_{1,1}});\ldots; \tilde{c}_{k_1};\ldots;
\pi_{j_{1,|u^{(1)}|}} (a_{i_1, j_{1, |u^{(1)}|}} ) ]
\nonumber \\
&&\times
\nu^* [\pi_{j_{2,1}} (a_{i_2, j_{2,1}});\ldots; \tilde{c}_{k_1}; \ldots;
\pi_{j_{2,|u^{(2)}|}} (a_{i_2, j_{2, |u^{(2)}|}} ) ] 
\nonumber \\
&&\times
\nu^* [\pi_{j_{3,1}} (a_{i_3, j_{3,1}});\ldots; \tilde{c}_{k_3};\ldots;
\pi_{j_{3,|u^{(3)}|}} (a_{i_3, j_{3, |u^{(3)}|}} ) ]
\nonumber \\
&&\times
\nu^* [\pi_{j_{4,1}} (a_{i_4, j_{4,1}});\ldots; \tilde{c}_{k_3}; \ldots;
\pi_{j_{4,|u^{(4)}|}} (a_{i_4, j_{4, |u^{(4)}|}} ) ] \Big\}
\nonumber \\
&=& O(\frac{1}{q^3}),
\nonumber \\
&& R^\ddag_{\{1\},\{2\},\{3\},\{4\}}
\nonumber \\
&=& E\Big\{ \frac{1}{ d^2 q^{10} (q-1)^2 \sigma_{\ell_1}^2 \sigma_{\ell_2}^2 } \sum_{i_1=1}^{q^2}
\sum_{1\leq i_2\leq q^2: i_2\neq i_1} \sum_{1\leq i_3\leq q^2: i_3\neq i_1, i_2} 
\sum_{1\leq i_4\leq q^2: i_4\neq i_1,i_2, i_3} 
\nonumber \\
&&\times
\sum_{0\leq u^{(1)}_1,\ldots, u^{(1)}_d\leq 1: |u^{(1)}|=\ell_1 +2}
\sum_{0\leq u^{(2)}_1,\ldots, u^{(2)}_d\leq 1: |u^{(2)}|=\ell_2 +2}
\nonumber \\
&&\times
\sum_{0\leq u^{(3)}_1,\ldots, u^{(3)}_d\leq 1: |u^{(3)}|=\ell_1 +2}
\sum_{0\leq u^{(4)}_1,\ldots, u^{(4)}_d\leq 1: |u^{(4)}|=\ell_2 +2}
\nonumber \\
&&\times
\sum_{k_1 =1}^d {\cal I} \{ k_1 \in \{j_{1,1},\ldots, j_{1,|u^{(1)}|} \} \cap \{ j_{2,1},\ldots, j_{2,|u^{(2)}|} \} \}
{\cal I} \{ a_{i_1, k_1} = a_{i_2, k_1} \}
\nonumber \\
&&\times
\sum_{k_3=1}^d {\cal I} \{ k_3 \in \{j_{3,1},\ldots, j_{3,|u^{(3)}|} \} \cap \{ j_{4,1},\ldots, j_{4,|u^{(4)}|} \} \}
{\cal I} \{ a_{i_3, k_3 } = a_{i_4, k_3 } \}
\nonumber \\
&&\times
\sum_{0\leq \tilde{c}_{k_1}\leq q-1: \tilde{c}_{k_1} \neq \pi_{k_1} (a_{i_1, k_1})}
\sum_{0\leq \tilde{c}_{k_3}\leq q-1: \tilde{c}_{k_3} \neq \pi_{k_3} (a_{i_3, k_3})}
\nonumber \\
&&\times
\nu^* [\pi_{j_{1,1}} (a_{i_1, j_{1,1}});\ldots; \tilde{c}_{k_1};\ldots;
\pi_{j_{1,|u^{(1)}|}} (a_{i_1, j_{1, |u^{(1)}|}} ) ]
\nonumber \\
&&\times
\nu^* [\pi_{j_{2,1}} (a_{i_2, j_{2,1}});\ldots; \tilde{c}_{k_1}; \ldots;
\pi_{j_{2,|u^{(2)}|}} (a_{i_2, j_{2, |u^{(2)}|}} ) ] 
\nonumber \\
&&\times
\nu^* [\pi_{j_{3,1}} (a_{i_3, j_{3,1}});\ldots; \tilde{c}_{k_3};\ldots;
\pi_{j_{3,|u^{(3)}|}} (a_{i_3, j_{3, |u^{(3)}|}} ) ]
\nonumber \\
&&\times
\nu^* [\pi_{j_{4,1}} (a_{i_4, j_{4,1}});\ldots; \tilde{c}_{k_3}; \ldots;
\pi_{j_{4,|u^{(4)}|}} (a_{i_4, j_{4, |u^{(4)}|}} ) ] \Big\}
\nonumber \\
&=& O(\frac{1}{q^4}),
\nonumber \\
&& R^\ddag_{\{1,2\},\{3,4\}} 
\nonumber \\
&=& E\Big\{ \frac{1}{ d^2 q^{10} (q-1)^2 \sigma_{\ell_1}^2 \sigma_{\ell_2}^2 } \sum_{i_1=i_2=1}^{q^2}
\sum_{1\leq i_3=i_4\leq q^2:i_3\neq i_2}
\sum_{0\leq u^{(1)}_1,\ldots, u^{(1)}_d\leq 1: |u^{(1)}|=\ell_1 +2}
\nonumber \\
&&\times
\sum_{0\leq u^{(2)}_1,\ldots, u^{(2)}_d\leq 1: |u^{(2)}|=\ell_2 +2}
\sum_{0\leq u^{(3)}_1,\ldots, u^{(3)}_d\leq 1: |u^{(3)}|=\ell_1 +2}
\sum_{0\leq u^{(4)}_1,\ldots, u^{(4)}_d\leq 1: |u^{(4)}|=\ell_2 +2}
\nonumber \\
&&\times
\sum_{k_1 =1}^d {\cal I} \{ k_1 \in \{j_{1,1},\ldots, j_{1,|u^{(1)}|} \} \cap \{ j_{2,1},\ldots, j_{2,|u^{(2)}|} \} \}
\nonumber \\
&&\times
\sum_{k_3=1}^d {\cal I} \{ k_3 \in \{j_{3,1},\ldots, j_{3,|u^{(3)}|} \} \cap \{ j_{4,1},\ldots, j_{4,|u^{(4)}|} \} \}
\nonumber \\
&&\times
\sum_{0\leq \tilde{c}_{k_1}\leq q-1: \tilde{c}_{k_1} \neq \pi_{k_1} (a_{i_1, k_1})}
\sum_{0\leq \tilde{c}_{k_3}\leq q-1: \tilde{c}_{k_3} \neq \pi_{k_3} (a_{i_3, k_3})}
\nonumber \\
&&\times
\nu^* [\pi_{j_{1,1}} (a_{i_1, j_{1,1}});\ldots; \tilde{c}_{k_1};\ldots;
\pi_{j_{1,|u^{(1)}|}} (a_{i_1, j_{1, |u^{(1)}|}} ) ]
\nonumber \\
&&\times
\nu^* [\pi_{j_{2,1}} (a_{i_2, j_{2,1}});\ldots; \tilde{c}_{k_1}; \ldots;
\pi_{j_{2,|u^{(2)}|}} (a_{i_2, j_{2, |u^{(2)}|}} ) ] 
\nonumber \\
&&\times
\nu^* [\pi_{j_{3,1}} (a_{i_3, j_{3,1}});\ldots; \tilde{c}_{k_3};\ldots;
\pi_{j_{3,|u^{(3)}|}} (a_{i_3, j_{3, |u^{(3)}|}} ) ]
\nonumber \\
&&\times
\nu^* [\pi_{j_{4,1}} (a_{i_4, j_{4,1}});\ldots; \tilde{c}_{k_3}; \ldots;
\pi_{j_{4,|u^{(4)}|}} (a_{i_4, j_{4, |u^{(4)}|}} ) ] \Big\}
\nonumber \\
&=& O(\frac{1}{q^3}),
\end{eqnarray*}
as $q\rightarrow\infty$. The last equality uses the assumption that $\ell_1<\ell_2$. 
Thus we conclude that
\begin{eqnarray*}
\frac{d(q-1)}{2(\ell_1+2) } E| E^{\cal W} ( \tilde{S}_{\ell_1,1} \tilde{S}_{\ell_2,1} ) | \leq 
\frac{d(q-1)}{2(\ell_1+2)} \{ E[ E^{\cal W} ( \tilde{S}_{\ell_1,1} \tilde{S}_{\ell_2,1} ) ]^2\}^{1/2} = O(\frac{1}{q^{1/2}}),
\end{eqnarray*}
as $q\rightarrow\infty$. This proves Lemma \ref{la:a.24}. \hfill $\Box$

{\sc Proof of Theorem \ref{tm:5.2}.}
Let $Z_1$ be a random variable having the standard (univariate) normal distribution and
$\xi = (\sigma_1/\sigma, \ldots, \sigma_{d-2}/\sigma)'$. Then $\|\xi \|^2= 1$ and 
$W = \xi' V$. For ease of exposition in the subsequent argument, we shall write $\xi= \xi(q)$ and $V=V(q)$.

We claim that $\xi' V\rightarrow Z_1$ in distribution as $q\rightarrow \infty$. We shall prove this
claim by contraposition. Suppose the claim is false. Then there exists an interval, say $[a, b) \subset {\mathbb R}$, 
such that $P(\xi'V \in [a, b))$ does not converge to $P( Z_1 \in [a, b))$ as $q\rightarrow \infty$.
Since $P(\xi' V \in [a, b)) \in [0, 1]$, by the compactness of $[0,1]$, there exists a subsequence, say $\xi'(q_k) V(q_k)$, of $\xi' V$
such that $P( \xi'(q_k) V(q_k) \in [a, b))$ converges to a number, say $L \neq P(Z_1\in [a, b))$. 
As $\| \xi(q_k)\| =1$,  there exists a further subsequence,
say $\xi' (q_{k_l}) V (q_{k_l})$, of $\xi'(q_k) V(q_k)$  such that $\xi(q_{k_l})$ 
converges to a point $\tilde{\xi}\in {\mathbb R}^{d-2}$ as $q_{k_l}\rightarrow \infty$.
This implies that $\xi'(q_{k_l}) V(q_{k_l}) - \tilde{\xi}' V (q_{k_l})\rightarrow 0$ in probability as $q_{k_l}\rightarrow \infty$
and hence $\xi'(q_{k_l}) V(q_{k_l})$ and $\tilde{\xi}' V(q_{k_l})$ have the same asymptotic distribution. 
Using Theorem \ref{tm:5.1} and $\|\tilde{\xi}\|^2 =1$, we observe that $\tilde{\xi}'V(q_{k_l})$ converges in law to the standard normal distribution
as $q_{k_l}\rightarrow \infty$.
Hence $\xi'(q_{k_l}) V(q_{k_l})$ converges in law to the same latter distribution. This is a contradiction and the claim is proved.

We observe from Theorem \ref{tm:3.1} and Proposition \ref{pn:3.1} that 
for $\int_{[0,1)^d} f^2_{rem} (x) dx >0$, we have
$\sigma^2_{oal}/\sigma^2 = 1 + O(q^{-1})$ and
$\sigma^2_{oas}/\sigma^2 = 1 + O(q^{-1})$ as $q\rightarrow \infty$.
Thus we conclude from Proposition \ref{pn:4.1} and Slutsky's theorem that
$W_{oal}$ and $W_{oas}$ both tend in law  to the standard (univariate) normal distribution
as $q\rightarrow \infty$. Finally using Theorem \ref{tm:3.1} and Proposition \ref{pn:3.1}, we have
$\lim_{q\rightarrow \infty} \sigma_{oal}^*/\sigma_{oal} = 1$. Hence for $[a, b) \subset {\mathbb R}$,
\begin{displaymath}
P( W^*_{oal} \in [a, b) ) = \frac{1}{d!} {\sum}^* P( W_{oal} \in [a, b)) +o(1)
\rightarrow P(Z_1\in [a, b)),
\end{displaymath}
as $q\rightarrow \infty$ where $\sum^*$ denotes summation over all the $d!$ permutations of the columns of $A^{**}$.
This proves that $W_{oal}^*$ converges in law to the standard normal distribution. The proof of Theorem \ref{tm:5.2} is complete.
\hfill $\Box$

\end{document}